\documentclass{article}
\usepackage[utf8]{inputenc}
\usepackage[a4paper, total={7in, 10in}]{geometry}
\usepackage{amsmath}
\usepackage{graphicx}
\usepackage{caption}
\usepackage{subcaption}
\usepackage{amsfonts}
\usepackage{amssymb}
\usepackage{amsthm}
\usepackage{mathtools}
\usepackage[noadjust]{cite}
\usepackage{url}
\newtheorem{theorem}{Theorem}[section]

\newtheorem{definition}{Definition}[subsection]
\newtheorem{remark}{Remark}[subsection]
\numberwithin{equation}{section}
\makeatletter
\def\maketag@@@#1{\hbox{\m@th\normalfont\normalsize#1}}
\makeatother
\DeclareMathOperator*{\esssup}{ess\,sup}

\providecommand{\keywords}[1]
{{\small\textbf{Key words.} #1}}
\providecommand{\AMS}[1]
{{\small\textbf{AMS subject classifications.} #1}}

\usepackage{xcolor}

\title{The least-squares RBF-PU method for linear elasticity in the diaphragm geometry\footnotemark[1]}
\author{Andreas Michael\footnotemark[3] \and Elisabeth Larsson\footnotemark[3] \and Pierre-Fr\'ed\'eric Villard\footnotemark[4] \and Davoud Mirzaei\footnotemark[3]}

\date{}

\begin{document}
\renewcommand{\thefootnote}{\fnsymbol{footnote}}
\footnotetext[1]{This work was supported by the Swedish Research Council, grant
no.~2020-03488. The computations were partially enabled by resources in projects SNIC 2022/22-238, NAISS 2023/22-580 and NAISS 2024/22-774 provided by the National Academic Infrastructure for Supercomputing in Sweden (NAISS) at UPPMAX, funded by the Swedish Research Council through grant agreement no. 2022-06725.}
\footnotetext[3]{ Department of Information Technology, Uppsala University, Uppsala, Sweden (andreas.michael@it.uu.se, elisabeth.larsson@it.uu.se, davoud.mirzaei@it.uu.se).}
\footnotetext[4]{Universit\'e de Lorraine, CNRS, Inria, LORIA, Nancy, France (pierrefrederic.villard@loria.fr).}

\maketitle

\begin{abstract}
The human diaphragm is vital for respiration and its simulation can aid in further understanding complications arising in individuals that have been subjected to mechanical ventilation. The diaphragm is very thin and hence its numerical simulation poses a challenge for most numerical methods. In this paper, we compute the deformation of real diaphragm geometries, which are modelled as linear elastic, and are subject to realistic boundary conditions. The method used is the unfitted least-squares radial basis function partition of unity method. The patches used are cylindrical and refinement is performed along the thin direction of the diaphragm. Both the unfitted nature of the method as well as the refinement technique make this method particularly well suited to solve problems formulated in thin geometries. We derive a theoretical proof showing that the method converges with high order to the solution of a well-posed linear elasticity boundary value problem. 
\end{abstract}

 \\

\AMS{65N12, 65N35, 65D12.}
\section{Introduction}
Understanding how the human respiratory system works is the driving application for the theory and method development in this work. As part of the INVIVE project~\cite{invive_project}, we aim to compute the deformation of the human diaphragm as a step towards understanding ventilator induced diaphragmatic dysfunction (VIDD), which is a negative side-effect of mechanical ventilation \cite{llano2012mechanisms}. The diaphragm is a skeletal muscle, whose deformation subject to external load can be modelled in the macroscopic regime by complex constitutive laws which consider the anisotropy of the muscle and its non-linear behaviour under high stresses. In this work we specifically use the diaphragm geometry, but in a simplified setting with isotropic linear elastic constitutive laws, to implement and analyse a high order least-squares (LS) radial basis function (RBF) partition of unity method (PUM) that can solve linear elastic partial differential equations (PDEs) on complex thin 3D geometries. 


The diaphragm is considered a thin body as its thickness is much smaller than its other two dimensions. The human diaphragms we consider have an approximate thickness of 2--5 mm, with the other two dimensions being determined by the size of the human torso. Figure~\ref{fig:diaphragm_Recon} shows examples of two geometries reconstructed from medical image data.

In engineering applications, and sometimes in biomechanical modelling, plane stress approximations are applied for thin geometries, such that the geometry is reduced from a thin volume to a surface. A shell finite element approximation was, e.g., applied to a human diaphragm simulation in~\cite{Pato11}, and the diaphragm geometry was modelled as a hemispherical surface in~\cite{CheBor22}.
While assuming plane stress reduces the dimensionality of the problem and improves computational efficiency, we believe it would not be accurate enough to use for a detailed simulation of the human diaphragm muscle mechanics.
The approximation only holds for very thin structures and becomes less accurate close to the edges of the thin geometry \cite{TheoryOfElasticity1951}. Additionally, the thickness of the diaphragm is not uniform, especially during the breathing cycle (see \cite{diaphragmRBFFD2022}). We are interested in the stress distribution along its thin dimension to fully capture the loading distribution both during natural breathing and during mechanical ventilation. 
A detailed 3D finite element (FEM) model of the diaphragm is derived in~\cite{Coelho18} and 3D FEM models involving also the lungs and rib cage can be found in~\cite{Zhang16,Ladjal25}. These models can predict the movement of the diaphragm well, but due to computational cost, do not resolve the diaphragm fully, making it hard to study the internal stress distribution.

A specific property of biological systems is the inherent smoothness of the geometries involved. For the diaphragm, we also expect that the physiological deformation is smooth. If this smoothness is captured by the mathematical and numerical modeling of the problem, a high-order method can be effective in reducing the computational cost. In this work we use a high-order unfitted LS-RBF-PU method~\cite{LarsonShcherbakovHeryudono2017,larsson2024rbfpartitionunitymethod}, which has specific advantages for computational domains with a thin dimension, such as the diaphragm geometry. Having an unfitted approximation lets us resolve the thin dimension with fewer degrees of freedom than would be needed for a fitted discretisation.


Using radial basis function approximation in a partition of unity method was first suggested by Babu\v{s}ka and Melenk in~\cite{BabMel97}.
The first realizations of the RBF-PUM idea can be found in~\cite{LaMo02,Wend02}.
The method has been used to solve PDEs in, e.g., option pricing~\cite{ShcheLa16}, and for glacial ice flow~\cite{AhlShche17}. Other types of RBF-PUM have also been introduced, such as the direct RBF-PUM by Mirzaei~\cite{Mirzaei21}.
It was shown in~\cite{LarsonShcherbakovHeryudono2017} that collocation RBF-PUM in the PDE context has limitations regarding approximation stability for large scale discretisations. An unfitted oversampled version of the method, LS-RBF-PUM, was introduced and was shown to have good stability properties, lower computational cost, and the method also simplifies meshfree node distribution significantly. LS-RBF-PUM was used in~\cite{larsson2024rbfpartitionunitymethod} for reconstruction of the diaphragm geometry, as well as for solving Poisson problems in the same geometry. Here we apply LS-RBF-PUM to the more challenging linear elastic system of PDEs for the first time. Additionally, we use medical image data to set up boundary conditions that reflect the real deformation problem for the diaphragm that we want to study.



There has previously not been a complete theoretical framework for analysis of RBF-PUM or LS-RBF-PUM applied to a PDE problem. In this work, we use analysis techniques from least squares FEM \cite{bochev2009least}, ideas from~\cite{stableRBFFD2021}, where convergence for a least-squares RBF-FD approximation was analysed, and the techniques introduced in~\cite{larsson2022numerical} for quadrature errors, to derive a complete and rather general convergence proof for unfitted LS-RBF-PUM. Similar proof techniques can be used for other well-posed PDE problems.

The paper is organized as follows: Section~\ref{sec:prob} introduces the displacement formulation of the continuous linear elasticity system of PDEs. Then in Section~\ref{sec:method}, the least-squares formulation of the PDE problem and its LS-RBF-PUM discretisation are derived. Section~\ref{sec:theory} is dedicated to the theoretical analysis of the continuous and the discretised problems. Implementation and numerical results are discussed in Section~\ref{sec:res}, followed by conclusions in Section~\ref{sec:conc}. In addition, two appendices provide detailed proofs for various results used in Section~\ref{sec:theory}.

\section{Problem formulation}\label{sec:prob} 

\subsection{Definitions of spaces, norms and notation}  \label{sec:Definitions}
We start by defining the spaces and norms used throughout this article. We consider functions defined over the bounded domain $\Omega\subseteq\mathbb{R}^d$, with boundary $\partial\Omega$. The Lebesgue inner products for scalar, vector and tensor valued functions $v:\mathbb{R}^d\rightarrow\mathbb{R}$, $\mathbf{v}:\mathbb{R}^d\rightarrow\mathbb{R}^d$, $\mathbf{A}:\mathbb{R}^d\rightarrow\mathbb{R}^{d\times d\times \dots \times d}$ are defined as
\begin{equation}
    \left(u,v\right)_{L_2(\Omega)} = \int_{\Omega}u(\mathbf{y})v(\mathbf{y})d\mathbf{y}, \hspace{0.5cm}\left(\mathbf{u},\mathbf{v}\right)_{L_2(\Omega)} = \int_{\Omega}\mathbf{u}(\mathbf{y})\cdot\mathbf{v}(\mathbf{y})d\mathbf{y}, \hspace{0.5cm}
    \left(\mathbf{A},\mathbf{B}\right)_{L_2(\Omega)} = \int_{\Omega}\mathbf{A}(\mathbf{y}):\mathbf{B}(\mathbf{y})d\mathbf{y}, 
    \label{eq:Lebesgue_Products}
\end{equation}
where we use the tensor inner product for order $n$ tensor valued functions $\mathbf{A},~\mathbf{B}$
\begin{equation}
    \mathbf{A}:\mathbf{B} = \sum_{i_1,i_2,\dots,i_n=1}^d A_{i_1i_2\dots i_n}B_{i_1i_2\dots i_n}.
\end{equation} 
We define the corresponding norm for $v:\mathbb{R}^d\rightarrow\mathbb{R}$ as
\begin{equation}
        \left\|v\right\|^2_{L_2(\Omega)} = \left(v,v\right)_{L_2(\Omega)} = \int_{\Omega}\left\|v\right\|^2_Td\mathbf{y},
    \label{eq:Lebesgue_Norms}
\end{equation}
where $\|\cdot\|_T$ denotes a specific tensor norm \cite{tensors2009}. The definition of the $L_2$-norm takes the same form for vectors $\mathbf{v}$ and tensors $\mathbf{A}$. We use the tensor norm to simplify differences in notation based on the tensor order of $v$. For a scalar function it becomes the absolute value, $\|\cdot\|_T \equiv |\cdot|$, for a vector valued function it is the Euclidean 2-norm, $\|\cdot\|_T \equiv \|\cdot\|_2$, and for a second order tensor it is the Frobenius norm, $\|\cdot\|_T\equiv \|\cdot\|_F$. The Lebesgue space is hence defined as 
\begin{equation}
    L_2(\Omega):=\left\{v:\left\|v\right\|_{L_2(\Omega)}<\infty\right\},
    \label{eq:Lebesgue_Spaces}
\end{equation}
for scalar functions $v$ and similarly for vectors $\mathbf{v}$ and tensors $\mathbf{A}$. The spaces are equipped with the Lebesgue inner products \eqref{eq:Lebesgue_Products}. Furthermore, we introduce the multi-index $\zeta = \left(\zeta_1,\zeta_2,\dots,\zeta_d\right)$, with order
\begin{equation}
    \left|\zeta\right| = \sum_{i=1}^d\zeta_i,
\end{equation}
where $d$ is the dimension. The derivative operator $D^\zeta$ applied to any scalar function $v(\mathbf{x}):\mathbb{R}^d\rightarrow \mathbb{R}$ is defined as 
\begin{equation}
    D^{\zeta}v = \frac{\partial^{|\zeta|}v}{\partial^{\zeta_1}x_1\partial^{\zeta_2}x_2\dots\partial^{\zeta_d}x_d},
    \label{eq:multiIndex}
\end{equation}
and similarly for a order $n$ tensor valued function $\mathbf{A}(\mathbf{x}):\mathbb{R}^d\rightarrow\mathbb{R}^{d\times d\times \dots \times d}$, we have the order $n$ tensor $D^{\zeta}\mathbf{A}$ with elements $D^{\zeta}A_{i_1i_2\dots i_n}$, for $i_1,i_2,\dots,i_n=1,\dots,d$.

We introduce the following Sobolev spaces (here also Hilbert spaces) for any tensor valued function $v$, non-negative integer $k \in\mathbb{N}_0$ and fraction $s\in\left(0,1\right)$ as
\begin{subequations}
\begin{align}
         W^{k}_2(\Omega) &\equiv H^k(\Omega):=\{v\in L_2(\Omega):\left\|v\right\|_{H^k} <\infty\}, \\
         W^{s}_2(\Omega) &\equiv H^{s}(\Omega)\hspace{1pt} :=\{v\in L_2(\Omega):\left\|v\right\|_{H^s} <\infty\},
\label{eq:Sobolev_Spaces}
\end{align}
\end{subequations}
where $W^{s}_2$ is a fractional Sobolev space. For a discussion of fractional spaces see \cite{FractionalSobolev2012}. Both spaces are equipped with norms
\begin{subequations}
\begin{align}
        & \left\|v\right\|^2_{H^k(\Omega)} = \sum_{|\zeta|\leq k}\left\|D^{\zeta}v\right\|^2_{L_2(\Omega)}, \label{eq:Sobolev_norm}\\
        & \left\|v\right\|^2_{H^s(\Omega)} = \left\|v\right\|^2_{L_2(\Omega)} + \left[v\right]^2_{H^{s}(\Omega)} = \left\|v\right\|^2_{L_2(\Omega)} + \int_{\Omega}\int_{\Omega}\frac{\left\|v(\mathbf{x}) - v(\mathbf{y})\right\|_T^2}{\left\|\mathbf{x} - \mathbf{y}\right\|^{d+2s}_2}d\mathbf{x}d\mathbf{y},
        \label{eq:FracSobolev_norm}
\end{align}
\end{subequations}
where the second term in the fractional norm, $\left[v\right]^2_{H^{s}}$, is a fractional semi-norm.

Note that given an order $\tau\in\mathbb{N}_0$ and the tensor $L_2$-{norm} \eqref{eq:Lebesgue_Norms}, we can also define the Sobolev norm as
\begin{equation}
    \left\|v\right\|^2_{H^k(\Omega)} = \sum_{\tau = 0}^k\left\|D^{\tau}v\right\|^2_{L_2(\Omega)}, \label{eq:Sobolev_norm_2}
\end{equation}
where, e.g., for a scalar function $v$, $D^0v = v$, $D^1 v \triangleq \nabla v$ is a vector with elements $\partial_i v_i$, $i = 1,\dots,d$, and $D^2 v$ is a second order tensor with elements  $\partial_{ij} v$, $i,~j = 1,\dots,d$. Sobolev norms over the boundary, $\partial\Omega$, are defined using tangential derivatives as in \cite[Definition 2.11]{SurfPDEs_Dziuk_Elliott_2013}. The definition for $k=1$ and a tensor valued function $v$ is
\begin{equation}
    \|v\|^2_{H^k(\partial\Omega)} = \|v\|^2_{L_2(\partial\Omega)} + \|\nabla_{\Gamma}v\|^2_{L_2(\partial\Omega)},
\end{equation}
where the tangential gradient is $\nabla_{\Gamma}v|_{\partial\Omega} = \nabla v|_{\partial\Omega} - \left(\nabla v|_{\partial\Omega}\cdot\mathbf{n}\right)\mathbf{n}$, where $\mathbf{n}$ is the surface normal.

Additionally, we define the native Hilbert space of the Gaussian kernel $\phi(r) = e^{-\varepsilon^2 r^2}$ \cite[Theorem 10.12]{Wendland_2004}, \cite{rieger2010sampling}
\begin{equation}
        \mathcal{N}_{G}(\mathbb{R}^d) := \left\{v\in C(\mathbb{R}^d)\cup L_2(\mathbb{R}^d):\left\|v\right\|^2_{\mathcal{N}_G}:=\int_{\mathbb{R}^d}\left|\hat{v}(\omega)\right|^2e^{\frac{r^2}{4\varepsilon^2}}d\omega < \infty \right\}, 
        \label{eq:NativeSpaceRd}
\end{equation}
where $\varepsilon$ is the shape parameter and $\hat{v}$ is the Fourier transform of $v$. We also define the native space of the Gaussian in a domain $\Omega$ as 
\begin{equation}
        \mathcal{N}_{G}(\Omega) := \left\{v\in \mathcal{N}_G(\mathbb{R}^d)~:~v|_{\Omega} = 0\right\}^{\perp_{\mathcal{N}_G}}.
        \label{eq:NativeSpace}
\end{equation}

Finally, we define the discrete $\ell_2$-norm and inner product for scalar functions $u,v:\mathbb{R}^d\rightarrow\mathbb{R}$ as given in \cite{stableRBFFD2021}
\begin{equation}
    \left(u,v\right)_{\ell_2(\Omega)} = \sum_{j=1}^MW_iu(\mathbf{y}_i)v(\mathbf{y}_i), \hspace{0.5cm} \left\|u\right\|^2_{\ell_2(\Omega)} = \left(u,u\right)_{\ell_2(\Omega)},
    \label{eq:l2}
\end{equation}
where $M$ is the number of evaluation points used to approximate the inner product and $W_i$ are 
quadrature weights for each evaluation/quadrature point $\mathbf{y}_i$. The weights are explicitly defined in Section \ref{sec:RBFPU_4}. We define the inner products and norms for vector and tensor valued functions similarly (as explained for the continuous case). The discrete counterpart to the continuous Sobolev space norm~\eqref{eq:Sobolev_norm}, $\|\cdot\|_{h_2(\Omega)}$, is defined similarly using the $\ell_2$-norm.

\subsection{The linear elasticity problem}
We aim to compute the displacement field $\mathbf{u} = \left(u_x,u_y,u_z\right)^T$ for thin linear elastic three-dimensional bodies. Thin bodies refer to the case where the proportion of one of the dimensions is considerably smaller than the other two (see for example Figure \ref{fig:diaphragm_Recon}). We define the displacement of a deformed body as
\begin{equation}
    \mathbf{u}(\mathbf{x}_{\mathrm{ref}}) = \mathbf{x}(\mathbf{x}_{\mathrm{ref}}) - \mathbf{x}_{\mathrm{ref}},
\end{equation}
where $\mathbf{x}$ is the position of a particle in the deformed or current configuration $\Omega$, and $\mathbf{x}_{\mathrm{ref}}$ is the position of the same particle in the undeformed or reference configuration $\Omega_{\mathrm{ref}}$. The linear elasticity constitutive law assumes small deformations $\|\nabla\mathbf{u}\|_2\ll1$, such that the displacement is considered equal when measured in either configuration. We assume that the body obeys momentum balance \cite[p.~145]{Holzapfel} described by Cauchy's equation of equilibrium, which leads to the following boundary value problem:
\begin{align}
    \begin{split}
        -\nabla\cdot\sigma = \mathbf{f}, & \hspace{0.5cm} \text{in} \hspace{0.1cm} \Omega, \\
        \kappa_1\sigma\cdot\mathbf{n} + \kappa_0\mathbf{u} = \kappa_1\mathbf{t} + \kappa_0\mathbf{g}, & \hspace{0.5cm} \text{on} \hspace{0.1cm} \partial\Omega,
    \end{split}
\end{align}
where $\Omega\subset\mathbb{R}^d$ is the current configuration which we refer to as the domain and $d=3$ for this problem, $\partial\Omega$ is its boundary with normal $\mathbf{n}$, $\sigma$ is the second order stress tensor, $\mathbf{f}$ is internal forcing, $\mathbf{t}$ is the traction on the boundary and $\mathbf{g}$ is the displacement on the boundary. The Robin boundary condition includes the coefficients $\kappa_0$ and $\kappa_1$ which are used to indicate the Dirichlet and Neumann parts of the boundary with a smooth transition in between. The system of partial differential equations is reformulated in terms of the displacement using the following linear constitutive relation between stress and displacement
\begin{equation}
\sigma = 2\mu\epsilon + \lambda tr(\epsilon)\mathbf{I},
\label{eq:ConstitutiveHook}
\end{equation}
where $\epsilon$ is the second order strain tensor given by the symmetric gradient of the displacement
\begin{equation}
    \epsilon = \frac{1}{2}(\nabla\mathbf{u} + \nabla^T\mathbf{u}),
\end{equation}
and where $tr(\epsilon) = \sum_{i=1}^d\epsilon_{ii}$, and $\mathbf{I}$ is the unit tensor. The Lam\'e parameters $\mu$ and $\lambda$ are defined as
\begin{equation}
\mu = \frac{E}{2(1+\nu)}, \hspace{1cm} \lambda = \frac{E\nu}{(1+\nu)(1-2\nu)},
\label{eq:LamePars}
\end{equation}
where $E$ is the Young's modulus and $\nu$ the Poisson's ratio of the modelled material, where $\nu\in[0,0.5)$. 

The resulting boundary value problem for the displacement is given by the Navier-Cauchy equations
\begin{align}
    \begin{split}
        -\mu\nabla^2\mathbf{u} - (\lambda + \mu)\nabla(\nabla\cdot\mathbf{u}) = \mathbf{f}, & \hspace{0.5cm} \text{in} \hspace{0.1cm} \Omega, \\
        \kappa_1(\mu(\nabla\mathbf{u} + \nabla^T\mathbf{u}) + \lambda(\nabla\cdot\mathbf{u})\mathbf{I})\cdot\mathbf{n} + \kappa_0\mathbf{u} = \kappa_1\mathbf{t} + \kappa_0\mathbf{g}, & \hspace{0.5cm} \text{on} \hspace{0.1cm} \partial\Omega,
    \end{split}
    \label{eq:linear_elasticity_displacement}
\end{align}
where $\nabla^2$ is the vector Laplacian. 

Since we will use a least squares method to solve the boundary value problem, we reformulate it as the following minimization problem solved for $\mathbf{u}$:
\begin{align}
    & \min_{\mathbf{u}\in H^2(\Omega)}J(\mathbf{u};\mathbf{f},\mathbf{t},\mathbf{g}), \\
    & J(\mathbf{u};\mathbf{f},\mathbf{t},\mathbf{g}) = \left\|\nabla \cdot\boldsymbol{\sigma} + \mathbf{f} \right\|^2_{L_2(\Omega)} + \left\|\kappa_0(\mathbf{u} - \mathbf{g}) +  \kappa_1(\boldsymbol{\sigma}\cdot\mathbf{n} - \mathbf{t})\right\|^2_{H^{1/2}(\partial\Omega)},
    \label{eq:energy_functional}
\end{align}
where we assume that the true solution $\mathbf{u}\in H^2(\Omega)$, and that the data $\mathbf{f}\in L_2(\Omega)$, $\mathbf{g}\in H^{1/2}(\partial\Omega)$, and $\mathbf{t}\in H^{1/2}(\partial\Omega)$. We also assume that $\kappa_0,\kappa_1\in C^1(\partial\Omega)$ are chosen such that $\kappa_1,\kappa_0\in (0,1)$. These spaces and the boundary norm $\left\|\cdot\right\|_{H^{1/2}(\partial\Omega)}$ are chosen such that the weak problem is well posed, as explained in Section~\ref{sec:theory}. The minimization problem can be expressed as the following weak problem solved for $\mathbf{u}\in H^2(\Omega)$
\begin{equation}
\alpha(\mathbf{u},\mathbf{v}) = l(\mathbf{v}), \hspace{1cm} \forall\mathbf{v}\in H^2(\Omega),
\label{eq:LSQWF}
\end{equation}
where $\alpha(\mathbf{u},\mathbf{v})$ is a bilinear form and $l(\mathbf{v})$ is a linear form, defined as
\begin{align}
    & \alpha(\mathbf{u},\mathbf{v}) = -\left(\nabla\cdot\boldsymbol{\sigma}(\mathbf{u}),\nabla\cdot\boldsymbol{\sigma}(\mathbf{v})\right)_{L_2(\Omega)} + \left(\kappa_0\mathbf{u} + \kappa_1\sigma(\mathbf{u})\cdot\mathbf{n},\kappa_0\mathbf{v} + \kappa_1\sigma(\mathbf{v})\cdot\mathbf{n}\right)_{H^{1/2}(\partial\Omega)}, \label{eq:ContBilinear}\\ 
    & l(\mathbf{v}) = \left(\mathbf{f},\nabla\cdot\boldsymbol{\sigma}(\mathbf{v})\right)_{L_2(\Omega)} + \left(\kappa_0\mathbf{g} + \kappa_1\mathbf{t},\kappa_0\mathbf{v} + \kappa_1\sigma(\mathbf{v})\cdot\mathbf{n}\right)_{H^{1/2}(\partial\Omega)}.  
    \label{eq:ContLinear}
\end{align}

\subsection{The diaphragm geometry}
The method used and further described in Section \ref{sec:RBFPU} has been developed for thin geometries and specifically to discretise  the linear elasticity problem on realistic reconstructed human diaphragm geometries. They were reconstructed by first segmenting 3D Computed Tomography (CT) scans manually. In this work we solve problems on two different reconstructed diaphragms which will henceforth be referred to as Diaphragm 1, shown in Figure \ref{fig:diaphragm1_Recon}, and Diaphragm 2, shown in Figure \ref{fig:diaphragm2_Recon}. The reconstruction is computed from the medical image data as described in \cite{larsson2024rbfpartitionunitymethod}. Hence the surface is implicitly parametrised using a level set function $\ell(\mathbf{y})$, whose zero level set defines the surface. We use $\ell(\mathbf{y})$ to generate both interior and boundary points. We will solve benchmark problems on these geometries both to support our theoretical estimates and to show that the method can work on problems arising from more realistic scenarios.

\begin{figure}
    \centering
    \begin{subfigure}{.3\textwidth}
        \centering
        \includegraphics[width=\linewidth]{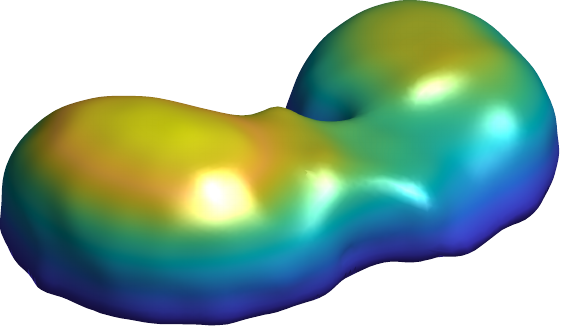}
        \caption{}
        \label{fig:diaphragm1_Recon}
    \end{subfigure}
    \hspace{2cm}
    \begin{subfigure}{.3\textwidth}
        \centering
        \includegraphics[width=\linewidth]{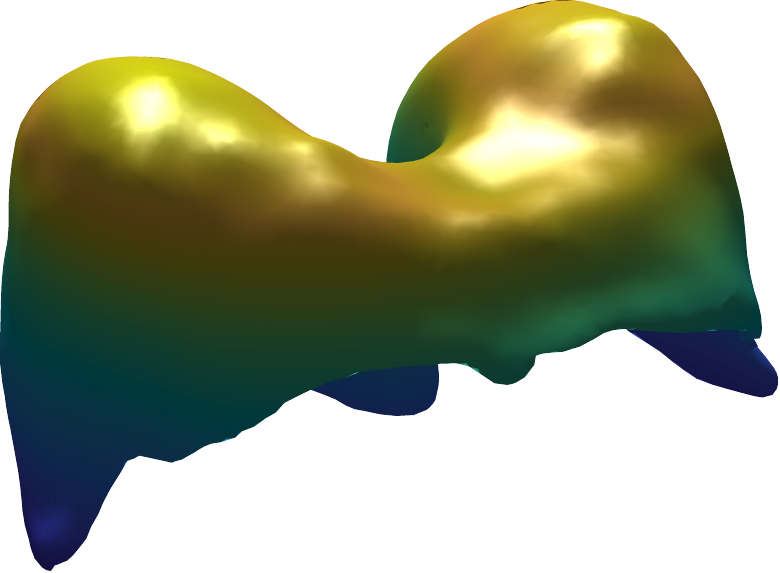}
        \caption{}
        \label{fig:diaphragm2_Recon}
    \end{subfigure}
    \caption{(a) Smooth reconstruction of Diaphragm 1 (anterosuperior view). The lower part of this diaphragm is not included in the image data. (b) Smooth reconstruction of Diaphragm 2 (anterior view).  }
    \label{fig:diaphragm_Recon}
\end{figure}

\section{The unfitted least-squares RBF-PU method}\label{sec:method}
We search for an approximate solution $\mathbf{u}_h$ to problem~\eqref{eq:LSQWF} from a finite dimensional space $V_h(\Omega)\subset H^2(\Omega)$. The bilinear and linear forms~\eqref{eq:ContBilinear} and~\eqref{eq:ContLinear} are defined using fractional Sobolev inner products, which are impractical to evaluate numerically. Hence, we approximate the bilinear form using weighted boundary terms as
\begin{align}
    & \tilde{\alpha}(\mathbf{u}_h,\mathbf{v}) = -\left(\nabla\cdot\boldsymbol{\sigma}(\mathbf{u}_h),\nabla\cdot\boldsymbol{\sigma}(\mathbf{v})\right)_{L_2(\Omega)} + \beta\left(\gamma\kappa_0\mathbf{u}_h + \kappa_1\sigma(\mathbf{u}_h)\cdot\mathbf{n},\gamma\kappa_0\mathbf{v} + \kappa_1\sigma(\mathbf{v})\cdot\mathbf{n}\right)_{L^{2}(\partial\Omega)}, \label{eq:ContWeightBilinear}\\ 
    & \tilde{l}(\mathbf{v}) = \left(\mathbf{f},\nabla\cdot\boldsymbol{\sigma}(\mathbf{v})\right)_{L_2(\Omega)} + \beta\left(\gamma\kappa_0\mathbf{g} + \kappa_1\mathbf{t},\gamma\kappa_0\mathbf{v} + \kappa_1\sigma(\mathbf{v})\cdot\mathbf{n}\right)_{L^{2}(\partial\Omega)},
    \label{eq:ContWeightLinear}
\end{align}
where the weights $\beta$ and $\gamma$ are used for ensuring that the new bilinear form is coercive for all $\mathbf{v}\in V_h(\Omega)$, which is used for deriving error estimates (see Section \ref{sec:theoryDisc}). An additional practical modification involves approximating the $L_2(\Omega)$ inner products using a quasi-Monte Carlo quadrature rule which corresponds to solving a modified discrete weak form to find a solution $\mathbf{u}_h\in V_h(\Omega)\subset H^2(\Omega)$ such that
\begin{equation}
    \alpha_h(\mathbf{u}_h,\mathbf{v}) = l_h(\mathbf{v}), \hspace{1cm} \forall\mathbf{v}\in V_h(\Omega),
    \label{eq:disc_WF}
\end{equation}
where the discrete bilinear and linear forms are defined using the $\ell_2$ inner product as
\begin{align}
    & \alpha_h(\mathbf{u}_h,\mathbf{v}) = -\left(\nabla\cdot\boldsymbol{\sigma}(\mathbf{u}_h),\nabla\cdot\boldsymbol{\sigma}(\mathbf{v})\right)_{\ell_2(\Omega)} + \beta\left(\gamma\kappa_0\mathbf{u}_h + \kappa_1\sigma(\mathbf{u}_h)\cdot\mathbf{n},\gamma\kappa_0\mathbf{v} + \kappa_1\sigma(\mathbf{v})\cdot\mathbf{n}\right)_{\ell_2(\partial\Omega)}, \label{eq:DiscBilinear}\\ 
        & l_h(\mathbf{v}) = \left(\mathbf{f},\nabla\cdot\boldsymbol{\sigma}(\mathbf{v})\right)_{\ell_2(\Omega)} + \beta\left(\gamma\kappa_0\mathbf{g} + \kappa_1\mathbf{t},\gamma\kappa_0\mathbf{v} + \kappa_1\sigma(\mathbf{v})\cdot\mathbf{n}\right)_{\ell_2(\partial\Omega)}.\label{eq:DiscLinear}
\end{align}

The space $V_h$ is spanned by a specific set of globally defined, but locally supported basis functions $\left\{\psi_i\right\}^N_{i=1}$, such that the global solution can be approximated as
\begin{equation}
    \mathbf{u}_h(\mathbf{y}) = \sum_{i=1}^{N}\mathbf{U}(\mathbf{x}_i)\psi_i(\mathbf{y}),
    \label{eq:Global_interpolant}
\end{equation}
where $\mathbf{U}(\mathbf{x}_i) = (U_x(\mathbf{x}_i),U_y(\mathbf{x}_i),U_z(\mathbf{x}_i))^T$ are nodal coefficients and $\psi_i$ is the basis function associated with the point~$\mathbf{x}_i$.

\subsection{Method formulation}\label{sec:RBFPU}

We use the oversampled and unfitted RBF-PU method to construct the global basis functions. Initially, the bounded domain $\Omega\subseteq\mathbb{R}^d$ is covered by overlapping patches $\{\Omega_k\}^{P}_{k=1}$, where $\Omega\subset\tilde{\Omega}=\cup_{k = 1}^{P}\Omega_k$. For this purpose, we use an adaptive algorithm, which generates one layer of cylindrical patches in $\mathbb{R}^3$. The generated patches are designed to be of similar volume, where $V_k = \mathfrak{C}_VR_k^{d-1}H_k$ is the volume of patch $\Omega_k$, such that they fully and relatively tightly enclose the thin body. The constant $\mathfrak{C}_V=\pi$ for $d=3$, the patch radius $R_k$ and the patch height  $H_k$ define the shape and size of the cylindrical patch.
Parameters such as patch size and patch overlap are controlled by the adaptive algorithm, further discussed in \cite{larsson2024rbfpartitionunitymethod}. 
We quantify the overlap of a cover using the minimum absolute overlap between patch $k$ and its neighbouring patches, $\delta_k$, which is defined as 
\begin{equation}
    \delta_k = \delta_0 R_k,
    \label{eq:overlap}
\end{equation}
where $\delta_0$ is the relative overlap, which is fixed for all patches. 
An example of patches generated over a 3D diaphragm geometry is shown in Figure \ref{fig:patches_1}.

\begin{figure}
\centering
    \begin{subfigure}{.4\textwidth}
        \centering
        \includegraphics[width=\linewidth]{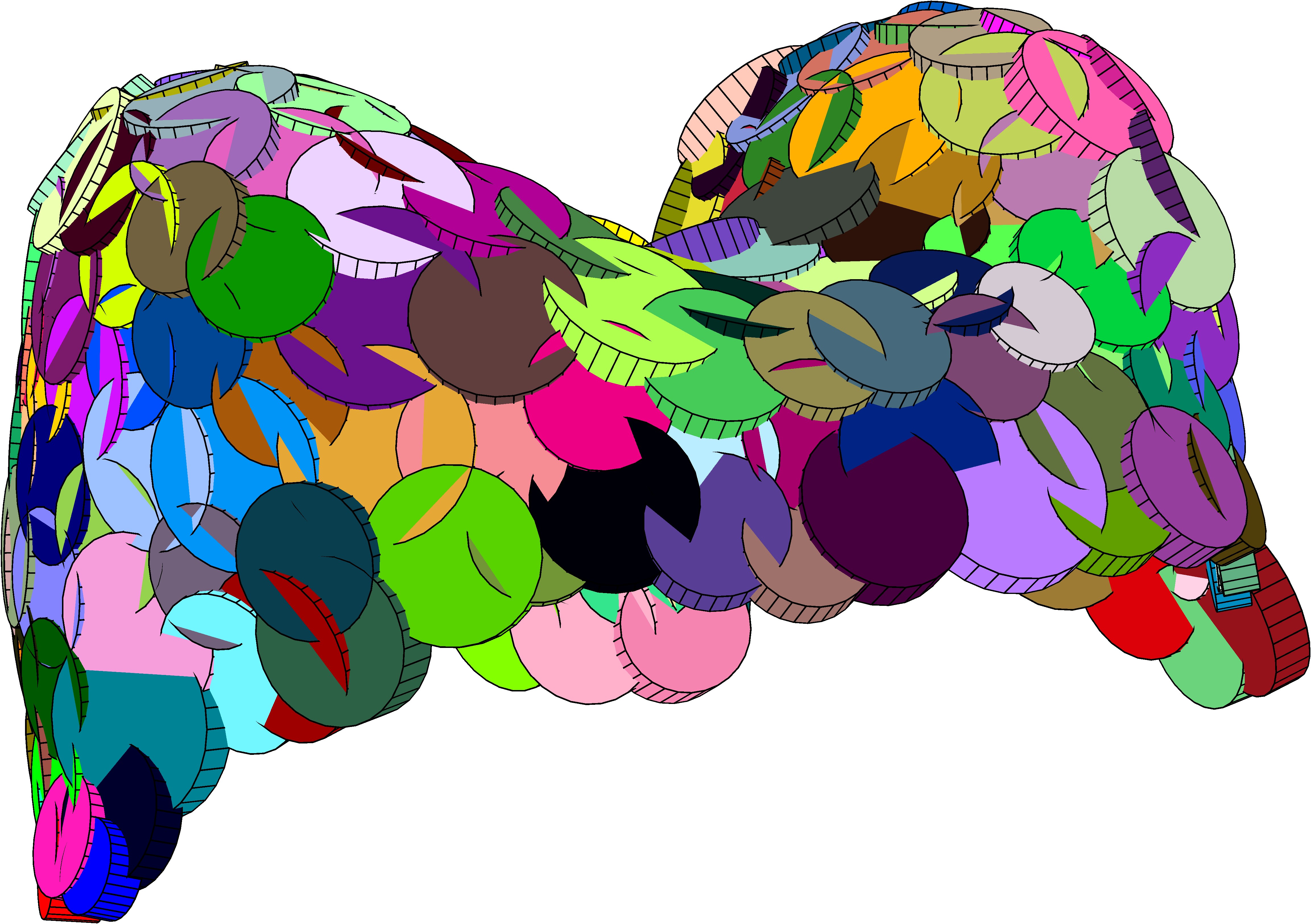}
        \caption{}
        \label{fig:patches_1}
    \end{subfigure}
    \hspace{2cm}
    \begin{subfigure}{.45\textwidth}
        \centering
        \includegraphics[width=\linewidth]{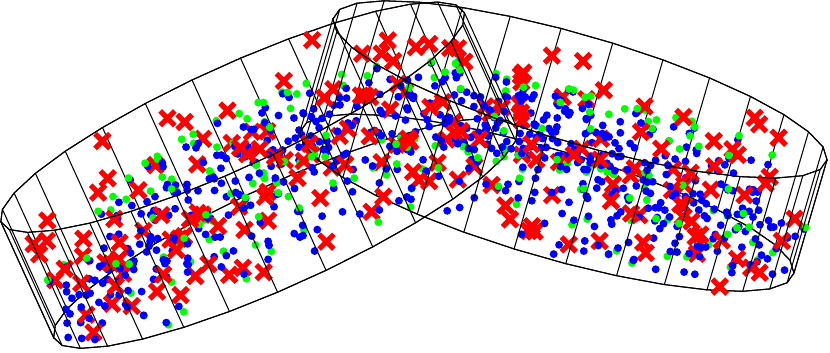}
        \caption{}
        \label{fig:patches_2}
    \end{subfigure}
    \caption{(a) Patch cover over Diaphragm 2 (anterior view). (b) Cylindrical overlapping patches with the centre point sets (red crosses), the boundary evaluation point sets (green dots), and the interior evaluation point sets (blue dots). An oversampling factor $q=5$ is used for the visualisation.}
    \label{fig:patches}
\end{figure}

The RBF-PU method is mesh-free so there is no underlying mesh connectivity but rather scattered point sets which cover the domain and patches. We use two different global point sets $Y = \{\mathbf{y}_j\}_{j=1}^M\subseteq \Omega$ and $X = \{\mathbf{x}_j\}_{j=1}^N\subseteq\tilde\Omega$, referred to as the evaluation or quadrature point set and the node or centre point set, respectively, shown in Figure \ref{fig:patches_2}. The naming conventions arise from the fact that the approximation is evaluated on the $Y$ point set, which is also used to numerically approximate the integrals in the weak formulation \eqref{eq:disc_WF}, and from the fact that the RBFs are centred on the $X$ point set. Both point sets are constructed using the Halton sequence and, since the method is oversampled, the number of evaluation points $M$ is larger than the number of centre points $N$. 

The centre points are constructed locally in each patch, such that $X= \cup_{k=1}^P X_k$, where $P$ is the number of patches. They are independent of the geometry and therefore, do not in general sample the boundary of the domain. Conversely, the $Y$ point set consists of points sampled both on the boundary and in the domain, indexed by the index sets $I_{\partial\Omega}$, and $I_{\Omega}$. We define the point subsets $Y_{\Omega} \triangleq \{\mathbf{y}_i\}_{i\in I_{\Omega}},~Y_{\partial\Omega} \triangleq \{\mathbf{y}_i\}_{i\in I_{\partial\Omega}}$. The boundary points and normals $\{\mathbf{n}^{(k)}_j\}_{j = 1}^{M_{\partial\Omega}}$ are computed using the implicit level set function $\ell(\mathbf{y})$ representing the surface.

The density of the evaluation and centre point sets is described by the fill distance and separation distance. The fill distance for a point set $Z = \{z_1,\dots,z_N\}$ in $\mathcal{B}$, where $\mathcal{B}\triangleq \Omega$ or $\mathcal{B}\triangleq \partial\Omega$ is defined as 
\begin{equation}
    h_f(Z,\mathcal{B}) = \sup_{\mathbf{x}\in\mathcal{B}}\min_{\mathbf{x}_j\in Z}\|\mathbf{x}-\mathbf{x}_j\|_2,
    \label{eq:fill}
\end{equation}
and the separation distance for the same point set is defined as
\begin{equation}
    h_s(Z,\mathcal{B}) = \frac{1}{2}\underset{\mathbf{x}_j,\mathbf{x}_k\in Z}{\min_{j\neq k}}\|\mathbf{x}_j-\mathbf{x}_k\|_2.
    \label{eq:sepDistance}
\end{equation}
For a quasi-uniform node set, the node quality can be measured using $c_{q,Z} > 0$, where $h_s(Z,\mathcal{B}) \leq h_f(Z,\mathcal{B}) \leq c_{q,Z} h_s(Z,\mathcal{B})$. Our definitions for the fill distance and separation distance use the Euclidean distance metric even for the boundary, as opposed to a distance metric intrinsic to $\partial\Omega$. However, these are equivalent for a sufficiently dense node set \cite[Theorem 6]{WrightFuselier2012}. We note that it was recently shown that the Halton sequence, used for constructing the node sets $X_k$, $Y_{\Omega}$ and $Y_{\partial\Omega}$, is not quasi-uniform in any $d\geq 2$ \cite{Goda2026_HaltonQuasiUniformity}. However, the discrete theoretical estimates in Section \ref{sec:theoryDisc} and Appendix \ref{app:intErrorEstimates} assume that the node sets are  quasi-uniform. The assumption is closer to the true behaviour for smaller numbers of Halton points, such as those used within each patch. 


Some specific fill distance measures for the centre point set $X = \cup_{k=1}^P X_k$ and the evaluation point sets $Y_{\Omega},~Y_{\partial\Omega}$ are defined as
\begin{alignat}{3}
    &h_k &\,=\,& h_f(X_k,\Omega_k),\\
    &h &\,=\,& \max_{k}h_k, 
    \label{eq:fillX}\\
    &h_y &\,=\,& h_f(Y_{\Omega},\Omega) 
    \label{eq:fillY}\\
    &h_{y,\partial\Omega} &\,=\,& h_f(Y_{\partial\Omega},\partial\Omega).
    \label{eq:fillYBnd}
\end{alignat}
%
Given the challenge of numerically computing the fill distance for a given point set, we introduce some approximate measures for the centre point fill distance in patch $\Omega_k$ as well as for the local evaluation point fill distance, given by
\begin{alignat}{3}
    & \tilde{h}_k &\,=\,& \frac{\left|\Omega_k\right|}{n}, 
    \label{eq:fillXapprox} \\
    & \tilde{h}_{j,\Omega} &\,=\,& \frac{1}{d+1}\sum_{\mathbf{y}_i\in Z_{j}^{\Omega}}\|\mathbf{y}_i - \mathbf{y}_j\|_2, 
    \label{eq:fillYapprox}\\
    & \tilde{h}_{j,\partial\Omega} &\,=\,& \frac{1}{d}\sum_{\mathbf{y}_i\in Z_{j}^{\partial\Omega}}\|\mathbf{y}_i-\mathbf{y}_j\|_2,  \label{eq:fillYBndApprox}
\end{alignat}
where $n = |X_k|$ is the number of centre points in patch $\Omega_k$, the point sets $Z_{j}^{\Omega}$ and $Z_{j}^{\partial\Omega}$ contain the $d+1$ and $d$ closest neighbours of the point $\mathbf{y}_j$, respectively, where $d$ is the dimension.

The global approximation is constructed from a local cardinal basis in each patch. To reduce computational cost we use a template centre point set $X_0$ in the reference patch $\Omega_0$, and define the linear map $T_k:\Omega_k\rightarrow\Omega_0$ such that any patch can be defined as
\begin{equation}
    \Omega_k:=\{\mathbf{x} = T_k^{-1}(\mathbf{x}^\prime)~:~\mathbf{x}^\prime\in\Omega_0\}.
    \label{eq:referencePatchMap_1}
\end{equation}
The linear map $T_k$ includes a translation and rotation of the patch as well as a scaling in the radial and vertical directions and is given by
\begin{equation}
    \mathbf{x}^\prime = T_k(\mathbf{x}) = S_kQ_k\mathbf{x} - C_k, 
    \label{eq:referencePatchMap}
\end{equation}
where $S_k$ is a diagonal scaling matrix, $Q_k$ is an orthogonal (rotation) matrix, and $C_k$ is the translation matrix, and $\mathbf{x}$ and $\mathbf{x}^\prime$ are column vectors. Note that for cylindrical patches in $d=3$ dimensions, the first two diagonal elements of $S_k$ are equal and scale the radius of the patch with $R_{\mathrm{sc},k}$, while the third element scales the height with $H_{\mathrm{sc},k}$.  
For more details on how the map is computed and used in the code, refer to \cite{larsson2024rbfpartitionunitymethod}.

The density of the template node set $X_0$ is also described by the fill distance measure $h_0 = h_f(X_0,\Omega_0)$, which is related to $h_k$ by a constant. The relation is given by
\begin{equation}
C_{0,k}h_k \leq h_0\leq C'_{0,k}h_k,
\label{eq:referenceFilltoPhysicalFill}
\end{equation}
where $C_{0,k}$ and $C'_{0,k}$ are constants dependent on the scaling to the reference patch. We compute the constants by first defining the pointwise 2-norm of the distance between two points in the reference patch given by
\begin{equation}
\|\mathbf{x}^\prime-\mathbf{x}^\prime_j\|_2 = \|T_k(\mathbf{x}) - T_k(\mathbf{x}_j)\|_2 = \|S_kQ_k\left(\mathbf{x} - \mathbf{x}_j\right)\|_2,
\end{equation}
where we used \eqref{eq:referencePatchMap}.
The property $\|Q_k\mathbf{x}\|_2 = \|\mathbf{x}\|_2$ for an orthonormal matrix $Q_k$ gives
\begin{equation}
    \min\left(R_{\mathrm{sc},k},H_{\mathrm{sc},k}\right)\|\mathbf{x} - \mathbf{x}_j\|_2\leq \|\mathbf{x}^\prime-\mathbf{x}^\prime_j\|_2 \leq  \max\left(R_{\mathrm{sc},k},H_{\mathrm{sc},k}\right)\|\mathbf{x} - \mathbf{x}_j\|_2.
\end{equation}
Taking the minimum distance over all centre points $\mathbf{x}_j$ and supremum over the patch gives \eqref{eq:referenceFilltoPhysicalFill}, where the constants are defined as $C_{0,k} = \min(R_{\mathrm{sc},k},H_{\mathrm{sc},k}),~C'_{0,k} = \max(R_{\mathrm{sc},k},H_{\mathrm{sc},k})$.

Given the transformation from the reference patch \eqref{eq:referencePatchMap_1}, the local approximation $\mathbf{u}_h^{(k)}:\mathbb{R}^d\rightarrow \mathbb{R}^d$ defined in patch $\Omega_k$ is given by
\begin{equation}
\mathbf{u}^{(k)}_h(\mathbf{x}) = \mathbf{u}^{(k)}_h(T_k^{-1}(\mathbf{x}^\prime)) \triangleq \tilde{\mathbf{u}}^{(k)}_h(\mathbf{x}^\prime)= \tilde{\mathbf{u}}_h^{(k)}(T_k(\mathbf{x})),
\label{eq:transf}
\end{equation}
where $\tilde{\mathbf{u}}^{(k)}_h$ is the approximant in the reference patch.

The approximant in $\Omega_0$ (for each patch $\Omega_k$) is constructed using a linear combination of RBFs
\begin{equation}
    \tilde{\mathbf{u}}^{(k)}_h(\mathbf{x}^\prime) = \sum_{j=1}^{n}\boldsymbol{\lambda}_j^{(k)}\phi(\|\mathbf{x}^\prime-\mathbf{x}^\prime_j\|_2),
    \label{eq:RBFinterpolantLocal}
\end{equation}
where $n$ is the number of centre points $\mathbf{x}_j'\in X_0$, $\tilde{\mathbf{u}}_{h}^{(k)}(\mathbf{x}^\prime) = ({\tilde{u}}_{h,1}^{(k)}(\mathbf{x}^\prime),\dots,{\tilde{u}_{h,d}^{(k)}}(\mathbf{x}^\prime))^T$,  $\boldsymbol{\lambda}_j^{(k)}= (\lambda_{j,1},\dots,\lambda_{j,d})^T$ are vector coefficients,  and $\phi:\mathbb{R}_+\rightarrow\mathbb{R}$ is the univariate basis function. 
Throughout this paper, we use the infinitely smooth Gaussian RBF, $\phi(r) = e^{-(\varepsilon r)^2}$, where $r = \|\mathbf{x}^\prime - \mathbf{x}_j'\|_2$ and $\varepsilon$ is the shape parameter.

When directly using this formulation for computations, the coefficient values can become large, especially for small shape parameters, for which the basis functions become nearly flat. It is therefore common to instead use the cardinal form of the approximation. To illustrate the change of basis, we first consider the approximation of a generic scalar function $u(\mathbf{x}^\prime)$ over the reference patch. 
\begin{equation}
    u(\mathbf{x}^\prime)=\sum_{j=1}^n\lambda_j\phi(\|\mathbf{x}^\prime-\mathbf{x}^\prime_j\|)\triangleq\sum_{j=1}^n\lambda_j\phi_j(\mathbf{x}^\prime)\triangleq A(\mathbf{x}^\prime,X_0)\lambda,
    \label{eq:card0}
\end{equation}
where $A(\mathbf{x}^\prime,X_0)=(\phi_1(\mathbf{x}^\prime),\ldots,\phi_n(\mathbf{x}^\prime))$ and $\lambda=(\lambda_1,\ldots,\lambda_n)^T$. If we let $u(X_0)=(u(\mathbf{x}_1^\prime, \ldots, u(\mathbf{x}_n^\prime))^T$, we can form the linear system $A(X_0,X_0)\lambda=u(X_0)$,
where $A:=A(X_0,X_0)$ is the local RBF interpolation matrix with elements $A_{ij} = \phi_j(\mathbf{x}_i^\prime)$, which is symmetric and strictly positive definite \cite{Schoenberg1938}. By formally solving this system for the coefficients, we get $\lambda=A^{-1}u(X_0)$. Substituting $\lambda$ into~\eqref{eq:card0} leads to
\begin{align}
\begin{split}
    u(\mathbf{x}^\prime)= A(\mathbf{x}^\prime,X_0)A^{-1}u(X_0)&=
    (\phi_1(\mathbf{x}^\prime),\ldots,\phi_n(\mathbf{x}^\prime))A^{-1}u(X_0)\\
    &\triangleq(\tilde{\psi}_1(\mathbf{x}^\prime),\ldots,\tilde{\psi}_n(\mathbf{x}^\prime))u(X_0)=\sum_{j=1}^n u(\mathbf{x}^\prime_j)\tilde{\psi}_j(\mathbf{x}^\prime),
    \end{split}
    \label{eq:cardinalBasis}
\end{align}
where $\tilde{\psi}_j(\cdot)$ are cardinal basis functions. The inverse of $A$ cannot be computed stably for small shape parameters, but the cardinal functions and their derivatives for Gaussian RBFs can be stably evaluated using the RBF-QR method~\cite{RBFQR}.



Going back the local approximation in patch $\Omega_k$ defined through~\eqref{eq:transf} and~\eqref{eq:RBFinterpolantLocal}, we use the cardinal basis from~\eqref{eq:cardinalBasis} and evaluate at a point $\mathbf{y}\in \Omega$ to get
\begin{equation}
    \mathbf{u}^{(k)}_h(\mathbf{y}) = \tilde{\mathbf{u}}^{(k)}_h(T_k(\mathbf{y})) = \tilde{\mathbf{u}}^{(k)}_h(\mathbf{y}') = \sum_{j=1}^{n}\tilde{\psi}_j(\mathbf{y}')\tilde{\mathbf{u}}^{(k)}_h(\mathbf{x}^\prime_j).
    \label{eq:RBFinterpolantLocal2}
\end{equation}
Note that for $\mathbf{y}\in\Omega\cap\Omega_k$, $\mathbf{y}'\in\Omega_0$.
Using partition of unity weight functions, the global approximation is then constructed as
\begin{equation}
    \mathbf{u}_h(\mathbf{y}) = \sum_{k=1}^{P}\sum_{j=1}^{n}w_k(\mathbf{y})\psi^{(k)}_j(\mathbf{y})\mathbf{u}_h^{(k)}(\mathbf{x}_j),
    \label{eq:global_func_approx}
\end{equation}
where $\psi_j^{(k)}(\mathbf{y})$ are the cardinal basis functions transformed to the physical patch $\Omega_k$, such that $\psi_j^{(k)}(\mathbf{y}) = \tilde{\psi}_j(T_k(\mathbf{y}))$. The compactly supported functions ${w_k(\mathbf{y})}\in C^{\gamma}(\Omega_k)$ form a $\gamma$-stable partition of unity. The partition of unity property implies that $\sum_{k=1}^Pw_k(\mathbf{y}) = 1$, for  $\mathbf{y}\in\Omega$ \cite[Definition 15.16]{wendland2005approximate}. Here, $\mathbf{u}_h(\mathbf{y})$ is a general global approximant and not necessarily the solution to \eqref{eq:disc_WF}. Note that we can also write \eqref{eq:global_func_approx} in terms of global basis functions as in, \eqref{eq:Global_interpolant}, using
\begin{equation}
    \psi_i(\mathbf{y}) =  w_k(\mathbf{y})\psi_j^{(k)}(\mathbf{y}),
    \label{eq:globalCardinalFunc}
\end{equation}
where we apply the index transformation $i = (k-1)n + j$. 

For the PDE problems we solve here, the weights are generated using the $C^2$ compactly supported Wendland function \cite{Wendland1995PiecewisePP}, defined for $d \leq 3$ as
\begin{equation}
    f_W(r) = (4r+1)(1-r)_{+}^4,
\end{equation}
which ensures that the global basis is sufficiently smooth for the PDE problem to be well-defined. We use a tensor product of the two-dimensional and one-dimensional compactly supported functions on the reference patch (see \cite{larsson2024rbfpartitionunitymethod}) and then use Shepard's method to construct the weight functions \cite{Shepard1968ATI,larsson2024rbfpartitionunitymethod}. Applying a linear operator to the global approximation results in
\begin{equation}
    \mathcal{L}\mathbf{u}_h(\mathbf{y}) = \sum_{k=1}^{P}\sum_{j=1}^{n}\mathcal{L}\left(w_k(\mathbf{y})\psi^{(k)}_j(\mathbf{y})\right)\mathbf{u}_h^{(k)}(\mathbf{x}_j)=\sum_{i=1}^N\mathcal{L}\psi_i(\mathbf{y})\mathbf{U}(\mathbf{x}_i).
    \label{eq:global_approx_linearOp}
\end{equation}


\subsection{Discrete least-squares formulation}\label{sec:RBFPU_4}
The linear elasticity problem is discretised using the oversampled and unfitted RBF partition of unity method. We write the boundary value problem \eqref{eq:linear_elasticity_displacement} in terms of the domain and boundary operators $\mathcal{L},~\mathcal{B}_1,~\mathcal{B}_0$
\begin{align}
        \mathcal{L}\mathbf{u} = \mathbf{f}, & \hspace{0.5cm} \text{in} \hspace{0.1cm} \Omega, \label{eq:StrongPDE}\\
        \left(\kappa_0\mathcal{B}_0 + \kappa_1\mathcal{B}_1\right)\mathbf{u} = \kappa_0\mathbf{g} + \kappa_1\mathbf{t}, & \hspace{0.5cm} \text{on} \hspace{0.1cm} \partial\Omega, 
    \label{eq:StrongBC}
\end{align}
where $\mathbf{u} = (u_x,u_y,u_z)^T \in H^2(\Omega)$, $\mathbf{f}\in L_2(\Omega),~\mathbf{g}\in H^{1/2}(\partial\Omega),~\mathbf{t}\in H^{1/2}(\partial\Omega),~\kappa_0\in C^1(\partial\Omega),~\kappa_1\in C^1(\partial\Omega)$. The fully expanded operators are given by
\makeatletter 
 \def\@eqnnum{{\normalsize \normalcolor (\theequation)}} 
  \makeatother
{\scriptsize
\begin{eqnarray}
    \begin{aligned}
    & \mathcal{L} = -\begin{pmatrix}(\lambda+2\mu)\nabla_{xx}+\mu(\nabla_{yy} + \nabla_{zz}) & (\lambda + \mu)\nabla_{xy} & (\lambda + \mu)\nabla_{xz} \\
    (\lambda + \mu)\nabla_{yx} & (\lambda + 2\mu)\nabla_{yy} + \mu(\nabla_{xx} + \nabla_{zz}) & (\lambda + \mu)\nabla_{yz} \\
    (\lambda + \mu)\nabla_{zx} & (\lambda + \mu)\nabla_{zy} & (\lambda + 2\mu)\nabla_{zz} + \mu(\nabla_{xx} + \nabla_{yy})\end{pmatrix}, \\
    & \mathcal{B}_1 = \begin{pmatrix}(\lambda + 2\mu)n_1\nabla_{x} + \mu(n_2\nabla_y + n_3\nabla_z) & \lambda n_1\nabla_{y} + \mu n_2\nabla_{x} & \lambda n_1\nabla_z + \mu n_3 \nabla_{x} \\
    \lambda n_2\nabla_{x} + \mu n_1\nabla_y & (\lambda + 2\mu)n_2\nabla_y + \mu(n_1\nabla_x + n_3\nabla_z) & \lambda n_2\nabla_z + \mu n_3\nabla_y \\
    \lambda n_3\nabla_x + \mu n_1\nabla_z & \lambda n_3\nabla_y + \mu n_2\nabla_z & (\lambda + 2\mu)n_3 \nabla_z + \mu(n_1\nabla_x + n_2\nabla_y)
    \end{pmatrix}, \\
    & \mathcal{B}_0 = I,
    \end{aligned}
    \label{eq:Elasticity_Operators}
\end{eqnarray}
}
where $\nabla_{x}$, $\nabla_{y}$, and $\nabla_{z}$ are derivatives in the $x$, $y$, and $z$ directions in a Cartesian coordinate frame.

The continuous, weighted bilinear and linear forms given in \eqref{eq:ContWeightBilinear} and \eqref{eq:ContWeightLinear} for an approximate solution, $\mathbf{u}_h\in V_h(\Omega)$, are
\begin{align}
\begin{split}
    &\tilde{\alpha}(\mathbf{u}_h,\mathbf{v}) = -\int_{\Omega} \left(\mathcal{L}\mathbf{u}_h\right)\cdot\left(\mathcal{L}\mathbf{v}\right)d\mathbf{y} + \beta\int_{\partial\Omega}\left(\kappa_1\mathcal{B}_1\mathbf{u}_h + \gamma\kappa_0\mathcal{B}_0\mathbf{u}_h\right)\cdot\left(\kappa_1\mathcal{B}_1\mathbf{v} + \gamma\kappa_0\mathcal{B}_0\mathbf{v}\right)d\mathbf{y}, \\
     &\tilde{l}(\mathbf{v}) = \int_{\Omega} \mathbf{f}\cdot\left(\mathcal{L}\mathbf{v}\right)d\mathbf{y} + \beta\int_{\partial\Omega}(\kappa_1\mathbf{t} + \gamma\kappa_0\mathbf{g})\cdot\left(\kappa_1\mathcal{B}_1\mathbf{v} + \gamma\kappa_0\mathcal{B}_0\mathbf{v}\right)d\mathbf{y}, \hspace{1cm} \forall\mathbf{v}\in V_h(\Omega).
\end{split}
\end{align}
We obtain the weighted weak formulation by equating the bilinear and linear forms. We additionally insert $\mathbf{u}_h = (\sum_{j=1}^{N}U_x(\mathbf{x}_j)\psi_j,~\sum_{j=1}^{N}U_y(\mathbf{x}_j)\psi_j,~\sum_{j=1}^{N}U_z(\mathbf{x}_j)\psi_j)^T$, see \eqref{eq:Global_interpolant}, 
and set the test function $\mathbf{v} = I\psi_i$, $i = 1,...,N$, where $I$ is the identity matrix, to get
\begin{align}
\begin{split}
     -&\sum_{j=1}^N\left(\int_{\Omega} \left(\mathcal{L}\psi_i\right)^T\left(\mathcal{L}\psi_j\right)d\mathbf{y}\right)\begin{pmatrix}
        U_x(\mathbf{x}_j) \\ U_y(\mathbf{x}_j) \\ U_z(\mathbf{x}_j)
    \end{pmatrix} + \\ & \beta\sum_{j=1}^N\left(\int_{\partial\Omega} \left(\kappa_1\mathcal{B}_1\psi_i + \gamma\kappa_0\mathcal{B}_0\psi_i\right)^T\left(\kappa_1\mathcal{B}_1\psi_j + \gamma\kappa_0\mathcal{B}_0\psi_j\right)d\mathbf{y}\right)\begin{pmatrix}
        U_x(\mathbf{x}_j) \\ U_y(\mathbf{x}_j) \\ U_z(\mathbf{x}_j)
    \end{pmatrix}  = \\ 
    & \int_{\Omega} \left(\mathcal{L}\psi_i\right)^T\mathbf{f}d\mathbf{y} + \beta\int_{\partial\Omega}\left(\kappa_1\mathcal{B}_1\psi_i + \gamma\kappa_0\mathcal{B}_1\psi_i\right)^T\left(\kappa_1\mathbf{t} + \gamma\kappa_0\mathbf{g}\right)d\mathbf{y}, \hspace{1cm} i = 1,...,N.
\end{split}
\label{eq:exactInt}
\end{align}
This can be expressed as the linear system of equations
\begin{equation}
    K \underline{U} = b,
    \label{eq:ExactIntLinearSystem}
\end{equation}
where $\underline{U} = (\underline{U}_x^T,\underline{U}_y^T,\underline{U}_z^T)^T$, with  $\underline{U}_x=(U_x(\mathbf{x}_1), \ldots, U_x(\mathbf{x}_N))^T$, and with the other two vectors similarly constructed, and $K$ is the stiffness matrix.

In practice, we approximate~\eqref{eq:ExactIntLinearSystem} using numerical quadrature. For a least-squares problem, the accuracy of the quadrature does not  affect the convergence order of the method. Therefore, we approximate the integrals using a simple, low-order quasi-Monte Carlo method.
The resulting bilinear and linear forms \eqref{eq:DiscBilinear} and \eqref{eq:DiscLinear} satisfy a discrete orthogonality property. We use the density of the point set $Y$ in the domain and on the boundary to construct the quadrature weights. For each quadrature point $\mathbf{y}_k$, we approximate the node density using the approximate local fill distances~\eqref{eq:fillYapprox} and \eqref{eq:fillYBndApprox}
\begin{align}
    \begin{split}
            W^{\Omega}_k & \triangleq \left|\Omega\right|\frac{\tilde{h}_k^d}{\sum_{j\in I_{\Omega}}\tilde{h}_j^d} , \hspace{1.6cm} \forall k\in I_{\Omega}, \\
            W^{\partial\Omega}_k & \triangleq \left|\partial\Omega\right|\frac{\tilde{h}_{k,\partial\Omega}^{d-1}}{\sum_{j\in I_{\partial\Omega}}\tilde{h}_{j,\partial\Omega}^{d-1}}, \hspace{0.75cm}  \forall k\in I_{\partial\Omega}.
            \label{eq:integrationWeights}
    \end{split}
\end{align}
Using the discrete norm (quadrature)~\eqref{eq:l2},
we approximate~\eqref{eq:exactInt} by
\begin{align}
    \begin{split}
         -&\left(\sum_{k \in I_{\Omega}}\left(\mathcal{L}\psi_i(\mathbf{y}_k)\right)^TW^{\Omega}_k\left(\mathcal{L}\psi_j(\mathbf{y}_k)\right)\right) \begin{pmatrix}
        U_x(\mathbf{x}_j) \\ U_y(\mathbf{x}_j) \\ U_z(\mathbf{x}_j)
        \end{pmatrix} + \\
        & \beta\left(\sum_{k \in I_{\partial\Omega}}\left(\kappa_1\mathcal{B}_1\psi_i(\mathbf{y}_k) + \gamma\kappa_0\mathcal{B}_0\psi_i(\mathbf{y}_k)\right)^TW^{\partial\Omega}_k\left(\kappa_1\mathcal{B}_1\psi_j(\mathbf{y}_k) + \gamma\kappa_0\mathcal{B}_0\psi_j(\mathbf{y}_k)\right)\right) \begin{pmatrix}
        U_x(\mathbf{x}_j) \\ U_y(\mathbf{x}_j) \\ U_z(\mathbf{x}_j) \end{pmatrix} = \\
        & \left(\sum_{k \in I_{\Omega}}\left(\mathcal{L}\psi_i(\mathbf{y}_k)\right)^TW^{\Omega}_k\,\mathbf{f}(\mathbf{y}_k)\right) + \\ & \beta\left(\sum_{k \in I_{\partial\Omega}}\left(\kappa_1\mathcal{B}_1\psi_i(\mathbf{y}_k) + \gamma\kappa_0\mathcal{B}_0\psi_i(\mathbf{y}_k)\right)^TW^{\partial\Omega}_k\left(\kappa_1\mathbf{t}(\mathbf{y}_k) + \gamma\kappa_0\mathbf{g}(\mathbf{y}_k)\right)\right), \hspace{1cm} i = 1,...,N,
    \end{split}
\end{align}
which can be expressed as a linear system
{\footnotesize
\begin{equation}
    \begin{pmatrix}
        D^{\mathcal{L}} \\ \kappa_1D^{\mathcal{B}_1} + \kappa_0D^{\mathcal{B}_0}
    \end{pmatrix}^T
    \begin{pmatrix}
        W^{\Omega} & 0 \\
        0 & W^{\partial\Omega}
    \end{pmatrix}
    \begin{pmatrix}
        D^{\mathcal{L}} \\ \kappa_1D^{\mathcal{B}_1} + \kappa_0D^{\mathcal{B}_0}
    \end{pmatrix} 
    \begin{pmatrix}
        \underline{U}_x \\ \underline{U}_y \\ \underline{U}_z
    \end{pmatrix} = 
    \begin{pmatrix}
        D^{\mathcal{L}} \\ \kappa_1D^{\mathcal{B}_1} + \kappa_0D^{\mathcal{B}_0}
    \end{pmatrix}^T
    \begin{pmatrix}
        W^{\Omega} & 0 \\
        0 & W^{\partial\Omega}
    \end{pmatrix}
    \begin{pmatrix}
        \underline{\mathbf{f}} \\ \kappa_1\underline{\mathbf{t}} + \kappa_0\underline{\mathbf{g}}
    \end{pmatrix},
    \label{eq:normalEqutions}
\end{equation}}where $D^{\mathcal{L}},~D^{\mathcal{B}_1},~D^{\mathcal{B}_0}$ are matrices of sizes $(3M_{\Omega}\times 3N)$, $(3M_{\partial\Omega}\times 3N)$, and $(3M_{\partial\Omega}\times 3N)$, respectively, corresponding to the discretised PDE operator and boundary operators. The diagonal weight matrices $W^{\Omega}$ and $W^{\partial\Omega}$ have sizes $(3M_{\Omega}\times 3M_{\Omega})$ and $(3M_{\partial\Omega}\times3M_{\partial\Omega})$ respectively. The above linear system corresponds to the normal equations of the least-squares problem solved for $\underline{U}$
\begin{equation}
    \min_{\underline{U}}\|D\underline{U} - F\|_2
    \label{eq:leastSquaresProb}
\end{equation}
where $D$ is a $(3M \times 3N)$ rectangular matrix and $F = (\underline{\mathbf{f}}^T,~\underline{\mathbf{t}}^T,~\underline{\mathbf{g}}^T)^T$ is a $(3M \times 1)$ vector, and the weights $\beta$ and $\gamma$ together with the quadrature weights provide the scaling of the weighted least-squares problem. 
\section{Theoretical analysis}\label{sec:theory}
We construct convergence error estimates for the LS-RBF-PU method applied to the weak form least-squares linear elasticity problem~\eqref{eq:LSQWF}. This includes a proof of well-posedness for the continuous problem followed by a best approximation estimate for the discrete problem. Employing sampling inequalities from the literature, we then derive an a-priori error estimate for our method. This estimate relies on the existence of a discrete coercivity estimate, which is derived by proving norm equivalence between the discrete and continuous norms and bilinear forms under sufficient oversampling. 

\subsection{Continuous analysis}\label{sec:theoryCont}
To the authors' knowledge, an explicit well-posedness proof for the second order least squares weak form of linear elasticity was not available in the literature, hence we also prove well-posedness of the continuous problem here. According to the Lax-Milgram Theorem \cite{LaxMilgramTheorem},~\cite[Theorem 2.7.7]{brenner2008mathematical} a unique solution $\mathbf{u}\in H^2(\Omega)$ exists if the bilinear form is coercive and continuous and the linear form is continuous. For the coercivity estimate we rely on estimates for elliptic PDEs in the Agmon-Douglis-Nirenberg setting \cite{ADN_2} and trace inequalities for the continuity estimates. Some results derived in this section require specific boundary regularity, which we assume is always available due to the geometry reconstruction method based on $C^{\infty}$ weight functions and infinitely smooth RBFs in the RBF-PU method described in \cite{larsson2024rbfpartitionunitymethod}. 

To show continuity of the linear and bilinear forms, we use a number of inequalities involving the PDE operator, vector valued functions, and the fractional norm. 
First, we have the trace inequality~\cite[Theorem 5.5]{necas2011direct} adapted for vector valued functions, $\mathbf{v}\in H^{2}(\Omega)$ 
\begin{align}
\begin{split}
    \left\|\nabla\mathbf{v}\cdot\mathbf{n}\right\|_{H^{1/2}(\partial\Omega)}\leq C_T\left\|\mathbf{v}\right\|_{H^{2}(\Omega)}, \\
    \left\|\mathbf{v}\right\|_{H^{1/2}(\partial\Omega)} \leq C'_T\left\|\mathbf{v}\right\|_{H^1(\Omega)},
    \end{split}
    \label{eq:Theorem5.5}
\end{align}
which hold when $\partial\Omega$ is of class $C^1$. 
Then, we use \cite[Lemma 5.3]{FractionalSobolev2012}, for $v\in H^{1/2}(\Omega)$
\begin{align}
\left\|\kappa v\right\|_{H^{1/2}(\Omega)}\leq C_h\left\|v\right\|_{H^{1/2}(\Omega)},
\label{eq:Lemma5.3}
\end{align}
which holds for any function $\kappa\in C^0(\Omega)$ and $\kappa:\Omega\rightarrow\left[-1,1\right]$. We note that the proof of \cite[Lemma 5.3]{FractionalSobolev2012} assumes that $\kappa:\Omega\rightarrow\left[0,1\right]$, but the range can be easily extended to $\left[-1,1\right]$. We use this inequality in a different setting, for norms over the boundary $\partial\Omega$. The components used in the proof should however carry over also to this case.  We show in Appendix \ref{app:ContinuityBounds} that similar inequalities hold for vector and tensor valued functions as follows:
\begin{align}
    \begin{split}
        \left\|\left(\nabla\cdot\mathbf{v}I\right)\cdot\mathbf{n}\right\|_{H^{1/2}(\partial\Omega)} &\leq C_f\left\|\nabla\cdot\mathbf{v}\right\|_{H^{1/2}(\partial\Omega)}, \\
        \left\|\nabla^a\mathbf{v}\cdot\mathbf{n}\right\|_{H^{1/2}(\partial\Omega)} &\leq C'_f\left\|\nabla^a\mathbf{v}\right\|_{H^{1/2}(\partial\Omega)}, 
    \end{split}
    \label{eq:Lemma5.3_modified}
\end{align}
where $\mathbf{n}$ is the normal to the boundary $\partial\Omega$, $\left\|\mathbf{n}\right\|_2 = 1$ and $\nabla^a\mathbf{v} = \nabla\mathbf{v} - \epsilon(\mathbf{v})$ is the antisymmetric gradient operator. 
Finally, we construct the following estimates, with their detailed derivation provided in Appendix \ref{app:ContinuityBounds}:
\begin{alignat}{9}
        & \left\|\nabla\cdot\mathbf{v}\right\|_{L_2(\Omega)}
        &&\leq 
        &\,d^{1/2}\left\|\nabla\mathbf{v}\right\|_{L_2(\Omega)}, 
        \quad\quad 
        &\left\|\nabla\left(\nabla\cdot\mathbf{v}\right)\right\|_{L_2(\Omega)} 
        &&\leq 
        &\,d^{1/2}\left\|D^2\mathbf{v}\right\|_{L_2(\Omega)}, \nonumber\\
        &\left\|\nabla^a\mathbf{v}\right\|_{L_2(\Omega)}
        &&\leq
        &\left\|\nabla\mathbf{v}\right\|_{L_2(\Omega)}, 
        \quad\quad  
        &\left\|\nabla\nabla^a\mathbf{v}\right\|_{L_2(\Omega)} 
        &&\leq 
        &\left\|D^2\mathbf{v}\right\|_{L_2(\Omega)}, \label{eq:generalBounds}\\
        & \left\|\epsilon(\mathbf{v})\right\|_{L_2(\Omega)} 
        &&\leq 
        &\left\|\nabla\mathbf{v}\right\|_{L_2(\Omega)}, 
        \quad\quad 
        &\left\|\nabla\cdot\epsilon(\mathbf{v})\right\|_{L_2(\Omega)}
        &&\leq 
        &\,d^{1/2}\left\|D^2\mathbf{v}\right\|_{L_2(\Omega)}, \nonumber  
\end{alignat}
where $D^2\mathbf{v}$ is the Hessian of $\mathbf{v}$, which is a tensor with elements $\partial_{kj}v_i$, for $i,j,k = 1,\ldots,d$.  
\begin{theorem}
    The bilinear and linear forms, $\alpha(\cdot,\cdot)$ and $l(\cdot)$, from \eqref{eq:LSQWF} are both continuous in the sense that 
    \begin{alignat}{3}
        \left|\alpha\left(\mathbf{u},\mathbf{v}\right)\right|&\leq C_2\left\|\mathbf{u}\right\|_{H^2(\Omega)}\left\|\mathbf{v}\right\|_{H^2(\Omega)},\quad\quad\quad  &\forall\mathbf{u},\mathbf{v}&\in H^2(\Omega), \\
        \left|l(\mathbf{v})\right| &\leq C_3\left\|\mathbf{v}\right\|_{H^2(\Omega)},\quad\quad\quad &\forall \mathbf{v}&\in H^2(\Omega),
    \end{alignat}
where $C_2$ is a constant that depends on the Lam\'e parameters $\mu$ and $\lambda$, the domain $\Omega$, the dimension $d$, the Robin boundary condition coefficients $\kappa_0$, $\kappa_1$ and the outward boundary normal, $\mathbf{n}$. The constant $C_3$ additionally depends on the data $\mathbf{f}$, $\mathbf{t}$, and $\mathbf{g}$.
    \label{theorem:Continuity}
\end{theorem}
\begin{proof}
    We start by using the Cauchy-Schwarz inequality 
    and the triangle inequality to get
    \begin{align}
        \begin{split}
            \left|\alpha\left(\mathbf{u},\mathbf{v}\right)\right| 
            = & \left|-\left(\nabla\cdot\sigma(\mathbf{u}),\nabla\cdot\sigma(\mathbf{v})\right)_{L_2(\Omega)} + \left(\kappa_0\mathbf{u} + \kappa_1\sigma(\mathbf{u})\cdot\mathbf{n},\kappa_0\mathbf{v} + \kappa_1\sigma(\mathbf{v})\cdot\mathbf{n}\right)_{H^{1/2}(\partial\Omega)} \right| \\
            \leq & \left\|\nabla\cdot\sigma(\mathbf{u})\right\|_{L_2(\Omega)}\left\|\nabla\cdot\sigma(\mathbf{\mathbf{v}})\right\|_{L_2(\Omega)} + \left\|\kappa_0\mathbf{u} + \kappa_1\sigma(\mathbf{u})\cdot\mathbf{n}\right\|_{H^{1/2}(\partial\Omega)}\left\|\kappa_0\mathbf{v} + \kappa_1\sigma(\mathbf{v})\cdot\mathbf{n}\right\|_{H^{1/2}(\partial\Omega)}.
            \label{eq:continuityBound1}
        \end{split}
    \end{align}
    We find bounds for the domain and boundary terms separately and then combine the estimates. Starting with the domain term, we use the constitutive law \eqref{eq:ConstitutiveHook}, noting that the trace of the strain is equal to the divergence of the displacement. The inequalities \eqref{eq:generalBounds} are then applied to bound both terms in the norm of the Hessian operator as
    \begin{align}
        \begin{split}
            \left\|\nabla\cdot\sigma(\mathbf{u})\right\|_{L_2(\Omega)} = & \left\|2\mu\nabla\cdot\epsilon(\mathbf{u}) + \lambda\nabla\cdot\left(\nabla\cdot\mathbf{u}I\right)\right\|_{L_2(\Omega)} \leq 2\mu\left\|\nabla\cdot\epsilon(\mathbf{u})\right\|_{L_2(\Omega)} + \lambda\left\|\nabla\cdot\left(\nabla\cdot\mathbf{u}I\right)\right\|_{L_2(\Omega)}
            \\ \leq & 
            \,2d^{1/2}\mu\left\|D^2\mathbf{u}\right\|_{L_2(\Omega)} + d^{1/2}\lambda\left\|D^2\mathbf{u}\right\|_{L_2(\Omega)} =  d^{1/2}\left(2\mu + \lambda\right)\left\|D^2\mathbf{u}\right\|_{L_2(\Omega)} \\ 
            \leq & \,d^{1/2}\left(2\mu + \lambda\right)\left\|\mathbf{u}\right\|_{H^{2}(\Omega)},        \label{eq:domainNormContinuity}  
        \end{split}
    \end{align}
    where, in the final step, we used the Sobolev norm definition \eqref{eq:Sobolev_Spaces}. Similarly we try to find a bound for the boundary norm in terms of the same Sobolev norm. We start by using the fractional Sobolev bound \eqref{eq:Lemma5.3}, the constitutive law \eqref{eq:ConstitutiveHook}, the definition for the symmetric gradient tensor $\epsilon(\mathbf{v}) = \nabla\mathbf{v} - \nabla^{\alpha}\mathbf{v}$, and the inequalities \eqref{eq:Lemma5.3_modified}, to get
    \begin{align}
        \begin{split}
            \left\|\kappa_0\mathbf{u} + \kappa_1\sigma(\mathbf{u})\cdot\mathbf{n}\right\|_{H^{1/2}(\partial\Omega)}\leq &\left\|\kappa_0\mathbf{u}\right\|_{H^{1/2}(\partial\Omega)} + \left\|\kappa_1\sigma(\mathbf{u})\cdot\mathbf{n}\right\|_{H^{1/2}(\partial\Omega)} \\
            \leq & \,C_{\kappa_0}\left\|\mathbf{u}\right\|_{H^{1/2}(\partial\Omega)} + C_{\kappa_1}\left\|\sigma(\mathbf{u})\cdot\mathbf{n}\right\|_{H^{1/2}(\partial\Omega)} \\
            \leq & \,C_{\kappa_0}\left\|\mathbf{u}\right\|_{H^{1/2}(\partial\Omega)} \\ & + C_{\kappa_1}\left\|2\mu\nabla\mathbf{u}\cdot\mathbf{n} - 2\mu\nabla^{\alpha}\mathbf{u}\cdot\mathbf{n}+\lambda\left(\nabla\cdot\mathbf{u} I\right)\cdot\mathbf{n}\right\|_{H^{1/2}(\partial\Omega)} \\
            \leq & \,C_{\kappa_0}\left\|\mathbf{u}\right\|_{H^{1/2}(\partial\Omega)} + 2\mu C_{\kappa_1}\left\|\nabla\mathbf{u}\cdot\mathbf{n}\right\|_{H^{1/2}(\partial\Omega)} \\ & + 2\mu C_{\kappa_1}\left\|\nabla^{\alpha}\mathbf{u}\cdot\mathbf{n}\right\|_{H^{1/2}(\partial\Omega)} + \lambda C_{\kappa_1}\left\|\left(\nabla\cdot\mathbf{u} I\right)\cdot\mathbf{n}\right\|_{H^{1/2}(\partial\Omega)} \\
            \leq & \,C_{\kappa_0}\left\|\mathbf{u}\right\|_{H^{1/2}(\partial\Omega)} + 2\mu C_{\kappa_1}\left\|\nabla\mathbf{u}\cdot\mathbf{n}\right\|_{H^{1/2}(\partial\Omega)} \\ & + 2\mu C'_fC_{\kappa_1}\left\|\nabla^{\alpha}\mathbf{u}\right\|_{H^{1/2}(\partial\Omega)} + \lambda C_fC_{\kappa_1}\left\|\nabla\cdot\mathbf{u} \right\|_{H^{1/2}(\partial\Omega)}.
        \end{split}
        \label{eq:bndStart}
    \end{align}
    Moreover, we use the trace inequalities \eqref{eq:Theorem5.5} to bound the boundary norms by norms over the domain as 
    \begin{align}
        \begin{split}
            \left\|\kappa_0\mathbf{u} + \kappa_1\sigma(\mathbf{u})\cdot\mathbf{n}\right\|_{H^{1/2}(\partial\Omega)} \leq & C'_TC_{\kappa_0}\left\|\mathbf{u}\right\|_{H^{1}(\Omega)} + 2\mu C_TC_{\kappa_1}\left\|\mathbf{u}\right\|_{H^{2}(\Omega)} \\ & + 2\mu C'_TC'_fC_{\kappa_1}\left\|\nabla^{\alpha}\mathbf{u}\right\|_{H^{1}(\Omega)} + \lambda C'_TC_fC_{\kappa_1}\left\|\nabla\cdot\mathbf{u} \right\|_{H^{1}(\Omega)}.
        \end{split}
    \end{align}
    We note that the following inequalities can be derived using \eqref{eq:generalBounds}: 
    \begin{align}
        \begin{split}
            \left\|\nabla^a\mathbf{v}\right\|^2_{H^1(\Omega)} & = \left\|\nabla^a\mathbf{v}\right\|^2_{L_2(\Omega)} + \left\|\nabla\nabla^a\mathbf{v}\right\|^2_{L_2(\Omega)} \\
            & \leq \left\|\nabla\mathbf{v}\right\|^2_{L_2(\Omega)} + \left\|D^2\mathbf{v}\right\|^2_{L_2(\Omega)} \leq \left\|\mathbf{v}\right\|^2_{H^2(\Omega)} \\
            \left\|\nabla\cdot\mathbf{v}\right\|^2_{H^1(\Omega)} & = \left\|\nabla\cdot\mathbf{v}\right\|^2_{L_2(\Omega)} + \left\|\nabla\left(\nabla\cdot\mathbf{v}\right)\right\|^2_{L_2(\Omega)} \\
            & \leq d\left(\left\|\nabla\mathbf{v}\right\|^2_{L_2(\Omega)} + \left\|D^2\mathbf{v}\right\|^2_{L_2(\Omega)}\right) \leq d\left\|\mathbf{v}\right\|^2_{H^2(\Omega)}.
        \end{split}
        \label{eq:extras}
    \end{align} 
    Finally, we use~\eqref{eq:bndStart}--\eqref{eq:extras} to get the final estimate for the boundary term
    \begin{align}
        \begin{split}
            \left\|\kappa_0\mathbf{u} + \kappa_1\sigma(\mathbf{u})\cdot\mathbf{n}\right\|_{H^{1/2}(\partial\Omega)} \leq & \left(C'_TC_{\kappa_0} + 2\mu C_TC_{\kappa_1}+ 2\mu C'_TC'_fC_{\kappa_1} + d^{1/2}\lambda C'_TC_fC_{\kappa_1}\right)\left\|\mathbf{u}\right\|_{H^{2}(\Omega)}.
            \label{eq:boundaryNormContinuity}
        \end{split}
    \end{align}
    We then insert the estimates \eqref{eq:domainNormContinuity} and \eqref{eq:boundaryNormContinuity} back into the initial continuity bound \eqref{eq:continuityBound1}, which gives
    \begin{align}
        \begin{split}
            \left|\alpha\left(\mathbf{u},\mathbf{v}\right)\right| \leq & \left(2d^{1/2}\mu + d^{1/2}\lambda\right)^2\left\|\mathbf{u}\right\|_{H^{2}(\Omega)}\left\|\mathbf{v}\right\|_{H^{2}(\Omega)} \\
            & + \left(C'_TC_{\kappa_0} + 2\mu C_TC_{\kappa_1}+ 2\mu C'_TC'_fC_{\kappa_1} + d^{1/2}\lambda C'_TC_fC_{\kappa_1}\right)^2\left\|\mathbf{u}\right\|_{H^2(\Omega)}\left\|\mathbf{v}\right\|_{H^2(\Omega)} \\
            = & \,C_2\left\|\mathbf{u}\right\|_{H^2(\Omega)}\left\|\mathbf{v}\right\|_{H^2(\Omega)},
        \end{split}
    \end{align}
where the constant $C_2 = C_2(C_{\kappa_0},C_{\kappa_1},C_f,C'_f,C_T,C'_T,d,\lambda,\mu)$.
Using similar techniques to bound the linear form results in the following estimate:
\begin{align}
    \begin{split}
        \left|l\left(\mathbf{v}\right)\right| \leq & \left(d^{1/2}\left(2\mu + \lambda\right)\left\|\mathbf{f}\right\| + \left(\left(C_{\kappa_0}\right)^2 + 2\mu C_{\kappa_0}C_{\kappa_1}C_T + 2\mu C_{\kappa_0}C_{\kappa_1}C'_fC'_T + d^{1/2}\lambda \left(C_{\kappa_1}\right)^2C_f C'_T\right)\right. \\
        & \left.  \left(\left\|\mathbf{g}\right\|_{H^{1/2}(\partial\Omega)} + \left\|\mathbf{t}\right\|_{H^{1/2}(\partial\Omega)}\right)\right)\left\|\mathbf{v}\right\|_{H^2(\Omega)} = C_3\left\|\mathbf{v}\right\|_{H^2(\Omega)},
    \end{split}
\end{align}
where the constant $C_3 = C_3(C_{\kappa_0},C_{\kappa_1},C_f,C'_f,C_T,C'_T,d,\lambda,\mu,\mathbf{f},\mathbf{g},\mathbf{t})$.
\end{proof}


\begin{theorem}
    Assume that $\Omega$ is a bounded domain in $\mathbb{R}^d$ with a smooth boundary, $\partial\Omega\in C^2$. Additionally assume the coefficients of the continuous differential operator \eqref{eq:StrongPDE} are of class $C(\bar{\Omega})$ and the coefficients of the boundary operator \eqref{eq:StrongBC} are of class $C^1(\partial\Omega)$. Further assume the data, $\mathbf{f}\in L_2(\Omega),~\mathbf{g}\in H^{1/2}(\partial\Omega),~\mathbf{t}\in H^{1/2}(\partial\Omega)$. Then the boundary value problem \eqref{eq:StrongPDE},~\eqref{eq:StrongBC} is ADN elliptic, uniformly elliptic, and regular elliptic, the boundary conditions are complementary, and the boundary value problem has a unique solution. The following coercivity bound for the bilinear form, $\alpha(\cdot,\cdot)$ from \eqref{eq:LSQWF}, holds
    \begin{align}
        & C_1\alpha\left(\mathbf{v},\mathbf{v}\right)\geq \left\|\mathbf{v}\right\|^2_{H^2(\Omega)}, \hspace{0.5cm} \forall \mathbf{v}\in H^2(\Omega),
    \end{align}
    where $C_1$ is a constant that depends on the Lam\'e parameters $\mu$ and $\lambda$, the domain $\Omega$, the dimension $d$, the Robin boundary condition coefficients $\kappa_0$ and $\kappa_1$, their derivatives $\nabla\kappa_0$ and $\nabla\kappa_1$, and the boundary $\partial\Omega$.          
    \label{theorem:Coercivity}
\end{theorem}
\begin{proof}
The above is a reformulated theorem coming from \cite[Theorem 10.5]{ADN_2}, \cite[Theorem D.1]{bochev2009least}, and \cite{JaRoĭtberg_1969}, but assuming that we have the minimal smoothness in our coefficients and data, such that a solution $\mathbf{u}\in H^2(\Omega)$ exists. Smoother solutions are possible given increased boundary, coefficient, and data smoothness. The proof involves showing uniqueness of the strong elasticity problem \eqref{eq:StrongPDE}, \eqref{eq:StrongBC}, as well as showing that the differential operator is ADN elliptic, uniformly elliptic, and regular elliptic, and that the boundary conditions are complementary. All definitions are provided in \cite{ADN_2}, \cite[Appendix D]{bochev2009least} and are explicitly shown for the linear elasticity problem in Appendix \ref{app:CoercivityBounds}. 
\end{proof}

\begin{theorem}
    If Theorems \ref{theorem:Continuity} and \ref{theorem:Coercivity} hold, then there exists a unique solution $\mathbf{u}\in H^2(\Omega)$  to 
    \begin{equation}
        \alpha\left(\mathbf{u},\mathbf{v}\right) = l(\mathbf{v}), \hspace{0.5cm}\forall \mathbf{v}\in H^{2}(\Omega),
    \end{equation}
    with $\alpha(\cdot,\cdot)$, $l(\cdot)$ as defined in \eqref{eq:ContBilinear} and \eqref{eq:ContLinear}.
    \label{theorem:WellPosedness}
\end{theorem}
\begin{proof}
    Given the continuity of the linear and bilinear forms from Theorem \ref{theorem:Continuity} and the coercivity from Theorem \ref{theorem:Coercivity}, we conclude well-posedness in view of the Lax-Milgram theorem \cite{LaxMilgramTheorem},~\cite[Theorem 2.7.7]{brenner2008mathematical}. 
\end{proof}

\subsection{Discrete analysis}\label{sec:theoryDisc}

In this section we derive an a-priori error estimate for the solution $\mathbf{u}_h\in V_h(\Omega)$ to the discrete least-squares problem \eqref{eq:disc_WF} as an approximation of the solution $\mathbf{u}\in H^2(\Omega)$ to the
continuous least-squares problem \eqref{eq:LSQWF}. 

Before proceeding with the discrete error analysis, we define the finite dimensional space $V_h(\Omega)$ and its properties. The space is a direct sum of the space spanned by the global basis functions in each dimension,
\begin{equation}
    V_h(\Omega) := \underset{1\leq j \leq d}{\bigoplus}~\text{span}\{\psi_i\}_{i=1}^N,
    \label{eq:FDspace}
\end{equation}
where $\oplus$ is the direct sum, $d$ is the dimension, $\psi_i$ are the functions defined in \eqref{eq:globalCardinalFunc}, associated with the discrete set $X$, with fill distance defined as in \eqref{eq:fillX}. 
We denote a vector valued function from this space by $\mathbf{v}\in V_h(\Omega)$.


The global function can also be expressed as $\mathbf{v = }\sum_{k=1}^Pw_k\mathbf{v}^{(k)}$, where
$\mathbf{v}^{(k)}(\mathbf{y}) = \tilde{\mathbf{v}}^{(k)}(T_k(\mathbf{y})) = (\tilde{v}_1^{(k)},\dots,\tilde{v}_d^{(k)})^T$ is the local approximation in patch $\Omega_k$. 
Each component function $\tilde{v}_i^{(k)}$ is constructed on the reference patch $\Omega_0$ \eqref{eq:RBFinterpolantLocal} using 
the positive definite Gaussian kernel $\phi(r) = e^{-\varepsilon^2r^2}$. Therefore, $\tilde{v}^{(k)}_i\in\mathcal{N}_{G}(\Omega_0)$, where $\mathcal{N}_{G}$ is the reproducing kernel Hilbert space (RKHS) of the Gaussian defined in~\eqref{eq:NativeSpace}, see also~\cite[Definition 10.9]{Wendland_2004} for further details on its properties. 

Since the local RBF approximations are constructed on the reference patch, we choose to express bounds of physical derivatives of $\mathbf{v}^{(k)}$ in terms of derivatives taken in the reference patch. We consider the pointwise tensor norm defined in Section \ref{sec:Definitions} and apply the chain rule, noting that $T_k(\mathbf{y})$ is a linear transformation, to get
\begin{align}
    \begin{split}
        \|D^{\tau}_{\mathbf{y}}\mathbf{v}^{(k)}(\mathbf{y})\|^2_T &= \|D^{\tau}_{\mathbf{y}}\tilde{\mathbf{v}}^{(k)}(\mathbf{y}')\|_T^2 = \sum_{i=1}^d\sum_{j_1,\dots,j_{\tau}=1}^d\left|\frac{\partial^{\tau}\tilde{v}_i^{(k)}(\mathbf{y}')}{\partial y_{j_1}\dots\partial y_{j_{\tau}}}\right|^2 \\
        & = \sum_{i=1}^d\sum_{j_1,\dots,j_{\tau}=1}^d\left|\sum_{k_1,\dots,k_{\tau}=1}^d
        \frac{\partial^{\tau}\tilde{v}_i^{(k)}(\mathbf{y}')}{\partial y'_{k_1}\dots\partial y'_{k_{\tau}}}\frac{\partial y'_{k_1}}{\partial y_{j_1}}\dots\frac{\partial y'_{k_{\tau}}}{\partial y_{j_{\tau}}}\right|^2 \\ 
        & \leq d^{\tau}\sum_{i=1}^d\sum_{j_1,\dots,j_{\tau}=1}^d\sum_{k_1,\dots,k_{\tau}=1}^d\left|\frac{\partial^{\tau}\tilde{v}_i^{(k)}(\mathbf{y}')}{\partial y'_{k_1}\dots\partial y'_{k_{\tau}}}\frac{\partial T_k(y_{k_1})}{\partial y_{j_1}}\dots\frac{\partial T_k(y_{k_{\tau}})}{\partial y_{j_{\tau}}}\right|^2 \\
        & \leq d^{\tau}\|J_{T_k}\|_T^{2\tau}\|D^{\tau}_{\mathbf{y}'}\tilde{\mathbf{v}}^{(k)}(\mathbf{y}')\|^2_T,
        \label{eq:derivativePhysical2Ref_derivation}
    \end{split}
\end{align}
where $\tau$ is the derivative order, $D^{\tau}\mathbf{v}^{(k)}$ is a tensor of order $\tau+1$, $J_{T_k}$ is the Jacobian of the transformation matrix and is given by $J_{T_k} = S_kQ_k$ \eqref{eq:referencePatchMap} and $D^{\tau}_{\mathbf{y}'},~D^{\tau}_{\mathbf{y}}$ are derivative operators taken in the reference and physical coordinates, respectively. 
%
%
Since the transformation is linear, we similarly have
\begin{equation}
\|D^{\tau}_{\mathbf{y}'}\tilde{\mathbf{v}}^{(k)}(\mathbf{y}')\|^2_T\leq
d^{\tau}\|J^{-1}_{T_k}\|_T^{2\tau}\|D^{\tau}_{\mathbf{y}}\mathbf{v}^{(k)}(\mathbf{y})\|^2_T.
\label{eq:deriv2}
\end{equation}
To simplify the notation, we omit the subscripts on the derivative operators
and use the notational convention $D^{\tau}\tilde{\mathbf{v}}^{(k)} \triangleq D^{\tau}_{\mathbf{y}'}\tilde{\mathbf{v}}^{(k)}$ and
$D^{\tau}\mathbf{v}^{(k)} \triangleq D^{\tau}_{\mathbf{y}}\mathbf{v}^{(k)}$.
Taking the square root and essential supremum of~\eqref{eq:derivativePhysical2Ref_derivation} and~\eqref{eq:deriv2} gives
\begin{equation}
    C_{\mathrm{ref}}'(\tau)\|D^{\tau}\tilde{\mathbf{v}}^{(k)}\|_{L_{\infty}(\Omega_0)} \leq \|D^{\tau}\mathbf{v}^{(k)}\|_{L_{\infty}(\Omega_k)} \leq  d^{\tau}\|S_k\|_T^{\tau}\|D^{\tau}\tilde{\mathbf{v}}^{(k)}\|_{L_{\infty}(\Omega_0)} = C_{\mathrm{ref}}(\tau)\|D^{\tau}\tilde{\mathbf{v}}^{(k)}\|_{L_{\infty}(\Omega_0)},
    \label{eq:derivativePhysical2Ref}
\end{equation} 
where we used the fact that $\|Q_k\|^2_{T} = d$, $~C_{\mathrm{ref}}(\tau) = d^{\tau}(2~R_{\mathrm{sc},k}^2 + H_{\mathrm{sc},k}^2)^{\tau/2}$ for cylindrical patches, where $d=3$ and $R_{\mathrm{sc},k}$ and $H_{\mathrm{sc},k}$ are the radial and vertical scaling between the reference patch and $\Omega_k$. For the lower bound we similarly get $C_{\mathrm{ref}}'(\tau) = d^{-\tau}(2R_{\mathrm{sc},k}^{-2} + H_{\mathrm{sc},k}^{-2})^{-\tau/2}$. 

By integrating ~\eqref{eq:derivativePhysical2Ref_derivation} and~\eqref{eq:deriv2} over the reference patch and using the relation 
\begin{equation}
dV_k = \det(J^{-1}_{T_k})dV_0 = \det(S_k^{-1})\det(Q_k^T)dV_0 = \det(S_k^{-1})dV_0 = dV_0/(R_{\mathrm{sc},k}^2~H_{\mathrm{sc},k})
\end{equation}
for the volume elements~\cite[p. 146]{lai2009continuumMech}, we get the following bounds for the $L_2$-norms
\begin{align}
        R_{\mathrm{sc},k}^2~H_{\mathrm{sc},k}\|D^{\tau}\mathbf{v}^{(k)}\|^2_{L_2(\Omega_k)} &\leq d^{\tau}\|J_{T_k}\|_T^{2\tau}\|D^{\tau}\tilde{\mathbf{v}}^{(k)}\|^2_{L_2(\Omega_0)},\label{eq:derivativePhysicalL2a}\\
        \|D^{\tau}\tilde{\mathbf{v}}^{(k)}\|^2_{L_2(\Omega_0)}&\leq R_{\mathrm{sc},k}^2~H_{\mathrm{sc},k}
        d^{\tau}\|J^{-1}_{T_k}\|_T^{2\tau}\|D^{\tau}\mathbf{v}^{(k)}\|^2_{L_2(\Omega_k)},
\label{eq:derivativePhysicalL2b}
\end{align}
leading to
\begin{align}
    \begin{split}
        \frac{C'_{\mathrm{ref}}(\tau)}{~R_{\mathrm{sc},k}~H_{\mathrm{sc},k}^{\frac{1}{2}}}\|D^{\tau}\tilde{\mathbf{v}}^{(k)}\|_{L_{2}(\Omega_0)} \leq \|D^{\tau}\mathbf{v}^{(k)}\|_{L_{2}(\Omega_k)} \leq \frac{C_{\mathrm{ref}}(\tau)}{R_{\mathrm{sc},k}~H_{\mathrm{sc},k}^{\frac{1}{2}}}\|D^{\tau}\tilde{\mathbf{v}}^{(k)}\|_{L_{2}(\Omega_0)}. 
    \label{eq:derivativeRef2PhysicalL2Norms}
    \end{split}
\end{align}

For norms on the boundary, we integrate over the surface $\partial\Omega\cap\Omega_k$, where the boundary is implicitly parametrised using a level set function $\ell(\mathbf{y})$ as defined in \cite{larsson2024rbfpartitionunitymethod}. In this case we scale the area element on the surface of the boundary $d\mathbf{S}_k = \|\mathbf{n}dA\|_2 = \|dA_0\det(J_{T_k}^{-1})J_{T_k}\mathbf{n}_{0,k}\|_2$, using the derivation from \cite[p. 145]{lai2009continuumMech}. Note that $\mathbf{n}_{0,k}$ is the normal on the boundary $\partial\Omega\cap\Omega_k$ in the reference patch and the integral bound is given by
\begin{align}
    \begin{split}
        \int_{\partial\Omega\cap\Omega_k} \|D^{\tau}\mathbf{v}^{(k)}\|_T^2 d\mathbf{S}_k & = \int_{\partial\Omega_{0,k}\cap\Omega_0}\|D^{\tau}\mathbf{v}^{(k)}\|_T^2\|\det(J_{T_k}^{-1})J_{T_k}\mathbf{n}_{0,k}dA_{0,k}\|_2 \\
        & \leq \left|\det(J_{T_k}^{-1})\right|\|J_{T_k}\|_2\int_{\partial\Omega_{0,k}\cap\Omega_0}\|D^{\tau}\mathbf{v}^{(k)}\|_T^2d\mathbf{S}_{0,k} \\
        & = \left(R_{\mathrm{sc},k}^2~H_{\mathrm{sc},k}\right)^{-1}\max{\left(R_{\mathrm{sc},k},~H_{\mathrm{sc},k}\right)}\int_{\partial\Omega_{0,k}\cap\Omega_0}\|D^{\tau}\mathbf{v}^{(k)}\|_T^2 d\mathbf{S}_{0,k} \\
        & \leq C^2_{\mathrm{ref}}(\tau)\left(R_{\mathrm{sc},k}^2~H_{\mathrm{sc},k}\right)^{-1}\max{\left(R_{\mathrm{sc},k},~H_{\mathrm{sc},k}\right)}\int_{\partial\Omega_{0,k}\cap\Omega_0}\|D^{\tau}\tilde{\mathbf{v}}^{(k)}\|_T^2d\mathbf{S}_{0,k}.
        \label{eq:derivativeRef2PhysicalL2NormBnd_1}
    \end{split}
\end{align}
Similarly we use the transformation in the opposite direction which gives the following norm bounds on the boundary
\begin{align}
    \begin{split}
        C'_{r,\partial\Omega}(\tau)\|D^{\tau}\tilde{\mathbf{v}}^{(k)}\|_{L_2(\partial\Omega_{0,k}\cap\Omega_0)} \leq \|D^{\tau}\mathbf{v}^{(k)}\|_{L_2(\partial\Omega\cap\Omega_k)} \leq  C_{r,\partial\Omega}(\tau)\|D^{\tau}\tilde{\mathbf{v}}^{(k)}\|_{L_2(\partial\Omega_{0,k}\cap\Omega_0)} 
        \label{eq:derivativeRef2PhysicalL2NormBnd}
    \end{split}
\end{align}
where $C_{r,\partial\Omega}(\tau) = C_{\mathrm{ref}}(\tau)\max{\left(R_{\mathrm{sc},k},~H_{\mathrm{sc},k}\right)}^{\frac{1}{2}}/(R_{\mathrm{sc},k}~H_{\mathrm{sc},k}^{\frac{1}{2}}),~C'_{r,\partial\Omega}(\tau) = C'_{\mathrm{ref}}(\tau)\min{\left(R_{\mathrm{sc},k},~H_{\mathrm{sc},k}\right)}^{\frac{1}{2}}/(R_{\mathrm{sc},k}~H_{\mathrm{sc},k}^{\frac{1}{2}})$ and $\partial\Omega_{0,k}\cap\Omega_0$ is the boundary cut by patch $\Omega_k$ and scaled to the reference patch. 
 
Additionally, we consider derivatives of the $\gamma$-stable partition of unity weight functions \eqref{eq:Global_interpolant}, $w_k\in C^{\gamma}(\Omega_k)$, which are compactly supported on their respective patches, and sum to unity at each point in the domain $\Omega$. For each weight function, we have
\begin{equation}
    \|D^{\zeta}w_k\|_{L_{\infty}(\Omega_k)}\leq G_{|\zeta|}\delta_k^{- |\zeta|},
    \label{eq:stabilityWeightFunc}
\end{equation}
where $\zeta$ is a multi-index satisfying $|\zeta|\leq\gamma$, $G_{|\zeta|}>0$ is a constant and $\delta_k$ is the minimum absolute overlap between patch $k$ and its neighbouring patches. This definition is equivalent to the one provided in \cite{larsson2024rbfpartitionunitymethod} and deviates slightly from \cite[Definition 15.16]{Wendland_2004}. 

\begin{theorem}
    Given a finite dimensional space $V_h(\Omega)$ of vector-valued functions as defined in \eqref{eq:FDspace}, spanned by basis functions $\psi_i$ constructed as in \eqref{eq:globalCardinalFunc}, using positive definite kernels $\phi:\mathbb{R}_+\rightarrow \mathbb{R}$ on a reference patch,  with native Hilbert space $\mathcal{N}_{G}(\Omega_0)$ and $\gamma$-stable weight functions $w_k\in C^{\gamma}(\Omega_k)$ with property \eqref{eq:stabilityWeightFunc}, we have $V_h(\Omega)\subset H^2(\Omega)$ for $\gamma \geq 2$.
    \label{theorem:VhInH2}
\end{theorem}
\begin{proof}
    We show that any vector-valued function $\mathbf{v}\in V_h(\Omega)$, where $\Omega\in\mathbb{R}^d$, is bounded in $H^2(\Omega)$. For this we use the property \eqref{eq:stabilityWeightFunc} and the bound 
    \begin{equation}
        \|v\|_{H^2(\Omega)} \leq C_{\mathcal{N}_G}\|v\|_{\mathcal{N}_G(\Omega)},\hspace{0.5cm} \forall v\in\mathcal{N}_G,
        \label{eq:H2includedInN}
    \end{equation}
    from \cite[Theorem 7.5]{rieger2010sampling}, where $\Omega\in\mathbb{R}^d$ is a bounded domain and $C_{\mathcal{N}_G}$ depends on $d$ and the shape parameter $\varepsilon$ of $\phi$. Using the definition of a global partition of unity approximation \eqref{eq:global_func_approx}, the definition of the Sobolev norm \eqref{eq:Sobolev_norm} and the triangle inequality, we get
    \begin{align}
    \begin{split}
    \|\mathbf{v}\|_{H^2(\Omega)} & = \|\sum_{k=1}^P\mathbf{v}^{(k)}w_k\|_{H^2(\Omega)} \leq \sum_{k=1}^P\|\mathbf{v}^{(k)}w_k\|_{H^2(\Omega\cap\Omega_k)} \\ 
    & = \sum_{k=1}^P \left(\sum_{\tau = 0}^2\|D^{\tau}(\mathbf{v}^{(k)} w_k)\|^2_{L_2(\Omega\cap\Omega_k)}\right)^{\frac{1}{2}} \leq \sum_{k=1}^P \sum_{\tau = 0}^2\|D^{\tau}(\mathbf{v}^{(k)} w_k)\|_{L_2(\Omega\cap\Omega_k)}.
    \label{eq:Th4.4_1}
    \end{split}
    \end{align}
    Using the Leibniz rule from \eqref{eq:LeibnizRuleOperator} but for the $L_2$-norm gives
    \begin{equation}
            \|\mathbf{v}\|_{H^2(\Omega)} \leq \sum_{k=1}^P\sum_{\tau = 0}^2C_T^{\frac{1}{2}}\sum_{q+\zeta = \tau}\frac{\tau!}{q!\zeta!}\|D^q\mathbf{v}^{(k)}\|_{L_2(\Omega_k)}\|D^{\zeta}w_k\|_{L_{\infty}(\Omega_k)},
    \end{equation}
    where we used $\|w_k\|_{L_{\infty}(\Omega_k\cap\Omega)}\leq \|w_k\|_{L_{\infty}(\Omega_k)}$ and $\|D^q\mathbf{v}^{(k)}\|_{L_2(\Omega_k\cap\Omega)}\leq \|D^q\mathbf{v}^{(k)}\|_{L_2(\Omega_k)}$. Additionally, we use \eqref{eq:derivativeRef2PhysicalL2Norms} to bound the image $\tilde{\mathbf{v}}^{(k)}$ of the function $\mathbf{v}^{(k)}$ in the reference patch, whose components $\tilde{v}^{(k)}_i$ are in the native space of the Gaussian $\mathcal{N}_G(\Omega_0)$ 
    \begin{align}
    \begin{split}
        \|\mathbf{v}\|_{H^2(\Omega)} & \leq \sum_{k=1}^P\sum_{\tau=0}^2C_T^{\frac{1}{2}}\sum_{q+\zeta = \tau}\frac{\tau!}{q!\zeta!}\frac{C_{\mathrm{ref}}(q)}{R_{\mathrm{sc},k}H^{\frac{1}{2}}_{\mathrm{sc},k}}\|D^{q}\tilde{\mathbf{v}}^{(k)}\|_{L_2(\Omega_0)}\|D^{\zeta}w_k\|_{L_{\infty}(\Omega_k)} \\
        & \leq \sum_{k=1}^P\sum_{\tau=0}^2C_T^{\frac{1}{2}}\sum_{q+\zeta = \tau}\frac{\tau!}{q!\zeta!}\frac{C_{\mathrm{ref}}(q)}{R_{\mathrm{sc},k}H^{\frac{1}{2}}_{\mathrm{sc},k}}\sum_{i=1}^d\|D^{q}\tilde{v}_i^{(k)}\|_{L_2(\Omega_0)}\|D^{\zeta}w_k\|_{L_{\infty}(\Omega_k)}.
    \end{split}
    \end{align}
    Using \eqref{eq:H2includedInN} gives
    \begin{equation}
        \|\mathbf{v}\|_{H^{2}(\Omega)} \leq \sum_{k=1}^P\sum_{\tau=0}^2C_T^{\frac{1}{2}}\sum_{q+\zeta = \tau}\frac{\tau!}{q!\zeta!}\frac{C_{\mathrm{ref}}(q)}{R_{\mathrm{sc},k}H^{\frac{1}{2}}_{\mathrm{sc},k}}\sum_{i=1}^d C_{\mathcal{N}_G}\|\tilde{v}_i^{(k)}\|_{\mathcal{N}_G(\Omega_0)}\|D^{\zeta}w_k\|_{L_{\infty}(\Omega_k)},
    \end{equation}
    and since the weight functions are $\gamma$-stable for $\gamma \geq 2$ we conclude that the function $\mathbf{v}\in V_h$ is bounded in the $H^2$-norm.
\end{proof}

The best approximation error estimate for the solution of \eqref{eq:disc_WF} is given in Theorem \ref{theorem:bestApproximation} below and has four prerequisites, a coercivity estimate for the weighted bilinear form, a discrete least squares orthogonality property, a quadrature error estimate and a continuity bound for the discrete bilinear form. We should also assume that the solution to \eqref{eq:LSQWF} has additional regularity which allows for the point-wise evaluations necessary for the $\ell_2$-norms in \eqref{eq:disc_WF}. As such we assume $\mathbf{u}\in H^2(\Omega)\cap C^{2,\gamma}(\bar{\Omega})$, where $C^{2,\gamma}(\bar{\Omega})$ is a Hölder space with exponent $\gamma > 0$, for each component of the vector valued function $\mathbf{u}$. From the Sobolev embedding theorem we get that it is sufficient to have $\mathbf{u}\in H^4(\Omega)$ for a Lipschitz domain $\Omega$ \cite[p. 85]{adams2003sobolev}. 

We start with a coercivity estimate for the weighted bilinear form $\tilde{\alpha}\left(\cdot,\cdot\right)$ defined in \eqref{eq:ContWeightBilinear}. We use the continuous coercivity estimate from Theorem \ref{theorem:Coercivity} as well as the following inverse inequality for the boundary 
\begin{equation}
    \|\kappa_1\sigma\left(\mathbf{v}\right)\cdot\mathbf{n} + \kappa_0\mathbf{v}\|^2_{H^{\frac{1}{2}}(\partial\Omega)} \leq C_{\mathrm{inv}}h^{-1}\|\kappa_1\sigma\left(\mathbf{v}\right)\cdot\mathbf{n} + h^{-1}\kappa_0\mathbf{v}\|^2_{L_2(\partial\Omega)} + C_{\mathrm{inv}}\|\nabla\cdot\sigma(\mathbf{v})\|^2_{L_2(\Omega)},
\end{equation}
which is derived in Appendix \ref{app:halfNorm}, using \cite[Lemma 4.5.3]{brenner2008mathematical}. The coercivity estimate for a function $\mathbf{v}\in V_h$ becomes
\begin{align}
    \begin{split}
        \|\mathbf{v}\|^2_{H^2(\Omega)} & \leq C_1\alpha\left(\mathbf{v},\mathbf{v}\right) = C_1\|\nabla\cdot\sigma(\mathbf{v})\|^2_{L_2(\Omega)} + C_1\|\kappa_1\sigma(\mathbf{v})\cdot\mathbf{n} + \kappa_0\mathbf{v}\|^2_{H^{1/2}(\partial\Omega)} \\
        & \leq C_1(1 + C_{\mathrm{inv}})\|\nabla\cdot\sigma(\mathbf{v})\|^2_{L_2(\Omega)} + C_1C_{\mathrm{inv}}h^{-1}\|\kappa_1\sigma(\mathbf{v})\cdot\mathbf{n} + h^{-1}\kappa_0\mathbf{v}\|^2_{L_2(\partial\Omega)}
         \\
        & \leq  C_1(1 + C_{\mathrm{inv}})\left(\|\nabla\cdot\sigma(\mathbf{v})\|^2_{L_2(\Omega)} + h^{-1}\|\kappa_1\sigma(\mathbf{v})\cdot\mathbf{n} + h^{-1}\kappa_0\mathbf{v}\|^2_{L_2(\partial\Omega)}\right)  \\ &  = C_1(1 + C_{\mathrm{inv}})\tilde{\alpha}\left(\mathbf{v},\mathbf{v}\right),
    \end{split}
    \label{eq:ContBilinear2WeightBilinear}
\end{align}
where $h$ is the fill distance \eqref{eq:fillX}. 

The least squares orthogonality property for the error of~\eqref{eq:disc_WF} can be expressed as
\begin{equation}
    \alpha_h\left(\mathbf{u}_h - \mathbf{u}, \mathbf{v}\right) = 0, \hspace{0.5cm} \forall\mathbf{v}\in V_h,
    \label{eq:LSQ_orthogonality}
\end{equation}
where $\mathbf{u}\in H^4(\Omega)$ is the solution to the continuous least-squares problem \eqref{eq:LSQWF}, and $\mathbf{u}_h\in V_h(\Omega)$ is the solution to~\eqref{eq:disc_WF}. 

Bounds for quadrature errors are derived in Appendix \ref{app:intErrorEstimates}. We use the following notation for the errors involved in the estimates below:
\begin{subequations}
\label{eq:genEigenProblems}
\begin{alignat}{3}  
    \|\mathbf{v}\|^2_{\ell_2(\Omega)} &- \|\mathbf{v}\|^2_{L_2(\Omega)} &&\leq \tau_0\|\mathbf{v}\|^2_{L_2(\Omega)}, 
    \label{eq:genEigenProblemL2}\\
    \|\mathbf{v}\|^2_{h^2(\Omega)} &- \|\mathbf{v}\|^2_{H^2(\Omega)} &&\leq \tau_2\|\mathbf{v}\|^2_{H^2(\Omega)}, 
    \label{eq:genEigenProblemH2}\\
    \tilde{\alpha}\left(\mathbf{v},\mathbf{v}\right) &- \alpha_h\left(\mathbf{v},\mathbf{v}\right) &&\leq \tau_{\alpha}\tilde{\alpha}\left(\mathbf{v},\mathbf{v}\right),
    \label{eq:genEigenProblemBilinear}
\end{alignat}
\end{subequations}
where the constants~$\tau_{0}$, $\tau_{2}$, and~$\tau_{\alpha}$ depend on the oversampling $q = \frac{M}{N}$, assuming that the node quality is sufficiently high. Similar quadrature error estimates were numerically investigated for the least squares RBF-FD method in \cite{larsson2022numerical,Larsson_2024}.

A continuity bound for the discrete bilinear form \eqref{eq:DiscBilinear} is derived using the bounds provided in Appendix \ref{app:ContinuityBounds}, but applied to the $\ell_2$-norms. For $\mathbf{v}\in V_h(\Omega)$, we have
\begin{align}
    \begin{split}
        \alpha_h\left(\mathbf{v},\mathbf{v}\right) &  = \|\nabla\cdot\sigma(\mathbf{v})\|^2_{\ell_2(\Omega)} + h^{-1}\|\kappa_1\sigma(\mathbf{v})\cdot\mathbf{n} + h^{-1}\kappa_0\mathbf{v}\|^2_{\ell_2(\partial\Omega)} \\
        & \leq d\left(2\mu + \lambda\right)^2\|D^2\mathbf{v}\|^2_{\ell_2(\Omega)} \\
        & + 2h^{-1}\max\left(\|\kappa_0\|_{L_{\infty}(\partial\Omega)}^2,\|\kappa_1\|_{L_{\infty}(\partial\Omega)}^2\right)\left(\|\sigma(\mathbf{v})\cdot\mathbf{n}\|^2_{\ell_2(\partial\Omega)} + h^{-2}\|\mathbf{v}\|^2_{\ell_2(\partial\Omega)}\right) \\
        & \leq  d\left(2\mu + \lambda\right)^2\|D^2\mathbf{v}\|^2_{\ell_2(\Omega)} \\
        & + 2\max\left(\|\kappa_0\|_{L_{\infty}(\partial\Omega)}^2,\|\kappa_1\|_{L_{\infty}(\partial\Omega)}^2\right)\left(\left(4\mu^2 + d\lambda^2\right)h^{-1}\|\nabla\mathbf{v}\|_{\ell_2(\partial\Omega)} + h^{-3}\|\mathbf{v}\|^2_{\ell_2(\partial\Omega)}\right) \\
        & \leq C_4\left(\|D^2\mathbf{v}\|^2_{\ell_2(\Omega)} + h^{-1}\|\nabla\mathbf{v}\|^2_{\ell_2(\partial\Omega)} + h^{-3}\|\mathbf{v}\|^2_{\ell_2(\partial\Omega)}\right),
    \end{split}
    \label{eq:discContinuity}
\end{align}
where the constant $C_4$ depends on $\lambda,~\mu,~d,~\kappa_0,~\text{and}~\kappa_1$.

\begin{theorem}
    Given a solution $\mathbf{u}\in H^4(\Omega)$ to the continuous least-squares projection problem \eqref{eq:LSQWF}, an arbitrary function $\mathbf{u}_I\in V_h(\Omega)$ and the solution $\mathbf{u}_h\in V_h(\Omega)$ to the discrete problem \eqref{eq:disc_WF}, the following approximation estimate holds 
    \begin{equation}
        \|\mathbf{u} - \mathbf{u}_h \|^2_{H^2(\Omega)}\leq C_5\alpha_h\left(\mathbf{u}-\mathbf{u}_I,\mathbf{u}-\mathbf{u}_I\right) + 2\|\mathbf{u}-\mathbf{u}_I\|^2_{H^2(\Omega)},
    \end{equation} 
    where $\alpha_h(\cdot,\cdot)$ is the discrete bilinear form defined in \eqref{eq:DiscBilinear} and $C_5$ is a constant depending on the quadrature error $\tau_{\alpha}$ and the coercivity constant $C_1$.
    \label{theorem:bestApproximation}
\end{theorem}

\begin{proof}
The proof closely resembles the proof in \cite[Theorem 5.2]{stableRBFFD2021}. We start by defining the approximation error $\mathbf{e} = \mathbf{u} - \mathbf{u}_h$ and the error $\mathbf{e}_I = \mathbf{u} - \mathbf{u}_I$ such that $\mathbf{e} = \mathbf{e}_I + \mathbf{e}_h$ and $\mathbf{e}_h = \mathbf{u}_I-\mathbf{u}_h\in V_h(\Omega)$, where we refer to $\mathbf{e}_I$ as an interpolation error since we later choose $\mathbf{u}_I$ to be a locally constructed partition of unity interpolant.

We first utilise the least squares orthogonality property from \eqref{eq:LSQ_orthogonality} which gives the discrete bilinear form for $\mathbf{e}_h$
\begin{align}
    \begin{split}
        \alpha_h\left(\mathbf{e}_h,\mathbf{e}_h\right) & = \alpha_h\left(\mathbf{e}-\mathbf{e}_I,\mathbf{e}_h\right) 
         = \alpha_h\left(\mathbf{e},\mathbf{e}_h\right) + \alpha_h\left(-\mathbf{e}_I,\mathbf{e}_h\right) \\ 
        & = \alpha_h\left(-\mathbf{e}_I,\mathbf{e}_h\right) = \alpha_h\left(\mathbf{e}_I,\mathbf{e}_I\right) + \alpha_h\left(-\mathbf{e}_I,\mathbf{e}\right) \\
        & = \alpha_h\left(\mathbf{e}_I,\mathbf{e}_I\right) + \alpha_h\left(-\mathbf{e},\mathbf{e}\right) + \alpha_h\left(\mathbf{e}_h,\mathbf{e}\right) =  \alpha_h\left(\mathbf{e}_I,\mathbf{e}_I\right) - \alpha_h\left(\mathbf{e},\mathbf{e}\right),
    \end{split}
    \label{eq:DiscBilinearBound}
\end{align}
and shows that $\alpha_h\left(\mathbf{e}_h,\mathbf{e}_h\right)\leq \alpha_h\left(\mathbf{e}_I,\mathbf{e}_I\right)$.
We then use the coercivity estimate \eqref{eq:ContBilinear2WeightBilinear}, to bound the \mbox{$H^2$-norm} as
\begin{align}
    \begin{split}
        \|\mathbf{e}\|^2_{H^2(\Omega)} & \leq 2\|\mathbf{e}_h\|^2_{H^2(\Omega)} + 2\|\mathbf{e}_I\|^2_{H^2(\Omega)}  \\
        & \leq 2C_1\left(1+C_{\mathrm{inv}}\right)\tilde{\alpha}\left(\mathbf{e}_h,\mathbf{e}_h\right) + 2\|\mathbf{e}_I\|^2_{H^2(\Omega)}.
    \end{split}
\end{align}
Using estimates \eqref{eq:genEigenProblemBilinear} and \eqref{eq:DiscBilinearBound} we derive the approximation error bound
\begin{align}
    \begin{split}
        \|\mathbf{e}\|^2_{H^2(\Omega)} & \leq 2C_1\left(1+C_{\mathrm{inv}}\right)\tilde{\alpha}\left(\mathbf{e}_h,\mathbf{e}_h\right) + 2\|\mathbf{e}_I\|^2_{H^2(\Omega)} \\
        & \leq 2C_1\frac{1+C_{\mathrm{inv}}}{1-\tau_{\alpha}}\alpha_h\left(\mathbf{e}_h,\mathbf{e}_h\right) + 2\|\mathbf{e}_I\|^2_{H^2(\Omega)}\\ 
        & \leq 2C_1\frac{1+C_{\mathrm{inv}}}{1-\tau_{\alpha}}\alpha_h\left(\mathbf{e}_I,\mathbf{e}_I\right) + 2\|\mathbf{e}_I\|^2_{H^2(\Omega)}.
    \end{split}
\end{align}
\end{proof}


Theorem \ref{theorem:bestApproximation} shows that the error for the approximate solution $\mathbf{u}_h\in V_h(\Omega)$ can be bounded using any function $\mathbf{u}_I\in V_h$. We choose $\mathbf{u}_I$ as a particular function constructed using interpolants in the reference patch $\Omega_0$. Let $\mathbf{u}(\mathbf{x})=(u_1(\mathbf{x}),\ldots,u_d(\mathbf{x}))^T\in H^4(\Omega)$ denote the global solution to \eqref{eq:LSQWF}. We restrict the global solution to one of the patches to form the local functions $u_i^{(k)}(\mathbf{x}) = u_i(\mathbf{x})|_{(\Omega_k\cap\Omega)}$. Then we transform the local functions into the reference patch to get $\tilde{u}_i^{(k)}(\mathbf{x}^\prime)=u_i^{(k)}(\mathbf{x})$, where $\mathbf{x^\prime}\in\Omega_{0,k}$ and $\Omega_{0,k}$ is the image of $\Omega_k\cap\Omega$ in the reference patch $\Omega_0$. For the theoretical analysis, we need to assume that $\tilde{u}_i^{(k)}\in \mathcal{N}_G(\Omega_{0,k})$. This allows us to form a natural extension~\cite[Theorem 10.46]{Wendland_2004} $E\tilde{u}_i^{(k)}$, such that $E\tilde{u}_i^{(k)}|_{\Omega_{0,k}}=\tilde{u}_i^{(k)}$ and
\begin{equation}
    \|E\tilde{u}_i^{(k)}\|_{\mathcal{N}_G(\Omega_0)} = \|\tilde{u}_i^{(k)}\|_{\mathcal{N}_G(\Omega_{0,k})}.
    \label{eq:normext}
\end{equation}

We define the Gaussian RBF interpolation operator over  $I_h: \mathcal{N}_G(\Omega_0) \rightarrow \mathcal{N}_G(\Omega_0)$ and apply it to the locally extended functions to get $\tilde{u}^{(k)}_{I,i} = I_h(E\tilde{u}_i^{(k)})$, and let
$\tilde{\mathbf{u}}^{(k)}_I = (\tilde{u}^{(k)}_{I,1},\dots,\tilde{u}^{(k)}_{I,d})$. We form the global function in the extended physical domain using the partition of unity weights
\begin{equation}
    \mathbf{u}_I(\mathbf{x}) = \sum_{k=1}^Pw_k(\mathbf{x})\mathbf{u}_I^{(k)}\left(\mathbf{x}\right), \hspace{2cm} \forall \mathbf{x}\in\cup_{k=1}^P\Omega_k,
    \label{eq:globalExtendedInterpolant}
\end{equation}
where $\mathbf{u}_I^{(k)}(\mathbf{x}) = \tilde{\mathbf{u}}^{(k)}_I(\mathbf{x}^\prime)$. We note that $\mathbf{u}_I$ is not a global interpolant, since the interpolation points in different patches do not coincide in the physical domain. Furthermore, we note that the global solution function $\mathbf{u}(\mathbf{x})$ can be expressed as 
\begin{equation}
    \mathbf{u}(\mathbf{x}) = \sum_{k=1}^P w_k(\mathbf{x})\mathbf{u}(\mathbf{x}), \hspace{2cm} \forall \mathbf{x}\in\Omega.
    \label{eq:globalTrueSolIntEst}
\end{equation}


The following theorem on generalized interpolation errors and their derivatives is based on a similar estimate for a fitted RBF-PU method derived in \cite[Theorem 15.19]{Wendland_2004}. Additionally, we use the sampling inequality \cite[Theorem 2.4]{rieger2010sampling} and the bound \eqref{eq:H2includedInN} for Gaussian kernels. A similar result for the Laplace operator of the interpolation error is shown in \cite{larsson2024rbfpartitionunitymethod}.

\begin{theorem}
    Let $\Omega\subset \mathbb{R}^d$ be a bounded domain partitioned into a single layer of overlapping cylindrical patches $\{\Omega_k\}_{k=1}^P$, with radius $R_k$ and height $H_k$, which form a regular cover of $\Omega$ such that every $\mathbf{x}\in\Omega$ is overlapped by at most $\chi$ patches. Let $T_k:\Omega_k\rightarrow\Omega_0$ be a linear map to a reference patch $\Omega_0$, which is star shaped with respect to a ball and satisfies an interior cone condition with radius $r$ and angle $\theta\in\left(0,\pi/2\right)$. Let $\{w_k\}$ be $\gamma$-stable partition of unity weight functions in the sense of \eqref{eq:stabilityWeightFunc}. For a function $\mathbf{u}\in H^{p+1}(\Omega)$, $p+1\geq 4$ \eqref{eq:globalTrueSolIntEst}, and its generalized partition of unity interpolant $\mathbf{u}_I\in V_h(\Omega)$ of the form \eqref{eq:globalExtendedInterpolant}, based on  a sufficiently dense discrete template node set $X_0\subset \Omega_0$ with fill distance $h_0 = h_f(X_0,\Omega_0)$, where $0 \leq h_0\leq \frac{r\sin\theta}{4(1+\sin\theta)l^2}$, the following holds for $|\alpha|\leq l$
\begin{align}
    \|D^{\alpha}\left(\mathbf{u}-\mathbf{u}_I\right)\|_{L_{\infty}(\Omega)} \leq C_I h_0^{p+1-|\alpha| - \frac{d}{2}}\max_{\substack{1\leq i\leq d \\ 1\leq k\leq P}} \ \|\tilde{u}^{(k)}_i\ \|_{\mathcal{N}_G(\Omega_{0,k})},
\end{align}
where $p>|\alpha| + \frac{d}{2} - 1$ is the degree of polynomial space with dimension $n = |X_0|$ in $d$ dimensions, $C_I$ is a constant that depends on $n$ and $\chi$, the multi-index $\alpha$, $d$, $p$, $\theta$, the maximum patch height $H$, and the shape parameter $\varepsilon$ of the kernel \eqref{eq:NativeSpaceRd}, and $\tilde{u}_i^{(k)}\in\mathcal{N}_G(\Omega_{0,k})$ is $u_i|_{(\Omega_k\cap\Omega)}$ transformed to $\Omega_0$ using $T_k$.
    \label{theorem:InterpolationError}
\end{theorem}
\begin{proof}
    We start the proof from a local interpolation estimate constructed using sampling inequalities for smooth functions in domains that are star shaped with respect to a ball. We modify the result of \cite[Theorem 2.4]{rieger2010sampling} by using the relation $h_0\leq\text{diam}(\Omega_k)\leq (n+1)h_0$, where $h_0 = h_f(\Omega_0,X_0)\leq \frac{r\sin\theta}{4(1+\sin\theta)l^2}$ is the local fill distance for a patch and $n = |X_0|$ is the number of local points. Then, for a function $v\in H^{p+1}(\Omega_0)$ we have
    \begin{equation}
        \|D^{\alpha}v\|_{L_{\infty}(\Omega_0)} \leq C_S h_0^{p + 1 - \frac{d}{2}-|\alpha|}|v|_{H^{p+1}(\Omega_0)} + 2h_0^{-|\alpha|}\|v|_{X_0}\|_{\ell_{\infty}(X_0)},
    \end{equation}
    where $C_S$ depends on the dimension $d$, the multi-index $\alpha$, the order $p$ and the angle $\theta$.

    
    For a function $v\in\mathcal{N}_G(\Omega_0)$, we use \cite[Theorem 7.5]{rieger2010sampling} to bound the Sobolev norm using the native space norm
    \begin{equation}
        \|D^{\alpha}v\|_{L_{\infty}(\Omega_0)} \leq C_{\mathcal{N}_G}C_Sh_0^{p +1 - \frac{d}{2}-|\alpha|}\|v\|_{\mathcal{N}_G(\Omega_0)} + 2h_0^{-|\alpha|}\|v|_{X_0}\|_{\ell_{\infty}(X_0)},
        \label{eq:LocalSamplingIneq}
    \end{equation}
    where $C_{\mathcal{N}_G}$ depends on $d$ and the shape parameter $\varepsilon$ as in~\eqref{eq:H2includedInN}, but also on $p$.
    
    If we now let $v=E\tilde{u}_i^{(k)}-\tilde{u}_{I,i}^{(k)}$, the second term becomes zero, and we have
    \begin{align}
        \|D^{\alpha}\left(E\tilde{u}_i^{(k)}-\tilde{u}_{I,i}^{(k)}\right)\|_{L_{\infty}(\Omega_0)}&\leq C_{\mathcal{N}_G}C_Sh_0^{p + 1 - \frac{d}{2}-|\alpha|}\|E\tilde{u}_i^{(k)}-\tilde{u}_{I,i}^{(k)}\|_{\mathcal{N}_G(\Omega_0)}\nonumber\\
        &\leq 2C_{\mathcal{N}_G}C_Sh_0^{p + 1 - \frac{d}{2}-|\alpha|}\|E\tilde{u}_i^{(k)}\|_{\mathcal{N}_G(\Omega_0)} \label{eq:LocalInterpEstimate}\\
        & = 2C_{\mathcal{N}_G}C_Sh_0^{p + 1 - \frac{d}{2}-|\alpha|}\|\tilde{u}_i^{(k)}\|_{\mathcal{N}_G(\Omega_{0,k})},\nonumber
    \end{align}
    where we first used the optimality property $\|I_h(v)\|_{\mathcal{N}_G(\Omega_0)}\leq\|v\|_{\mathcal{N}_G(\Omega_0)}$, and then the extension property~\eqref{eq:normext}.
    
    We then proceed to bound the global generalized interpolation error $\mathbf{e}_I = \mathbf{u} - \mathbf{u}_I$. 
    Using sub-additivity we get a bound for the derivatives of the global error
\begin{align}
    \begin{split}
        \|D^{\alpha}\left(\mathbf{u}-\mathbf{u}_I\right)\|_{L_{\infty}(\Omega)} & = \|\sum_{k=1}^P D^{\alpha}(w_k\mathbf{u} - w_k\mathbf{u}_{I}^{(k)})\|_{L_{\infty}(\Omega)} \leq \sum_{i=1}^d\sum_{k=1}^P\|D^{\alpha}(w_k(u_i  - u_{I,i}^{(k)}))\|_{L_{\infty}(\Omega)}  \\
        & =  \sum_{i=1}^d\sum_{k=1}^P\|D^{\alpha}(w_k(u_i^{(k)}  - u_{I,i}^{(k)}))\|_{L_{\infty}(\Omega_k\cap\Omega)},
    \end{split}
    \label{eq:subadd}
\end{align}
where we used that $w_k$ are compactly supported. We further expand the bound by using the Leibniz rule for partial derivatives as expressed in \cite[Section 3.2]{Comtet1974}, and then that any point $\mathbf{x}\in\Omega$ is overlapped by at most $\chi$ patches

\begin{align}
    \begin{split}
        \|D^{\alpha}\mathbf{e}_I\|_{L_{\infty}(\Omega)} & \leq \sum_{i=1}^d\sum_{k=1}^P\|D^{\alpha}(w_k(u_i^{(k)}  - u_{I,i}^{(k)}))\|_{L_{\infty}(\Omega_k\cap\Omega)} \\ 
        & = \sum_{i=1}^d\sum_{k=1}^P\|\sum_{q+ \zeta = \alpha}\frac{\alpha!}{q!\zeta!}D^q(u_i^{(k)} - u_{I,i}^{(k)}) D^{\zeta}w_k \|_{L_{\infty}(\Omega_k\cap\Omega)} \\ 
        & \leq d \chi\max_{\substack{1\leq i\leq d \\ 1\leq k\leq P}}\sum_{q+ \zeta = \alpha}\frac{\alpha!}{q!\zeta!}\|D^q(u_i^{(k)} - u_{I,i}^{(k)}) D^{\zeta}w_k \|_{L_{\infty}(\Omega_k\cap\Omega)} \\
        & \leq d \chi\max_{\substack{1\leq i\leq d \\ 1\leq k\leq P}}\sum_{q+ \zeta = \alpha}\frac{\alpha!}{q!\zeta!}\|D^q(u_i^{(k)} - u_{I,i}^{(k)})\|_{L_{\infty}(\Omega_k\cap\Omega)}\|D^{\zeta}w_k \|_{L_{\infty}(\Omega_k)},
    \end{split}
\end{align}
where $\zeta$ and $q$ are multi-indices.
Note that an equivalent detailed expression of the Leibniz rule in multiple dimensions is given by \cite[Proposition 6]{Hardy2006}. We transform the norm of the local interpolation error into the reference patch using~\eqref{eq:derivativePhysical2Ref}, where we include the whole patch and the extended function at this point
\begin{equation}
    \|D^{\alpha}\mathbf{e}_I\|_{L_{\infty}(\Omega)} \leq d \chi\max_{\substack{1\leq i\leq d \\ 1\leq k\leq P}}\sum_{q+ \zeta = \alpha}\frac{\alpha!}{q!\zeta!}C_{\mathrm{ref}}\|D^q(E\tilde{u}_i^{(k)} - u_{I,i}^{(k)})\|_{L_{\infty}(\Omega_0)}\|D^{\zeta}w_k \|_{L_{\infty}(\Omega_k)}
\end{equation}
Then, using the local interpolation estimate \eqref{eq:LocalInterpEstimate} and the weight function property \eqref{eq:stabilityWeightFunc} gives 
\begin{equation}
        \|D^{\alpha}\mathbf{e}_I\|_{L_{\infty}(\Omega)} \leq 2d \chi\max_{\substack{1\leq i\leq d \\ 1\leq k\leq P}}\sum_{q+ \zeta = \alpha}\frac{\alpha!}{q!\zeta!}G_{|\zeta|}C_{\mathcal{N}_G}C_SC_{\mathrm{ref}}h_0^{p + 1 - \frac{d}{2}-|q|}\delta_k^{- |\zeta|}\|\tilde{u}_i^{(k)}\|_{\mathcal{N}_G(\Omega_{0,k})} 
        \label{eq:localInterpolationEstimate}
\end{equation}
where $\delta_k$ is the minimum absolute overlap of patch $\Omega_k$. 
The overlap for patch $k$ is computed as $\delta_k = R_k\delta_0$, where $R_k$ is the radius of the patch and $\delta_0$ is the relative overlap. The aim is to refine such that $\delta_k\propto h_0$ and all summands converge with the same order. For our patches when $d=3$ the patch volume is computed as $|\Omega_k| \triangleq V_k = \mathfrak{C}_V R_k^{d-1}H_k$, where $\mathfrak{C}_V = \pi$. Moreover, we assume that the height $H_k$ remains constant during refinement and is proportional to the thickness of the thin geometry, since we maintain one layer of patches during refinement. For $d=3$, with cylindrical patches and a sufficiently dense node set $X_k$ such that $h_s(X_k,\Omega_k)\leq \max(R_k,H_k)$, the patch volume can be lower bounded by the separation distance \eqref{eq:sepDistance} as
\begin{align}
\begin{split}
    \frac{4}{3}\pi nh_s(X_k,\Omega_k)^d & \leq V_k + \pi R_k^2h_s\left(X_k\Omega_k\right) + 2\pi R_k H_k h_s\left(X_k\Omega_k\right) + \pi H_kh_s^2\left(X_k,\Omega_k\right) + \pi^2 R_kh_s^2\left(X_k,\Omega_k\right) \\ 
    & \leq (5+\pi) V_k,
\end{split}
\end{align}
where we added a band of thickness $h_s(X_k,\Omega_k)$ around the cylinder to ensure all spheres with radius $h_s(X_k,\Omega_k)$ are included in the volume. The above estimate only holds for a sufficiently small separation distance, which corresponds to a sufficiently large number of points $n$. The condition for the number of centre points can also be derived and is given by $n\geq (1+\sqrt{R_k^2 + (0.5H_k)^2} / \min(R_k,H_k))^3$. Additionally we construct an upper bound and use the node quality to get 
\begin{align}
    \begin{split}
        \frac{4}{3}\pi\left(5+\pi\right) n c_{q,X_k}^{1/d} h_k^{d} & \leq V_k \leq \frac{4}{3}\pi n h_k^{d} \\
        \tilde{C}_V h_{k}^{\frac{d}{d-1}} & \leq R_k \leq C_V h_k^{\frac{d}{d-1}}
    \end{split}
    \label{eq:fillX_VolumeRel}
\end{align}
where $c_{q,X_k}$ is the node quality and $h_k = h_f(X_k,\Omega_k)$ the fill distance for point set $X_k\subset\Omega_k$ \eqref{eq:fillDistance_Pts}. This indicates that $\delta_0$ should be refined as $h_k^{-1/(d-1)}$, which gives
\begin{equation}
    \tilde{C}_V \left(C'_{0,k}\right)^{-1}h_0 \leq \tilde{C}_Vh_k \leq \delta_k= R_k\delta_0 \leq C_V h_k^{\frac{d}{d-1}}h_k^{-\frac{1}{d-1}} = C_V h_k \leq C_V(C_{0,k})^{-1}h_0,
    \label{eq:OverlapFillDistRelation}
\end{equation}
where we used \eqref{eq:referenceFilltoPhysicalFill} and $C_V, \tilde{C}_V$ are constants dependent on the patch height $H_k$, the number of points $n$ and the dimension $d$. 

Using \eqref{eq:OverlapFillDistRelation} we have the global interpolation estimate
\begin{align}
    \begin{split}
        \|D^{\alpha}\mathbf{e}_I\|_{L_{\infty}(\Omega)}\leq  d\chi\mathfrak{C}_{\alpha}
        h_0^{p + 1 -\frac{d}{2}-|\alpha|}\max_{\substack{1\leq i\leq d \\ 1\leq k\leq P}}\|\tilde{u}_i^{(k)}\|_{\mathcal{N}_G(\Omega_{0,k})},
        \label{eq:globalInterpEstimate}
    \end{split}
\end{align}
where $\mathfrak{C}_{\alpha}$ is a constant dependent on the multi-index $\alpha$, the dimension $d$, the number of points $n$, the angle $\theta$ of the interior cone condition, the largest patch height $H_k$, and the scaling of the largest patch in the domain to the reference patch.
\end{proof}

\begin{remark}
We briefly note the important result by Schaback in \cite{schaback1996approximation} on achieving optimal orders when approximating less regular functions in Sobolev spaces using the Gaussian basis. This is achievable for specific approximants (not interpolants) in the native space of the Gaussian. Practically, the locally extended functions, $E\tilde{u}_i^{(k)}$, do not need to be in $\mathcal{N}_{G}(\Omega_0)$ for convergence to be achieved. 
\end{remark}

\begin{remark}
    The condition $h_0\leq\frac{r\sin\theta}{4(1+\sin\theta)l^2}$ comes from~\cite[Theorem 2.2]{rieger2010sampling}. It is based on a worst case bound for interpolation in one dimension. For the cylindrical patches that we use, $\theta=\pi/4$ and $r=\min(R_0/\sqrt{2},H_0/2)$, where $R_0$ is the radius and $H_0$ the height of the reference patch. For $l=2$, this leads to $h_0\leq\frac{\min(\sqrt{2}R_0,H_0)}{32(1+\sqrt{2})}\triangleq h^*$, indicating that $\mathcal{O}(10^5)$ points would be needed in a patch. However, for the node sets that we use, we get the expected convergence behaviour for much smaller node sets. We can instead investigate whether our specific node sets satisfy the conditions of Theorem 2.2 from~\cite{rieger2010sampling}. Then there needs to be coefficients $a_j^{(\alpha)}$ such that
    \begin{eqnarray}
        \sum_{j=1}^Np(\mathbf{x}_j)a_j^{(\alpha)}(\mathbf{x}) & = &D^\alpha(p(\mathbf{x})), \quad \forall p\in\mathbb{P}_l^d(\Omega_0),\label{eq:cond1}\\
        \sum_{j=1}^N|a_j^{(\alpha)}(\mathbf{x})| &\leq &2h_0^{-|\alpha|}, \label{eq:cond2}
    \end{eqnarray}
    where $\mathbb{P}_l^d(\Omega_0)$ is the space of polynomials of degree $\leq l$ in $d$ dimensions. By applying $L_1$-minimization to~\eqref{eq:cond1} and then computing $h_0$ from~\eqref{eq:cond2} for $|\alpha|=2$ we get the following results when comparing with the numerically estimated fill distance $h_0$ using definition~\eqref{eq:fill} for a unit patch $H_0=R_0=1$:\\
    \begin{center}
        \begin{tabular}{r|rrrrrrrrrr}
        $n$ & 20 & 35  & 56 & 84 & 120 & 165 & 220 & 286 & 364\\\hline
        $h^*$ & 0.32 & 0.34 & 0.35 & 0.35 & 0.35 & 0.35 & 0.35 & 0.35 & 0.35 \\ \hline
        $h_0$ & 0.65 & 0.60 & 0.49 & 0.47 & 0.35 & 0.28 & 0.27 & 0.24 & 0.21 
        \end{tabular}
    \end{center}
    The conditions are numerically satisfied for our specific Halton node sets for $n\geq 120$. 
    Due to the nature of our RBF-PUM approximation, where node sets are selected from a predetermined set, the convergence results from Theorem~\ref{theorem:InterpolationError} can be applied for much larger fill distances than are indicated by the general requirement of the theorem. The results for $h_0$ are however sensitive to high patch aspect ratios. When $R_0\gg H_0$ or $R_0\ll H_0$ the approximation properties deteriorate due to the thin dimension where the resolution becomes insufficient. 
\end{remark}


The a-priori estimate for our method is derived by using the best approximation estimate from Theorem \ref{theorem:bestApproximation} and the interpolation estimate from Theorem \ref{theorem:InterpolationError}. The error between the true solution $\mathbf{u}$ and the approximate solution $\mathbf{u}_h$ is bounded by
\begin{align}
\begin{split}
    \|\mathbf{u}-\mathbf{u}_h\|^2_{H^{2}(\Omega)} & \leq C_5\mathbf{\alpha}_h\left(\mathbf{u}-\mathbf{u}_I,\mathbf{u}-\mathbf{u}_I\right) + 2\|\mathbf{u}-\mathbf{u}_I\|^2_{H^2(\Omega)}
    \\
    & \leq C_4C_5\left(\|D^2\left(\mathbf{u}-\mathbf{u}_I\right)\|^2_{\ell_2(\Omega)} + h^{-1}\|\nabla\left(\mathbf{u}-\mathbf{u}_I\right)\|^2_{\ell_2(\partial\Omega)} + h^{-3}\|\mathbf{u}-\mathbf{u}_I\|^2_{\ell_2(\partial\Omega)}\right) \\ 
    & + 2\|\mathbf{u}-\mathbf{u}_I\|^2_{H^2(\Omega)},
\end{split}
\end{align}
where $h$ is the fill distance \eqref{eq:fillX} and we also used the discrete continuity estimate \eqref{eq:discContinuity}. Expanding the $H^2$-norm and bounding the discrete and the continuous $L_2$-norms by the infinity norm we have
\begin{align}
    \begin{split}
         \|\mathbf{u}-\mathbf{u}_h\|^2_{H^{2}(\Omega)} & \leq C_4C_5\|D^2\left(\mathbf{u}-\mathbf{u}_I\right)\|^2_{\ell_2(\Omega)} + C_4C_5h^{-1}\|\nabla\left(\mathbf{u}-\mathbf{u}_I\right)\|^2_{\ell_2(\partial\Omega)} \\ & + C_4C_5h^{-3}\|\mathbf{u}-\mathbf{u}_I\|^2_{\ell_2(\partial\Omega)} + 2\|D^2\left(\mathbf{u}-\mathbf{u}_I\right)\|^2_{L_2(\Omega)} + 2\|\nabla\left(\mathbf{u}-\mathbf{u}_I\right)\|^2_{L_2(\Omega)} + 2\|\mathbf{u}-\mathbf{u}_I\|^2_{L_2(\Omega)} \\
         & \leq |\Omega|\left(C_4C_5 + 2\right)\|D^2\left(\mathbf{u}-\mathbf{u}_I\right)\|^2_{L_{\infty}(\Omega)} + |\partial\Omega|C_4C_5 h^{-1}\|\nabla\left(\mathbf{u}-\mathbf{u}_I\right)\|^2_{L_{\infty}(\Omega)} \\ & + 
         |\partial\Omega|C_4C_5 h^{-3}\|\mathbf{u}-\mathbf{u}_I\|^2_{L_{\infty}(\Omega)} + 2|\Omega|\|\nabla\left(\mathbf{u}-\mathbf{u}_I\right)\|^2_{L_{\infty}(\Omega)} + 2|\Omega|\|\mathbf{u}-\mathbf{u}_I\|^2_{L_{\infty}(\Omega)},
         \label{eq:AprioriEstimateP1}
    \end{split}
\end{align}
where we used that the quadrature weights in the $\ell_2$-norm reproduce the measure of the domain of integration, see~\eqref{eq:integrationWeights}.
We use Theorem \ref{theorem:InterpolationError} by first rewriting the derivative operators in terms of a multi-index $\alpha$ as 
\begin{align}
    \begin{split}
        \|\mathbf{u}-\mathbf{u}_h\|^2_{H^{2}(\Omega)} & \leq 2|\Omega|\left(C_4C_5 + 2\right)\sum_{|\alpha| = 2}\|D^{\alpha}\left(\mathbf{u}-\mathbf{u}_I\right)\|^2_{L_{\infty}(\Omega)} \\
        & + \left(2|\Omega| + |\partial\Omega|C_4C_5h^{-1}\right) \sum_{|\alpha|=1}\|D^{\alpha}\left(\mathbf{u}-\mathbf{u}_I\right)\|^2_{L_{\infty}(\Omega)} \\   & + 
         \left(2|\Omega| + |\partial\Omega|C_4C_5h^{-3}\right)\|\mathbf{u}-\mathbf{u}_I\|^2_{L_{\infty}(\Omega)} \\
        & \leq C_I^2d\Bigl(2|\Omega|\left(C_4C_5 + 2\right) h_0^{2\left(p-1-\frac{d}{2}\right)} + \left(2|\Omega| + |\partial\Omega|C_4C_5h^{-1}\right)h_0^{2\left(p-\frac{d}{2}\right)} \\ 
         & + \left(2|\Omega| + |\partial\Omega|C_4C_5h^{-3}\right)h_0^{2\left(p+1-\frac{d}{2}\right)}\Bigr)\max_{\substack{1\leq i\leq d \\ 1\leq k\leq P}}\|\tilde{u}_i^{(k)}\|_{\mathcal{N}_G(\Omega_{0,k})}^2,
    \end{split}
\end{align}
where $C_I$ is the maximum squared constant between all terms from the global interpolation estimate \eqref{eq:globalInterpEstimate}.

Using \eqref{eq:referenceFilltoPhysicalFill} and \eqref{eq:fillX} to bound the reference patch fill distance $h_0$ by the global fill distance measure $h$, and collecting the leading terms gives the a-priori estimate
\begin{equation}
    \|\mathbf{u}-\mathbf{u}_h\|_{H^{2}(\Omega)} \leq Ch^{p-1-\frac{d}{2}}\max_{\substack{1\leq i\leq d \\ 1\leq k\leq P}}\|\tilde{u}_i^{(k)}\|_{\mathcal{N}_G(\Omega_{0,k})},
    \label{eq:FinalApriori}
\end{equation}
where $C = \max_{1\leq k \leq P}{(C'_{0,k})^{p-1-d/2}}C_Id(2|\Omega|(C_4C_5 + 1))^{1/2}$. In this a-priori estimate we convert all fill distance measures into the global one, $h$ \eqref{eq:fillX}. However, to support future development of adaptive refinement algorithms, it is more useful to derive the estimate with respect to the local patch-wise fill distance measure $h_k$. If we convert to the patch-wise fill distance after \eqref{eq:localInterpolationEstimate} and then avoid bounding the overlap by $h_0$, we have
\begin{align}
    \begin{split}
        \|\mathbf{u}-\mathbf{u}_h\|_{H^{2}(\Omega)} & \leq C'_Id^{\frac{1}{2}}\Bigl(\sqrt{2}|\Omega|^{\frac{1}{2}}\left(C_4C_5 + 1\right)^{\frac{1}{2}}\max_{\substack{1\leq i\leq d \\ 1\leq k\leq P}} h_k^{p-1-\frac{d}{2}} \|\tilde{u}_i^{(k)}\|_{\mathcal{N}_G(\Omega_{0,k})} \\ 
        & + \left(2|\Omega| + |\partial\Omega|C_4C_5h^{-1}\right)^{\frac{1}{2}}\max_{\substack{1\leq i\leq d \\ 1\leq k\leq P}}h_k^{p-\frac{d}{2}}\|\tilde{u}_i^{(k)}\|_{\mathcal{N}_G(\Omega_{0,k})} \\ 
         & + \left(2|\Omega| + |\partial\Omega|C_4C_5h^{-3}\right)^{\frac{1}{2}}\max_{\substack{1\leq i\leq d \\ 1\leq k\leq P}}h_k^{p+1-\frac{d}{2}}\|\tilde{u}_i^{(k)}\|_{\mathcal{N}_G(\Omega_{0,k})}\Bigr),
    \end{split}
\end{align}
where the first term with the local fill distance measure is the leading term. We note that the second and third terms still include the global measure coming from the inverse estimate on the boundary.

\begin{remark}
    The refinement strategy that we use, with refinement only in the radial direction, is effective as long as patches are wide enough, but eventually the increasing relative overlap becomes too large. At that point, the refinement strategy should change to split the largest dimension of the patch or to split all dimensions uniformly instead. Our implementation does not contain this transition, and the theoretical analysis also leaves this aspect out. From a practical point of view, in three dimensions, it is generally too costly to refine that far, but the implementation and the analysis can both be modified to include this aspect if the need arises.
\end{remark}
\section{Numerical experiments}\label{sec:res}
The method implementation is written in \texttt{MATLAB} and all experiments have been performed in \texttt{MATLAB} R2024a on a Linux 64-bit machine with 1TB RAM and an \textit{AMD EPYC 7502P} 32 core processor, where 32 out of the 64 threads are used. 
\subsection{Error measures and node generation}
We perform numerical experiments for three benchmark problems using the two diaphragm geometries shown in Figure \ref{fig:diaphragm_Recon}. We investigate the deformation for these problems and produce convergence plots for the normalised error in displacement measured using either the discrete $\ell_2$-norm or the discrete $h_2$-norm as
\begin{align}
    E_{\ell} = \frac{\|\mathbf{u} - \mathbf{u}_h\|_{\ell_2(\bar{\Omega})}}{\|\mathbf{u}\|_{\ell_2(\bar{\Omega})}},
    \label{eq:numericalErrorl2} \\
    E_{h} = \frac{\|\mathbf{u} - \mathbf{u}_h\|_{h_2(\bar{\Omega})}}{\|\mathbf{u}\|_{h_2(\bar{\Omega})}},
    \label{eq:numericalErrorh2}
\end{align}
which are computed using the evaluation points both in the interior and on the boundary $Y = \{\mathbf{y}_j\}_{j=1}^M\subseteq \bar{\Omega}$.  The numerical approximation $\mathbf{u}_h$ is the numerical solution of \eqref{eq:leastSquaresProb} and $\mathbf{u}$ is either a manufactured smooth solution (Problem 1) or a highly refined numerical solution (Problems 2 and 3). 

As described in Theorem \ref{theorem:InterpolationError} we perform patch refinement in the tangential direction to the diaphragm surface which means the patch radii $R_k$ decrease, while the heights $H_k$ remain mostly constant and are limited by the thickness of the diaphragm. This leads to algebraic convergence rates as given in \eqref{eq:FinalApriori} for different numbers of local points, $n$, corresponding to a polynomial space dimension related to the order as $n = \binom{p+d}{d}$, where the spatial dimension $d = 3$ in this case.

The refinement is measured through an approximation of the fill distance of the centre point set $X$. It is computed by taking the maximum approximate local fill distance from \eqref{eq:fillXapprox} as $h = \max_k\tilde{h}_k$. Given the equivalence relation between the volume of the patch and its fill distance \eqref{eq:fillX_VolumeRel}, we note that the approximate fill distance used for the numerical experiments is related by a constant to the theoretical value \eqref{eq:fillX}. Since the theoretical definition is not used in any of the numerical experiments, whenever $h$ is used henceforth, it refers to the approximate measure of the fill distance for the $X$ point set. The stencil centre point set $X_0$ is selected from the Halton sequence \cite{halton1960efficiency} in a box enclosing the reference patch and we keep the points generated inside the patch itself. This point set is then transformed to $X_k$ for each patch using the transformation~\eqref{eq:referencePatchMap}.

The evaluation point sets $Y_{\Omega}$ and $Y_{\partial\Omega}$ are constructed using the implicit level set function $\ell(\mathbf{y})$ used to define the boundary. The level set as well as a set of boundary points, $Y_{\mathrm{bnd}}$, are computed using the RBF-PU method as described in \cite{larsson2024rbfpartitionunitymethod}. The size of this set is $|Y_{\textrm{bnd}}| = 10726$ and also includes segments forming a surface mesh which is used for visualisation of results in the forthcoming sections. The total number of evaluation points to be generated is computed using the oversampling factor $q$ as $M = qN$, where $N = |X|$ and they are split between domain and boundary points. The number of points are split by using the following requirement of equal node density  
\begin{equation}
    \frac{|\partial\Omega|}{C_{\partial\Omega/\Omega}s^{d-1}} + \frac{|\Omega|}{s^d} = M,
\end{equation}
which is solved for the variable $s$ that represents the node distance. The constant $C_{\partial\Omega/\Omega}$ balances the volume and surface node distributions to be as equal as possible in terms of fill distance (or imbalanced according to a user defined target ratio). 
The number of points to generate in the domain and boundary are then given as $M_{\Omega} = |\Omega|/s^d$ and $M_{\partial\Omega} = |\partial\Omega|/(C_{\partial\Omega/\Omega}s^{d-1})$. Here, we use $C_{\partial\Omega/\Omega} = 1.6$ which numerically provides equivalent fill distance measures between the boundary and interior sampling as shown in Figures \ref{fig:evalFillDistances1} and~\ref{fig:evalFillDistances2} for both diaphragm geometries. The oversampling used for these plots is $q = 2.5$ and the fill distance is computed as $h_y = \max_{j\in I_{\Omega}}\tilde{h}_{j,\Omega}$ with the approximate measure of the fill distance used for both the domain~\eqref{eq:fillYapprox} and boundary points~\eqref{eq:fillYBndApprox}. We use $h_y$ to refer to the approximate measure of the evaluation point fill distance hereafter. Figure \ref{fig:fillDistanceRatio} indicates that the ratio between fill distances for the evaluation and centre node sets remains constant during refinement given that $q$ is constant. The ratio is expected to be approximately $q^{1/d} \approx 0.74$.

\begin{figure}
\centering
    \begin{subfigure}{.32\textwidth}
        \centering
        \includegraphics[width=\linewidth]{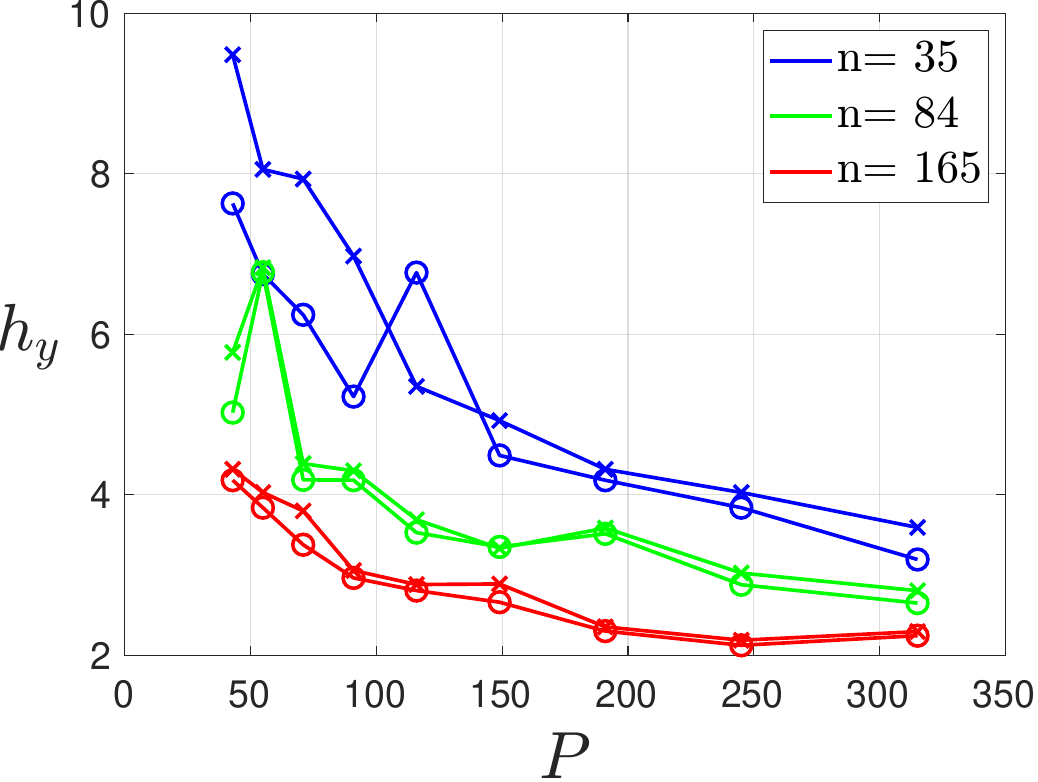}
        \caption{}
        \label{fig:evalFillDistances1}
    \end{subfigure}
    \hspace{.1cm}
    \begin{subfigure}{.32\textwidth}
        \centering
        \includegraphics[width=\linewidth]{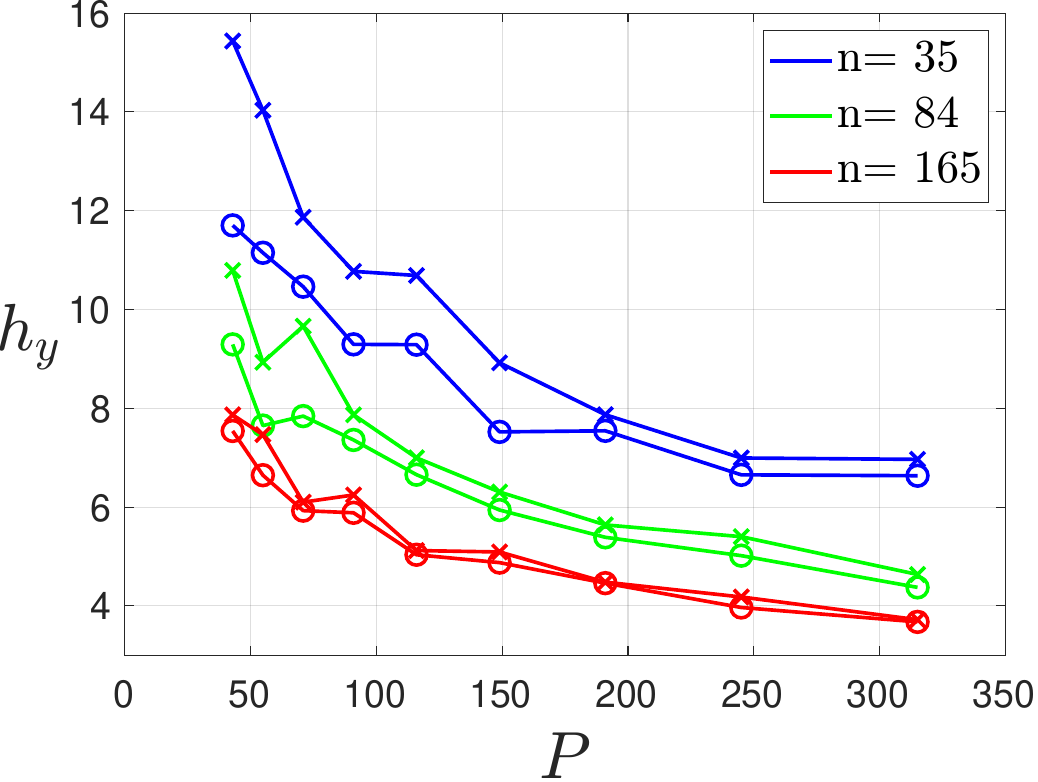}
        \caption{}
        \label{fig:evalFillDistances2}
    \end{subfigure}
    \hspace{.1cm}
    \begin{subfigure}{.32\textwidth}
        \centering
        \includegraphics[width=\linewidth]{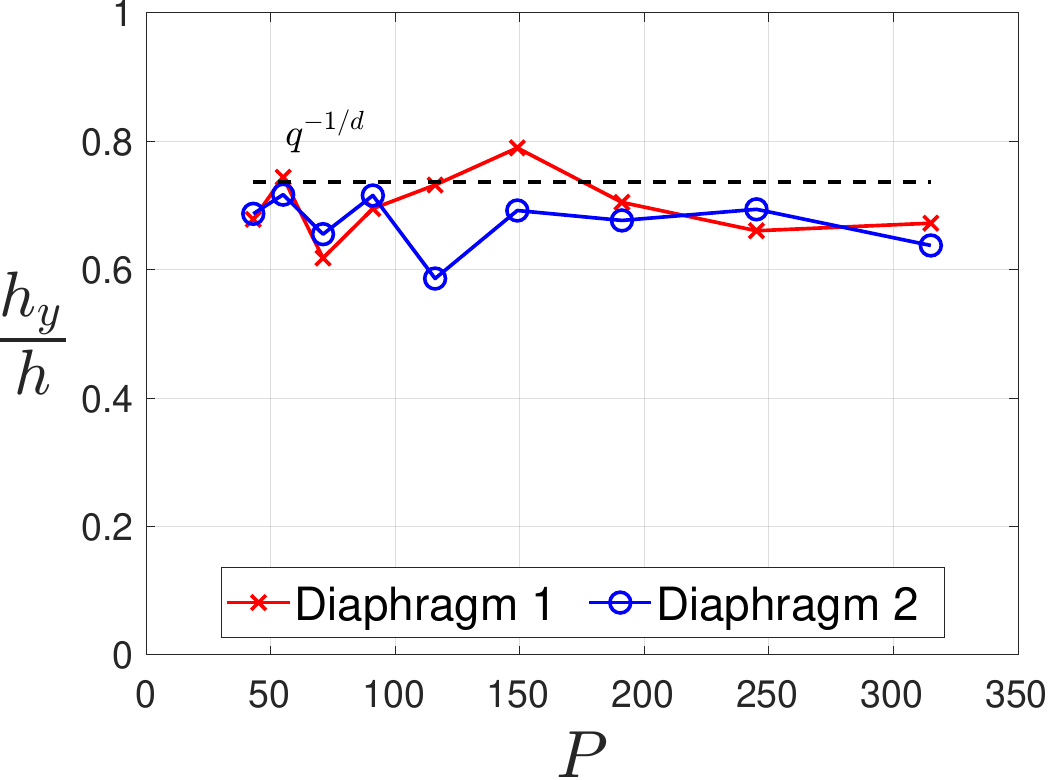}
        \caption{}
        \label{fig:fillDistanceRatio}
    \end{subfigure}
    \caption{Fill distances, $h_y$, of evaluation points in the domain (crosses) and on the boundary (circles) against the number of patches. The fill distance is plotted for a varying number of centre points $n$ and is given for points generated with $C_{\Omega/\partial\Omega} = 1.6$ and oversampling $q = 2.5$ for (a) Diaphragm 1 and (b) Diaphragm 2. (c) Ratio between the fill distance of evaluation points in $\Omega$ and of the centre points plotted against the number of patches for both diaphragm geometries. Data shown is only for $n = 165$ local centre points.}
    \label{fig:evalFillDistances}
\end{figure}

To generate evaluation points, we used a fixed geometry, and hence, a fixed patch cover and a fixed level set function, while the patch cover used by the PDE solver varies depending on the refinement level. Nodes are generated locally using the patch structure, which reduces cost and allows for a node density that is adapted to the patch size, which directly controls $\tilde{h}_k$. To generate points in the interior of the domain, we use the level set function, which has negative values inside the domain, to determine which points to include in the set $Y_{\Omega}$. To generate the boundary points we first select points in the interior of the domain that are close to the boundary, and then move them to the boundary using Newton iterations of the form
\begin{equation}
    \mathbf{y}_{\textrm{new}} = \mathbf{y} - \ell\left(\mathbf{y}\right)\frac{\nabla \ell\left(\mathbf{y}\right)}{\|\nabla \ell\left(\mathbf{y}\right)\|_2},
\end{equation}
where $\ell (\mathbf{y})$ is the level set function evaluated at the interior points. The tolerance used for convergence is $10^{-9}$ and the maximum number of iterations is set to $100$. An example of the evaluation and centre point sets was shown in Figure \ref{fig:patches_2}.

We note that the reconstruction method described in \cite{larsson2024rbfpartitionunitymethod} may generate some spurious zero level sets away from the geometry, especially when working in three dimensions. Their potential existence introduces complications when generating evaluation points close to the boundary of the domain. We deal with this issue by looking for outliers after points inside the domain are selected. Outliers are points with a distance larger than $\tau_{d}\tilde{h}_{\textrm{ref}}$ from every other point, where $\tau_{d}$ is the node distance tolerance and $\tilde{h}_{\textrm{ref}}$ is an approximate measure of the fill distance of the generated trial set of points. Spurious zero level sets depend on the geometry itself as well as the reconstruction, so we use $\tau_{d} = 1.5$ for Diaphragm 1 and $\tau_{d} = 1.2$ for Diaphragm 2.

\subsection{Results and discussion}
The linear elasticity problem \eqref{eq:linear_elasticity_displacement} is discretised and numerically solved in Diaphragm 1 or Diaphragm 2 using the method described in Sections \ref{sec:RBFPU} and \ref{sec:RBFPU_4}. The resulting linear system of equations is given by \eqref{eq:normalEqutions} and \eqref{eq:leastSquaresProb} with different data and boundary conditions for each problem. The resulting oversampled and sparse linear system is solved in \texttt{MATLAB} using \texttt{mldivide} which uses a sparse QR decomposition through the \texttt{SuiteSparseQR} package \cite{DavisSuiteSparse}. 

In this section we first discuss and justify the parameter choices made in the algorithm before presenting the three benchmark problems solved in the diaphragm geometries.

\subsubsection{Parameter values}

The method involves first creating the patch cover over the domain and boundary, generating the centre and evaluation points and subsequently formulating the system of equations \eqref{eq:leastSquaresProb} using the local RBF basis and weight functions. Several parameters involved in the method formulation are selected. A certain number of patches, $P$, are generated, each with $n$ centre points. These two parameters are varied in the following convergence studies. Regarding the cover generation, we select a fixed overlap for $P = 20$, given by $\tilde{\delta}_0$, since the relative overlap $\delta_0$ increases during refinement (see \eqref{eq:OverlapFillDistRelation}). Figure \ref{fig:overlapConv} shows that the effect of the overlap on the error is minimal, unless it becomes very large. In this case patches that are fully overlapped are detected by the algorithm and deleted, while the remaining are extended, coarsening the point cloud. This is necessary to ensure the problem is discretely coercive, as indicated in the quadrature error derivation shown in Appendix \ref{app:intErrorEstimates}.  Hence we choose a relatively small overlap at $P = 20$, $\tilde{\delta}_0 = 0.1$. Even though a small overlap results in large weight function gradients in the overlapped regions, the effect is less significant for solving a PDE than for the reconstruction (see \cite{larsson2024rbfpartitionunitymethod}). Additionally, a small overlap results in a less dense linear system, as shown in Figure \ref{fig:spyMatrix}, which hence is less computationally expensive to solve.

The oversampling, $q$, influences the quadrature error \eqref{eq:integrationError}. Figure \ref{fig:oversampConv} shows how just refining with respect to oversampling affects the error of Problem 1. Given the limited gain from increasing $q$ and the increase in cost for larger values of oversampling, we choose $q = 2.5$. Note that $q > 1$ is a strict condition for each patch.
Generally, the conditioning of the least squares problem is worse for a low oversampling leading to large errors and hence we specifically require $q > 1.2$ for all patches (see \cite{larsson2022numerical} for numerical experiments looking at oversampling for the least squares RBF-FD method). 

The shape parameter for the local approximation \eqref{eq:RBFinterpolantLocal} is chosen as $\varepsilon = 0.1$. Basis functions closer to the flat limit can produce better approximations of smooth functions, while conditioning issues are addressed by computing the interpolation matrix inverse using the RBF-QR algorithm \cite{RBFQR,Larsson_Heryudono_Michael_Piret_2026}.

The material parameters present in the boundary value problem \eqref{eq:StrongPDE} are the Young's modulus and the Poisson ratio. These are determined by uniaxial tests along the fibre direction on specimens from 6 human diaphragms by the authors in \cite{Gaur2016}. The paper in question focuses on material parameters at different strain rates that can occur during motor vehicle accidents. We expect low strain rates as well as low strain during breathing or ventilation, hence we use the Young's modulus value for the toe region of the stress-strain curve which is $E = 0.5$~MPa. However, we note that our linear elastic model is limited in describing the deformation of the diaphragm, given that the uniaxial stress strain behaviour of specimens suggests the need for a hyperelastic or bilinear model along the fibre direction. Additionally, we note that the relatively small sample set in \cite{Gaur2016} leads to results with a very large standard deviation. The Poisson ratio $\nu = 0.434$ for these diaphragm samples is given in the modelling paper \cite{Gaur2019} by some of the same authors. Moreover, as also noted in \cite{Gaur2019} the density of the diaphragm is very close to the density of water so we use $\rho = 1000~\text{kg/m}^3$. The Young's modulus and the Poisson ratio are used to compute the Lam\'e parameters \eqref{eq:LamePars} while the density is used to compute the forcing in Problems 2 and 3. 

The weighted least squares problem solved also includes the constants $\beta=\gamma=h^{-1}$, which are determined by the coercivity bound for the weighted bilinear form, $\tilde{\alpha}(\cdot,\cdot)$ \eqref{eq:ContBilinear2WeightBilinear}. We additionally multiply the Neumann and domain terms ($D^{\mathcal{B}_1},~D^{\mathcal{L}},~\mathbf{f},~\mathbf{t}$) in the least-squares problem \eqref{eq:normalEqutions} by $0.5/\mu$ which significantly improves conditioning. 

\begin{figure}[ht!]
\centering
    \begin{subfigure}{.34\textwidth}
        \centering
        \includegraphics[width=\linewidth]{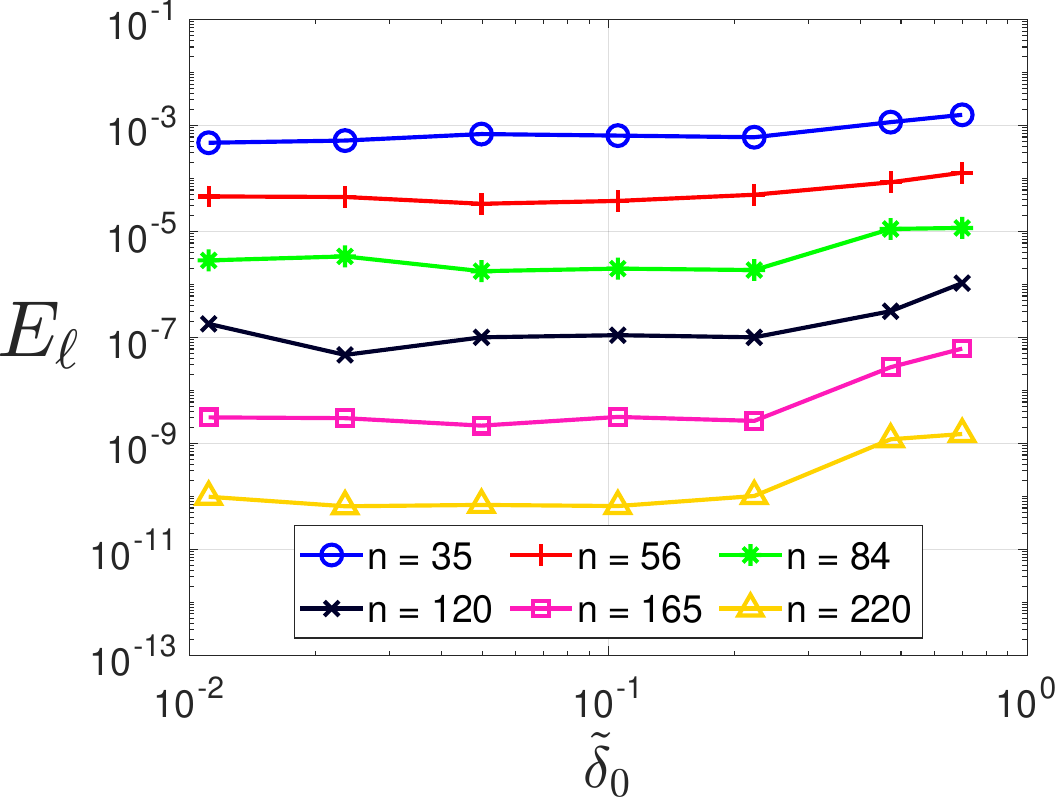}
        \caption{}
        \label{fig:overlapConv}
    \end{subfigure}
    \hspace{.3cm}
    \begin{subfigure}{.34\textwidth}
        \centering
        \includegraphics[width=\linewidth]{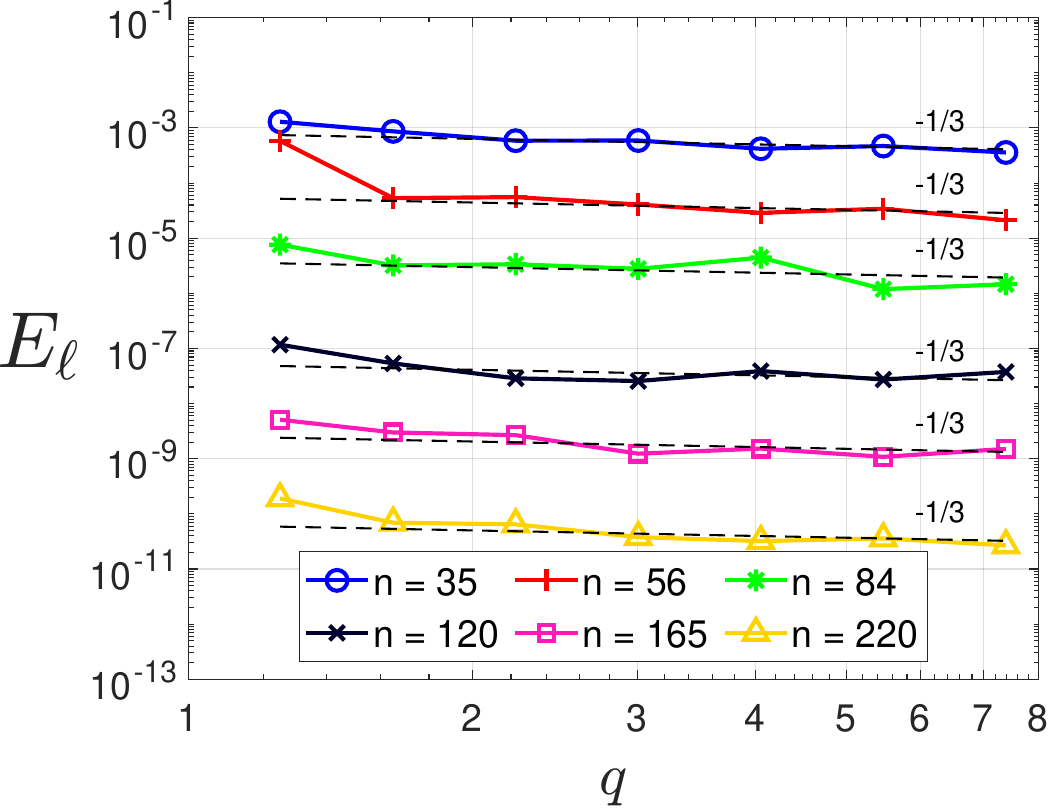}
        \caption{}
        \label{fig:oversampConv}
    \end{subfigure}
    \hspace{.3cm}
    \begin{subfigure}{.25\textwidth}
        \centering
        \includegraphics[width=\linewidth]{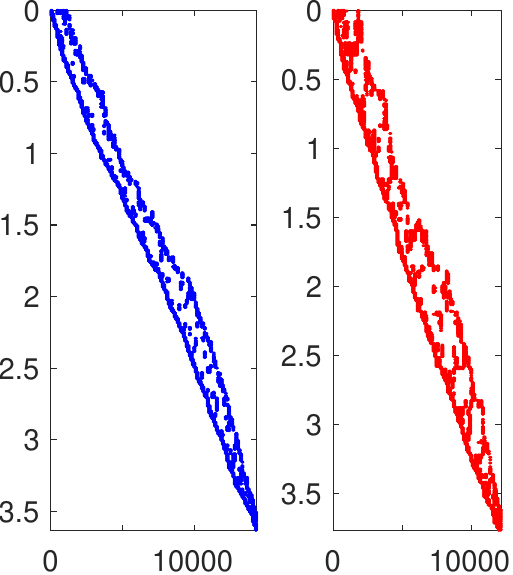}
        \caption{}
        \label{fig:spyMatrix}
    \end{subfigure}
    \caption{Numerical error from solving Problem 1 (smooth manufactured solution)  (a) with varying overlap and fixed oversampling $q = 2.5$ (b) with varying oversampling and fixed overlap 
   $\tilde{\delta}_0 = 0.1$. The solution is computed using $P = 120$ patches and $n = 35,~56,~84,~120,~165, \mbox{ and } 220$ local points. (c) Global rectangular approximation matrices~D \eqref{eq:leastSquaresProb} constructed using a set of $P = 120$ patches, $n = 35$ local points, an oversampling factor $q = 2.5$ and overlap $\tilde{\delta}_0 = 0.01$ (left) and $\tilde{\delta}_0 = 0.3$ (right). The density of the two matrices is $0.013$ (left) and $0.023$ (right). }
    \label{fig:paramConv}
\end{figure}

\subsubsection{Problem 1: Smooth manufactured solution}

\begin{figure}[ht!]
\centering
    \begin{subfigure}{.51\textwidth}
        \centering
        \includegraphics[width=\linewidth]{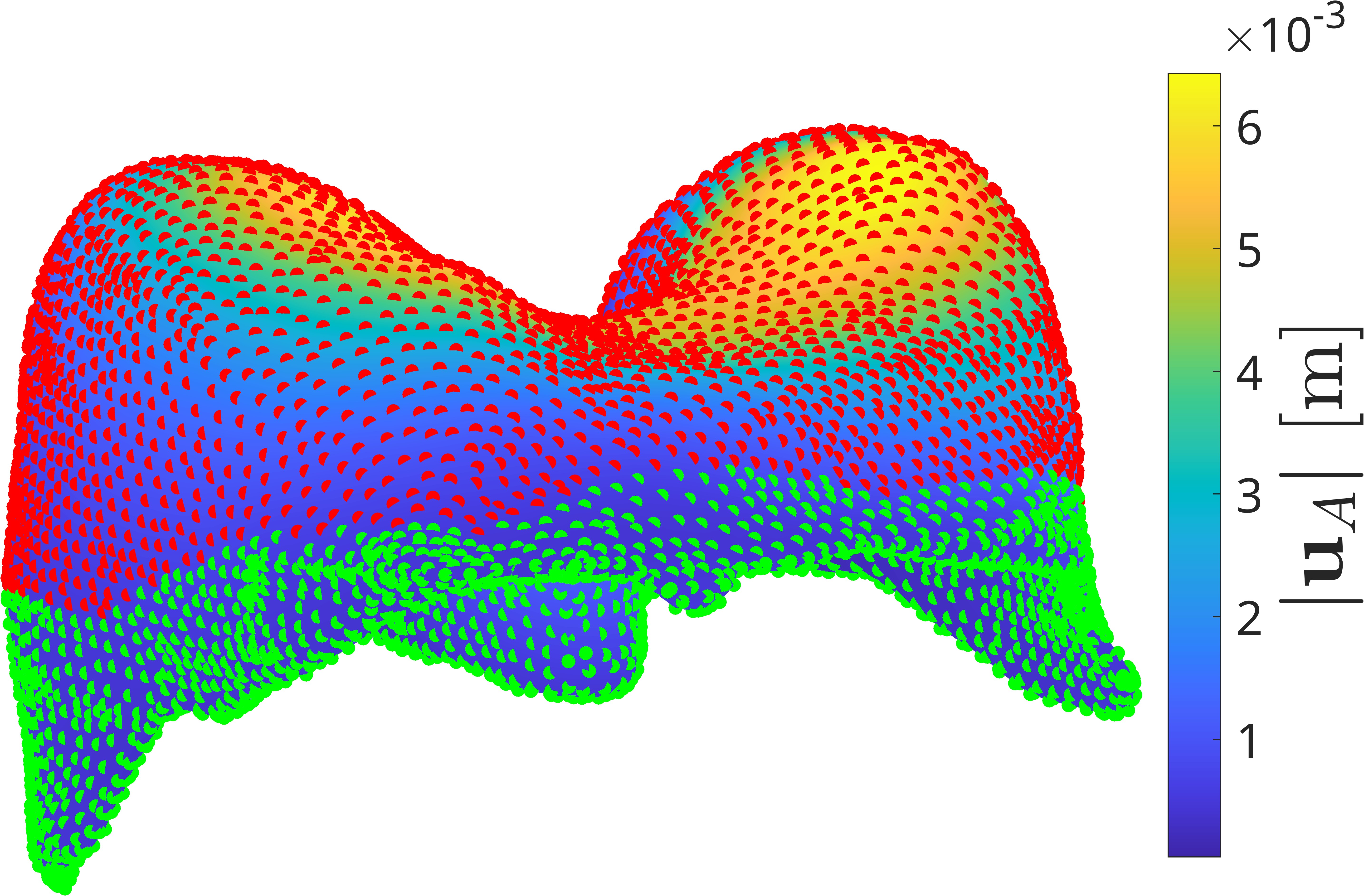}
        \caption{}
        \label{fig:problem1_form}
    \end{subfigure}
    \hspace{.3cm}
    \begin{subfigure}{.42\textwidth}
        \centering
        \includegraphics[width=\linewidth]{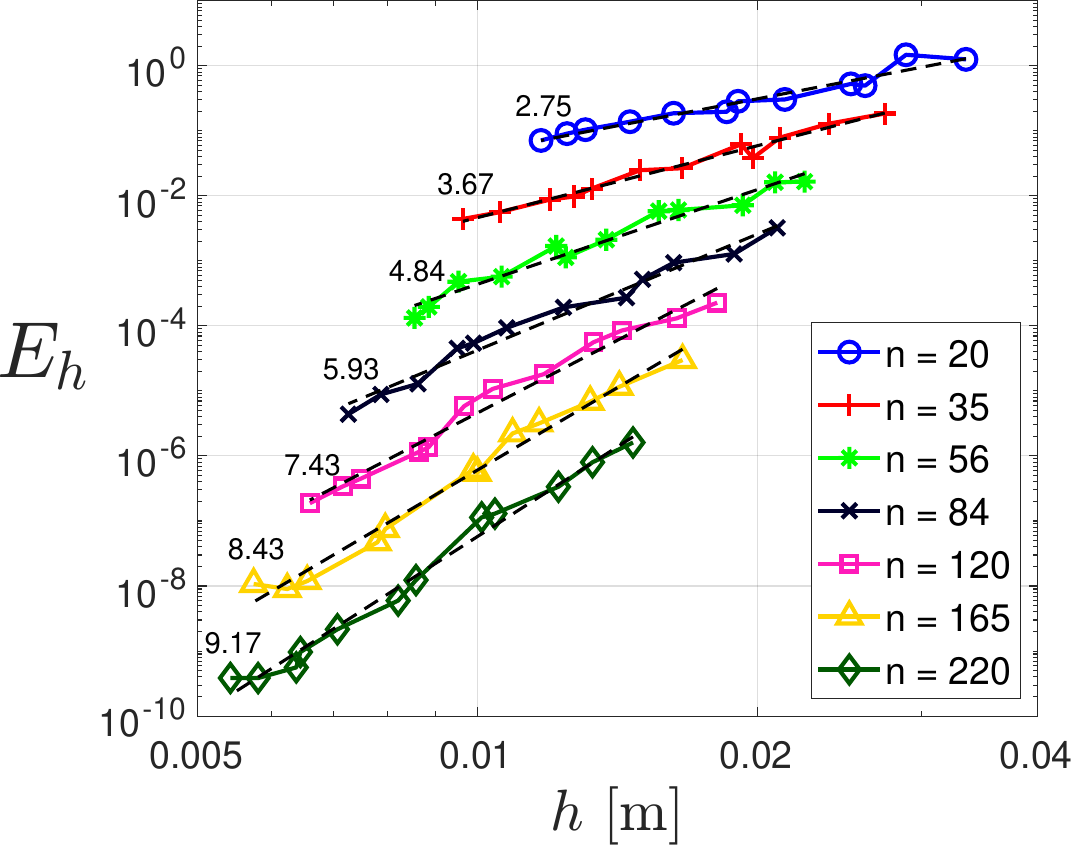}
        \caption{}
        \label{fig:problem1_conv}
    \end{subfigure}
    \caption{(a) Diaphragm 2 and the magnitude of the smooth manufactured solution \eqref{eq:smooth_ManSol} on the boundary points $Y_{\textrm{bnd}}$. The Neumann and the Dirichlet regions are shown by red and green points respectively. (b) Problem 1: convergence study for patch numbers $P = 21$ to $P = 315$ and the number of local centre points $n$ corresponding polynomials of degree $p = 3,4,\ldots, 9$. }
    \label{fig:problem1}
\end{figure}

We solve \eqref{eq:leastSquaresProb} assuming the existence of a smooth solution which is used to compute both the forcing, $\mathbf{f}$ and boundary conditions $\mathbf{g}$ and $\mathbf{t}$. The manufactured solution is given by 
\begin{equation}
    \mathbf{u}_A(\mathbf{y}) = \begin{pmatrix}
        -5\times 10^{-3} \left(\frac{\left(x-x_{\min}\right)\left(x-x_{\max}\right)}{L} + \frac{\left(x_{\max}-x_{\min}\right)^2}{4L}\right) \\
        -2\times 10^{-3}\frac{\left(y-y_{\min}\right)^4}{W^3} \\
        \frac{\left(x-2x_{\min}\right)^2\left(x-1.5x_{\max}\right)^2}{L^4}\frac{\left(y-y_{\min}\right)^2\left(y-y_{\max}\right)^2}{W^4}\frac{\left(z-z_{\min}\right)^2}{H}e^{-\left(\frac{x-\bar{x}}{L}\right)^2 - 
        \left(\frac{y-\bar{y}}{W}\right)^2 - 
        \left(\frac{z-\bar{z}}{H}\right)^2}
    \end{pmatrix},
    \label{eq:smooth_ManSol}
\end{equation}
where $\mathbf{y} = \left(x,y,z\right) \in\bar{\Omega}$, $\bar{\mathbf{y}} = \left(\bar{x},\bar{y},\bar{z}\right)$ is an approximation of the centre of mass of $\bar{\Omega}$, computed by using the point set $Y_{\textrm{bnd}}$, and both the position and the displacement are given in meters. We also note that $x_{\max} = \max_{\mathbf{\mathbf{y}}\in Y_{\textrm{bnd}}} x, ~x_{\min} = \min_{\mathbf{y}\in Y_{\textrm{bnd}}} x,~ L = x_{\max} - x_{\min}$, and equivalently for $W,~H$ in the $y$ and $z$ directions respectively. The solution is chosen such that it loosely replicates the movement of the diaphragm during mechanical ventilation. As such the displacement amplitude is largest on the domes which move in the superior (positive $z$) direction, while the region closer to the sternum, the intercostal muscles and the spine displaces slightly towards the centre of mass of the diaphragm. The amplitude of the smooth solution on the point set $Y_{\textrm{bnd}}$ is shown in Figure \ref{fig:problem1_form}. The figure also indicates the area of the diaphragm where the Dirichlet and the Neumann conditions are enforced, showing that for this problem we use discontinuous Robin coefficients $\kappa_0$ and $\kappa_1$. Even though this contradicts the coercivity requirements in Theorem \ref{theorem:Coercivity} the existence of the manufactured solution is guaranteed by construction.

The convergence study in Figure \ref{fig:problem1_conv} plots the discrete error \eqref{eq:numericalErrorh2} against the fill distance. The convergence rates are expected to be $p - 2.5$ \eqref{eq:FinalApriori} which corresponds to $0.5,~1.5,~2.5,~3.5,~4.5,~5.5,~\mbox{and}~6.5$ for the number of local points used in the study. The numerical rates are significantly higher which is potentially because the manufactured solution can be easily approximated by Gaussians. The convergence plot for $E_{\ell}$ \eqref{eq:numericalErrorl2} is not shown in the figure, since the orders are very similar to the ones for the discrete Sobolev norm of the error. Since the error in the displacement is more physical and the convergence rates are very similar, we simply plot the $E_{\ell}$ error for Problems 2 and 3. 
                       
\subsubsection{Problem 2: Dirichlet conditions from medical image data}

\begin{figure}[ht!]
\centering
    \begin{subfigure}[b]{.4\textwidth}
        \centering
        \includegraphics[width=\linewidth]{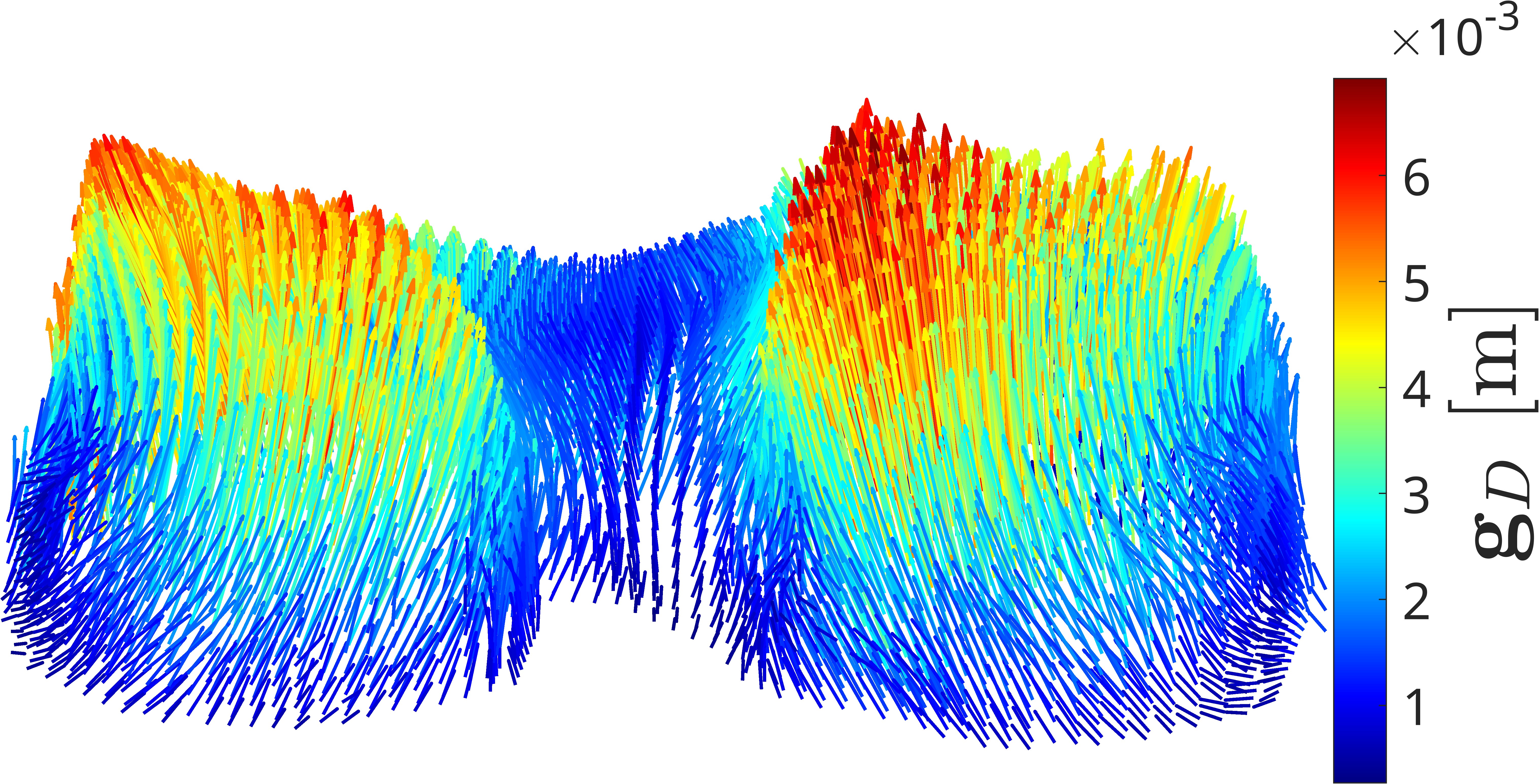}
        \vspace{.2cm}
        \includegraphics[width=\linewidth]{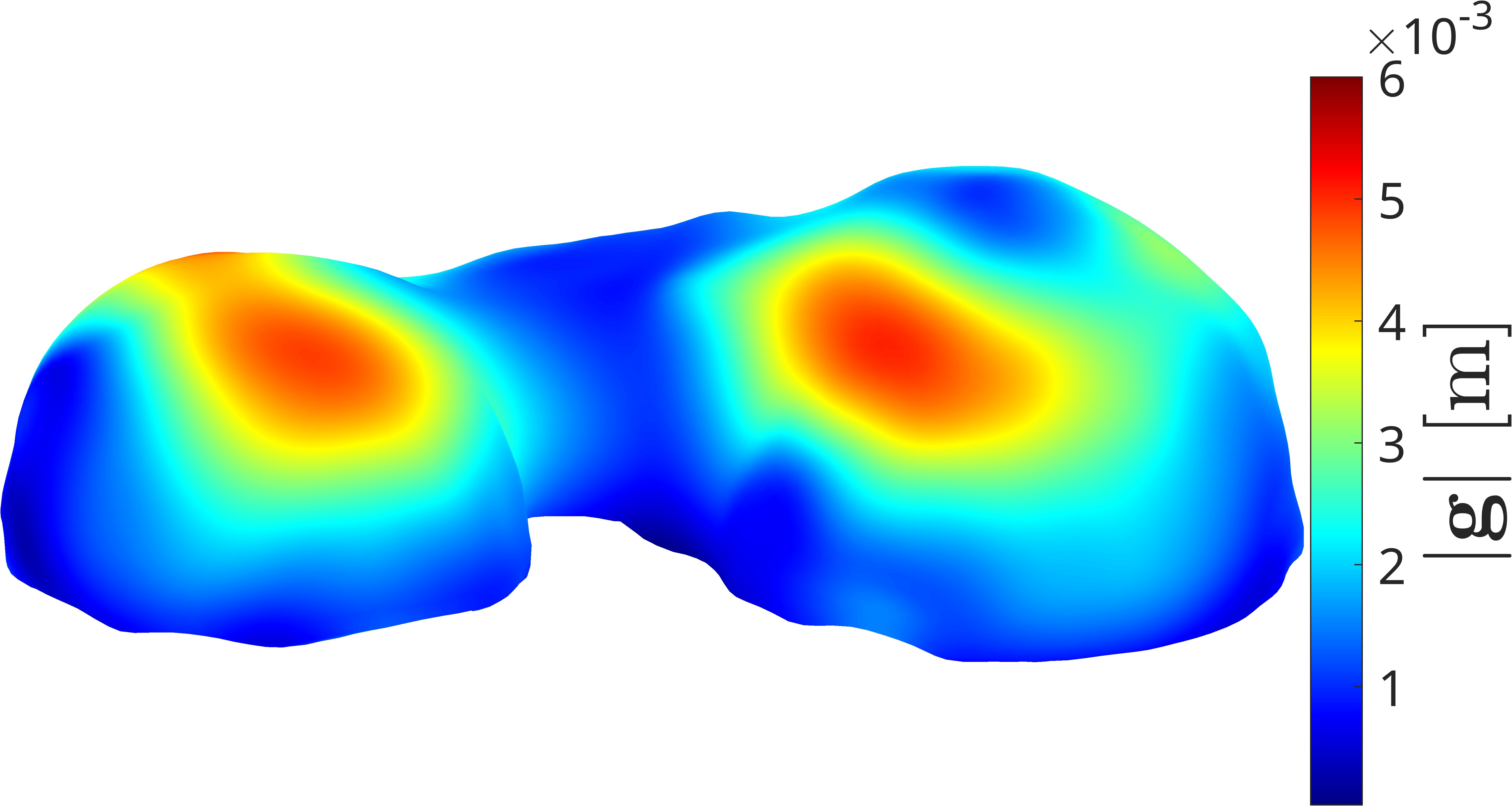}
        \caption{}
        \label{fig:problem2_BC}
    \end{subfigure}
    \hspace{.2cm}
    \begin{subfigure}[b]{.45\textwidth}
        \centering
        \includegraphics[width=\linewidth]{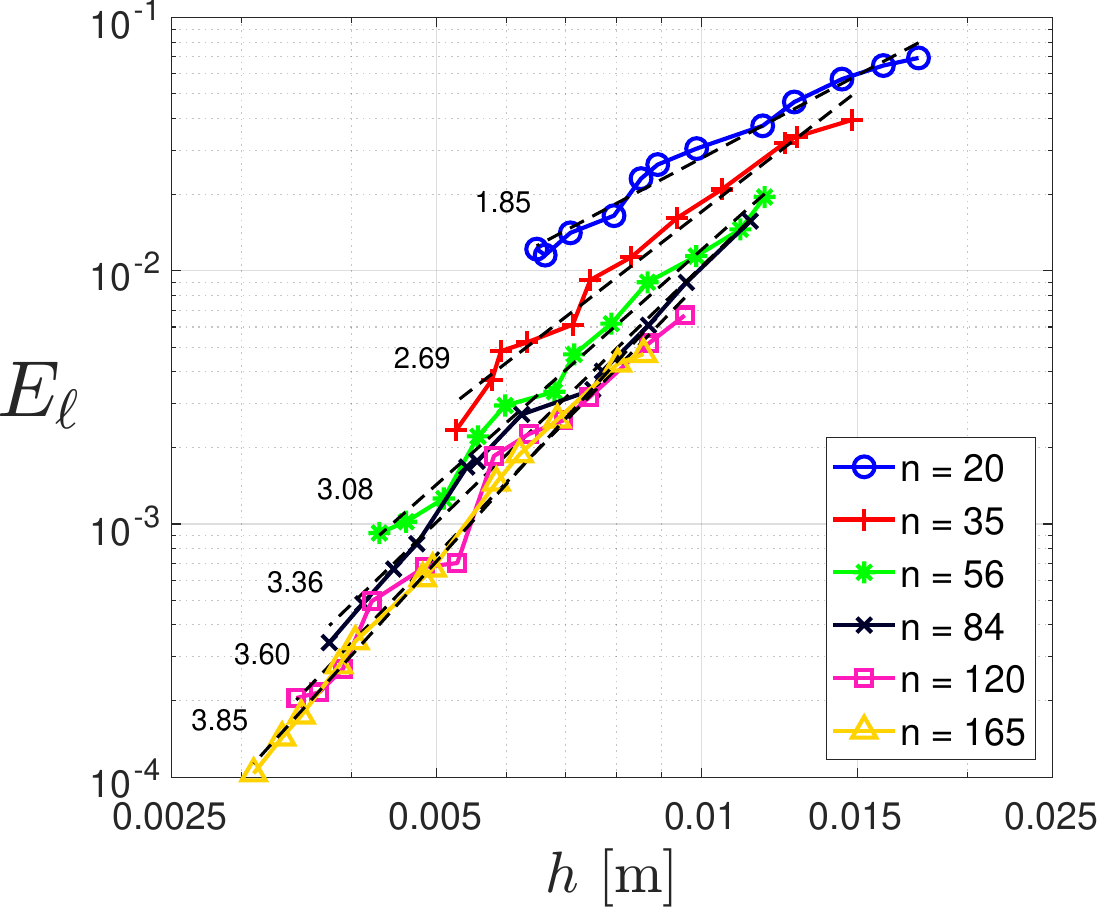}
        \caption{}
        \label{fig:problem2_conv}
    \end{subfigure}
    \caption{(a) Deformation vector field extracted from CT images on the point set $X_{\mathrm{BC}}$ (top) and the magnitude of the smoothed Dirichlet boundary condition on the point set $Y_{\textrm{bnd}}$ (bottom). (b) Problem 2: convergence study for patch numbers $P = 21$ to $P = 315$ and the number of local centre points $n$ corresponding to the dimension of polynomials of degree $p = 3,4,\dots 8$. The error is computed against a reference numerical solution with $P = 315,~n = 220$, giving $207900$ degrees of freedom. }
    \label{fig:problem2}
\end{figure}
For this problem we solve \eqref{eq:leastSquaresProb} assuming $\kappa_0 = 1$ everywhere on the boundary, meaning that we impose Dirichlet boundary conditions only. Additionally, we set the forcing $\mathbf{f} = (0,0,-\rho g)$, where $\rho = 1000~\text{kg/m}^3$ is the density of the diaphragm and $g = 9.81~\text{m/s}^2$ is the gravitational acceleration. The Dirichlet condition is constructed from the expected physical movement of Diaphragm 1 during respiration.

Given that the pre-processing step for our simulation starts with the segmentation of CT images, we are able to use a given set of images taken from the same diaphragm during respiration at different time steps. Two images of the same diaphragm are then used to create a displacement vector field which provides a reference movement. The vector field is provided on a specific point set, $X_{\mathrm{BC}}$, by using the image registration module of the Medical Image Registration ToolKit (MIRTK) \cite{Mirtk}. The displacement vector field on this point set, $\mathbf{g}_D$, is shown in Figure \ref{fig:problem2_BC} (top).

The vector field $\mathbf{g}_D$ is defined pointwise and so it is fully discrete. It describes a step in the respiratory motion of Diaphragm 1 and hence is used to compute a Dirichlet boundary condition on $\partial\Omega$. We require that the Dirichlet boundary condition $\mathbf{g}$ is smooth enough to ensure the existence of a solution to the continuous problem (see Theorem \ref{theorem:Coercivity}). Therefore, we solve the following discrete least-squares projection problem using the LS-RBF-PU method
\begin{equation}
 \min_{\mathbf{g}}\|\mathbf{g} - \mathbf{g}_D\|^2_{\ell_2\left(X_{\mathrm{BC}}\right)} + C_{\Delta}h^{-4}\|\Delta\mathbf{g}\|^2_{\ell_2\left(X_{\mathrm{BC}}\right)},
 \label{eq:problem2DirichletProjection}
\end{equation}
where $\|\mathbf{g}\|_{\ell_2(X_{\mathrm{BC}})} = \sum_{\mathbf{x}_i\in X_{\mathrm{BC}}}^M W_i\|\mathbf{g}(\mathbf{x}_i)\|^2_2$ is the discrete $\ell_2$-norm on the point set $X_{\mathrm{BC}}$ and $h$ is the fill distance on the centre point set $X$ constructed for this problem, with weight $C_{\Delta} = 1.5$ for the regularisation term which is used to smooth the boundary condition. This problem is solved once, using a cover of $P = 30$ patches, $n = 35$ local points, an overlap of $\tilde{\delta}_0 = 0.4$ and shape parameter $\varepsilon = 0.1$. The weight functions used in each patch are generated using the bump function $f_W(r) = \exp{(-1/(1-r^2))},~r < 1$, which gives $\mathbf{g}\in C^{\infty}(\partial\Omega)$. Ensuring $\mathbf{g}$ is much smoother than the minimum requirement for well-posedness is necessary to improve convergence rates. Once solved, the boundary condition is simply evaluated on the boundary evaluation points $Y_{\partial\Omega}$. In Figure \ref{fig:problem2_BC} we show the smoothed boundary condition on the point set $Y_{\textrm{bnd}}$.

Figure \ref{fig:problem2_conv} shows that high convergence rates can be achieved for this problem, but the rates are lower than the theoretically expected. Given that this is not the case for the manufactured problem (Problem 1), we assume that more refinement is required to capture the larger gradients and second derivatives that are present in the solution. Moreover, the sharpness of the surface gradient and the surface Hessian of the level set function $\ell (\mathbf{y})$ at the lower edge influence convergence rates.



\subsubsection{Problem 3: Robin boundary conditions compatible with rib rotation}

\begin{figure}[ht!]
\centering
    \begin{subfigure}{.26\textwidth}
        \centering
        \includegraphics[width=\linewidth]{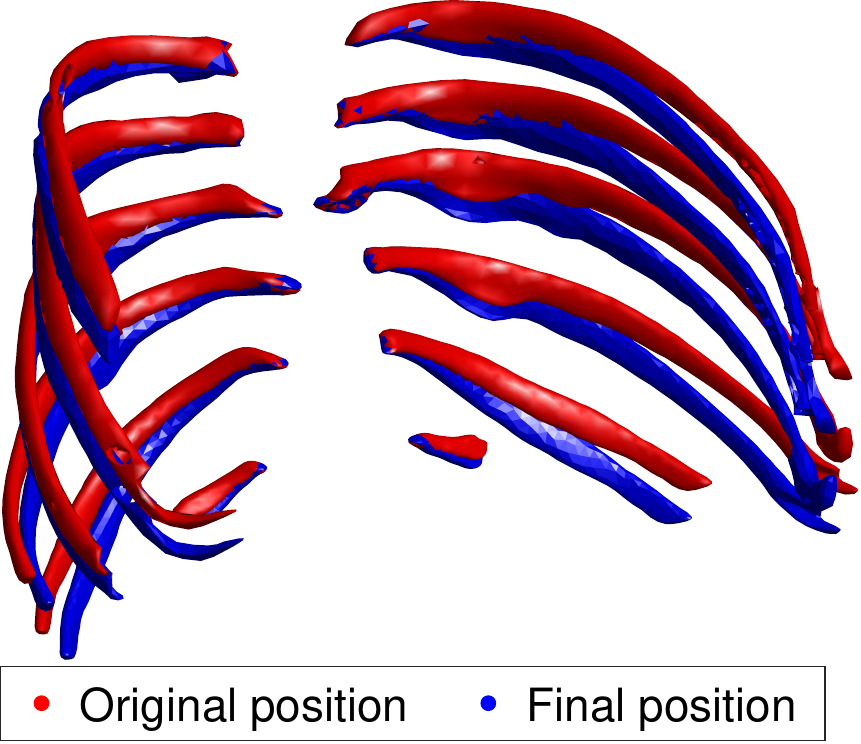}
        \caption{}
        \label{fig:problem3_ribs}
    \end{subfigure}
    \hspace{.3cm}
    \begin{subfigure}{.32\textwidth}
        \centering
        \includegraphics[width=\linewidth]{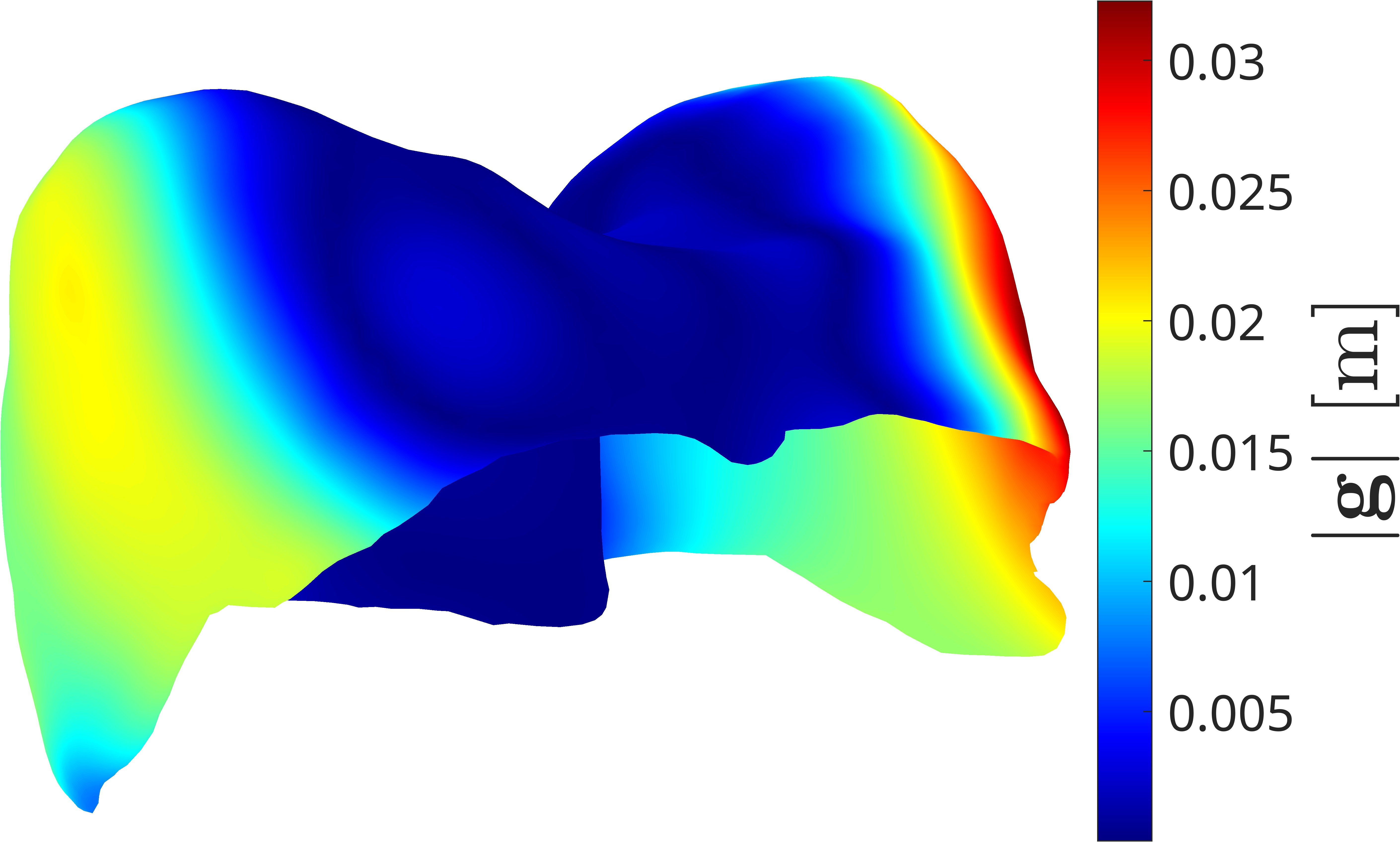}
        \caption{}
        \label{fig:problem3_Dirichlet}
    \end{subfigure}
    \hspace{.3cm}
    \begin{subfigure}{.32\textwidth}
        \centering
        \includegraphics[width=\linewidth]{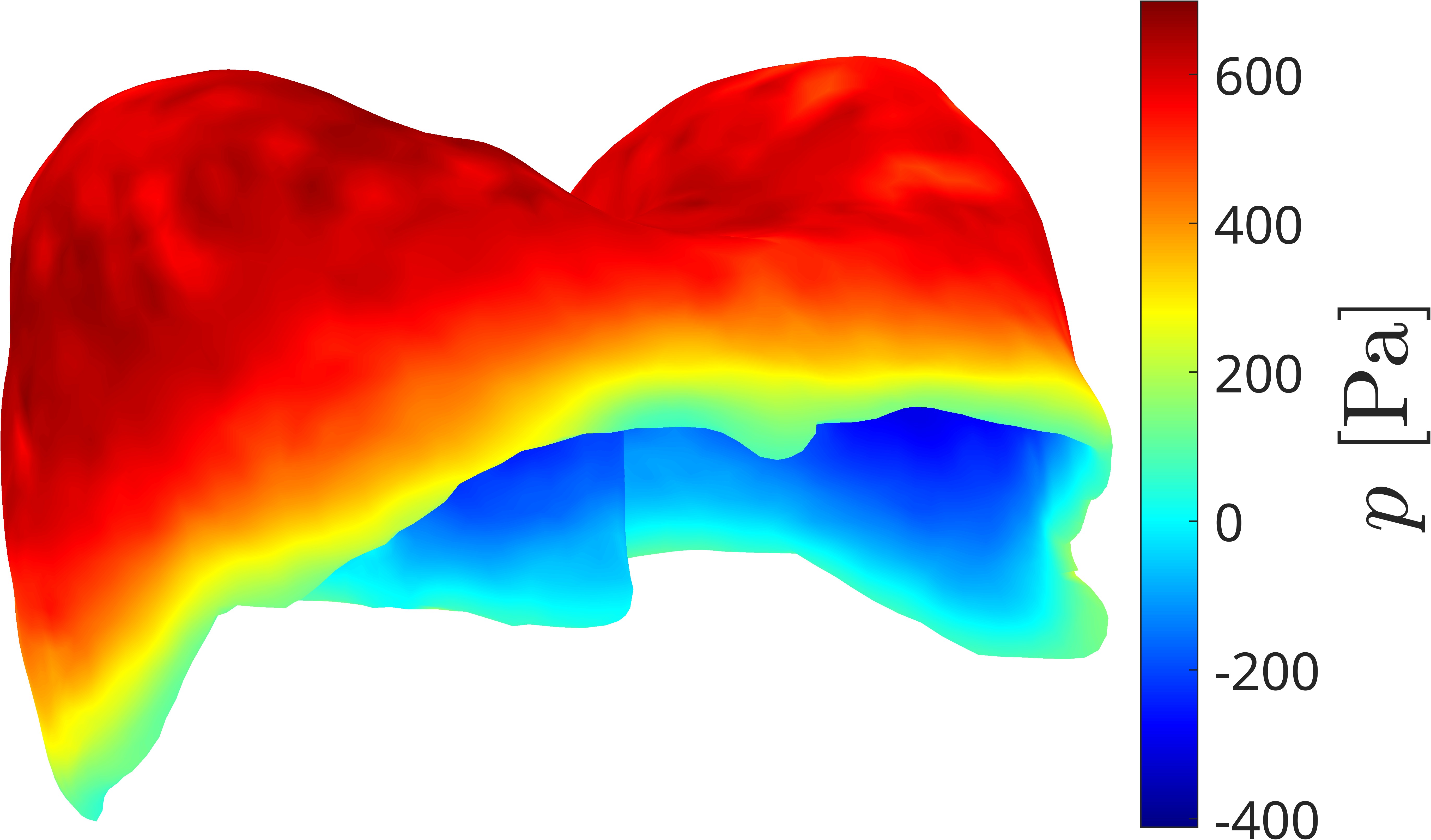}
        \caption{}
        \label{fig:problem3_Neumann}
    \end{subfigure}
    \caption{(a) Rib movement (anterior view). (b) Magnitude of the smoothed Dirichlet boundary condition, $\mathbf{g}$, computed based on the rib rotation and evaluated on the point set $Y_{\mathrm{bnd}}$. (c) Smoothed pressure $p$ used for computing the traction (Neumann) boundary condition on the point set $Y_{\mathrm{bnd}}$. The condition is computed assuming an intra-abdominal pressure $p_A = 800$ Pa and a pleural pressure $p_L = -500$ Pa. }
    \label{fig:problem3_Ribs_k0_k1}
\end{figure}

\begin{figure}
\centering
    \begin{subfigure}{.42\textwidth}
        \centering
        \includegraphics[width=\linewidth]{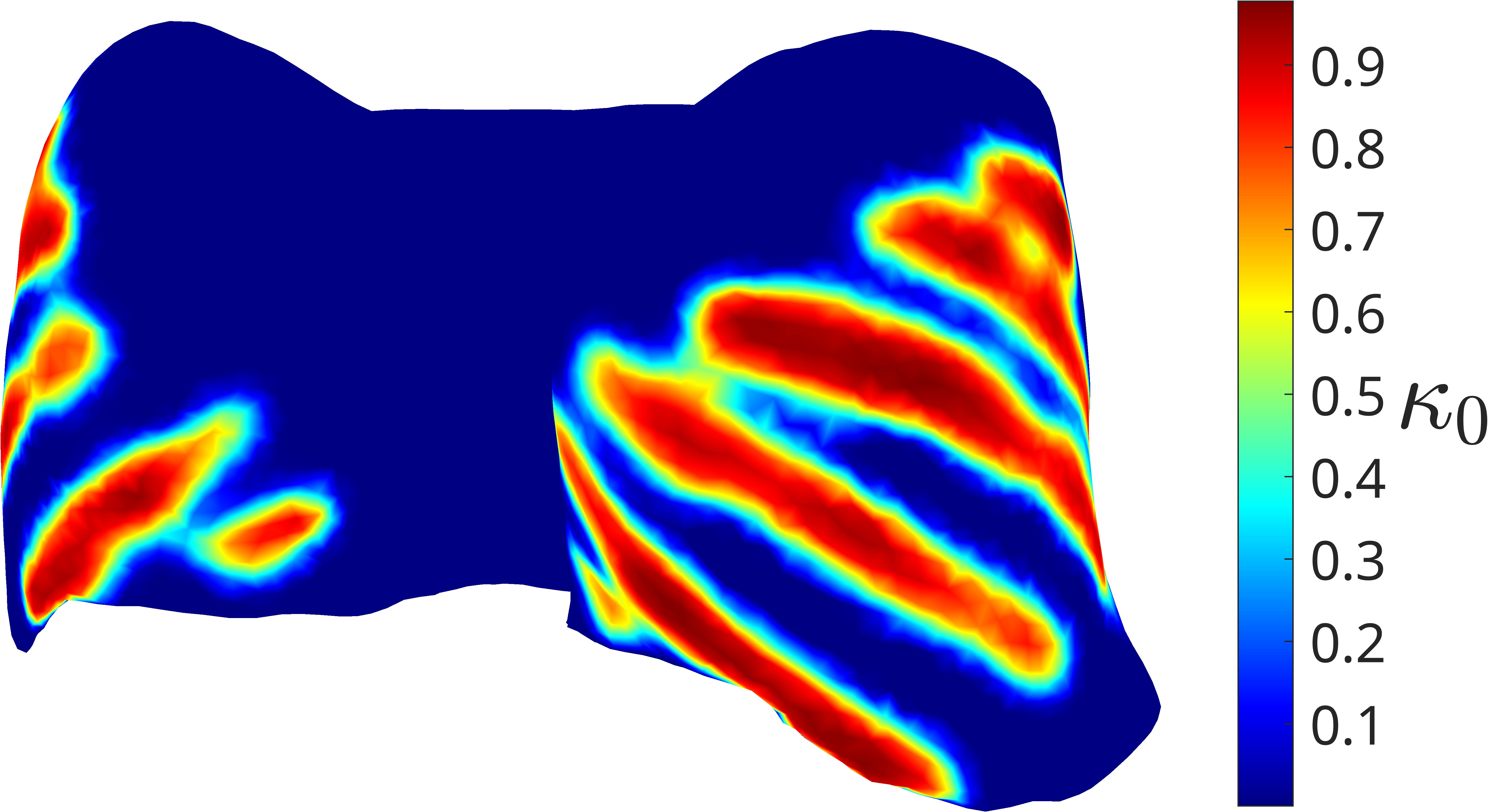}
        \caption{}
        \label{fig:problem3_k1}
    \end{subfigure}
    \hspace{.3cm}
    \begin{subfigure}{.42\textwidth}
        \centering
        \includegraphics[width=\linewidth]{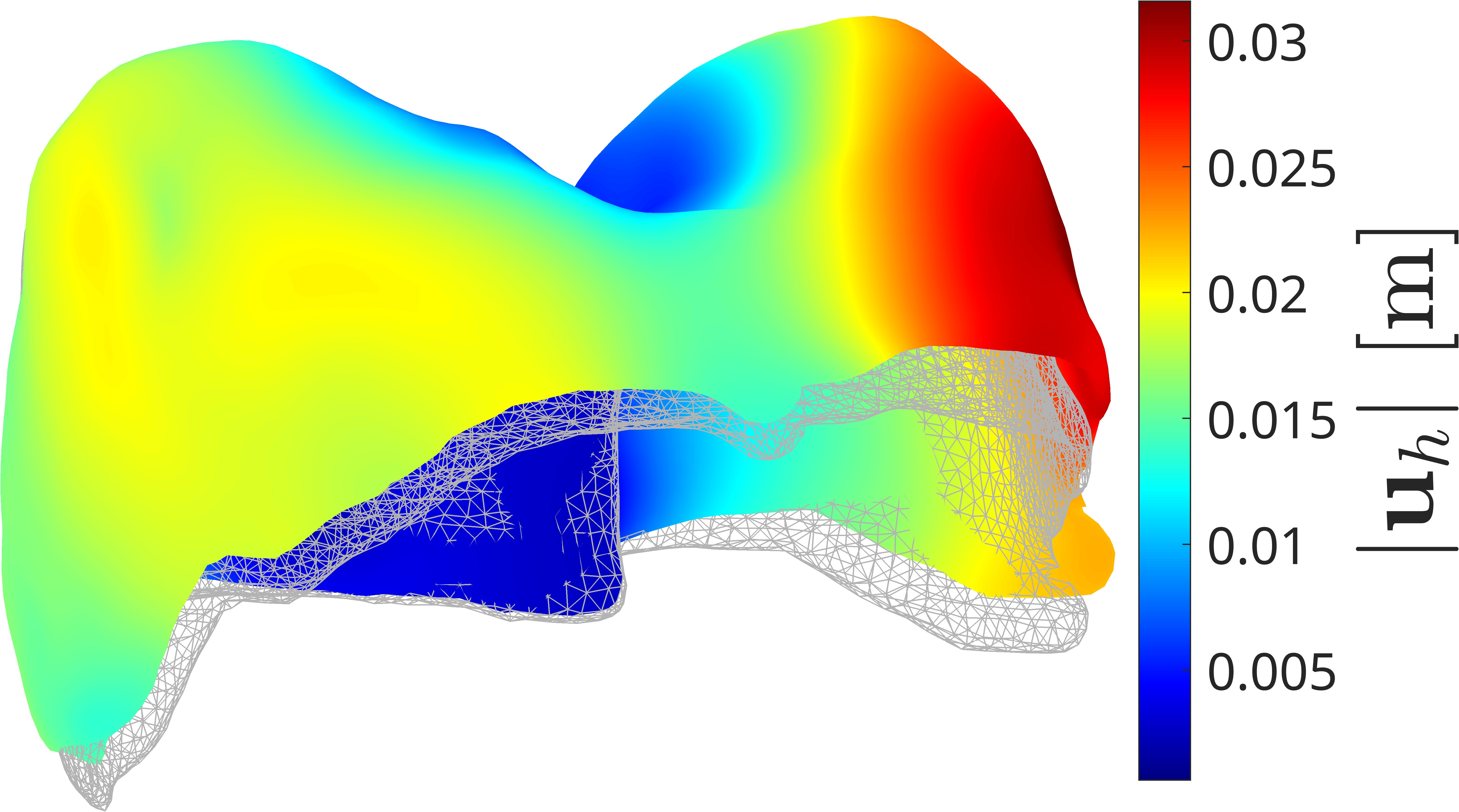}
        \caption{}
        \label{fig:problem3_Sol}
    \end{subfigure}
    \caption{(a) Robin coefficient $\kappa_0$ on the point set $Y_{\mathrm{bnd}}$. Note that the coefficient $\kappa_1 = 1- \kappa_0$. The coefficient is evaluated on the point set $Y_{\mathrm{bnd}}$ (b) Magnitude of the numerical solution of problem 3 computed with $P = 315$ patches and $n = 220$ local points in each patch and $q = 2.5,~\tilde{\delta}_0 = 0.1,~\varepsilon = 0.1$. The reference position of Diaphragm 2 is indicated by the gray wireframe. }
    \label{fig:problem3DirichletNeumann}
\end{figure}

\begin{figure}
\centering
    \begin{subfigure}{.45\textwidth}
        \centering
        \includegraphics[width=\linewidth]{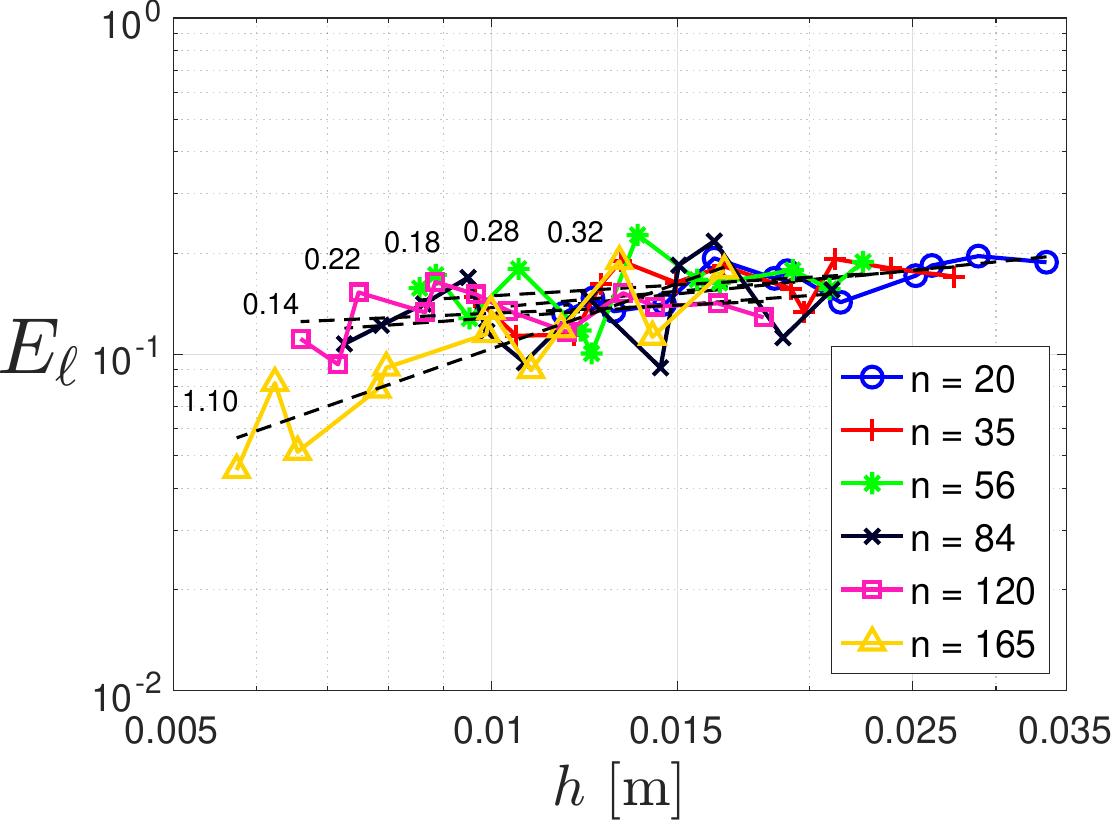}
        \caption{}
        \label{fig:problem3_conv}
    \end{subfigure}
    \hspace{.3cm}
    \begin{subfigure}{.45\textwidth}
        \centering
        \includegraphics[width=\linewidth]{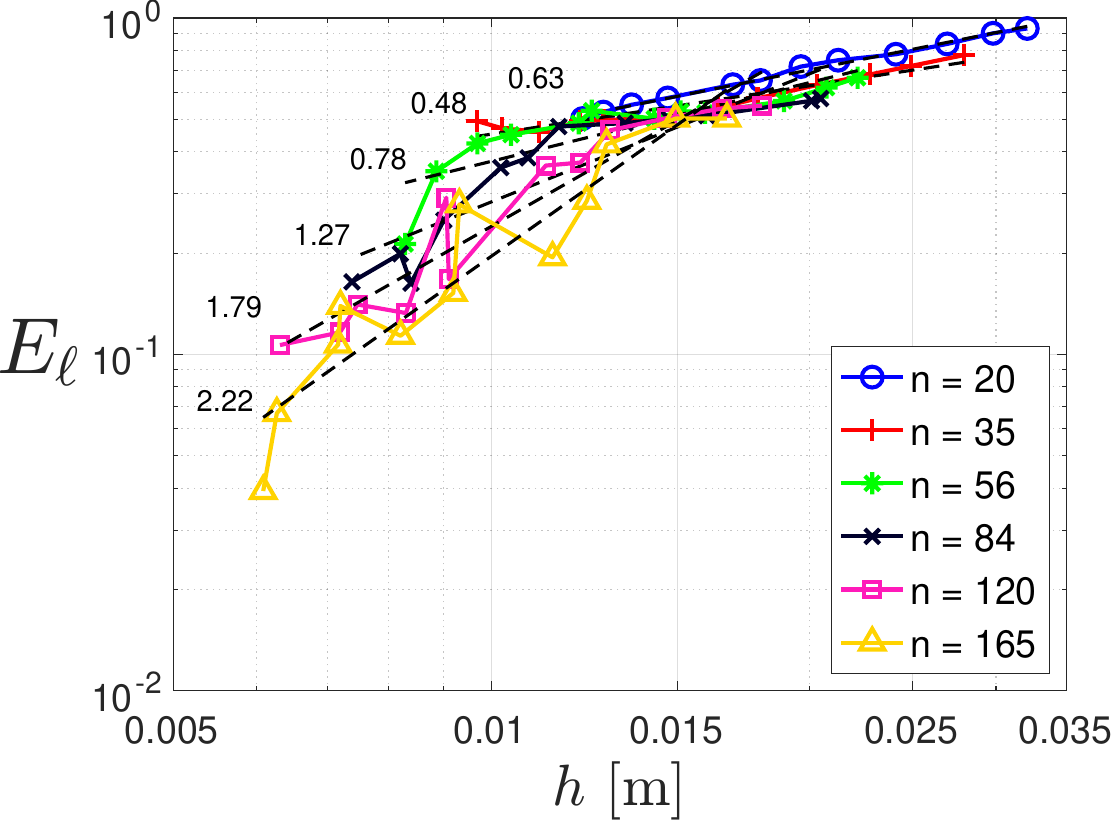}
        \caption{}
        \label{fig:problem3_residual}
    \end{subfigure}
    \caption{Convergence study for Problem 3 using patch numbers $P = 21$ to $P = 315$ and the number of local centre points $n$ corresponding to the dimension of polynomials of degree $p = 3,4,\dots 8$. The normalised error, $E_{\ell}$, is plotted against the fill distance and it is computed against a reference numerical solution with $P = 315,~n=220$ which corresponds to $207900$ degrees of freedom. The two plots correspond to the different shape parameter values that were used for constructing the global basis functions, with (a) $\varepsilon = 0.1$ and (b) $\varepsilon = 3$. }
    \label{fig:problem3}
\end{figure}
In this benchmark, we solve \eqref{eq:leastSquaresProb} with Robin boundary conditions, using $0 < \kappa_0 < 1$ and $\kappa_1 = \kappa_0 - 1$. Similarly to Problem 2, the forcing term is equal to a constant $\mathbf{f} = (0,~0~,-\rho g)$ and the boundary conditions are defined based on realistic parameters related to respiration. These include the rotation of the ribs connected to the diaphragm and the intra-abdominal and pleural pressures.

The diaphragm is attached to the costal cartilages of ribs 7--12 and additionally connects to rib $12$ through the lateral arcuate ligament, vertebrae L1--L3 through crura and arcuate ligaments, as well as the sternum and pericardium \cite{ryan2010anatomy}. We extract the rib geometries from the CT data used for the reconstruction of Diaphragm~2. The strategy is to take the rib rotation into account to compute boundary conditions and form there the deformation of the diaphragm. Note that we omit the simulation of cartilage and ligaments making this a significant simplification of the real connectivity between the diaphragm, rib cage and other skeltal parts. 

To model the rib kinematics, we first extracted the rotation axes for each rib by computing the eigenvectors of the inertia tensor derived from their 3D geometries as in \cite{villard2011}. It consists of two orthogonal rotation axes per rib: the pump-handle axis (aligned with the rib) and the bucket-handle axis (orthogonal to both the rib and the vertical axis) as described in \cite{wilson1987}. A centre of rotation for each rib was estimated based on the proximity to the spine mesh. The local kinematic parameters were then used to drive rigid rotations of individual ribs.

We obtain rotation angles for the pump-handle and the bucket handle axes from \cite{wilson2001respiratory}, where CT scans of 5 individuals were assessed. The rotation angles for ribs 7--9 are $7.9,~7.9,~6$ for the bucket-handle axis and $6.6,~6.2,~6.3$ for the pump-handle axis.
Given that the rotations for ribs 10--12 are not provided in the literature, we simply use $5$ degrees for both axes. This is also a simplification, especially for ribs 11--12 which are expected to experience an additional caliper-like motion \cite{BastirThoraxMotion2017}. The complete movement of ribs 7--12 is illustrated in Figure \ref{fig:problem3_ribs}.

We use this rib movement to generate the Dirichlet boundary condition, $\mathbf{g}$. Note that we construct the boundary conditions once on the reconstruction points $Y_{\mathrm{bnd}}$ and then simply evaluate the conditions on the point set used to solve the PDE problem. We start by setting the deformation of each point in $Y_{\mathrm{bnd}}$ equal to the deformation of the closest point on the ribs but set the deformation for points significantly far from the ribs to zero, with a smooth transition enabled by the use of a sigmoid function. The resulting displacement vector field is denoted by $\mathbf{g}_D$ and to obtain the smoothed continuous boundary condition, $\mathbf{g}\in C^{\infty}(\partial\Omega)$, we solve the following discrete least squares projection problem
\begin{equation}
    \min_{\mathbf{g}}\|\mathbf{g}- \mathbf{g}_D\|^2_{\ell_2(Y_{\mathrm{bnd}})} + C_{\Delta}h^{-4}\|\Delta\mathbf{g}\|_{\ell_2(Y_{\mathrm{bnd}})},
    \label{eq:problem3DirichletProjection}
\end{equation}
which is similar to \eqref{eq:problem2DirichletProjection}, but we set $C_{\Delta} = 100$. The large value is necessary to penalise large derivatives and discontinuities, which exist as a result of the way we construct the data $\mathbf{g}_D$. The resulting boundary condition is shown in Figure \ref{fig:problem3_Dirichlet} and is computed using $P = 80$ patches, $n = 35$ local points, an overlap of $\tilde{\delta}_0 = 0.4$ and shape parameter $\varepsilon = 0.1$. As in \eqref{eq:problem2DirichletProjection} we use the bump function to construct the weights in each patch.

For the traction (Neumann) condition we note that the inferior part of the diaphragm is under intra-abdominal pressure and the superior part under pleural pressure. Values for the pleural pressure vary between inspiration and expiration. Assuming that we start the simulation prior to the inspiration stage we use a pleural pressure of $p_L = -500$~Pa and an intra-abdominal pressure of $p_A = 800$~Pa \cite{guyton2006textbook}. Here it is important to note that the point set $Y_{\mathrm{bnd}}$ is split into three subsets, one for the points on the superior part of the diaphragm, $Y_{\mathrm{outer}}$, one for points on the inferior, $Y_{\mathrm{inner}}$ and one for the points in between, $Y_{\mathrm{edge}}$. Using the pressure parameters and the point sets we set up and solve the following projection problem
\begin{align}
\begin{split}
    & \min_{p}\|p - p_A\|^2_{\ell_2(Y_{\mathrm{inner}})} + \|p - p_L\|^2_{\ell_2(Y_{\mathrm{outer}})} \\  & +  C_{\nabla}h^{-2}\left(\|\nabla p\cdot \mathbf{t}_1\|^2_{\ell_2(Y_{\mathrm{bnd}})}  
    + \|\nabla p\cdot \mathbf{t}_2\|^2_{\ell_2(Y_{\mathrm{bnd}})}\right) + C_{\mathcal{H}^S}h^{-4}\|\mathcal{H}^S p\|^2_{\ell_2(Y_{\mathrm{bnd}})},
    \end{split}
    \label{eq:problem3NeumannProjection}
\end{align}
where $\mathbf{t}_1(\mathbf{y}),~\mathbf{t}_2(\mathbf{y})$ are the tangents at the boundary point $\mathbf{y}\in Y_{\mathrm{bnd}}$ and $\mathcal{H}^Sp$ is the surface Hessian of the pressure. For a definition of the surface Hessian see \cite{larsson2024rbfpartitionunitymethod}. Note that here we regularise with the surface Hessian term as opposed to the Laplace operator since we wish to smooth tangentially to the surface, while the gradient terms reduce small oscillations which appear due to the approximation with a very smooth basis. The constants used are $C_{\nabla} = 2$ and $C_{\mathcal{H}^S} = 0.25$ and the smoothed pressure, $p$, is shown in Figure \ref{fig:problem3_Neumann}. The same parameters as were used for solving \eqref{eq:problem3DirichletProjection} are used and the resulting traction is then computed using the surface normals as $\mathbf{t} = p\mathbf{n}$.

The smooth coefficient, $\kappa_0$, shown in Figure \ref{fig:problem3_k1} is similarly computed. In this case however, we construct a new set of evaluation points using $P = 450$ patches, with $n = 35$ local points, an overlap of $\tilde{\delta} = 0.1$ and oversampling factor $q = 4$. The evaluation points set on the boundary is $Y_{\mathrm{robin}}$. From this set we flag the points on the outer surface of the diaphragm that are closest to the ribs, $Y_{\mathrm{near}}$, as well as the points that are further away, $Y_{\mathrm{far}}$. This leaves a set of points in between that is not given any conditions which reduces oscillations in the solution. We then solve the following projection problem:
\begin{equation}
    \min_{\tilde{\kappa}_0}\|\tilde{\kappa}_0-1\|^2_{\ell_2(Y_{\mathrm{near}})} + \|\tilde{\kappa}_0\|^2_{\ell_2(Y_{\mathrm{far}})} + h^{-4}\|\mathcal{H}^S\tilde{\kappa}_0\|^2_{\ell_2(Y_{\mathrm{robin}})}.
\end{equation}
Using a bump function to construct the weights in each patch ensures that $\tilde{\kappa}_0\in C^{\infty}(\partial\Omega)$. However, we are not able to guarantee the strict conditions $0<\kappa_0<1$ given the existence of oscillations. Hence, after evaluating $\tilde{\kappa}_0$ on the evaluation point sets used for solving Problem 3, we use a generalised logistic function as a filter which gives the effective Robin coefficient
\begin{equation}
    \kappa_0 = \frac{1}{(1 + Qe^{-B\tilde{\kappa}_0})^{1/\zeta}},
\end{equation}
where $\zeta = 0.1,~ Q = 2.98,~B = 7.45$ and $\kappa_1 = 1-\kappa_0$.

Using the smoothed forcing, the boundary conditions, and the Robin coefficients we solve \eqref{eq:leastSquaresProb}. The computed deformation for the most refined solution is shown in Figure \ref{fig:problem3_Sol}, along with the reference position of the diaphragm. The deformation closely matches what is expected based on the movement of the ribs (see Figure \ref{fig:problem3_ribs}), where the ventral part of the diaphragm moves upwards and the dorsal part mostly stays fixed. As in Problems 1 and 2 we construct the global basis functions using Gaussians with a shape parameter $\varepsilon = 0.1$. However, we observe from the convergence study in Figure \ref{fig:problem3_conv} that the numerical solution converges to the reference solution, but with relatively low convergence rates. Even though we ensure that the boundary conditions, the reconstruction and the Robin coefficients are smooth enough, we observe that they are challenging to fully resolve. In particular, note that large gradients are present due to the sharpness of the Robin coefficients used. Given the existence of large gradients we assume that a larger shape parameter will result in a better approximation and higher convergence rates. In Figure \ref{fig:problem3_residual} we present the convergence study with a basis constructed using $\varepsilon = 3$, which shows higher convergence rates. However, the scale of the error is larger for coarser point sets and very similar for the finer point sets. Note that for the solution with the larger shape parameter we do not use the RBF-QR stable computation method to construct the global basis.

\section{Conclusions}\label{sec:conc}

In this paper we used the LS-RBF-PU method to solve the linear elasticity problem in two 3D diaphragm geometries. We proved that the continuous least squares problem is well-posed under specific conditions and that solutions computed using the LS-RBF-PU method converge to the continuous solution. The derived error estimate for the method is  numerically validated by solving a smooth manufactured problem on the diaphragm geometry. 

Additionally, we performed numerical experiments constructed using the expected diaphragm movement extracted from medical image data and experimentally determined parameters from the literature. We have shown that the method is able to produce numerical results for realistic and complex problems on the thin diaphragm geometry. While the theory shows that the smoothness of the data, the boundary and the Robin coefficients are necessary for convergence, results indicate that it can still be expensive to resolve such problems sufficiently on a 3D thin geometry.

Additional refinement becomes a resource limitation but also indicates the direction necessary for future work on this problem. Areas with large gradients arising from the shape of the boundary or sharpness of the coefficients can be more finely resolved using an adaptive patch refinement algorithm. Moreover, solving the mixed pressure-displacement formulation for linear elasticity could assist in reducing the numerical errors, to the extent that the large Poisson ratio of the diaphragm affects the scale of the error. Some work has already been done in the master thesis \cite{cao2025using}.  Similarly the development of a partition of unity method that solves a Galerkin projection of the PDE can reduce smoothness requirements and improve conditioning of the approximation matrices. Work on using mesh-free quadrature \cite{BrunoDavydov2025} for local RBF-based Galerkin methods was initiated by the master thesis \cite{yangmei2025investigation}. Finally, we note that the linear model used is a simplification of the expected hyperelastic and anisotropic behaviour of the human diaphragm.

\appendix
\section{Bounds used for continuous analysis}\label{app:ContinuousBounds}
\setlength\parindent{15pt}
\subsection{Bounds used for proving Theorem \ref{theorem:Continuity}}\label{app:ContinuityBounds}
First, we prove the modified bounds for fractional Sobolev spaces \eqref{eq:Lemma5.3_modified}. Starting from the definition of the fractional Sobolev norm, we set
\begin{align}
    \begin{split}
       \left\|\left(\nabla\cdot\mathbf{v}I\right)\cdot\mathbf{n}\right\|^2_{H^{1/2}(\partial\Omega)} & = \left\|\left(\nabla\cdot\mathbf{v}I\right)\cdot\mathbf{n}\right\|^2_{L_2(\partial\Omega)} + \left[\left(\nabla\cdot\mathbf{v}I\right)\cdot\mathbf{n}\right]^2_{H^{1/2}(\partial\Omega)} \\
       & \leq \sum_{i=1}^d \|n_i\|^2_{L_{\infty}(\partial\Omega)}\left\|\nabla\cdot\mathbf{v}\right\|^2_{L_2(\partial\Omega)} + \sum_{i=1}^d\left[\left(\nabla\cdot\mathbf{v}\right)n_i\right]^2_{H^{1/2}(\partial\Omega)} \\
       & \leq d\left\|\nabla\cdot\mathbf{v}\right\|^2_{L_2(\partial\Omega)} + C\left[\nabla\cdot\mathbf{v}\right]^2_{H^{1/2}(\partial\Omega)} \\
       & \leq C^2_f\left\|\nabla\cdot\mathbf{v}\right\|^2_{H^{1/2}(\partial\Omega)},
       \label{eq:divVdotN}
    \end{split}
\end{align}
where we used \eqref{eq:Lemma5.3} to for the semi-norm bound and the constant $C_f = \sqrt{\max(C,d)}$. For the bound on\linebreak $\left\|\nabla^a\mathbf{v}\cdot\mathbf{n}\right\|_{H^{1/2}(\partial\Omega)}$ we have the following estimate for the $L_2$-norm by using $\|\mathbf{n}\|^2_{L_{\infty}(\partial\Omega)} = 1$ 
\begin{align}
    \left\|\nabla^a\mathbf{v}\cdot\mathbf{n}\right\|_{L_2(\partial\Omega)} \leq \left\|\nabla^a\mathbf{v}\right\|_{L_2(\partial\Omega)}.
    \label{eq:AntiSymmetricL2Norm}
\end{align}

Moreover we show the following bound on the fractional semi-norm
\begin{align}
\begin{split}
    \left[\nabla^a\mathbf{v}\cdot\mathbf{n}\right]^2_{H^{1/2}(\partial\Omega)} & \leq d\sum_{i,j=1}^d\left[A_{ij}n_j\right]^2_{H^{1/2}(\partial\Omega)} \leq d\sum_{i,j=1}^d C_{ij}\left[A_{ij}\right]^2_{H^{1/2}(\partial\Omega)} \\
    & \leq d\max_{i,j}C_{ij}\sum_{i,j=1}^d\left[A_{ij}\right]^2_{H^{1/2}(\partial\Omega)} = d\max_{i,j}C_{ij}\left[\nabla^a\mathbf{v}\right]^2_{H^{1/2}(\partial\Omega)}
\end{split}
\label{eq:AntiSymmetricHalfSemiNorm}
\end{align}
where we used the substitution $\nabla^a\mathbf{v} = A$ and \eqref{eq:Lemma5.3}. Adding bounds \eqref{eq:AntiSymmetricHalfSemiNorm},~\eqref{eq:AntiSymmetricL2Norm} we have
\begin{equation}
    \left\|\nabla^a\mathbf{v}\cdot\mathbf{n}\right\|^2_{H^{1/2}(\partial\Omega)}\leq \left(C'_f\right)^2\left\|\nabla^a\mathbf{v}\right\|^2_{H^{1/2}(\partial\Omega)}.
\end{equation}

Additionally, we derive bounds for the estimates \eqref{eq:generalBounds}. Note that we use the tensor norm as described in Section \ref{sec:Definitions}. For example we use the contraction operator to get $\|\nabla\mathbf{v}\|^2_T = \nabla\mathbf{v}:\nabla\mathbf{v} = \sum_{i,j=1}^d \partial_j v_i\partial_j v_i = \sum_{i,j=1}^d \left|\partial_j v_i\right|^2$. Similarly for a third order tensor $\|D^2\mathbf{v}\|^2_T = D^2\mathbf{v}:D^2\mathbf{v} = \sum_{i,j,k=1}^d \partial_{kj} v_i\partial_{kj} v_i = \sum_{i,j,k=1}^d |\partial_{kj} v_i|^2$. In the bounds to follow we use the fact that the squared tensor norm of the Hessian can be written as
\begin{align}
    \begin{split}
        \left\|D^2\mathbf{v}\right\|_T^2 = \sum_{i,j,k=1}^d
        \left|\partial_{kj} v_i\right|^2 = \sum_{i\neq k}^d\left|\partial_{ki}v_i\right|^2 + \sum_{i\neq j}^d\left|\partial_{jj}v_i\right|^2 + \sum_{j\neq k}^d\left|\partial_{kj}v_k\right|^2 + \sum_{i=1}^d\left|\partial_{ii}v_i\right|^2+\sum_{i\neq j\neq k}^d\left|\partial_{kj}v_i\right|^2.
    \end{split}
\end{align}

Starting with the bound for the strain tensor, $\left\|\epsilon\right\|_{L_2(\Omega)}$ we note that
\begin{equation}
    \left\|\nabla\mathbf{v}\right\|^2_T = \left\|\epsilon(\mathbf{v})+\nabla^A\mathbf{v}\right\|^2_T = \left\|\epsilon(\mathbf{v})\right\|^2_T + 2\epsilon(\mathbf{v}):\nabla^A\mathbf{v} + \left\|\nabla^A\mathbf{v}\right\|^2_T, 
\end{equation}
where $\nabla^A\mathbf{v}$ is the anti-symmetric part of the gradient tensor, hence $\epsilon(\mathbf{v}):\nabla^A\mathbf{v} = 0$. We integrate the pointwise tensor norm on both sides to get
\begin{equation}
    \|\epsilon(\mathbf{v})\|_{L_2(\Omega)}\leq\|\nabla\mathbf{v}\|_{L_2(\Omega)}.
\end{equation}
We similarly develop the rest of the bounds and then integrate to get \eqref{eq:generalBounds}. For the bound on $\left\|\nabla\cdot\mathbf{v}\right\|_{L_2(\Omega)}$ we have the following:
\begin{align}
    \begin{split}
        \left|\nabla\cdot\mathbf{v}\right|^2 = 
        \left|\sum_{i=1}^d\partial_i v_i\right|^2\leq d\sum^{d}_{i=1}\left|\partial_i v_i\right|^2\leq d\left\|\nabla\mathbf{v}\right\|^2_T,
        \label{eq:divV}
    \end{split}
\end{align}
while $\|\left(\nabla\cdot\mathbf{v}\right)I\|_T^2 \leq d^2\|\nabla\mathbf{v}\|^2_T$. For the bound on the divergence term $\left\|\nabla\cdot\left(\left(\nabla\cdot\mathbf{v}\right)I\right)\right\|_{L_2(\Omega)}$ we have
\begin{equation}
    \left\|\nabla\cdot\left(\left(\nabla\cdot\mathbf{v}\right)I\right)\right\|_2^2 = \left\|\nabla\left(\nabla\cdot\mathbf{v}\right)\right\|^2_2=
    \sum_{j=1}^d\left|\sum_{i=1}^d\partial_{ji}v_i\right|^2\leq d\sum_{i,j=1}^d\left|\partial_{ji}v_i\right|^2\leq d\left\|D^2\mathbf{v}\right\|^2_T,
    \label{eq:DivDivVI}
\end{equation}
while the gradient term $\left\|\nabla\left(\left(\nabla\cdot\mathbf{v}\right)I\right)\right\|_T^2\leq d^2\|D^2\mathbf{v}\|^2_T$.
Moreover, we construct a bound for the term $\left\|\nabla\cdot\epsilon\right\|_{L_2(\Omega)}$ as
\begin{align}
    \begin{split}
        \left\|\nabla\cdot\epsilon\right\|^2_2 & = \sum_{i=1}^d\left|\sum_{j=1}^d\partial_j\epsilon_{ij}\right|^2 = \frac{1}{4}\sum_{i=1}^d\left|\sum_{j=1}^d\left(\partial_{ji}v_j + \partial_{jj}v_i\right)\right|^2 \\
        & \leq \frac{d}{4}\sum_{i,j=1}^d\left|\partial_{ji}v_j + \partial_{jj}v_i\right|^2 \leq \frac{d}{2}\sum_{i,j=1}^d\left|\partial_{ji}v_j\right|^2 + \frac{d}{2}\sum_{i,j=1}^d\left|\partial_{jj}v_i\right|^2 \\
        & = \frac{d}{2}\left(2\sum_{i=1}^d\left|\partial_{ii}v_i\right|^2 + \sum_{i\neq j}^d\left|\partial_{ji}v_j\right|^2 + \sum_{i\neq j}^d\left|\partial_{jj}v_i\right|^2\right) \leq d\left\|D^2\mathbf{v}\right\|_T^2.
        \label{eq:DivStrain}
    \end{split}
\end{align}
Finally for the bounds on $\left\|\nabla^a\mathbf{v}\right\|_{L_2(\Omega)}$, $\left\|\nabla\nabla^a\mathbf{v}\right\|_{L_2(\Omega)}$ we have
\begin{align}
    \begin{split}
        \left\|\nabla^a\mathbf{v}\right\|_T^2 & = \sum_{i,j=1}^d\left|\frac{1}{2}\left(\partial_j v_i - \partial_i v_j\right)\right|^2 = \sum_{i\neq j}^d\left|\frac{1}{2}\left(\partial_j v_i - \partial_i v_j\right)\right|^2 \\
        & \leq \frac{1}{2}\sum_{i\neq j}^d\left|\partial_j v_i\right|^2 + \left|\partial_i v_j\right|^2 = \sum_{i\neq j}^d \left|\partial_j v_i\right|^2 \leq \left\|\nabla\mathbf{v}\right\|_T^2
        \label{eq:gradaV},
    \end{split}
\end{align}
\begin{align}
    \begin{split}
        \left\|\nabla \nabla^a\mathbf{v}\right\|_T^2 & = \sum_{i\neq j, k=1}^d\left|\frac{1}{2}\left(\partial_{kj} v_i - \partial_{ki} v_j\right)\right|^2 \\
        & \leq \frac{1}{2}\sum_{i\neq j, k=1}^d\left|\partial_{kj} v_i\right|^2 + \left|\partial_{ki} v_j\right|^2 = \sum_{i\neq j, k=1}^d \left|\partial_{kj} v_i\right|^2 \\
        & = \sum_{i\neq j}^d\left|\partial_{jj}v_i\right|^2 + \sum_{k\neq j}^d\left|\partial_{kj}v_k\right|^2 \leq \left\|D^2\mathbf{v}\right\|_T^2
        \label{eq:nablaNablaA_D2}
    \end{split}
\end{align}

Some additional bounds including second derivatives and normals are required for the proofs in Appenidix \ref{app:DiscreteBounds}. They are given as 
\begin{align}
    \|\nabla\left(\nabla^{\alpha}\mathbf{v}\cdot\mathbf{n}\right)\|_T^2 & = \sum_{i,j=1}^d\left|\sum_{k=1}^d\frac{1}{2}\left(\partial_{jk}v_i - \partial_{ji}v_k\right)n_k\right|^2 \\
    & \leq \sum_{k=1}^d\left|n_k\right|^2\sum_{i,j,k=1}^d\left|\frac{1}{2}\partial_{jk}v_i - \partial_{ji}v_k\right|^2 = \|\nabla\nabla^{\alpha}\mathbf{v}\|^2_T \leq \|D^2\mathbf{v}\|^2_T,
\end{align}
where we used the Cauchy-Schwartz inequality and \eqref{eq:nablaNablaA_D2}. 


Furthermore, required bounds including third derivative operators for Appendix \ref{app:DiscreteBounds} are given as
\begin{align}
    \begin{split}
        &\|\nabla\left(\nabla\cdot\epsilon\left(\mathbf{v}\right)\right)\|^2_T = \sum_{i,k=1}^d\left|\sum_{j=1}^d\frac{1}{2}\left(\partial_{kjj}v_i + \partial_{kji}v_j\right)\right|^2 \leq \frac{d}{2}\sum_{i,j,k=1}^d\left|\partial_{kjj}v_i\right|^2 + \frac{d}{2}\sum_{i,j,k=1}^d\left|\partial_{kji}v_j\right|^2 \leq d\|D^3\mathbf{v}\|^2_T, \\
        & \|\nabla\left(\nabla\cdot\left(\nabla\cdot\mathbf{v}I\right)\right)\|_T^2 = \|\nabla\left(\nabla\left(\nabla\cdot\mathbf{v}\right)\right)\|_T^2 = \sum_{j,k=1}^d\left|\sum_{i=1}^d\partial_{kji}v_i\right|^2 \leq d\sum_{i,j,k=1}^d\left|\partial_{kji}v_i\right|^2\leq d\|D^3\mathbf{v}\|^2_T.
        \label{eq:gradDivDivV}
    \end{split}
\end{align}
\subsection{Bounds used for proving Theorem \ref{theorem:Coercivity}}\label{app:CoercivityBounds}

We will show that the strong problem is ADN elliptic, uniformly elliptic and regular elliptic, while the boundary conditions satisfy the complementing conditions. These definitions are provided in \cite{ADN_2}, \cite{bochev2009least}, \cite{JaRoĭtberg_1969} and are also given here. We initially define a generalized notation for the linear differential and boundary operators from \eqref{eq:StrongBC} 
\begin{align}
    \mathcal{L}_{ij}(\mathbf{x},D) = \sum_{|\zeta|\leq t_{ij}}\alpha^{ij}_{\zeta}\left(\mathbf{x}\right)D^{\zeta}, \quad & i,j=1,\dots,N \\
    \mathcal{B}_{lj}(\mathbf{x},D) = \sum_{|\zeta|\leq r_{lj}}\beta^{lj}_{\zeta}\left(\mathbf{x}\right)D^{\zeta}, \quad & j=1,\dots,N,~l=1,\dots,M,
\end{align}
where $\alpha_{\zeta}^{ij}$ and $\beta_{\zeta}^{lj}$ are the coefficients of the domain and boundary operators, respectively. Additionally, $\zeta$ is a multi-index as defined in \eqref{eq:multiIndex}. The order of the operators is related to sets of integer weights, $\{s_i\}_{i=1,\dots,N},~\{t_j\}_{j=1,\dots,N}$, for operator $\mathcal{L}_{ij}(\mathbf{x},D)$, since $t_{ij} = s_i + t_j$. The integer weights are chosen such that certain ellipticity conditions for the differential operator are met.  Similarly for operator $\mathcal{B}_{lj}(\mathbf{x},D)$ the sets of integer weights are $\{r_l\}_{l=1,\dots,M},~\{t_j\}_{j=1,\dots,N}$, where the weights are chosen such that the order of the operator is related to $r_{lj} = r_l + t_j$. For elliptic operators the relationships between the derivative orders and the integer weights are given in Definition \ref{def:ADNelliptic} and \eqref{eq:BoundaryOperatorWeightDependence} below. 

\begin{definition}[{\cite[Definition~D.1]{bochev2009least}}]
    A differential operator, $\mathcal{L}(\mathbf{x},D)$, is Agmon–Douglis–Nirenberg (ADN) elliptic if 
    \begin{itemize}
        \item \textup{$\text{deg}\mathcal{L}_{ij}(\mathbf{x},D) = 0$}, given $s_i + t_j < 0$,
        \item \textup{$\text{deg}\mathcal{L}_{ij}(\mathbf{x},D) \leq t_{ij}$}, where $t_{ij} = s_i + t_j$, given $s_i + t_j \geq 0$,
        \item  \textup{$\text{det}\mathcal{L}^P(\mathbf{x},\boldsymbol{\xi})\neq 0, \hspace{0.5cm} \forall\boldsymbol{\xi}\neq 0$},
    \end{itemize}
    where $s_i, t_j$ are integer weights, $\mathcal{L}(\mathbf{x},\boldsymbol{\xi})$ is the symbol of the differential operator \textup{$\mathcal{L}(\mathbf{x},D)$ and \\$\text{deg}~\mathcal{L}_{ij}(\mathbf{x},D)$} is the highest derivative degree for a specific element of the derivative operator. We also note that $\mathcal{L}^P(\mathbf{x},\boldsymbol{\xi})$ is the principal part of the symbol, defined such that \textup{$\text{deg}\mathcal{L}^P_{ij}(\mathbf{x},\boldsymbol{\xi}) = t_{ij}$}. 
    \label{def:ADNelliptic}
\end{definition}

\begin{definition}
    A differential operator, $\mathcal{L}(\mathbf{x},D)$, is uniformly elliptic if it is ADN elliptic and there exists a positive ellipticity constant $A$, such that 
    \textup{
    \begin{equation}
        A^{-1}\left|\boldsymbol{\xi}\right|^{2m}\leq \left|\text{det}\mathcal{L}^P(\mathbf{x},\boldsymbol{\xi})\right|\leq A\left|\boldsymbol{\xi}\right|^{2m},
        \label{eq:UniformElliptic}
    \end{equation}}
\end{definition}

\begin{definition}[{\cite[Definition~D.2]{bochev2009least}}]
    A differential operator, $\mathcal{L}(\mathbf{x},D)$, is regular elliptic if the polynomial $\mathcal{L}^P(\mathbf{x},\boldsymbol{\nu}+\tau\mathbf{n})$ only has $2m$ imaginary roots in $\tau$, half of which are positive and half negative. Here $\mathbf{x}\in\partial\Omega$ is any boundary point with normal $\mathbf{n}$ and tangent $\boldsymbol{\nu}$. This is called the supplementary condition.
\end{definition}

We note that the integer weights for the boundary operator for an elliptic problem are chosen in a similar manner to the PDE operator. Specifically, the weights should satisfy 
\begin{equation}
    \text{deg}\mathcal{B}_{lj}(\mathbf{x},\boldsymbol{\xi}) \leq r_l+ t_j,
    \label{eq:BoundaryOperatorWeightDependence}
\end{equation}
where $r_l$ are an additional set of integer weights and $l=1,\dots,m$. The principal part of $\mathcal{B}_{lj}$ is defined as
\begin{equation}
    \text{deg}\mathcal{B}_{lj}^P(\mathbf{x},\boldsymbol{\xi}) = r_l + t_j.
\end{equation}
\begin{definition}[{\cite[Definition~D.4]{bochev2009least}}]
    Boundary conditions satisfy the complementing condition if, for any point $\mathbf{x}\in\partial\Omega$ and any real tangent vector, $\boldsymbol{\nu}$ and normal vector, $\mathbf{n}$ at $\mathbf{x}$, the rows of the matrix given by
    \begin{equation}
        A = \mathcal{B}^P\left(\mathbf{x},\boldsymbol{\nu}+\tau\mathbf{n}\right)\mathcal{L}'\left(\mathbf{x},\boldsymbol{\nu}+\tau\mathbf{n}\right),
    \end{equation}
    are linearly independent modulo $M^{+}\left(\mathbf{x},\boldsymbol{\nu},\mathbf{n},\tau\right)$, which is defined as 
    \begin{equation}
        M^{+}\left(\mathbf{x},\boldsymbol{\nu},\mathbf{n},\tau\right) = \prod_{k=1}^m\left(\tau-\tau^{+}_k\right),
    \end{equation}
    and is a polynomial constructed by factorizing the determinant of the principal part of the symbol \textup{$\text{det}\mathcal{L}^P\left(\mathbf{x},\boldsymbol{\nu}+\tau\mathbf{n}\right)$}, where $\tau^{+}_k$ are its positive imaginary roots. $\mathcal{L}'\left(\mathbf{x},\boldsymbol{\xi}\right)$ is the adjugate of $\mathcal{L}^P$ and is defined as \textup{
    \begin{equation}
        \mathcal{L}'\left(\mathbf{x},\boldsymbol{\xi}\right) = \text{det}\mathcal{L}^P\left(\mathcal{L}^P\right)^{-1}. 
    \end{equation}}
    Hence the complementing condition can be succinctly stated. The following identities
    \begin{equation}
        \sum_{i=1}^m C_i A_{ij}\left(\tau\right) = M^{+}p_j(\tau),
        \label{eq:complementingBC}
    \end{equation}
    hold for $j = 1,\dots,N$ and any polynomial $p_j(\tau)$, if and only if $C_i = 0$.
\end{definition}

Initially we note that the uniqueness for the strong linear elasticity problem provided in \eqref{eq:StrongPDE} and \eqref{eq:StrongBC} is proven for different sets of boundary conditions in \cite[Theorem 4.4.1]{knops1971uniqueness}. Note that the Robin boundary conditions we use can be converted to a condition that ensures uniqueness as long as the functions $\kappa_0,~ \kappa_1$ are strictly positive.

To show that the linear elasticity problem is ADN elliptic we note that our problem has 3 equations and 3 unknowns such that $i,~j = 1,...,3$. The weights can be defined as $s_i = 0,~t_1 = t_2 = t_3 = 2$ such that the second point in definition \ref{def:ADNelliptic} is fulfilled. The choice of weights also shows that our problem is homogeneous elliptic \cite{JaRoĭtberg_1969}.  We express the principal part of the symbol of our PDE operator given in \eqref{eq:Elasticity_Operators} as
\makeatletter 
 \def\@eqnnum{{\normalsize \normalcolor (\theequation)}} 
  \makeatother
{\small\begin{equation}
    \mathcal{L}^P(\mathbf{x},\boldsymbol{\xi}) = -\begin{pmatrix}(\lambda+2\mu)\xi^2_1+\mu(\xi^2_2 + \xi^2_3) & (\lambda + \mu)\xi_1\xi_2 & (\lambda + \mu)\xi_1\xi_3 \\
    (\lambda + \mu)\xi_1\xi_2 & (\lambda + 2\mu)\xi_2^2 + \mu(\xi^2_1 + \xi^2_3) & (\lambda + \mu)\xi_2\xi_3 \\
    (\lambda + \mu)\xi_1\xi_3 & (\lambda + \mu)\xi_2\xi_3 & (\lambda + 2\mu)\xi^2_3 + \mu(\xi^2_1 + \xi^2_2)\end{pmatrix},
\end{equation}}
\flushleft where $\boldsymbol{\xi}^{\alpha} = \left(\xi^{\alpha_1}_1,\xi^{\alpha_2}_2,\xi^{\alpha_3}_3\right)$ are monomials representing derivative operators $D^{\alpha} = \frac{\partial^{\left|\alpha\right|}}{\partial x_1^{\alpha_1}\partial x_2^{\alpha_2}\partial x_3^{\alpha_3}}$, with multi-index $\alpha = \left(\alpha_1,\alpha_2,\alpha_3\right)$. The determinant is
\begin{equation}
    \text{det}\mathcal{L}^P(\mathbf{x},\boldsymbol{\xi}) = -\mu^2(\lambda+2\mu)(\xi_1^2 + \xi_2^2 + \xi_3^2)^3 = -\mu^2(\lambda + 2\mu)\left|\boldsymbol{\xi}\right|^6 \neq 0,
    \label{eq:ADN_elliptic}
\end{equation}

which shows that the linear elasticity operator is also uniformly elliptic assuming $\mu,~\lambda > 0$ with the PDE order being $\sum_i^3 s_i + t_i = 6$ and $m = 3$ \eqref{eq:UniformElliptic}.

Additionally, we show that the linear elasticity operator is regular elliptic by computing the following polynomial in $\tau$
\begin{align}
    \begin{split}
        \text{det}\mathcal{L}^P(\mathbf{x},\boldsymbol{\nu}+\tau\mathbf{n}) & = \mu^2(\lambda+2\mu)\left(\left|\boldsymbol{\nu}\right|^2 + 2\tau\boldsymbol{\nu}\cdot\mathbf{n} + \left|\tau\mathbf{n}\right|^2\right) \\
        & = \left(\tau^2 + 1\right)^3 \\
        & = \left(\tau - i\right)^3\left(\tau + i\right)^3. 
    \end{split}
\end{align}

The boundary operator for our linear elasticity problem is defined as $\mathcal{B}(\mathbf{x},D) = \kappa_0\mathcal{B}_0 + \kappa_1\mathcal{B}_1$ (see \eqref{eq:Elasticity_Operators}). We define the integer weights as $r_1 = r_2 = r_3 = -1$ such that $r_l + t_j = 1,~\forall i,j$. The principal part of the operator is
{\scriptsize\begin{equation}
    \mathcal{B}^P(\mathbf{x},\boldsymbol{\xi}) = \kappa_1\begin{pmatrix}(\lambda + 2\mu)n_1\xi_1 + \mu(n_2\xi_2 + n_3\xi_3) & \lambda n_1\xi_2 + \mu n_2\xi_1 & \lambda n_1\xi_3 + \mu n_3\xi_1 \\
    \lambda n_2\xi_1 + \mu n_1\xi_2 & (\lambda + 2\mu)n_2\xi_2 + \mu(n_1\xi_1 + n_3\xi_3) & \lambda n_2\xi_3 + \mu n_3\xi_2 \\
    \lambda n_3\xi_1 + \mu n_1\xi_3 & \lambda n_3\xi_2 + \mu n_2\xi_3 & (\lambda + 2\mu)n_3 \xi_3 + \mu(n_1\xi_1 + n_2\xi_2)
    \end{pmatrix}.
\end{equation}}
We also define the adjugate of the principal PDE operator as
{\scriptsize\begin{equation}
    \mathcal{L}'(\mathbf{x},\boldsymbol{\xi}) = \mu\left|\boldsymbol{\xi}\right|^2\begin{pmatrix}
        -\left(\lambda+2\mu\right)\left(\xi_2^2 + \xi_3^2\right)-\mu\xi_1^2 & \left(\lambda + \mu\right)\xi_1\xi_2 & \left(\lambda + \mu\right)\xi_1\xi_3 \\
        \left(\lambda + \mu\right)\xi_1\xi_2 & -\left(\lambda+2\mu\right)\left(\xi_1^2 + \xi_3^2\right)-\mu\xi_2^2 & \left(\lambda+\mu\right)\xi_2\xi_3 \\
        \left(\lambda + \mu\right)\xi_1\xi_3 & \left(\lambda + \mu\right)\xi_2\xi_3 & -\left(\lambda + 2\mu\right)\left(\xi_1^2 + \xi_2^2\right)-\mu\xi_3^2
    \end{pmatrix},
\end{equation}}
\flushleft while the polynomial $M^{+}$ for our problem is defined as $M^{+} = (\tau-i)^3$. We don't explicitly formulate the matrix, $\mathcal{B}^P\left(\mathbf{x},\boldsymbol{\nu}+\tau\mathbf{n}\right)\mathcal{L}'\left(\mathbf{x},\boldsymbol{\nu}+\tau\mathbf{n}\right)$ for brevity. Note that the $N$ identities from \eqref{eq:complementingBC} should hold for any $\mathbf{x}\in\partial\Omega$ with zero coefficients, $C_i = 0,~\forall i$, if the boundary conditions satisfy the complementing condition. Given rotation invariance of the domain and boundary operators for the linear elasticity problem, we can rotate the coordinate system such that any boundary point aligns with a specific normal and tangent. Hence without loss of generality we fix $\mathbf{x}\in\partial\Omega$ and pick the tangent and the normal as $\boldsymbol{\nu} = \left(1,0,0\right)^T$ and $\mathbf{n} = \left(0,0,-1\right)^T$ respectively. Then the following identities should hold if and only if $C_1 = C_2 = C_3 = 0$
\begin{align}
    -\mu^2\left(\tau + i\right)\kappa_1\left(C_3+C_1\tau\right)\left(\tau^2\left(\lambda 
    + 2\mu\right)-\lambda\right) = \left(\tau - i\right)^2 p_1(\tau), \label{eq:LECompCond1} \\
    -\mu^2\kappa_1\left(\tau + i\right)^2C_2\tau\left(\lambda+2\mu\right) = \left(\tau - i\right)p_2(\tau), \label{eq:LECompCond2} 
    \end{align}
    \vspace{-0.65cm}
\begin{align}
    \begin{split}
    & \mu\left(\tau + i\right)\kappa_1\left[C_1\left(\mu\left(\lambda+2\mu\right) + \mu^2\tau^2 -\tau^2\mu\left(\lambda + \mu\right)\right)\right. \\ & \left. + C_3\left(\left(-\left(\lambda+2\mu\right)-\mu\tau^2\right)\left(\tau\left(\lambda+2\mu\right)\right) + \tau\lambda\left(\mu+\lambda\right)\right)\right]= \left(\tau - i\right)^2 p_3(\tau).    \end{split}
    \label{eq:LECompCond3}
\end{align}
In general we assume that $\mu,~\lambda>0$. Hence, \eqref{eq:LECompCond2} only holds if $C_2 = 0$ since the polynomial on the left hand side, does not have the root $\tau = i$. The identity \eqref{eq:LECompCond1} holds if $C_3=-iC_1$. Substituting this relation in \eqref{eq:LECompCond3} the identity simplifies to
\begin{equation}
    \mu^2 i\left(\tau +i\right)\kappa_1 C_1\left[\left(\lambda+2\mu\right)\tau^2 + i\left(2\lambda + 2\mu\right)\tau + \left(2\mu+\lambda\right)\right] = \left(\tau - i\right)p_3(\tau),
\end{equation}
which does not hold unless $C_1 = 0$. This in turn means that $C_3 = 0$ and \eqref{eq:LECompCond1},~\eqref{eq:LECompCond2},~\eqref{eq:LECompCond3} hold if and only if $C_1=C_2=C_3$. Hence the linear elasticity problem in \eqref{eq:StrongPDE},~\eqref{eq:StrongBC} is elliptic and satisfies the complementing condition.

\section{Bounds used for discrete estimates}\label{app:DiscreteBounds}
\subsection{A note on the Leibniz rule for differential operators}\label{app:leibnizDiffOp}
\setlength\parindent{15pt}

Given our definition of the Sobolev norm \eqref{eq:Sobolev_norm_2} where $D^{\tau}$ is a differential operator of order $\tau\in\mathbb{N}_0$, we work with the tensor norm of $D^{\tau}\mathbf{v}$. See for example the transformation of derivatives to a reference patch \eqref{eq:derivativePhysical2Ref_derivation}. Additionally given we use a partition of unity method, we take derivatives of products between the weight functions and local function, $w_k\mathbf{v}^{(k)}$, meaning we use the Leibniz rule for partial derivatives. This is generally provided in terms of one partial derivative defined in using a multi-index (see \cite{Hardy2006}, \cite{Comtet1974}). Given that our derivations require a bound in terms of differential operators we use the Leibniz rule to bound a general operator in the tensor norm as
\begin{align}
    \begin{split}
    \|D^{\tau}(\mathbf{v}^{(k)}w_k)\|^2_T = \sum_{i=1}^d\sum_{|\alpha| = \tau}\frac{\tau!}{\alpha!}\left|D^{\alpha}(v_i^{(k)} w_k)\right|^2 = \sum_{i=1}^d\sum_{|\alpha| = \tau}\frac{\tau!}{\alpha!}\left|\sum_{\beta + \gamma = \alpha}\frac{\alpha!}{\beta!\gamma!}D^{\beta}v_i^{(k)} D^{\gamma}w_k\right|^2,
    \end{split}
\end{align}
where $\alpha,~\beta,~\gamma$ are multi indices \eqref{eq:multiIndex} and in the first step we used the multinomial theorem to write the tensor norm, while in the second step we used the Leibniz rule for partial derivatives. We aim to construct an upper bound of the above norm in terms of differential operators as opposed to sums of derivatives. Hence, we bound the squared sum by the sum of squares using inequality $(x_1+x_2+\dots+x_n)^2\leq n(x_1^2 + x_2^2 + \dots + x_n^2)$, which gives
\begin{equation}
    \|D^{\tau}(\mathbf{v}^{(k)}w_k)\|_T^2 \leq \sum_{i=1}^d\sum_{|\alpha|=\tau}\frac{\tau!}{\alpha!}\prod_{j=1}^d\left(\alpha_j + 1\right)\sum_{\beta+\gamma = \alpha}\left(\frac{\alpha!}{\beta!\gamma!}\right)^2|D^{\beta}v_i^{(k)}|^2|D^{\gamma}w_k|^2.
\end{equation}

Additionally, using the inequality of arithmetic and geometric means (AM-GM inequality) we obtain a bound for the multi-index product $\prod_{j=1}^d \alpha_j + 1$ with respect to the sum of the same indices which gives
\begin{equation}
    \|D^{\tau}(\mathbf{v}^{(k)}w_k)\|_T^2  \leq \left(1+\frac{\tau}{d}\right)^d\tau!\sum_{i=1}^d\sum_{|\alpha|=\tau}\sum_{\beta+\gamma = \alpha}\left(\frac{\alpha!}{\beta!\gamma!}\right)^2|D^{\beta}v_i^{(k)}|^2|D^{\gamma}w_k|^2,
\end{equation}
where we also use the bound for the multinomial coefficient. We attempt to reformulate the bound in terms of tensor norms on both sides, which gives

\begin{equation}
\|D^{\tau}(\mathbf{v}^{(k)}w_k)\|_T^2  \leq \left(1+\frac{\tau}{d}\right)^d\tau!\sum_{i=1}^d\sum_{q+\zeta = \tau}\sum_{|\beta| = q}\sum_{|\gamma|=\zeta}\left(\frac{(\beta+\gamma)!}{\beta!\gamma!}\right)^2|D^{\beta}v_i^{(k)}|^2|D^{\gamma}w_k|^2.
\end{equation}
Additionally, we use inequality $2^{-d+1}\tau! \leq (\beta + \gamma)! < \tau!$ \cite{GiordanoFactorialBounds2001}, where $d$ is the dimension which determines the number of indices in $\beta,~\gamma$, and collecting the tensor norm terms gives

\begin{align}
    \begin{split}
        \|D^{\tau}(\mathbf{v}^{(k)}w_k)\|_T^2 & \leq 2^{-2d+2}\left(1+\frac{\tau}{d}\right)^d\tau!\sum_{i=1}^d\sum_{q+\zeta = \tau}\left(\frac{\tau!}{q!\zeta!}\right)^2\sum_{|\beta| = q}\frac{q!}{\beta!}|D^{\beta}v_i^{(k)}|^2\sum_{|\gamma|=\zeta}\frac{\zeta!}{\gamma!}|D^{\gamma}w_k|^2\\
        & = 2^{-2d+2}\left(1+\frac{\tau}{d}\right)^d\tau!\sum_{i=1}^d\sum_{q + \zeta = \tau}\left(\frac{\tau!}{q!\zeta!}\right)^2\|D^{q}v_i^{(k)}\|_T^2\|D^{\zeta}w_k\|_T^2, \\
        & = C_T\sum_{q + \zeta = \tau}\left(\frac{\tau!}{q!\zeta!}\right)^2\|D^{q}\mathbf{v}^{(k)}\|_T^2\|D^{\zeta}w_k\|_T^2.
    \end{split}
\end{align}

In some cases this bound is used for different norms ($L_2$-norm or $L_{\infty}$-norm). This can be done by first taking the square root of the above expression and then using the triangle inequality. Taking the essential supremum and using the triangle inequality we have the bound for the $L_{\infty}$-norm
\begin{equation}
    \|D^{\tau}(\mathbf{v}^{(k)}w_k)\|_{L_{\infty}(\mathcal{B})} \leq C_T^{\frac{1}{2}}\sum_{q + \zeta = \tau}\frac{\tau!}{q!\zeta!}\|D^{q}\mathbf{v}^{(k)}\|_{L_{\infty}(\mathcal{B})}\|D^{\zeta}w_k\|_{L_{\infty}(\mathcal{B})},
    \label{eq:LeibnizRuleOperator}
\end{equation}
which can be similarly shown for other norms and is used in the following subsections of this appendix.
\subsection{The quadrature error estimates}\label{app:intErrorEstimates}
\setlength\parindent{15pt}
We construct an error bound between continuous and discrete inner products as they appear in the generalised eigenvalue bounds \eqref{eq:genEigenProblemL2},~\eqref{eq:genEigenProblemH2},~\eqref{eq:genEigenProblemBilinear}. We aim to show that the bounds only depend on the oversampling factor $q$ and not on the fill distance $h$. 

We start by constructing a general estimate that can be used for all three bounds. We consider a domain of integration $\mathcal{B}\subseteq\mathbb{R}^{d}$, which is either the entire domain $\Omega$, or its boundary $\partial\Omega$. We split the integrals over $\mathcal{B}$ into Voronoi regions centred around each point in the set $Z$. The Voronoi regions are defined as in \cite{Tominec2025} 
\begin{equation}
    V_j = \{\mathbf{y}\in\mathcal{B}~|~\|\mathbf{y} - \mathbf{y}_j\|_2 \leq \|\mathbf{y} - \mathbf{y}_i\|_2,~i\neq j,~i = 1,\dots,M\}, \hspace{0.5cm} j = 1,\dots,M,
    \label{eq:VoronoiRegions}
\end{equation}
where $M = |Z|$ and $\mathcal{B}= \cup_{j=1}^MV_j$. This definition is equivalent to restricted Voronoi cells as described in \cite{edelsbrunner1994triangulating},~\cite{Yan2009}, since we use the Euclidean distance metric and restrict the cells to $\mathcal{B}$.

The quadrature error is defined between norms of tensors of order one \eqref{eq:genEigenProblemL2}, \eqref{eq:genEigenProblemBilinear} and three \eqref{eq:genEigenProblemH2}. We denote a general tensor by $A:\mathbb{R}^d\rightarrow\mathbb{R}^{d\times\dots\times d}$ and define the auxiliary discrete norm $\|f\|^2_{\ell_2^*(\mathcal{B})}=\sum_{j=1}^M|V_j|\,f(\mathbf{y}_j)$. We expand the $\ell_2$-norm to get
\begin{equation}
    \|A\|^2_{\ell_2(\mathcal{B})} = \sum_{j=1}^MW_j\|A(\mathbf{y}_j)\|_T^2 = \sum_{j=1}^M\frac{W_j}{|V_j|}|V_j|\,\|A(\mathbf{y}_j)\|_T^2, 
\end{equation}
leading to 
\begin{equation}
    C_{W,0}\|A\|^2_{\ell_2^*(\mathcal{B})}\leq \|A\|^2_{\ell_2(\mathcal{B})} \leq C_{W,1}\|A\|^2_{\ell_2^*(\mathcal{B})},
\end{equation}
where $C_{W,0}=\min_{1\leq j\leq M}W_j/|V_j|$ and $C_{W,1}=\max_{1\leq j\leq M}W_j/|V_j|$. 
Then, for some $C_{W,0}\leq C_W\leq C_{W,1}$ the error is given by
\begin{align}
    \begin{split}
        \left|\|A\|^2_{\ell_2(\mathcal{B})} - \|A\|^2_{L_2(\mathcal{B})}\right| & = \left|C_W\left(\|A\|^2_{\ell_2^*(\mathcal{B})}-\|A\|^2_{L_2(\mathcal{B})}\right)+(C_W-1)\|A\|^2_{L_2(\mathcal{B})} \right|\\
        & \leq C_{W,1}\left|\|A\|^2_{\ell_2^*(\mathcal{B})}-\|A\|^2_{L_2(\mathcal{B})}\right| + C_{W,2}\|A\|^2_{L_2(\mathcal{B})},\\
    \end{split}
    \label{eq:Cw}
\end{align}
where $C_{W,2}=\max_{i=0,1}|C_{W,i}-1|$.

We first focus on the term involving the auxiliary discrete norm, where the integral over $\mathcal{B}$ can now be split into a sum of integrals over Voronoi regions. We let $f(\mathbf{y})=\|A(\mathbf{y})\|_T^2$, and bound the term as
\begin{align}
    \begin{split}
        \left|\|A\|^2_{L_2(\mathcal{B})} - \|A\|^2_{\ell^{*}_2(\mathcal{B})}\right| &  
        = \left|\int_{\mathcal{B}}f(\mathbf{y})d\mathbf{y}- \sum_{j=1}^M\left|V_j\right|f(\mathbf{y}_j)\right| \\
        & = \left|\sum_{j=1}^M\int_{V_j}f(\mathbf{y})-f(\mathbf{y}_j)d\mathbf{y}\right| \leq \sum_{j=1}^M\int_{V_j}\left|f(\mathbf{y}) - f(\mathbf{y}_j)\right|d\mathbf{y} \\
        & \leq \sum_{j=1}^M\int_{V_j}\left|\nabla f(\xi_j)\cdot\left(\mathbf{y}-\mathbf{y}_j\right)\right|d\mathbf{y},
    \end{split}
\end{align}
where we used the mean value theorem for $f(\mathbf{y}_j)$ in the Voronoi region $V_j$, and $\xi_j\in\mathrm{co}(V_j)$ is a point in the convex hull of the Voronoi region~\cite{THOMPSON2003717}. Note that $\xi_j$ may fall outside the domain $\Omega$, but $\xi_j\in\cup_{K=1}^P\Omega_k$ is always true due to the convexity of the patches. Using the Cauchy-Schwartz inequality we have
\begin{align}
    \begin{split}
        \left|\|A\|^2_{L_2(\mathcal{B})} - \|A\|^2_{\ell^{*}_2(\mathcal{B})}\right| \leq \sum_{j=1}^M\|\nabla f(\xi_j)\|_2\int_{V_j}\|\mathbf{y} - \mathbf{y}_j\|_2d\mathbf{y} \leq \sum_{j=1}^M\left|V_j\right|\|\nabla f\|_{L_{\infty}(\text{co}(V_j))}\|\mathbf{y}-\mathbf{y}_j\|_{L_{\infty}(V_j)},
        \label{eq:intError_beforeFillDist}
    \end{split}
\end{align}
where $\|\cdot\|_{L_{\infty(V_j)}} = \esssup_{\mathbf{y}\in V_j} \|\cdot\|_2$. In the case $\mathcal{B}=\partial\Omega$, $|V_j|$ in~\eqref{eq:intError_beforeFillDist} corresponds to the area of $\partial\Omega\cap V_j$, which is given by
$\int_A \frac{\|\nabla\ell(\mathbf{y})\|_2}{\|\nabla\ell(\mathbf{y})\cdot \mathbf{p}\|_2}\,\mathrm{d}A$, where $A$ is the projection of the surface onto a plane with normal $\mathbf{p}$, chosen to be aligned with the overall surface normal direction. As long as the geometry is well resolved by the sampling, i.e., the normal direction never varies so much that it becomes orthogonal to $\mathbf{p}$, the integrand is bounded. For the projected surface, we have $|A|\leq \mathrm{diam}(V_j)^2$. Then, given a node set $Z$, which can be the boundary sampling $Y_{\partial\Omega}$ or the interior sampling $Y_{\Omega}$, we have the following bound on the size of the local integration domains with respect to the diameter of the Voronoi region
\begin{equation}
     |V_j| \leq C_{V_j}\text{diam}(V_j)^{\nu},
     \label{eq:VoronoiSize_diameter}
\end{equation}
where $\nu = \text{dim}(\mathcal{B})$, $C_{V_j}\leq1$ for $\mathcal{B}=\Omega$, and $C_{V_j}\approx 1$
for $\mathcal{B}=\partial\Omega$, if the geometry is well resolved. 
The diameter itself can be bounded using the fill distance \eqref{eq:fill} as
\begin{equation}
    c_{q,Z}^{-1}h_f(Z,\mathcal{B}) \leq \text{diam}(V_j) \leq 2h_f(Z,\mathcal{B}),
    \label{eq:VoronoiDiameter_fill}
\end{equation}
where $h_f(Z,\mathbf{B})$ is the fill distance and $c_{q,Z}$ is the node quality measure relating the fill distance and separation distance \eqref{eq:sepDistance}. The above can be shown by using the triangle inequality for the upper bound and by first showing that the lower bound holds for the separation distance. 
Using these relations together with $\|\mathbf{y} - \mathbf{y}_j\|_{L_{\infty}(V_j)}\leq \text{diam}(V_j)$ in \eqref{eq:intError_beforeFillDist} gives
\begin{align}
\begin{split}
   \left|\|A\|^2_{L_2(\mathcal{B})} - \|A\|^2_{\ell^{*}_2(\mathcal{B})}\right| & \leq \sum_{j=1}^M C_{V_j}\text{diam}(V_j)^{\nu+1} \|\nabla f\|_{L_{\infty}(\text{co}(V_j))} \leq 2^{\nu + 1}h_f(Z,\mathcal{B})^{\nu+1}\sum_{j=1}^M C_{V_j}\|\nabla f\|_{L_{\infty}(\text{co}(V_j))} \\
   & \leq C_{\mathrm{vor}}(\mathcal{B},Z)h_f(Z,\mathcal{B})^{\nu+1}\sum_{j=1}^M\|\nabla f\|_{L_{\infty}(\text{co}(V_j))},
   \label{eq:integrationBound_1}
\end{split}
\end{align}
where $C_{\mathrm{vor}}(\mathcal{B},Z) = 2^{\nu+1}\max_{1\leq j\leq M}C_{V_j}$. 
Next, we expand the gradient term involving the tensor $A$ 
\begin{align}
    \begin{split}
        \|\nabla f\|^2_2 & = \left\|\,\nabla\left\|A\right\|^2_T\,\right\|^2_2 = \sum_{j=1}^d\left|\sum_{i_1,i_2,\dots i_{t_0}=1}^d\partial_j\left|\alpha_{i_1i_2\dots i_{t_0}}\right|^2\right|^2 \\
        & = 4\sum_{j=1}^d\left|\sum_{i_1,i_2,\dots i_{t_0}=1}^d\alpha_{i_1i_2\dots i_{t_0}}\partial_j\alpha_{i_1i_2\dots i_{t_0}}\right|^2 
        \\
        & \leq 4\|A\|^2_T\|\nabla A\|_T^2,
    \end{split}
\end{align}
where $t_0$ is the order of the tensor.
Taking the square root and essential supremum of both sides we have
\begin{equation}
    \|\nabla f\|_{L_{\infty}(\text{co}(V_j))}\leq 2
    \|A\|_{L_{\infty}(\text{co}(V_j))}\|\nabla A\|_{L_{\infty}(\text{co}(V_j))},
    \label{eq:gradfAA}
\end{equation}
which can be substituted into \eqref{eq:integrationBound_1} to give
\begin{equation}
    \left|\|A\|^2_{L_2(\mathcal{B})} - \|A\|^2_{\ell^{*}_2(\mathcal{B})}\right| \leq  2
    C_{\mathrm{vor}}(\mathcal{B},Z)h_f(Z,\mathcal{B})^{\nu+1}\sum_{j=1}^M\|\nabla A\|_{L_{\infty}(\text{co}(V_j))}\|A\|_{L_{\infty}(\text{co}(V_j))}.
    \label{eq:integrationBound_21}
\end{equation}

For all relevant quadrature bounds, the tensor $A$ is either a derivative of $\mathbf{v}\in V_h$, such that $A = D^{\tau}\mathbf{v}$, or a linear operator applied to $\mathbf{v}$, such that $A = \mathcal{L}(\mathbf{v})$, whose tensor norm can be bounded by a combination of derivatives as
\begin{equation}
    \|A\|^2_T = \|\mathcal{L}\left(\mathbf{v}\right)\|^2_T \leq \sum_{\tau \in \mathcal{S}} C^2_{\mathcal{L},\tau}\|D^{\tau}\mathbf{v}\|^2_T,
    \label{eq:genLinearOpBound}
\end{equation}
where $\mathcal{S}\subset \mathbb{N}_0$ is a set of orders based on the linear operator and $C_{\mathcal{L},\tau}$ are constants dependent on the operator. For example if $A = D^2\mathbf{v}$, then $\mathcal{S} = \{2\}$ and $C_{\mathcal{L},2} = 1$. Substituting definition \eqref{eq:genLinearOpBound} into the quadrature bound gives
\begin{equation}
        \left|\|A\|^2_{L_2(\mathcal{B})} - \|A\|^2_{\ell^{*}_2(\mathcal{B})}\right| \leq  2
        C_{\mathrm{vor}}(\mathcal{B},Z)h_f(Z,\mathcal{B})^{\nu+1}\sum_{\tau\in\mathcal{S}}C_{\mathcal{L},\tau}\sum_{\chi\in \mathcal{S}+1}C_{\nabla\mathcal{L},\chi}\sum_{j=1}^M\|D^{\chi}\mathbf{v}\|_{L_{\infty}(V_j)}\|D^{\tau}\mathbf{v}\|_{L_{\infty}(V_j)},
        \label{eq:integrationBound_2}
\end{equation}
where $\|\nabla A\|_T$ has been replaced by a sum of derivatives of orders given by the set $\mathcal{S} + 1$ and constants $C_{\nabla\mathcal{L},\chi}$. In the case of $A = D^{\tau}\mathbf{v}$ we omit the sums and $\chi = \tau + 1$. 

The functions in our discrete space can be written as $\mathbf{v} = \sum_{k=1}^Pw_k\mathbf{v}^{(k)}$, where $\mathbf{v}^{(k)}$ are local approximations. 
To simplify the notation we drop $\text{co}(V_j)$ and for the rest of the section we let $V_j$ indicate the convex hull of the Voronoi region $V_j$. 
If $V_j\subset\Omega_k$, we can use $\|\cdot\|_{L_\infty(V_j)}\leq \|\cdot\|_{L_\infty(\Omega_k)}$. Since the overlap between patches for locations $\mathbf{y}\in\bar{\Omega}$ is chosen such that $\delta>\mathrm{diam}(V_k)$, we have
\begin{equation}
    \left\{\begin{array}{lcl}
         V_j\subset\Omega_k,& & \mathbf{y}_j\in\Omega_k,\\
         V_j\subset\Omega_{k_i},&\mbox{for some } i\in\{1,\ldots,s_k\}, &\mathbf{y}_j\in\cap_{i=1}^{s_k}\Omega_{k_i}, 
    \end{array}\right.
\end{equation}
where $s_k$ is the number of overlapping patches at $\mathbf{y}_j$.
That is, a Voronoi region $V_j$ is always fully covered by at least one of the patches associated with its evaluation point $\mathbf{y}_j$. We define the patch index set $I_k=\{j\,|\,\Omega_j\cap\Omega_k\cap\mathcal{B}\neq\varnothing\}$ and let the number of evaluation points $\mathbf{y}_j\in\Omega_k$ be denoted by $m_k$. Then we can bound the sum over the Voronoi regions by using a sum over the overlapping patches and the patch-wise norms as

\begin{align}
    \begin{split}
        \sum_{j=1}^M\|D^{\chi}\mathbf{v}\|_{L_{\infty}(V_j)}\|D^{\tau}\mathbf{v}\|_{L_{\infty}(V_j)}  \leq & \sum_{k=1}^Pm_k\|D^{\chi}\mathbf{v}\|_{L_{\infty}(\Omega_k)}\|D^{\tau}\mathbf{v}\|_{L_{\infty}(\Omega_k)} \\
         = & \sum_{k=1}^Pm_k\|D^{\chi}\sum_{j\in I_k}\mathbf{v}^{(j)}w_j\|_{L_{\infty}(\Omega_k)}\|D^{\tau}\sum_{j\in I_k}\mathbf{v}^{(j)}w_j\|_{L_{\infty}(\Omega_k)} \\
         \leq & \sum_{k=1}^Pm_k\sum_{j\in I_k}\|D^{\chi}(\mathbf{v}^{(j)}w_j)\|_{L_{\infty}(\Omega_j)}\sum_{j\in I_k}\|D^{\tau}(\mathbf{v}^{(j)}w_j)\|_{L_{\infty}(\Omega_j)}, 
        \label{eq:integrationBound_2_2}
     \end{split}
\end{align}
where the triangle inequality was used in the last step.

Furthermore, we bound the number of evaluation points $m_k$ in each patch using the fill distance measures \eqref{eq:fillX}, \eqref{eq:fillY}, and \eqref{eq:fillYBnd}. The relation
\begin{equation}
    c_1M^{-\frac{1}{\nu}} \leq h_f(Z,\mathcal{B})\leq c_2M^{-\frac{1}{\nu}}
    \label{eq:fillDistance_Pts}
\end{equation}
where $M = |Z|$ is shown in \cite[Proposition 14.1]{Wendland_2004} for $\mathcal{B}=\Omega$. For the fill distance on the boundary, $\mathcal{B} =\partial\Omega$, the lower bound can be proven by taking the sum over all Voronoi regions and using \eqref{eq:VoronoiSize_diameter} and~\eqref{eq:VoronoiDiameter_fill}, while the upper bound can be proven by the fact that $\cup_{j=1}^M B(\mathbf{y}_j,h_s(\partial\Omega,Y_{\partial\Omega}))\cap\partial\Omega \subseteq |\partial\Omega|$, where $B(\mathbf{y}_j,h_s(\partial\Omega,Y_{\partial\Omega}))$ are $d$ dimensional balls around each point on the boundary and the size of each restricted ball is larger than a disc with the separation distance \eqref{eq:sepDistance} as radius $|B(\mathbf{y}_j,h_s(\partial\Omega,Y_{\partial\Omega}))\cap\partial\Omega| \geq h_s(\partial\Omega,Y_{\partial\Omega})^{\nu}\pi^{\nu/2}/\Gamma(\nu/2 + 1)$, where $h_s$ is the separation distance on the boundary.

The number of evaluation points in each patch can be bounded in terms of the fill distance by first assuming that $\min_{k}m_k\leq \chi M/P$, where $\chi$ is the maximum number of overlapping patches in the domain and we have considered the limiting case of full overlap ($\delta_0=1$).
Using \eqref{eq:fillDistance_Pts} then gives
\begin{equation}
    \min_{1\leq k\leq P}m_k \leq \chi \frac{M}{P} \leq \chi c_1^{\nu}\frac{h_f(Z,\mathcal{B})^{-\nu}}{P}.
\end{equation}
Multiplying the numerator and denominator by $|\Omega_k|$ and using the ratio between the largest and smallest patch in the domain $c_{\mathrm{vol}} = \min_k|\Omega_k|/\max_k|\Omega_k|$, we have
\begin{equation}
    \min_{1\leq k\leq P}m_k\leq \chi c_1^{\nu} \frac{h_f(Z,\mathcal{B})^{-\nu}|\Omega_k|}{P|\Omega_k|} \leq \chi c_1^{\nu} \frac{h_f(Z,\mathcal{B})^{-\nu}|\Omega_k|}{c_{\mathrm{vol}}|\Omega|} \leq \frac{4}{3}\pi n \chi c_1^{-\nu}c_{\mathrm{vol}}^{-1}|\Omega|^{-1} h_f(Z,\mathcal{B})^{-\nu}h_k^{d}, 
    \label{eq:cVol_FirstUse}
\end{equation}
where in the final step we used $P\max_{k}|\Omega_k| \geq |\Omega|$ and $|\Omega_k|\leq \frac{4}{3}\pi n h_k^d$ for the centre point fill distance \eqref{eq:fillX} and \eqref{eq:fillX_VolumeRel}. Additionally, assuming the adaptive refinement algorithm constrains the variation of patch of sizes during refinement, we can use $\max_{k}m_k\leq \tilde{C}_{m}\min_{k}m_k$ and we arrive at the bound
\begin{equation}
    m_k\leq C_m h_f(Z,\mathcal{B})^{-\nu}h_k^{d},
\end{equation}
where the constant $C_m = \frac{4}{3}\pi \tilde{C}_m n \chi c_1^{-\nu}c_{\mathrm{vol}}^{-1}|\Omega|^{-1}$. 
Using this estimate for $m_k$ in \eqref{eq:integrationBound_2_2} gives
\begin{equation}
    \sum_{j=1}^M\|D^{\chi}\mathbf{v}\|_{L_{\infty}(V_j)}\|D^{\tau}\mathbf{v}\|_{L_{\infty}(V_j)}  \leq  C_mh_f(Z,\mathcal{B})^{-\nu}\sum_{k=1}^P h_k^d \sum_{j\in I_k}\|D^{\chi}(\mathbf{v}^{(j)}w_j)\|_{L_{\infty}(\Omega_j)}\sum_{j\in I_k}\|D^{\tau}(\mathbf{v}^{(j)}w_j)\|_{L_{\infty}(\Omega_j)},
\end{equation}
which combined with the Leibnitz rule from \eqref{eq:LeibnizRuleOperator} becomes
{\small
 \begin{align}
     \begin{split}
        \sum_{j=1}^M\|D^{\chi}\mathbf{v}\|_{L_{\infty}(V_j)}\|D^{\tau}\mathbf{v}\|_{L_{\infty}(V_j)} \leq C_mC_Th_f(Z,\mathcal{B})^{-\nu}&\sum_{k=1}^Ph_k^d\sum_{j\in I_k}\sum_{\tilde{q}+ \tilde{\zeta} = \chi}\frac{\chi!}{\tilde{q}!\tilde{\zeta}!}\|D^{\tilde{q}}\mathbf{v}^{(j)}\|_{L_{\infty}(\Omega_j)}\|D^{\tilde{\zeta}}w_j\|_{L_{\infty}(\Omega_j)} \\
        & \sum_{j\in I_k}\sum_{q+ \zeta = \tau}\frac{\tau!}{q!\zeta!}\|D^{q}\mathbf{v}^{(j)}\|_{L_{\infty}(\Omega_j)}\|D^{\zeta}w_j\|_{L_{\infty}(\Omega_j)},
    \end{split}
\end{align}}where $C_T$ is the maximum of the constants in the two terms and $q,~\tilde{q},~\zeta,~\tilde{\zeta}\in\mathbb{N}_0$.


Our approximation is constructed in the reference patch, $\Omega_0$ \eqref{eq:RBFinterpolantLocal}, using the cardinal basis defined in \eqref{eq:cardinalBasis}. We aim to use properties of this cardinal basis, so we express the derivatives of $\mathbf{v}^{(j)}$ using derivatives in the reference patch. Hence we use \eqref{eq:derivativeRef2PhysicalL2Norms} to express the sum as
\begin{align}
    \begin{split}
        \sum_{j=1}^M\|D^{\chi}\mathbf{v}\|_{L_{\infty}(V_j)}\|D^{\tau}\mathbf{v}\|_{L_{\infty}(V_j)} \leq & C_mC_Th_f(Z,\mathcal{B})^{-\nu}\sum_{k=1}^Ph_k^d \\
        & \sum_{j\in I_k}\sum_{\tilde{q}+ \tilde{\zeta} = \chi}\frac{\chi!}{\tilde{q}!\tilde{\zeta}!}C_{ref}(\tilde{q})\|D^{\tilde{q}}\tilde{\mathbf{v}}^{(j)}\|_{L_{\infty}(\Omega_0)}\|D^{\tilde{\zeta}}w_j\|_{L_{\infty}(\Omega_j)} \\
        & \sum_{j\in I_k}\sum_{q+ \zeta = \tau}\frac{\tau!}{q!\zeta!}C_{ref}(q)\|D^q\tilde{\mathbf{v}}^{(j)}\|_{L_{\infty}(\Omega_0)}\|D^{\zeta}w_j\|_{L_{\infty}(\Omega_j)}.
        \label{eq:integrationBound_3}
    \end{split}
\end{align}
We additionally aim to modify \eqref{eq:integrationBound_3} such that it is bounded by the $L_2$-norm for \eqref{eq:genEigenProblemL2}, and the $H^2$-norm for \eqref{eq:genEigenProblemH2} and \eqref{eq:genEigenProblemBilinear} over the entire domain $\Omega$. To achieve this, we require three specific estimates involving the local function $\tilde{\mathbf{v}}^{(k)}$. All these estimates include a scaling argument and certain assumptions on our local cardinal basis. Starting with the scaling argument, consider any bounded domain $K\subseteq \Omega_0\in \mathbb{R}^d$, which is scaled to a region $\hat{K}$ defined as
\begin{equation}
    \hat{K} := \{\mathbf{y}'/h_0:\hspace{0.2cm}\forall\mathbf{y}'\in K\}, 
    \label{eq:scaled_K}
\end{equation}
where $h_0$ is the fill distance in the reference patch $\Omega_0$. The local approximation is defined as $\tilde{\mathbf{v}}^{(k)} = \sum_{j=1}^n\tilde{\mathbf{v}}^{(k)}(\mathbf{x}'_j)\tilde{\psi}_j^{(k)}$ and 
we define the cardinal basis in the scaled region such that $\hat{\psi}_j^{(k)}(\hat{\mathbf{y}}) = \tilde{\psi}_j^{(k)}(\mathbf{y}')$. 

Both our local cardinal basis and its derivatives, $D^{\alpha}\tilde{\psi}_j^{(k)}$, are linearly independent in $\mathbb{R}^d$ since they are constructed using the positive definite Gaussian kernel \cite[Theorem 6.10]{wendland2005approximate}. This can be shown by setting $\tilde{\mathbf{v}}^{(k)} = 0$, taking the Fourier transform and following the second part in the proof of \cite[Theorem 16.4]{wendland2005approximate}. Moreover, given the Gaussian and its derivatives are analytic, we use \cite[Lemma 5.22]{kuchment2016overview} which states that if an analytic function is zero in a connected open domain of positive measure $K$ it is identically zero. Specifically, $\|\tilde{\mathbf{v}}^{(k)}\|_{L_2(K)} = 0,~\|D^{\alpha}\tilde{\mathbf{v}}^{(k)}\|_{L_2(K)} = 0,~\text{iff}~ \tilde{\mathbf{v}}^{(k)}(\mathbf{x}_j') = 0,~j=1,\dots,n$. Analyticity of the Gaussian was also used in \cite{SOMMARIVA2026116983} to prove unisolvency of unsymmetric collocation.

The three local estimates are the following:
\begin{enumerate}
    \item 
      A bound for the infinity norm of the local function derivatives, similar to \eqref{eq:stabilityWeightFunc} for weight functions. We start with the pointwise tensor norm (as defined in Section \ref{sec:Definitions}) 
        \begin{align}
            \begin{split}
                    \|D^{\alpha}\tilde{\mathbf{v}}^{(k)}\|_T & = \|\sum_{j=1}^n\tilde{\mathbf{v}}^{(k)}(\mathbf{x}'_j)D^{\alpha}\tilde{\psi}_j^{(k)}\|_T \leq \sum_{i=1}^d\max_{1\leq j \leq n}\left|\tilde{v}_i^{(k)}(\mathbf{x}'_j)\right|\left\|\sum_{j=1}^n D^{\alpha}\tilde{\psi}_j^{(k)}\right\|_T \\
                    & \leq \sum_{i=1}^d\max_{1\leq j\leq n}\left|\tilde{v}_i^{(k)}(\mathbf{x}'_j)\right| h_0^{-|\alpha|}\left\|\sum_{j=1}^nD^{\alpha}\hat{\psi}_j^{(k)}\right\|_T \\
                    & \leq \ell_{\alpha}(\hat{\mathbf{y}})h_{0}^{-|\alpha|}d\max_{\substack{1\leq j\leq n \\ 1\leq i\leq d}}\left|\tilde{v}_i^{(k)}(\mathbf{x}'_j)\right|,
            \end{split}
        \end{align}
        where $\alpha$ is a multi-index
        $\tilde{\mathbf{v}}^{(k)}(\mathbf{x}'_j)$ are the nodal values at the point $\mathbf{x}'_j \in X_0$, 
        and $\ell_{\alpha}(\hat{\mathbf{y}}) = \sum_{j=1}^n\|D^{\alpha}\hat{\psi}_j^{(k)}(\hat{\mathbf{y}})\|_T$ is a Lebesgue function. Taking the essential supremum over the reference patch we have
        \begin{equation}
             \|D^{\alpha}\tilde{\mathbf{v}}^{(k)}\|_{L_{\infty}(\Omega_0)} \leq L_{\alpha}h_{0}^{-|\alpha|}d\max_{\substack{1\leq j\leq n \\ 1\leq i\leq d}}\left|\tilde{v}_i^{(k)}(\mathbf{x}'_j)\right| \leq L_{\alpha}h_0^{-|\alpha|}d\|\tilde{\mathbf{v}}^{(k)}\|_{L_{\infty}(\Omega_0)},
            \label{eq:lebesgueConstBound}
        \end{equation}
        where $L_{\alpha}=\|\ell_\alpha(\hat{\mathbf{y}})\|_{L_{\infty}(\Omega_0)}$ is a Lebesgue constant.
    \item 
        The second estimate is an upper bound of the infinity norm of $\tilde{\mathbf{v}}^{(k)}$ by its $L_2$-norm in $\Omega_0$. We compute this bound for each scalar component of $\tilde{\mathbf{v}}^{(k)}$, $\tilde{v}_i^{k}$, and then sum over all dimensions. We define $\|\tilde{v}_i^{(k)}\|^2_{L_{\infty}(\Omega_0)} = |\tilde{v}_i^{(k)}(\mathbf{y}^{*})|^2$ and note that the maximum gradient of this function is
        \begin{equation}
            \|\nabla \tilde{v}_i^{(k)}\|_{L_{\infty}(\Omega_0)} \leq L_{\nabla}h_{0}^{-1}\max_{1\leq j\leq n}\left|\tilde{v}_i^{(k)}(\mathbf{x}'_j)\right| \leq L_{\nabla}h_{0}^{-1}\left|\tilde{v}_i^{(k)}(\mathbf{y}^{*})\right|,
            \label{eq:infty_L2Bound_derivative}
        \end{equation}
        where we used \eqref{eq:lebesgueConstBound}, $L_{\nabla} = \sum_{|\alpha|=1} L_{\alpha}$ and the fact that the maximum nodal values are smaller than the maximum function value.
        
        We aim to relate the two norms by considering the smallest possible $L_2$-norm for a given $L_{\infty}$-norm. The smallest $L_2$-norm is defined by integrating a continuous radial function that decays as fast as possible from the maximum value $|\tilde{v}_i^{(k)}(\mathbf{y}^{*})|$. Using \eqref{eq:infty_L2Bound_derivative} the function is defined as
        \begin{equation}
            v(r) = \begin{cases}
            \left(1 - L_{\nabla}h_0^{-1}r\right)\left|\tilde{v}_i^{(k)}(\mathbf{y}^{*})\right|,    & \hspace{0.1cm} r\leq\frac{h_0}{L_{\nabla}},\\
            0, & \hspace{0.1cm} \text{otherwise},
            \end{cases}
        \end{equation}
        where $r = \|\mathbf{y}' - \mathbf{y}^{*}\|_2$. Integrating the squared function over a ball $B(\mathbf{y}',h_0/(L_{\nabla}))$ gives the lower bound 
        \begin{align}
            \begin{split}
                \|\tilde{v}_i^{(k)}\|^2_{L_2(\Omega_0)} & \geq \|\tilde{v}_i^{(k)}\|^2_{L_{\infty}(\Omega_0)}\int_{B\cap\Omega_0}\left|1 - L_{\nabla}h_0^{-1}r\right|^2d\mathbf{y'} \\ 
                & = c_{f}\|\tilde{v}_i^{(k)}\|^2_{L_{\infty}(\Omega_0)}\frac{2\pi^{\frac{d}{2}}}{\Gamma\left(\frac{d}{2}\right)}\int_{0}^{\frac{h_0}{L_{\nabla}}}r^{d-1}\left(1 -2L_{\nabla}h_0^{-1}r + L_{\nabla}^2h_0^{-2}r^2\right) dr,
            \end{split}
        \end{align}
        where $c_f$ is the fraction of the ball that is inside $\Omega_0$ given its centre is inside $\Omega_0$ and $\Gamma$ is the gamma function. Integrating gives
        \begin{align}
            \begin{split}
                \|\tilde{v}_i^{(k)}\|^2_{L_2(\Omega_0)} & \geq c_f \|\tilde{v}_i^{(k)}\|^2_{L_{\infty}(\Omega_0)}\frac{2\pi^{\frac{d}{2}}}{\Gamma\left(\frac{d}{2}\right)}\left[\frac{r^d}{d}-2L_{\nabla}h_0^{-1}\frac{r^{d+1}}{d+1} + L_{\nabla}^2h_0^{-2}\frac{r^{d+2}}{d+2}\right]_0^{\frac{h_0}{L_{\nabla}}} \\
                & = c_f\frac{2\pi^{d/2}}{L_{\nabla}^d\Gamma(\frac{d}{2} + 1)\left(d+1\right)\left(d+2\right)}h_0^d\|\tilde{v}_i^{(k)}\|^2_{L_{\infty}(\Omega_0)},
            \end{split}
        \end{align}
        and taking the sum over all dimensions and the square root on both sides gives 
        \begin{equation}
            \|\tilde{\mathbf{v}}^{(k)}\|_{L_{\infty}(\Omega_0)}\leq  C_{\infty}h_0^{-\frac{d}{2}}\|\tilde{\mathbf{v}}^{(k)}\|_{L_2(\Omega_0)}.
            \label{eq:inverseEstimateLinfty}
        \end{equation}
    \item 
    The third estimate is a local generalised eigenvalue bound that relates the $L_2$-norm of the function in the entire patch with the $L_2$-norm in a specific region of the patch. 
    We consider a region $K\subseteq \Omega_0$ and define the nodal value vector $\mathbf{v}_k= (\tilde{v}_1^{(k)}(X_0)^T,\dots,\tilde{v}_d^{(k)}(X_0)^T)^T$ to express the norm relation as
    \begin{align}
        \begin{split}
            \frac{\|\tilde{\mathbf{v}}^{(k)}\|^2_{L_2(\Omega_0)}}{\|\tilde{\mathbf{v}}^{(k)}\|^2_{L_2(K)}} = \frac{\mathbf{v}_k^T\mathbf{M}_{\Omega_0}\mathbf{v}_k}{\mathbf{v}_k^T\mathbf{M}_K\mathbf{v}_k} = \frac{h_0^d\mathbf{v}_k^T\widehat{\mathbf{M}}_{{\Omega}_0}\mathbf{v}_k}{h_0^d\mathbf{v}_k^T\widehat{\mathbf{M}}_{{K}}\mathbf{v}_k} =~\hat{\lambda}^{{\Omega}_0}_{{K}},
            \label{eq:localEigenBoundMass}
        \end{split}
    \end{align}
    where $\hat{\lambda}^{{\Omega}_0}_{{K}}$ is a generalised eigenvalue computed in the scaled patch.
    The mass matrices $\mathbf{M}_{\Omega_0}$ and $\mathbf{M}_{K}$ are block diagonal matrices with $d$ identical $n \times n$ blocks, with elements given by
    \begin{equation}
    M_{ij} = \left(\tilde{\psi}_i^{(k)},\tilde{\psi}_j^{(k)}\right)_{L_2(K)} = \int_{K}\tilde{\psi}_i^{(k)}\tilde{\psi}_j^{(k)}dK = h_0^{d}\int_{\hat{K}}\hat{\psi}_i^{(k)}\hat{\psi}_j^{(k)}d\hat{K} = h_0^{d}\widehat{M}_{ij}, \quad i,j=1,\dots,n.    
    \label{eq:localMassMatrix}
    \end{equation}
    Similar bounds relating the function and its derivatives are needed for estimates \eqref{eq:genEigenProblemH2} and~\eqref{eq:genEigenProblemBilinear}, we have
    \begin{align}
        \begin{split}
            \frac{\|\tilde{\mathbf{v}}^{(k)}\|^2_{L_2(K)}}{\|D^{\gamma} \tilde{\mathbf{v}}^{(k)}\|^2_{L_2(K)}} & = \frac{\|\tilde{\mathbf{v}}^{(k)}\|^2_{L_2(K)}}{\sum_{|\alpha| = \gamma}\frac{\alpha!}{\gamma!}\|D^{\alpha} \tilde{\mathbf{v}}^{(k)}\|^2_{L_2(K)}} = \frac{\mathbf{v}_k^T\mathbf{M}_{K}\mathbf{v}_k}{\mathbf{v}_k^T\mathbf{K}^\gamma_K\mathbf{v}_k} 
            =  \frac{h_0^{d}}{h_0^{d}h_0^{-2\gamma}}\frac{\mathbf{v}_k^T\widehat{\mathbf{M}}_{{K}}\mathbf{v}_k}{\mathbf{v}_k^T\widehat{\mathbf{K}}^\gamma_{{K}}\mathbf{v}_k} =~{}_{\gamma}\hat{\lambda}^{{K}}_{{K}}h_0^{2\gamma},
            \label{eq:localEigenBoundStifness}
        \end{split}
    \end{align}
    where $\gamma$ is the derivative order and $\alpha$ is a multi-index \eqref{eq:multiIndex}. In the following derivation we use $\gamma = 1,2$ and $\mathbf{K}^\gamma_K$ is a block diagonal stiffness matrix with $d$ blocks of size $n\times n$. The blocks are defined as in \eqref{eq:localMassMatrix} but with inner product $(D^{\alpha}\tilde{\psi}_i^{(k)},D^{\alpha}\tilde{\psi}_j^{(k)})_{L_2(K)}$. 
    
\end{enumerate}

Using these three estimates we first return to \eqref{eq:integrationBound_3} and insert~\eqref{eq:lebesgueConstBound} together with~\eqref{eq:stabilityWeightFunc} for the weight function derivatives, which gives
\begin{align}
    \begin{split}
        \sum_{j=1}^M\|D^{\chi}\mathbf{v}\|_{L_{\infty}(V_j)}\|D^{\tau}\mathbf{v}\|_{L_{\infty}(V_j)} \leq &   C_mC_Th_f(Z,\mathcal{B})^{-\nu}\\
        & \sum_{k=1}^Ph_k^d \sum_{j\in I_k}\|\tilde{\mathbf{v}}^{(j)}\|_{L_{\infty}(\Omega_0)}\sum_{\tilde{q}+ \tilde{\zeta} = \chi}\frac{\chi!}{\tilde{q}!\tilde{\zeta}!}C_{ref}(\tilde{q})\delta^{-\tilde{\zeta}}_j h_0^{-\tilde{q}}C_{\mathrm{GL}}(\tilde{\zeta},\tilde{q}) \\
        & \sum_{j\in I_k}\|\tilde{\mathbf{v}}^{(j)}\|_{L_{\infty}(\Omega_0)}\sum_{q+ \zeta = \tau}\frac{\tau!}{q!\zeta!}C_{ref}(q)\delta^{-\zeta}_j h_0^{-q}C_{\mathrm{GL}}(\zeta,q),
    \end{split}
\end{align}
where $C_{\mathrm{GL}}(\zeta,q) = dG_\zeta\sum_{|\alpha|=\zeta}\frac{\alpha!}{\zeta!}\sum_{|\beta|=q}\frac{\beta!}{q!}L_{\beta}$. We use \eqref{eq:OverlapFillDistRelation} and then \eqref{eq:referenceFilltoPhysicalFill} to bound the overlap and then the local fill distance by the fill distance in the reference patch, $h_0$, which gives
\begin{align}
    \begin{split}
        \sum_{j=1}^M\|D^{\chi}\mathbf{v}\|_{L_{\infty}(V_j)}\|D^{\tau}\mathbf{v}\|_{L_{\infty}(V_j)} \leq    & ~C_mC_Th_f(Z,\mathcal{B})^{-\nu} \\ & \sum_{k=1}^Ph_k^d \sum_{j\in I_k}\|\tilde{\mathbf{v}}^{(j)}\|_{L_{\infty}(\Omega_0)}\sum_{\tilde{q}+ \tilde{\zeta} = \chi}\frac{\chi!}{\tilde{q}!\tilde{\zeta}!}C_{ref}(\tilde{q})\tilde{C}_V^{-\tilde{\zeta}}h^{-\tilde{\zeta}}_j h_0^{-\tilde{q}}C_{\mathrm{GL}}(\tilde{\zeta},\tilde{q}) \\
        & \sum_{j\in I_k}\|\tilde{\mathbf{v}}^{(j)}\|_{L_{\infty}(\Omega_0)}\sum_{q+ \zeta = \tau}\frac{\tau!}{q!\zeta!}C_{ref}(q)\tilde{C}_V^{-\zeta}h^{-\zeta}_j h_0^{-q}C_{\mathrm{GL}}(\zeta,q) \\
        \leq &~C_mC_Th_f(Z,\mathcal{B})^{-\nu}  h_0^{d}h_0^{-\chi-\tau} \\ & \sum_{k=1}^PC_{0,k}^{-d}\sum_{j\in I_k}\|\tilde{\mathbf{v}}^{(j)}\|_{L_{\infty}(\Omega_0)}\sum_{\tilde{q}+ \tilde{\zeta} = \chi}\frac{\chi!}{\tilde{q}!\tilde{\zeta}!}C_{ref}(\tilde{q})\left(\frac{\tilde{C}_V}{C'_{0,j}}\right)^{-\tilde{\zeta}}C_{\mathrm{GL}}(\tilde{\zeta},\tilde{q}) \\
        & \sum_{j\in I_k}\|\tilde{\mathbf{v}}^{(j)}\|_{L_{\infty}(\Omega_0)} \sum_{q+ \zeta = \tau}\frac{\tau!}{q!\zeta!}C_{ref}(q)\left(\frac{\tilde{C}_V}{C'_{0,j}}\right)^{-\zeta}C_{\mathrm{GL}}(\zeta,q).
    \end{split}
\end{align}

Combining the constants and taking the maximum over all patches we have
\begin{equation}
    \sum_{j=1}^M\|D^{\chi}\mathbf{v}\|_{L_{\infty}(V_j)}\|D^{\tau}\mathbf{v}\|_{L_{\infty}(V_j)} \leq C_mC_{p}h_f(Z,\mathcal{B})^{-\nu}  h_0^{d}h_0^{-\chi-\tau}\sum_{k=1}^P\|\tilde{\mathbf{v}}^{(k)}\|^2_{L_{\infty}(\Omega_0)},  
\end{equation}
where the new constant is given by
\begin{align}
\begin{split}
    C_{p} = &C_T\chi^2\max_{1\leq k\leq P}C_{0,k}^{-d}\max_{1\leq k\leq P}\sum_{\tilde{q}+ \tilde{\zeta} = \chi}\frac{\chi!}{\tilde{q}!\tilde{\zeta}!}C_{ref}(\tilde{q})\left(\frac{\tilde{C}_V}{C'_{0,k}}\right)^{-\tilde{\zeta}}C_{\mathrm{GL}}(\tilde{\zeta},\tilde{q})
    \\ & \sum_{q+ \zeta = \tau}\frac{\tau!}{q!\zeta!}C_{ref}(q)\left(\frac{\tilde{C}_V}{C'_{0,k}}\right)^{-\zeta}C_{\mathrm{GL}}(\zeta,q),
    \end{split}
\end{align}
and $\chi = \max_{1\leq k\leq P}|I_k|$ is the largest number of overlapping patches in the domain which we multiply with to remove the additional sum over $I_k$.

The second estimate, \eqref{eq:inverseEstimateLinfty}, gives a bound with respect to the $L_2$-norm as
\begin{equation}
    \sum_{j=1}^M\|D^{\chi}\mathbf{v}\|_{L_{\infty}(V_j)}\|D^{\tau}\mathbf{v}\|_{L_{\infty}(V_j)}\leq C_mC_{p}C_{\infty}h_f(Z,\mathcal{B})^{-\nu} h_0^{-\chi-\tau}\sum_{k=1}^P\|\tilde{\mathbf{v}}^{(k)}\|^2_{L_2(\Omega_0)},
\end{equation}
where $C_{\infty}$ here is the maximum constant from \eqref{eq:inverseEstimateLinfty} over all patches.
Using both \eqref{eq:localEigenBoundMass} and \eqref{eq:localEigenBoundStifness} we have
\begin{equation}
    \sum_{j=1}^M\|D^{\chi}\mathbf{v}\|_{L_{\infty}(V_j)}\|D^{\tau}\mathbf{v}\|_{L_{\infty}(V_j)}\leq C_mC_{p}C_{\infty}~{}_{\gamma}\hat{\lambda}^{K}_{K}~\hat{\lambda}^{\Omega_0}_{K}h_f(Z,\mathcal{B})^{-\nu}h_0^{2\gamma - \chi - \tau}\sum_{k=1}^P\|D^{\gamma}\tilde{\mathbf{v}}^{(k)}\|^2_{L_2(K)},
\end{equation}
for a general region $K\subseteq\Omega_0$ and order $\gamma$.

    Additionally, we scale the $L_2$-norm by mapping the function to the physical patch by first using \eqref{eq:derivativeRef2PhysicalL2Norms} and assuming $K\subseteq\Omega_0$. Since we aim to combine the sum into a global $L_2$-norm, we choose $K = K_k$ which is the non-overlapped patch region in the reference domain for patch $\Omega_k$. The generalised eigenvalues are the largest from each patch problem, ${}_{\gamma}\hat{\lambda}_{{K}}^{{K}} = \max_{1\leq k \leq P}{}_{\gamma}\hat{\lambda}_{{K}_k}^{{K}_k}$ and $\hat{\lambda}_{{K}}^{{\Omega}_0} = \max_{1\leq k \leq P}\hat{\lambda}_{{K}_k}^{{\Omega}_0}$. We assume that each patch has a non-overlapped region with non-zero measure. In the physical domain this region is denoted by $\Omega\cap K_k$. We also use \eqref{eq:referenceFilltoPhysicalFill} for the fill distance over the physical domain giving
\begin{align}
    \begin{split}
        \sum_{j=1}^M\|D^{\chi}\mathbf{v}\|_{L_{\infty}(V_j)}\|D^{\tau}\mathbf{v}\|_{L_{\infty}(V_j)}  \leq & ~C_mC_{p}C_{\infty}~{}_{\gamma}\hat{\lambda}^{{K}}_{{K}}~\hat{\lambda}^{{\Omega}_0}_{{K}}C_{0,k}^{2\gamma -\chi - \tau}h_f(Z,\mathcal{B})^{-\nu}h_k^{2\gamma -\chi - \tau} \\ 
        & \sum_{k=1}^P\left(\frac{~R_{\mathrm{sc},k}~H_{\mathrm{sc},k}^{\frac{1}{2}}}{C'_{ref}(\gamma)}\right)^2\|D^{\gamma}\mathbf{v}^{(k)}\|^2_{L_2(K_k\cap\Omega)} \\
        \leq & ~C_mC_{p}C_{\infty}~{}_{\gamma}\hat{\lambda}^{{K}}_{{K}}~\hat{\lambda}^{{\Omega}_0}_{{K}}C_{0,k}^{2\gamma -\chi - \tau}h_f(Z,\mathcal{B})^{-\nu}h_k^{2\gamma -\chi - \tau} \\
        & \max_{1\leq k\leq P}\left(\frac{~R_{\mathrm{sc},k}~H_{\mathrm{sc},k}^{\frac{1}{2}}}{C'_{ref}(\gamma)}\right)^2\|D^{\gamma}\mathbf{v}\|^2_{L_2(\Omega)}.
        \label{eq:integrationBound_4}
    \end{split}
\end{align}
%
%
Combining the above with \eqref{eq:integrationBound_2} gives
\begin{equation}
    \left|\|A\|^2_{L_2(\mathcal{B})} - \|A\|^2_{\ell_2^{*}(\mathcal{B})}\right|\leq C_{\mathrm{int}}h_f\left(Z,\mathcal{B}\right)\sum_{\tau\in\mathcal{S}}C_{\mathcal{L},\tau}\sum_{\chi\in\mathcal{S}+1}C_{\nabla\mathcal{L},\chi}h^{2\gamma - \chi -\tau}\|D^{\gamma}\mathbf{v}\|^2_{L_2(\Omega)}, \hspace{0.5cm} \forall\mathbf{v}\in V_h,
    \label{eq:integrationError}
\end{equation}
where $h_f(Z,\mathcal{B})$ is the fill distance for the evaluation points $Z$ on $\mathcal{B}$ and $h$ is a measure of the global centre point fill distance \eqref{eq:fillX}. $C_{\mathrm{int}}$ is the quadrature error constant that depends on the domains $\mathcal{B}$ and $\Omega$, the number of overlapping patches $I_k$, the orders $\tau$ and $\gamma$, the scaling between the physical patches and the reference patch $R_{\mathrm{sc},k}$ and $H_{\mathrm{sc},k}$, the non-overlapped regions of each patch $\Omega\cap K_k$, the number of local centre points $n$, the node quality of the evaluation point set $Z$ and of the reference patch centre points $X_0$. Note that in case the power of $h$ is negative we use $h_k^{-1} \leq (\min_{k} h_{k})^{-1} = \tilde{h}^{-1}$ and by using the factor between the largest and smallest patches in the domain $c_{\mathrm{vol}}$ (see \eqref{eq:cVol_FirstUse}) and the relation between the fill distance and the patch volume \eqref{eq:fillX_VolumeRel}, we can relate the minimum fill distance and maximum fill distance in all patches to get $\tilde{h}^{-1} \leq c_{\mathrm{vol}}^{-1/d}c_q^{-1/d^2}(5+\pi)^{-1/d}h^{-1}$.

More crucially $C_{\mathrm{int}}$ depends on the generalised eigenvalues \eqref{eq:localEigenBoundMass},~\eqref{eq:localEigenBoundStifness} and the Lebesgue constant of the approximation in the reference patch \eqref{eq:lebesgueConstBound}. These all depend on the scaled cardinal basis $\hat{\psi}_j^{(k)}$. The scaled and unscaled cardinals have equal values, $\hat{\psi}_j^{(k)}(\hat{\mathbf{y}}) = \tilde{\psi}_j^{(k)}(\mathbf{y}^\prime)$, but for different underlying variables.
For the Gaussian, we have $\phi(r) = e^{-(\varepsilon
r)^2} = e^{-(\hat{\varepsilon} \hat{r})^2}$, where  
$\hat{r} = \|\mathbf{y}'/h_0 - \mathbf{x}'_j/h_0\|_2$. Hence the effective shape parameter becomes $\hat{\varepsilon} = h_0\varepsilon$. We note that as $h_0$ is decreased through refinement, the effective shape parameter also decreases, which leads to increasingly flat basis functions and eventually to the polynomial limit \cite{RBFQR}. 
Additionally, given the shape of our patches and the fact that we avoid refining in the direction normal to the domain surface, the aspect ratio of the patches increases during refinement, which also affects both the Lebesgue constants and generalised eigenvalues. The scale of the Lebesgue constants as well as their dependence on the shape parameter and aspect ratio is investigated numerically in Figure~\ref{fig:Leb}. We note that small shape parameters and high derivatives correspond to the largest (but still bounded) constants, while the aspect ratio influences the constants in a less systematic way. Additionally we note that the Lebesgue constants $L_j$ for derivatives here are defined as $L_j = \max_{|\alpha| = j}L_{\alpha}$, where $\alpha$ is a multi-index for $L_{\alpha}$ as defined in \eqref{eq:lebesgueConstBound}.
\begin{figure}[!htb]
\centering
\includegraphics[width=0.32\textwidth]{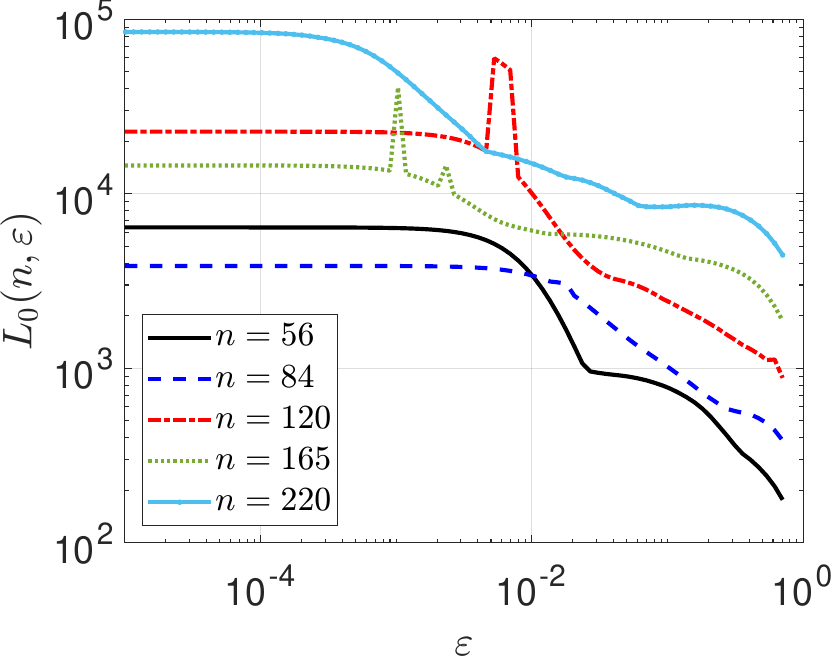}\hspace{0.1cm}
\includegraphics[width=0.32\textwidth]{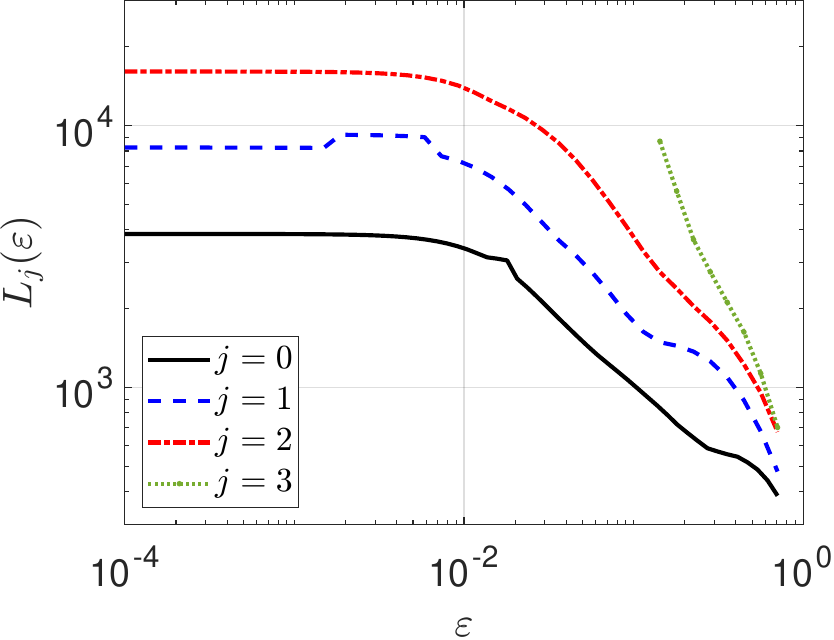}\hspace{0.1cm}
\includegraphics[width=0.32\textwidth]{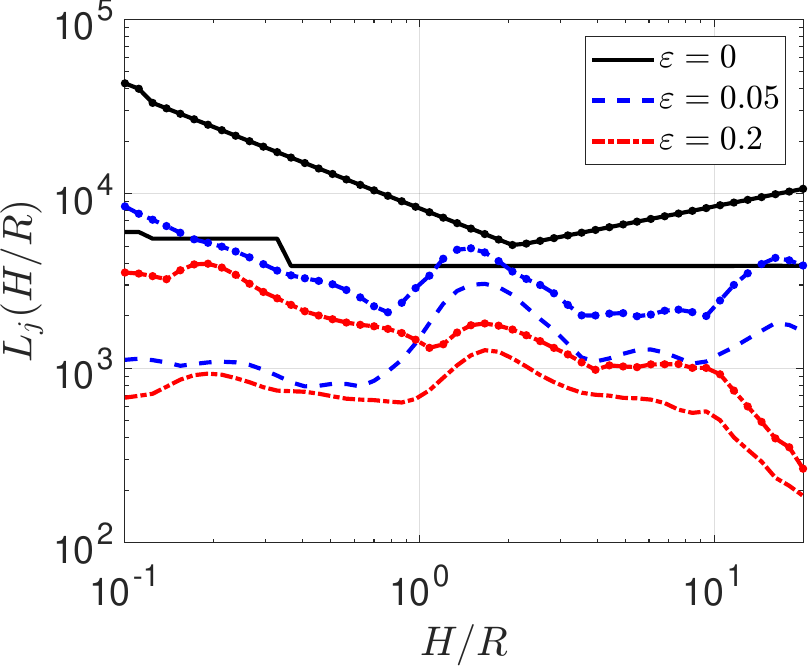}
\caption{The Lebesgue constant $L_0$ for function evaluation as a function of the shape parameter $\hat{\varepsilon}$ for different numbers of local points $n$ (left). The Lebesgue constants $L_j$ for derivatives of order $j=0,\ldots,3$ for $n=84$ local points (middle). Note that RBF-QR only supports derivatives up to order two, and hence the third derivatives cannot be computed stably for smaller shape parameter values. The Lebesgue constants $L_0$ (no marker) and $L_1$ (dot marker) for different patch aspect ratios for three different shape parameters and $n=84$ local points (right).}
\label{fig:Leb}
\end{figure}

The eigenvalue bounds are more expensive to approximate numerically, but are expected to show a similar dependence on the parameters due to norm equivalence relations. To show the expected scale of these constants we compute the largest eigenvalues for simplified versions of bounds \eqref{eq:localEigenBoundMass} and \eqref{eq:localEigenBoundStifness} which include norms of a scalar valued function $v = \sum_{j=1}^n\tilde{\psi}_j v(\mathbf{x}_j')$ constructed using the cardinal basis \eqref{eq:cardinalBasis} in the reference patch $\Omega_0$. These eigenvalues as well as the corresponding eigenfunctions for two test problems are shown in Figure~\ref{fig:Eig}. 
\begin{figure}
    \centering  \includegraphics[width=0.31\linewidth]{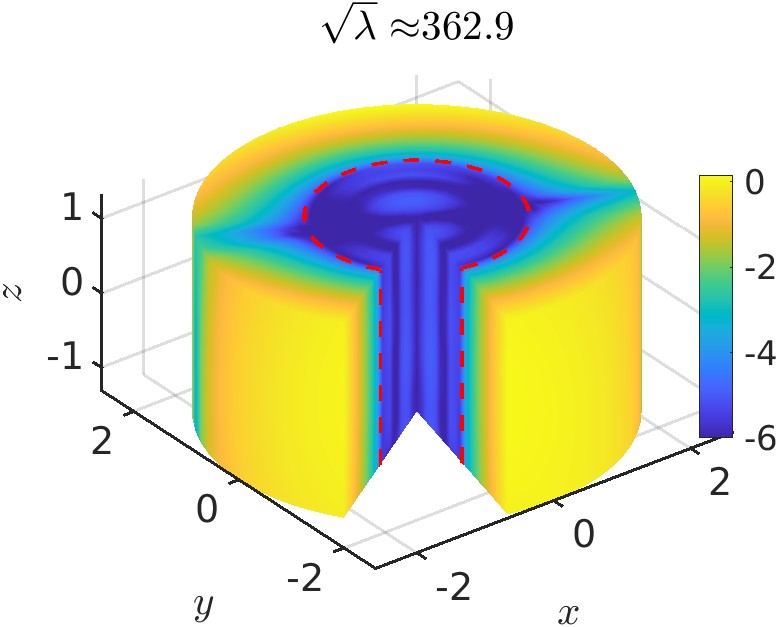} \hspace{0.2cm}
\includegraphics[width=0.31\linewidth]{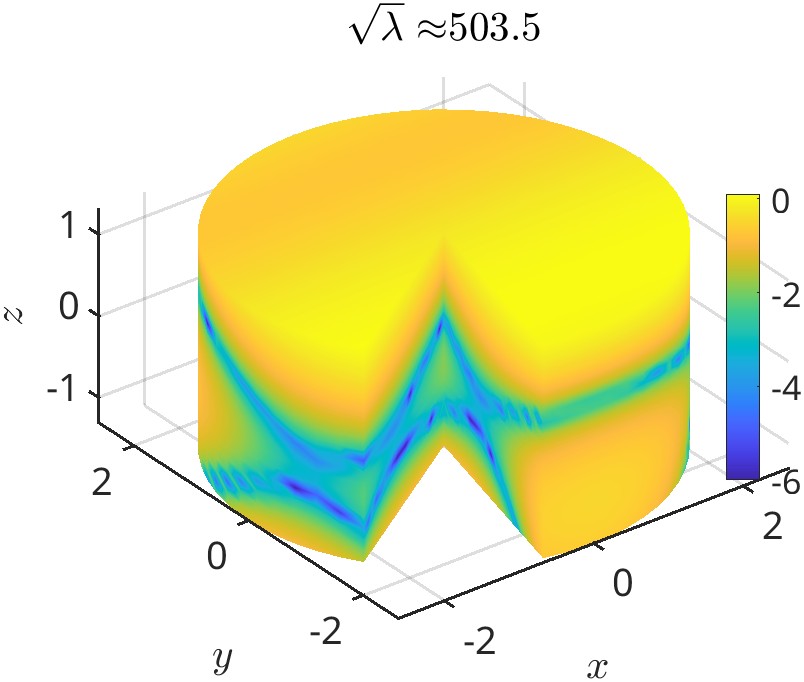}\hspace{0.2cm}
\includegraphics[width=0.31\linewidth]{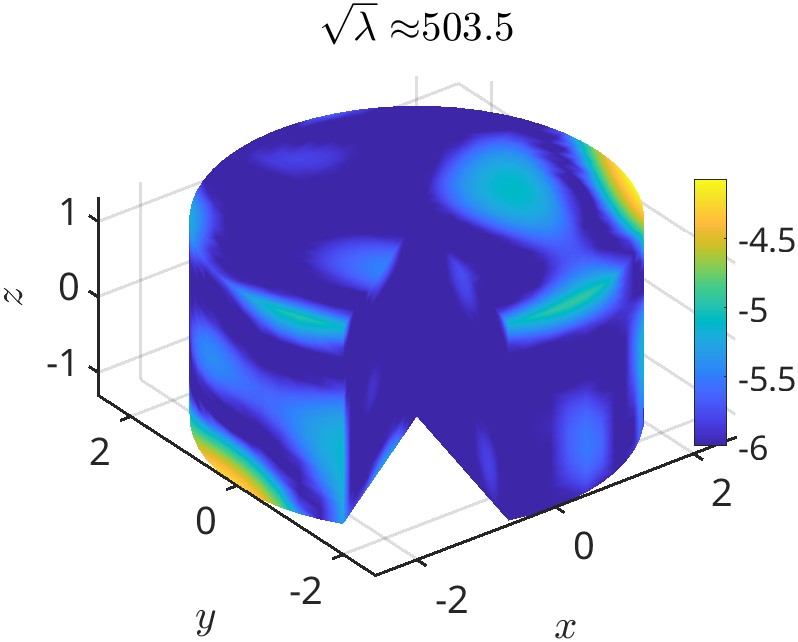}
    \caption{For patch eigenfunctions $v$ in the reference patch $\Omega_0$ with interior part of the domain $K$ (dashed red line) we show $\log_{10}|v|$ for the largest eigenvalue $\lambda=\max_{v}\|v\|_{L_2(\Omega_0)}^2/\|v\|^2_{L_2(K)}$ (left) and $\log_{10}|v|$ (middle) as well as $\log_{10}|v_x|$ (right) for the largest eigenvalue $\lambda=\max_{v}\|v\|_{L_2(\Omega_0)}^2/\|v_x\|^2_{L_2(\Omega_0)}$.
     In both cases, the number of local points $n=56$, the aspect ratio $H/R = 1$ and the shape parameter $\varepsilon=0.5$.}
    \label{fig:Eig}
\end{figure}

We can arrive at the quadrature error bound \eqref{eq:genEigenProblemL2} by using \eqref{eq:Cw} and \eqref{eq:integrationError}, where we choose $A = \mathbf{v}$, $\mathcal{B} = \Omega$, the interior evaluation point set $Z = Y_{\Omega}$, and the set of orders $\mathcal{S} = \{0\}$, while we choose order $\gamma = 0$. This gives the following bound
\begin{align}
\begin{split}
    \left|\|\mathbf{v}\|_{L_2(\Omega)}^2 - \|\mathbf{v}\|_{\ell^{*}_2(\Omega)}^2\right| & \leq C_{W,1}\left|\|\mathbf{v}\|^2_{L_2(\Omega)}-\|\mathbf{v}\|^2_{\ell^{*}_2(\Omega)}\right| + C_{W,2}\|\mathbf{v}\|^2_{L_2(\mathcal{B})} \\
    & \leq \left(C_{\mathrm{int}}C_{W,1}\frac{h_y}{h} + C_{W,2}\right)\|\mathbf{v}\|^2_{L_2(\Omega)},
\end{split}
\end{align}
where $h_y$ is the evaluation point fill distance.

\noindent For \eqref{eq:genEigenProblemH2} we have
\begin{align}
    \begin{split}
        \|\mathbf{v}\|^2_{h^2(\Omega)} - \|\mathbf{v}\|^2_{H^2(\Omega)}  & = \sum_{\tau = 0}^2 \|D^{\tau}\mathbf{v}\|^2_{\ell_2(\Omega)} + \|D^{\tau}\mathbf{v}\|^2_{L_2(\Omega)} \\
        & \leq C_{W,1}\sum_{\tau = 0}^2\left|\|D^{\tau}\mathbf{v}\|^2_{L_2(\Omega)} - \|D^{\tau}\mathbf{v}\|^2_{\ell_2^{*}(\Omega)}\right| + C_{W,2}\|\mathbf{v}\|^2_{H^2(\Omega)},
    \end{split}
\end{align}
where we used \eqref{eq:Cw}. Then choosing $A = D^{\tau}\mathbf{v}$, $\mathcal{B} = \Omega$, $Z = Y_{\Omega}$ and for each term in the sum we set $\mathcal{S} = \{\tau\}$ and $\gamma = \tau$ in \eqref{eq:integrationError} which gives 
\begin{equation}
    \|\mathbf{v}\|^2_{h^2(\Omega)} - \|\mathbf{v}\|^2_{H^2(\Omega)} \leq \left(C_{\mathrm{int}}C_{W,1}\frac{h_y}{h} + C_{W,2}\right)\|\mathbf{v}\|^2_{H^2(\Omega)}
\end{equation}
where in this case $C_{\mathrm{int}}$ is the maximum constant over the three terms in the $H^2$-norm. Note that when deriving this bound we rely on third order derivatives of $\mathbf{v}$, since $\chi = 3$ for the highest order term, while the generating Wendland functions have only $2$ continuous derivatives (see \cite[Chapter 9]{Wendland_2004}). However, for the chosen $C^2$ Wendland generating functions, the third derivative is simply discontinuous at $r = 0$, for which the essential supremum ($L_{\infty}$-norm) is well defined.

The quadrature error for the the bilinear form \eqref{eq:genEigenProblemBilinear} can be expressed by two terms, one in the domain and one on the boundary as
\begin{align}
    \begin{split}
        \tilde{\alpha}\left(\mathbf{v},\mathbf{v}\right) -  \alpha_h\left(\mathbf{v},\mathbf{v}\right)  = &  \|\nabla\cdot\sigma\left(\mathbf{v}\right)\|^2_{L_2(\Omega)} - \|\nabla\cdot\sigma\left(\mathbf{v}\right)\|^2_{\ell_2(\Omega)} \\
        & + h^{-1} \|\kappa_1\sigma\left(\mathbf{v}\right)\cdot\mathbf{n} + h^{-1}\kappa_0\mathbf{v}\|^2_{L_2(\partial\Omega)} - h^{-1}\|\kappa_1\sigma\left(\mathbf{v}\right)\cdot\mathbf{n} + h^{-1}\kappa_0\mathbf{v}\|^2_{\ell_2(\partial\Omega)} \\
        \leq & C_{W,1}\left| \|\nabla\cdot\sigma\left(\mathbf{v}\right)\|^2_{L_2(\Omega)} - \|\nabla\cdot\sigma\left(\mathbf{v}\right)\|^2_{\ell^*_2(\Omega)} \right| \\
        & + C_{W,1}h^{-1}\left| \|\kappa_1\sigma\left(\mathbf{v}\right)\cdot\mathbf{n} + h^{-1}\kappa_0\mathbf{v}\|^2_{L_2(\partial\Omega)} - \|\kappa_1\sigma\left(\mathbf{v}\right)\cdot\mathbf{n} + h^{-1}\kappa_0\mathbf{v}\|^2_{\ell^*_2(\partial\Omega)} \right| \\
        & + C_{W,2}\tilde{\alpha}\left(\mathbf{v},\mathbf{v}\right),
        \label{eq:integrationErrorBilinear_1}
    \end{split}
\end{align}
where we used \eqref{eq:Cw} and $h$ is the centre point fill distance \eqref{eq:fillX}. In this case both terms include a linear operator for which we compute the constants given in bound \eqref{eq:genLinearOpBound}. We start with the domain operator which is bounded by the Hessian as
\begin{align}
    \begin{split}
        \|\nabla\cdot\sigma\left(\mathbf{v}\right)\|^2_2 & \leq 8\mu^2\|\nabla\cdot\epsilon\left(\mathbf{v}\right)\|_2^2 + 2\lambda^2\|\nabla\cdot\left(\nabla\cdot\mathbf{v} I\right)\|_2^2 \\
        & \leq 2d\left(4\mu^2 + \lambda^2\right)\|D^2\mathbf{v}\|^2_T,
    \end{split}
\end{align}
where we used bounds \eqref{eq:DivDivVI},~\eqref{eq:DivStrain}. The derivative term is given by
\begin{align}
    \begin{split}
        \|\nabla\left(\nabla\cdot\sigma\left(\mathbf{v}\right)\right)\|_T^2 & \leq 8\mu^2\|\nabla\left(\nabla\cdot\epsilon\left(\mathbf{v}\right)\right)\|^2_T + 2\lambda^2\|\nabla\left(\nabla\cdot\left(\nabla\cdot\mathbf{v} I\right)\right)\|_T^2 \\
        & \leq 2d\left(4\mu^2+\lambda^2\right)\|D^3\mathbf{v}\|^2_T,
    \end{split}
\end{align}
where we used \eqref{eq:gradDivDivV}. Hence results show that for $A = \nabla\cdot\sigma$, we have $\mathcal{S} = \{2\}$ and $C^2_{\mathcal{L},2} = C^2_{\nabla\mathcal{L},3} = 2d\left(4\mu^2 + \lambda^2\right)$. Additionally choosing $\mathcal{B} = \Omega$, $Z = Y_{\Omega}$, $\gamma = 2$ and using \eqref{eq:integrationError} gives the quadrature error for the domain term
\begin{equation}
    \left| \|\nabla\cdot\sigma\left(\mathbf{v}\right)\|^2_{L_2(\Omega)} - \|\nabla\cdot\sigma\left(\mathbf{v}\right)\|^2_{\ell^*_2(\Omega)} \right| \leq C_{\mathrm{int},\Omega}\frac{h_y}{h}\|D^2\mathbf{v}\|^2_{L_2(\Omega)},
    \label{eq:DomainIntErrorl2star}
\end{equation}
where $C_{\mathrm{int},\Omega} = 2d(4\mu^2 + \lambda^2)C_{\mathrm{int}}$. 

For the 2-norm of the boundary term we have
\begin{align}
    \begin{split}
        \|\kappa_1\sigma\left(\mathbf{v}\right)\cdot\mathbf{n} + h^{-1}\kappa_0\mathbf{v}\|_2^2 & \leq 2\|\sigma\left(\mathbf{v}\right)\cdot\mathbf{n}\|_2^2 + 2h^{-2}\|\mathbf{v}\|^2_2 \leq 2\|\sigma\left(\mathbf{v}\right)\|_T^2 + 2h^{-2}\|\mathbf{v}\|^2_2 \\
        & = 2\|2\mu\nabla\mathbf{v} - 2\mu\nabla^{\alpha}\mathbf{v}+\lambda\left(\nabla\cdot \mathbf{v}I\right)\|^2_T + 2h^{-2}\|\mathbf{v}\|_2^2 \\
        & \leq 4\mu^2\left(8\|\nabla\mathbf{v}\|_T^2 + 8\|\nabla^{\alpha}\mathbf{v}\|_T^2\right) + 4\lambda^2\|\nabla\cdot\mathbf{v} I\|_T^2 + 2\|\mathbf{v}\|_2^2 \\
        & \leq \left(32\mu^2+4d^2\lambda^2\right)\|\nabla\mathbf{v}\|_T^2 + 2h^{-2}\|\mathbf{v}\|_2^2,
        \label{eq:bndTerm_1}
    \end{split}
\end{align}
where we used $\|\kappa_1\|_{L_{\infty}(\partial\Omega)}=\|\kappa_2\|_{L_{\infty}(\partial\Omega)}=1$, the Cauchy-Schwartz inequality and $\|\mathbf{n}\|_2 = 1$, as well as ~\eqref{eq:divV} and \eqref{eq:gradaV}. The gradient of the boundary term is given by
\begin{align}
    \begin{split}
        \|\nabla\left(\kappa_1\sigma\left(\mathbf{v}\right)\cdot\mathbf{n}\right) + h^{-1}\nabla\left(\kappa_0\mathbf{v}\right)\|_T^2\leq 2\|\nabla\left(\kappa_1\sigma\left(\mathbf{v}\right)\cdot\mathbf{n}\right)\|_T^2 + 2h^{-2}\|\nabla\left(\kappa_0\mathbf{v}\right)\|_T^2.
        \label{eq:gradBndTerm_2_beginning}
    \end{split}
\end{align}

The terms in~\eqref{eq:gradBndTerm_2_beginning} are treated separately since the Robin coefficients, $\kappa_1,~\kappa_0$, and the normals, $\mathbf{n}$, are smooth functions on the boundary (see Theorem \ref{theorem:Coercivity}). We start with the first term and use the product rule which gives
\begin{align}
    \begin{split}
        \|\nabla\left(\kappa_1\sigma\cdot\mathbf{n}\right)\|_T^2 & = \sum_{i,j=1}^d\left|\sum_{k=1}^d\left(\partial_j\kappa_1\right)\sigma_{ik}n_k + \left(\partial_j\sigma_{ik}\right)\kappa_1n_k + \left(\partial_j n_k\right)\sigma_{ik}n_k\right|^2 \\
        & \leq 4\sum_{i,j}^d\left|\sum_{k=1}^d\left(\partial_j\kappa_1\right)\sigma_{ik}n_k\right|^2 + 4\sum_{i,j}^d\left|\sum_{k=1}^d\left(\partial_j\sigma_{ik}\right)\kappa_1 n _k\right|^2 + 4\sum_{i,j}\left|\sum_{k=1}^d\kappa_1\left(\partial_j n_k\right)\sigma_{ik}\right|^2 \\
        & \leq 4\sum_{i,j}^d\left(\sum_{k=1}^d\left|\left(\partial_j\kappa_1\right)\sigma_{ik}\right|^2\right)\left(\sum_{k=1}^d\left|n_k\right|^2\right) + 4\sum_{i,j}^d\left(\sum_{k=1}^d\left|\left(\partial_j\sigma_{ik}\right)\kappa_1\right|^2\right)\left(\sum_{k=1}^d\left|n_k\right|^2\right) \\ & 
        + 4\sum_{i,j}^d\left(\sum_{k=1}^d\left|\kappa_1\sigma_{ik}\right|^2\right)\left(\sum_{k=1}^d\left|\partial_j n_k\right|^2\right)\\
        & = 4\|\nabla\kappa_1\|^2_2\|\mathbf{n}\|^2_2\|\sigma\|^2_T + 4\left|\kappa_1\right|^2\|\mathbf{n}\|_2^2\|\nabla\sigma\|_T^2+4\left|\kappa_1\right|^2\|\nabla\mathbf{n}\|_2^2\|\sigma\|_T^2 \\
        & \leq 4\left(\|\nabla\kappa_1\|_2^2+\|\nabla\mathbf{n}\|_T^2\right)\|\sigma\|_T^2 + 4\|\nabla\sigma\|_T^2,
        \label{eq:gradBndTerm_1}
    \end{split}
\end{align}
where we used $\|\mathbf{n}\|_2^2 = 1$ and $\|\kappa_1\|_{L_{\infty}(\partial\Omega)}=\|\kappa_2\|_{L_{\infty}(\partial\Omega)}=1$. The stress is given by $\sigma = 2\mu\nabla\mathbf{v}-2\mu\nabla^{\alpha}\mathbf{v}+\lambda\left(\nabla\cdot\mathbf{v}\right)I$, and its tensor norm is bounded by
\begin{equation}
        \|\sigma\|^2_T\leq 8\mu^2\|\nabla\mathbf{v}-\nabla^{\alpha}\mathbf{v}\|_T^2+2\lambda^2\|\left(\nabla\cdot\mathbf{v}\right)I\|_T^2 \leq 2\left(8\mu^2 + \lambda^2d^2\right)\|\nabla\mathbf{v}\|_T^2,
\end{equation}
where we used \eqref{eq:divV} and \eqref{eq:gradaV}. Similarly we have
\begin{equation}
    \|\nabla\sigma\|_T^2\leq 8\mu^2\|D^2\mathbf{v} - \nabla\nabla^{\alpha}\mathbf{v}\|_T^2 + 2\lambda^2\|\nabla\left(\nabla\cdot\mathbf{v}I\right)\|_T^2 \leq 2\left(8\mu^2+\lambda^2d^2\right)\|D^2\mathbf{v}\|^2_T,
\end{equation}
where we used \eqref{eq:DivDivVI} and \eqref{eq:nablaNablaA_D2}. Combining with \eqref{eq:gradBndTerm_1} gives
\begin{equation}
    \|\nabla\left(\kappa_1\sigma\left(\mathbf{v}\right)\cdot\mathbf{n}\right)\|_T^2\leq 8\left(8\mu^2 + \lambda^2d^2\right)\left(\|\nabla\kappa_1\|_2^2+\|\nabla\mathbf{n}\|_T^2\right)\|\nabla\mathbf{v}\|_T^2 + 8\left(8\mu^2+\lambda^2d^2\right)\|D^2\mathbf{v}\|_T^2
    \label{eq:bndTerm_21}
\end{equation}
The second term is bounded by
\begin{align}
    \begin{split}
        \|\nabla\left(\kappa_0\mathbf{v}\right)\|_T^2 \leq \|\nabla\kappa_0\|_2^2\|\mathbf{v}\|^2_2 + \|\nabla\mathbf{v}\|_T^2.
        \label{eq:bndTerm_22}
    \end{split}
\end{align}
Combining \eqref{eq:bndTerm_21} and \eqref{eq:bndTerm_22} in \eqref{eq:gradBndTerm_2_beginning} we have
\begin{align}
\begin{split}
    & \|\nabla\left(\kappa_1\sigma\left(\mathbf{v}\right)\cdot\mathbf{n}\right) + h^{-1}\nabla\left(\kappa_0\mathbf{v}\right)\|_T^2\leq 16\left(8\mu^2+\lambda^2d^2\right)\|D^2\mathbf{v}\|^2_T \\ & +\left(2h^{-2} + 16\left(8\mu^2+\lambda^2d^2\right)\left(\|\nabla\kappa_1\|_2^2+\|\nabla\mathbf{n}\|_T^2\right)\right)\|\nabla\mathbf{v}\|_T^2 + 2h^{-2}\|\nabla\kappa_0\|_2^2\|\mathbf{v}\|_2^2.
    \label{eq:gradientOfBoundary2Norm}
\end{split}
\end{align}

The results \eqref{eq:gradientOfBoundary2Norm} and \eqref{eq:bndTerm_1} show that that for $A = \kappa_1\sigma\cdot\mathbf{n} + \kappa_1\mathbf{v}$ we have the order set $\mathcal{S} = \{0,1\}$ and coefficients $C^2_{\mathcal{L},1} = 32\mu^2 + 4d^2\lambda^2$ and $C^2_{\mathcal{L},0} = 2h^{-2}$ , while for the gradient of the same term we have constants $C^2_{\nabla\mathcal{L},2} = 16(8\mu^2 + \lambda^2d^2)$ and $C^2_{\nabla\mathcal{L},1} = (2h^{-2} + 16(8\mu^2 + \lambda^2d^2)(\|\nabla\kappa_1\|^2_{L_{\infty}(\partial\Omega)} + \|\nabla\mathbf{n}\|^2_{L_{\infty}(\partial\Omega)}))$ and $C^2_{\nabla\mathcal{L},0} = 2h^{-2}\|\nabla\kappa_0\|^2_{L_{\infty}(\partial\Omega)}$. Using these results as well as $\mathcal{B} =\partial\Omega,~Z = Y_{\partial\Omega}$, $\gamma = 2$ and substituting this into \eqref{eq:integrationError} gives
{\small
\begin{align}
    \begin{split}
        & \left| \|\kappa_1\sigma\left(\mathbf{v}\right)\cdot\mathbf{n} + h^{-1}\kappa_0\mathbf{v}\|^2_{L_2(\partial\Omega)} - \|\kappa_1\sigma\left(\mathbf{v}\right)\cdot\mathbf{n} + h^{-1}\kappa_0\mathbf{v}\|^2_{\ell^*_2(\partial\Omega)} \right| \leq C_{\mathrm{int}}h_{y,\partial\Omega} \left(2\left(8\mu^2 + d^2\lambda^2\right)^{\frac{1}{2}}h^{-1} + \sqrt{2}h^{-1}h^{0}\right) \\
        & \biggl(4\left(8\mu^2+\lambda^2d^2\right)^{\frac{1}{2}} h^{-2}+ \left(2h^{-2} + 16\left(8\mu^2 + \lambda^2d^2\right)\left(\|\nabla\kappa_1\|^2_{L_{\infty}(\partial\Omega)} + \|\nabla\mathbf{n}\|^2_{L_{\infty}(\partial\Omega)}\right)\right)^{\frac{1}{2}}h^{-1} \\ 
        & + \sqrt{2}h^{-1}\|\nabla\kappa_0\|_{L_{\infty}(\partial\Omega)}h^{0}\biggr)h^{4}\|D^2\mathbf{v}\|^2_{L_2(\Omega)} \\ 
        & \leq C_{\mathrm{int}}h_{y,\partial\Omega}\biggl(8h\left(8\mu^2+d^2\lambda^2\right) + 4h\sqrt{2}\left(8\mu^2+d^2\lambda^2\right)^{\frac{1}{2}} + 2h\sqrt{2}\left(8\mu^2 + d^2\lambda^2\right)^{\frac{1}{2}} \\ 
        & + 8h^{2}\left(8\mu^2 + d^2\lambda^2\right)\left(\|\nabla\kappa_1\|^2_{L_{\infty}(\partial\Omega)} + \|\nabla\mathbf{n}\|_{L_{\infty}(\partial\Omega)}^2\right)^{\frac{1}{2}}  + 2h^{2}\sqrt{2}\|\nabla\kappa_0\|_{L_{\infty}(\partial\Omega)}\left(8\mu^2 + d^2\lambda^2\right)^{\frac{1}{2}} + 2h \\ 
        & + 4h^{2}\sqrt{2}\left(8\mu^2 + d^2\lambda^2\right)^{\frac{1}{2}}\left(\|\nabla\kappa_1\|^2_{L_{\infty}(\partial\Omega)} + \|\nabla\mathbf{n}\|_{L_{\infty}(\partial\Omega)}^2\right)^{\frac{1}{2}} + 2h^{2}\|\nabla\kappa_0\|_{L_{\infty}(\partial\Omega)}\biggr)\|D^2\mathbf{v}\|^2_{L_2(\Omega)} \\
        & \leq C_{\mathrm{int},\partial\Omega}h_{y,\partial\Omega}h\|D^2\mathbf{v}\|^2_{L_2(\Omega)}.
        \label{eq:BndIntErrorl2star}
    \end{split}
\end{align}}

\noindent Substituting \eqref{eq:DomainIntErrorl2star} and~\eqref{eq:BndIntErrorl2star} into \eqref{eq:integrationErrorBilinear_1} gives the full bilinear form quadrature error 
\begin{equation}
    \tilde{\alpha}\left(\mathbf{v},\mathbf{v}\right) -  \alpha_h\left(\mathbf{v},\mathbf{v}\right) \leq C_{W,1}C_{\mathrm{int},\Omega}\frac{h_y}{h}\|D^2\mathbf{v}\|^2_{L_2(\Omega)} + C_{W,1}C_{\mathrm{int},\partial\Omega}h_{y,\partial\Omega}\|D^2\mathbf{v}\|^2_{L_2(\Omega)} + C_{W,2}\tilde{\alpha}\left(\mathbf{v},\mathbf{v}\right),
\end{equation}
and the final estimate \eqref{eq:genEigenProblemBilinear} is given by 
\begin{equation}
   \tilde{\alpha}\left(\mathbf{v},\mathbf{v}\right) - \alpha_h\left(\mathbf{v},\mathbf{v}\right)  \leq C_{\mathrm{int},\tilde{\alpha}}\frac{h_y}{h}\tilde{\alpha}\left(\mathbf{v},\mathbf{v}\right),
\end{equation}
where we used the coercivity of the bilinear form \eqref{eq:ContBilinear2WeightBilinear} and $C_{\mathrm{int},\tilde{\alpha}} = \max(C_{W,1}C_{\mathrm{int},\Omega}C_1(1+C_{\mathrm{inv}}), C_{W,1}C_{\mathrm{int},\partial\Omega}C_1(1+C_{\mathrm{inv}}),C_{W,2})$. Additionally, we refine the interior and boundary evaluation points such that $h_y\approx h_{y,\partial\Omega}$.

\subsection{The inverse estimate for the fractional $H^{1/2}$-norm}\label{app:halfNorm}
The coercive continuous bilinear form \eqref{eq:ContBilinear} includes a fractional Sobolev norm of the Robin boundary operator. Directly implementing this norm in our method is not straightforward and hence we use an inverse inequality to bound the fractional norm with the $L_2$-norm on the boundary as is done in least squares finite element methods \cite{bochev2009least}. The final form of the inverse estimate is
\begin{equation}
    \|\kappa_1\sigma\left(\mathbf{v}\right)\cdot\mathbf{n} + \kappa_0\mathbf{v}\|^2_{H^{\frac{1}{2}}(\partial\Omega)} \leq C_{\mathrm{inv}}h^{-1}\|\kappa_1\sigma\left(\mathbf{v}\right)\cdot\mathbf{n} + h^{-1}\kappa_0\mathbf{v}\|^2_{L_2(\partial\Omega)} + C_{\mathrm{inv}}\|\nabla\cdot\sigma(\mathbf{v})\|^2_{L_2(\Omega)}, \hspace{0.5cm} \forall\mathbf{v}\in V_h.
    \label{eq:boundaryInverseEstimate}
\end{equation}
We start by using interpolation between Hilbert spaces based on~\cite[Theorem B.11]{mclean2000strongly}. The result ensures the existence of an interpolation inequality on a $C^{1,1}$ domain, given by 
\begin{equation}
    \|\kappa_1\sigma\left(\mathbf{v}\right)\cdot\mathbf{n} + \kappa_0\mathbf{v}\|^2_{H^{\frac{1}{2}}(\partial\Omega)}\leq C_i\|\kappa_1\sigma\left(\mathbf{v}\right)\cdot\mathbf{n} + \kappa_0\mathbf{v}\|_{H^1(\partial\Omega)}\|\kappa_1\sigma\left(\mathbf{v}\right)\cdot\mathbf{n} + \kappa_0\mathbf{v}\|_{L_2(\partial\Omega)},
    \label{eq:inverse}
\end{equation}
under the assumption that the boundary condition $(\kappa_1\sigma(\mathbf{v})\cdot\mathbf{n} + \kappa_0\mathbf{v})\in H^1(\partial\Omega)$. This holds for functions $\mathbf{v}\in V_h$, which are constructed using $C^2$ weight functions, over a boundary $\partial\Omega\in C^2$, and with coefficients $\kappa_1,~\kappa_0\in C^1(\partial\Omega)$ (see \ref{theorem:Coercivity}). It should be noted that we use \cite[Definition 2.11]{SurfPDEs_Dziuk_Elliott_2013} for the Sobolev norm, such that the squared tangential gradient is integrated over the boundary $\partial\Omega$. 
In \cite{mclean2000strongly}, where the above interpolation inequality is provided, the norms are computed by assuming a Lipschitz domain and splitting the boundary using partition of unity into hypographs, which can be integrated over a lower dimensional space $\mathbb{R}^{d-1}$. However, it is possible to show that the interpolation inequality holds for our definition as well, since fractional norms and $L_2$-norms are equivalent between the two definitions, and the norm of the tangential derivative on the boundary is larger than the gradient term in the norm definition of \cite{mclean2000strongly}. 

Squaring~\eqref{eq:inverse} and using the definition of the $H^1$-norm gives
{\footnotesize
\begin{align}
    \begin{split}
        \|\kappa_1\sigma\left(\mathbf{v}\right)\cdot\mathbf{n} + \kappa_0\mathbf{v}\|^4_{H^{\frac{1}{2}}(\partial\Omega)} \leq C^2_i\Bigl(&\|\kappa_1\sigma\left(\mathbf{v}\right)\cdot\mathbf{n} + \kappa_0\mathbf{v}\|^2_{L_2(\partial\Omega)}\|\nabla_{\Gamma}\left(\kappa_1\sigma\left(\mathbf{v}\right)\cdot\mathbf{n} + \kappa_0\mathbf{v}\right)\|^2_{L_2(\partial\Omega)} 
         + \|\kappa_1\sigma\left(\mathbf{v}\right)\cdot\mathbf{n} + \kappa_0\mathbf{v}\|^4_{L_2(\partial\Omega)}\Bigr),
         \label{eq:inverseEstimateGradient}
    \end{split}
\end{align}}where the tangential derivative on the boundary is defined as $\nabla_{\Gamma}\mathbf{v}|_{\partial\Omega} = \nabla\mathbf{v}|_{\partial\Omega} - \left(\nabla\mathbf{v}|_{\partial\Omega}\cdot\mathbf{n}\right)\mathbf{n}$. 
The objective is to bound the first term by a scaled $L_2$-norm of the boundary operator. This is done by first rewriting the term with respect to the local approximant $\mathbf{v}^{(k)}$ in each patch $\Omega_k$ and then scaling the local patches to the reference patch where the cardinal basis \eqref{eq:cardinalBasis} is defined. The scaling related to the fill distance $h_0$ of the reference patch is then extracted from derivatives and integration elements, and the constant part of the bound is formulated as an eigenproblem in a patch with unit fill distance.  
This mirrors the scaling done in Appendix~\ref{app:intErrorEstimates}.  

We start by substituting the definition for a function in our space $\mathbf{v} = \sum_{k=1}^P w_k\mathbf{v}^{(k)}$ into the first term in the right hand side of~\eqref{eq:inverseEstimateGradient}. Using the triangle inequality and the product rule for derivatives gives
{\footnotesize
\begin{align}
    \begin{split}
         & \|\kappa_1\sigma(\mathbf{v})\cdot\mathbf{n} + \kappa_0\mathbf{v}\|^2_{L_2(\partial\Omega)}\|\nabla_{\Gamma}\left(\kappa_1\sigma\left(\mathbf{v}\right)\cdot\mathbf{n} + \kappa_0\mathbf{v}\right)\|^2_{L_2(\partial\Omega\cap\Omega)} \leq \Biggl(2\chi\sum_{k=1}^P\|w_k(\kappa_1\sigma(\mathbf{v}^{(k)})\cdot\mathbf{n} + \kappa_0\mathbf{v}^{(k)})\|^2_{L_2(\partial\Omega\cap\Omega_k)} \\
         & + \|\tilde{\sigma}\left(w_k,\mathbf{n}\right)\kappa_1\mathbf{v}^{(k)}\|^2_{L_2(\partial\Omega\cap\Omega_k)}\Biggr)\Biggl(4\chi\sum_{k=1}^P \|w_k\nabla_{\Gamma}(\kappa_1\sigma(\mathbf{v}^{(k)})\cdot\mathbf{n} + \kappa_0\mathbf{v}^{(k)})\|^2_{L_2(\partial\Omega\cap\Omega_k)}
       \\ 
        & + \|(\kappa_1\sigma(\mathbf{v}^{(k)})\cdot\mathbf{n}+\kappa_0\mathbf{v}^{(k)})\otimes\nabla_{\Gamma} w_k\|^2_{L_2(\partial\Omega\cap\Omega_k)}  + \|\tilde{\sigma}(w_k,\mathbf{n})\otimes\nabla_{\Gamma}(\kappa_1\mathbf{v}^{(k)})\|^2_{L_2(\partial\Omega\cap\Omega_k)}
        + \|\kappa_1\nabla_{\Gamma}\tilde{\sigma}(w_k,\mathbf{n})\mathbf{v}^{(k)}\|^2_{L_2(\partial\Omega\cap\Omega_k)}\Biggr),
    \end{split}
\end{align}}where $\chi$ is the maximum number of overlapping patches, where we have defined the weight function operator $\tilde{\sigma}(w_k,\mathbf{n}) = (\nabla w_k)\cdot\mathbf{n} I + \nabla w_k \otimes\mathbf{n} + \mathbf{n}\otimes\nabla w_k$, where $\otimes$ is the outer product,  and where the tangential gradient term $(\nabla_{\Gamma}\tilde{\sigma})\mathbf{v}^{(k)} = (\partial_l\tilde{\sigma}_{ij})v^{(k)}_j - (n_en_l\partial_l\tilde{\sigma}_{ij})v^{(k)}_j$. Using Hölder's inequality we can bound the terms with outer products, and using the Cauchy-Schwartz inequality we can bound the last term which gives
\begin{align}
    \begin{split}
        &\|\kappa_1\sigma(\mathbf{v})\cdot\mathbf{n} + \kappa_0\mathbf{v}\|^2_{L_2(\partial\Omega)}\|\nabla_{\Gamma}\left(\kappa_1\sigma\left(\mathbf{v}\right)\cdot\mathbf{n} + \kappa_0\mathbf{v}\right)\|^2_{L_2(\partial\Omega\cap\Omega)} \leq \\ 
        8\chi^2\Biggl(& \sum_{k=1}^P\|w_k\|^2_{L_{\infty}(\partial\Omega\cap\Omega_k)}\|\kappa_1\sigma(\mathbf{v}^{(k)})\cdot\mathbf{n}+\kappa_0\mathbf{v}^{(k)}\|_{L_2(\partial\Omega\cap\Omega_k)} + \|\tilde{\sigma}\left(w_k,\mathbf{n}\right)\|^2_{L_{\infty}(\Omega\cap\partial\Omega_k)}\|\kappa_1\mathbf{v}^{(k)}\|^2_{L_2(\Omega\cap\Omega_k)}\Biggr)
        \\ \Biggl( & \sum_{k=1}^P \|w_k\|^2_{L_{\infty}(\partial\Omega\cap\Omega_k)}\|\nabla_{\Gamma}(\kappa_1\sigma(\mathbf{v}^{(k)})\cdot\mathbf{n} + \kappa_0\mathbf{v}^{(k)})\|^2_{L_2(\partial\Omega\cap\Omega_k)} 
        \\ 
        + & \|\nabla_{\Gamma} w_k\|^2_{L_{\infty}(\partial\Omega\cap\Omega_k)}\|(\kappa_1\sigma(\mathbf{v}^{(k)})\cdot\mathbf{n}+\kappa_0\mathbf{v}^{(k)})\|^2_{L_2(\partial\Omega\cap\Omega_k)} \\  
        + & \|\tilde{\sigma}(w_k,\mathbf{n})\|^2_{L_{\infty}(\partial\Omega\cap\Omega_k)}\|\nabla_{\Gamma}(\kappa_1\mathbf{v}^{(k)})\|^2_{L_2(\partial\Omega\cap\Omega_k)} + \|\nabla_{\Gamma}\tilde{\sigma}(w_k,\mathbf{n})\|^2_{L_{\infty}(\partial\Omega\cap\Omega_k)}\|\kappa_1\mathbf{v}^{(k)}\|^2_{L_2(\partial\Omega\cap\Omega_k)}\Biggr).
        \label{eq:numerator_1}
    \end{split}
\end{align}
We expand the component terms in their tensor norms as
{\footnotesize
\begin{align}
    \begin{split}
        \|\nabla_{\Gamma}w_k\|_2^2 & = \sum_{i=1}^d\left|\partial_iw_k - \sum_{j=1}^d n_in_j\partial_jw_k\right|^2 \leq 2\|\nabla w_k\|^2_2 + 2\sum_{i=1}^d\left|\sum_{j=1}^dn_in_j\partial_jw_k\right|^2 \\ & 
        \leq 2\|\nabla w_k\|^2_2 + 2\sum_{i=1}^d\left(\sum_{j=1}^d |n_in_j|^2\right)\left(\sum_{j=1}^d|\partial_j w_k|^2\right) = 4\|\nabla w_k\|_2^2 \\
        \|\tilde{\sigma}(w_k,\mathbf{n})\|_T^2 & = \sum_{i,j=1}^d\left|\sum_{l=1}^d(\partial_l w_k)n_l\delta_{ij} + (\partial_i w_k)n_j + (\partial_jw_k)n_i \right|^2 \\
        & \leq 4\sum_{i,j=1}^d\left(\sum_{l=1}^d|\partial_lw_k|^2\right)\left(\sum_{l=1}^d|n_l\delta_{ij}|^2\right) + 8\|\nabla w_k\otimes\mathbf{n}\|_T^2 = (8+4d)\|\mathbf{n}\|^2_2\|\nabla w_k\|^2_2 = (8+4d)\|\nabla w_k\|^2_2 \\
        \|\nabla_{\Gamma}\tilde{\sigma}(w_k,\mathbf{n})\|^2_T &\leq 4\|\nabla \tilde{\sigma}(w_k,\mathbf{n})\|^2_T = 4\sum_{i,j,l=1}^d\left|\sum_{m=1}^d\partial_l(\partial_{m}w_k n_m)\delta_{ij} + \partial_l(\partial_i w_k n_j) + \partial_l(\partial_jw_k n_i))\right|^2 \\ 
        & = 4\sum_{i,j,l=1}^d\Biggl|\sum_{m=1}^d(\partial_{lm}w_k)n_m\delta_{ij} + (\partial_l n_m)(\partial_m w_k)\delta_{ij} + (\partial_{li}w_k)n_j + (\partial_i w_k)(\partial_l n_j) + (\partial_{lj}w_k)n_i + (\partial_j w_k)(\partial_l n_i)\Biggr|^2 \\
        & \leq 24d\|\mathbf{n}\|_2^2\|D^2w_k\|_T^2 + 24d\|\nabla\mathbf{n}\|_T^2\|\nabla w_k\|_2^2 + 48\|D^2 w_k \otimes\mathbf{n}\|_T^2 + 48\|\nabla w_k\otimes\nabla\mathbf{n}\|_T^2 \\
        & = (24d+48)\|D^2w_k\|^2_T + (24d+48)\|\nabla\mathbf{n}\|_T^2\|\nabla w_k\|_2^2,
    \end{split}
\end{align}}where $\delta_{ij}$ is the Kronecker delta and we used the Cauchy-Schwartz inequality for the terms with repeated indices and Hölder's inequality for outer products as well as $\|\mathbf{n}\|_2^2=1$. We combine \eqref{eq:stabilityWeightFunc}, \eqref{eq:OverlapFillDistRelation} and \eqref{eq:referenceFilltoPhysicalFill} to bound the norms involving the weight functions in terms of the reference patch fill distance $h_0$, and we can then express \eqref{eq:numerator_1} as
{\small
\begin{align}
    \begin{split}
        &\|\kappa_1\sigma(\mathbf{v})\cdot\mathbf{n}+\kappa_0\mathbf{v}\|^2_{L_2(\partial\Omega)}\|\nabla_{\Gamma}\left(\kappa_1\sigma\left(\mathbf{v}\right)\cdot\mathbf{n} + \kappa_0\mathbf{v}\right)\|^2_{L_2(\partial\Omega)} \leq 8\chi^2\Biggl(\sum_{k=1}^P\|\kappa_1\sigma(\mathbf{v}^{(k)})\cdot\mathbf{n} + \kappa_0\mathbf{v}^{(k)}\|^2_{L_2(\partial\Omega\cap\Omega_k)} \\ & +\left(8+4d\right)G_1\sum_{\left|\alpha\right|=1}\left(\frac{\tilde{C}_V}{C'_{0,k}}\right)^{-2}\!\!h_0^{-2}\|\kappa_1\mathbf{v}^{(k)}\|^2_{L_2(\partial\Omega\cap\Omega_k)}\Biggr)\Biggl(\sum_{k=1}^P \|\nabla_{\Gamma}(\kappa_1\sigma(\mathbf{v}^{(k)})\cdot\mathbf{n} + \kappa_0\mathbf{v}^{(k)})\|^2_{L_2(\partial\Omega\cap\Omega_k)} 
        \\ & + 4G_1\sum_{|\alpha|=1}\left(\frac{\tilde{C}_V}{C'_{0,k}}\right)^{-2}\!\!h_0^{-2} \|(\kappa_1\sigma(\mathbf{v}^{(k)})\cdot\mathbf{n}+\kappa_0\mathbf{v}^{(k)})\|^2_{L_2(\partial\Omega\cap\Omega_k)} \\ &  
        + \left(8+4d\right) G_1\sum_{|\alpha|=1}\left(\frac{\tilde{C}_V}{C'_{0,k}}\right)^{-2}\!\!h_0^{-2} \|\nabla_{\Gamma}(\kappa_1\mathbf{v}^{(k)})\|^2_{L_2(\partial\Omega\cap\Omega_k)} \\ 
        & + (48 + 24d)\Biggl(2G_2\sum_{|\alpha|=2}\left(\frac{\tilde{C}_V}{\tilde{C}_{0,k}}\right)^{-4}\!\!h_0^{-4}  + G_1\sum_{|\alpha|=1}\left(\frac{\tilde{C}_V}{\tilde{C}_{0,k}}\right)^{-2}\!\!h_0^{-2}\|\nabla\mathbf{n}\|^2_{L_{\infty}(\partial\Omega\cap\Omega_k)}\Biggr)\|\kappa_1\mathbf{v}^{(k)}\|^2_{L_2(\partial\Omega\cap\Omega_k)}\Biggr).
        \label{eq:numerator_2}
    \end{split}
\end{align}}

We further aim to reformulate the local patch-wise norms with respect to the reference patch $\Omega_0$ to construct local eigenvalue bounds similar to \eqref{eq:localEigenBoundMass},~\eqref{eq:localEigenBoundStifness}. We focus on the boundary operator term to show how the transformation to the reference patch and subsequent scaling to a patch where $h_0=1$ is achieved. We use the definition of stress and the local approximant from \eqref{eq:RBFinterpolantLocal2}, which gives
{\small
\begin{align}
    \begin{split}
        \kappa_1\sigma(\mathbf{v}^{(k)})\cdot\mathbf{n} + \kappa_0\mathbf{v}^{(k)} & = \kappa_1\left(\mu\nabla\mathbf{v}^{(k)} + \mu\nabla^T\mathbf{v}^{(k)} + \lambda\nabla\cdot\mathbf{v}^{(k)}I\right)\cdot\mathbf{n} + \kappa_0\mathbf{v}^{(k)} \\
        & = 
        \sum_{i=1}^n\left(\kappa_1\mu(\nabla\psi_i^{(k)}\cdot\mathbf{n})I + \kappa_1\mu(\mathbf{n}(\nabla\psi_i^{(k)})^T)^T + \kappa_1\lambda\mathbf{n}(\nabla\psi_i^{(k)})^T+\kappa_0\psi_i^{(k)}I\right)\mathbf{v}^{(k)}(\mathbf{x}_i).
        \label{eq:denominator_2}
    \end{split}
\end{align}}The linear transformation to the reference patch \eqref{eq:referencePatchMap} is such that $\psi_i^{(k)}(\mathbf{y}) = \tilde{\psi}_i^{(k)}(\mathbf{y}')$, with cardinal basis $\tilde{\psi}_i^{(k)}$, and similarly for the level set function describing the boundary, evaluated locally in patch $\Omega_k$, $\ell(\mathbf{y})=\tilde{\ell}^{(k)}(\mathbf{y}')$. Hence, using the chain rule, we can rewrite the gradient of the basis functions and the level set function using the corresponding functions in the reference patch to get
\begin{equation}
    \nabla_{\mathbf{y}}\psi_i^{(k)} = J_{T_{k}}^T\nabla_{\mathbf{y}'}\tilde{\psi}^{(k)}_i, \hspace{1cm} \mathbf{n} = \frac{\nabla_{\mathbf{y}}\ell}{\|\nabla_{\mathbf{y}}\ell\|_2} =\frac{J_{T_{k}}^T\nabla_{\mathbf{y}'}\tilde{\ell}^{(k)}}{\|J_{T_{k}}^T\nabla_{\mathbf{y}'}\tilde{\ell}^{(k)}\|_2} =\frac{J_{T_{k}}^T \mathbf{n}_{0,k}}{\|J_{T_{k}}^T\mathbf{n}_{0,k}\|_2} = \frac{J_{T_{k}}^T \mathbf{n}_{0,k}}{\|S_k\mathbf{n}_{0,k}\|_2},
\end{equation}
where $\mathbf{n}_{0,k} = \nabla_{\mathbf{y}'}\tilde{\ell}^{(k)}/\|\nabla_{\mathbf{y}'}\tilde{\ell}^{(k)}\|_2$ is the normal of the boundary in patch $\Omega_k$ when transformed to the reference patch, and the Jacobian $J_{T_k} = S_kQ_k$. As in \eqref{eq:derivativePhysical2Ref} we omit the subscript in the gradients and transform~\eqref{eq:denominator_2} to the reference patch, which gives
{\footnotesize
\begin{align}
\begin{split}
    & \kappa_1\sigma(\mathbf{v}^{(k)})\cdot\mathbf{n} + \kappa_0\mathbf{v}^{(k)} = \sum_{i=1}^n\left(\tilde{\kappa}^{(k)}_1\tilde{\mathcal{B}}_1\tilde{\psi}^{(k)}_i + \tilde{\kappa}^{(k)}_0\tilde{\mathcal{B}}_0\tilde{\psi}^{(k)}_i\right)\mathbf{v}^{(k)}\left(\mathbf{x}_i\right)\\ 
    & = \sum_{i=1}^n \left(\frac{\tilde{\kappa}^{(k)}_1}{\|S_k\mathbf{n}_{0,k}\|_2}\left(\mu \nabla^T\tilde{\psi}_i^{(k)} S_k^TS_k\mathbf{n}_{0,k}I+ \mu J^T_{T_k}\nabla\tilde{\psi}_i^{k}\mathbf{n}_{0,k}^T J_{T_k}+\lambda J^T_{T_k}\mathbf{n}_{0,k}\nabla^T\tilde{\psi}_i^{(k)}J_{T_{k}}\right)+\tilde{\kappa}^{(k)}_0\tilde{\psi}_i^{(k)}I\right)\mathbf{v}^{(k)}(\mathbf{x}_i)\\
    & = Q_k^T\sum_{i=1}^n \left(\frac{\tilde{\kappa}^{(k)}_1}{\|S_k\mathbf{n}_{0,k}\|_2}\left(\mu (S_k\nabla\tilde{\psi}_i^{(k)})^T (S_k\mathbf{n}_{0,k})I+ \mu (S_k\nabla\tilde{\psi}_i^{k})(S_k\mathbf{n}_{0,k})^T +(\lambda S_k\mathbf{n}_{0,k})(S_k\nabla\tilde{\psi}_i^{(k)})^T\right)+\tilde{\kappa}^{(k)}_0\tilde{\psi}_i^{(k)}I\right)\underbrace{Q_{k}\mathbf{v}^{(k)}(\mathbf{x}_i)}_{\mathbf{w}(x_i)}\\
    &\triangleq Q_k^T\sum_{i=1}^n \left(\tilde{\kappa}^{(k)}_1\tilde{\mathcal{B}}_1^S\tilde{\psi}^{(k)}_i + \tilde{\kappa}^{(k)}_0\tilde{\mathcal{B}}_0\tilde{\psi}^{(k)}_i\right)\mathbf{w}^{(k)}\left(\mathbf{x}_i\right),
    \end{split}
    \label{eq:bndOpOmegak}
\end{align}}%
where functions and operators marked with a tilde are the transformed counterparts of the corresponding functions in patch $\Omega_k$. For the third equality, we used that $I = Q_k^TQ_k$ and that matrices commute with the scalar coefficients multiplying $I$. Note that for isotropic scaling $S_k=sI$, we get $\tilde{\mathcal{B}}_1^S\tilde{\psi}^{(k)}_i=s\mathcal{B}_1\psi^{(k)}_i$.
To take the boundary $L_2$-norm in patch $\Omega_k$, we first square~\eqref{eq:bndOpOmegak} to get the pointwise norm, and then transform the integration over the boundary to the reference patch. We note that the relation between the area integration elements in $\Omega_k$ and $\Omega_0$ is given by $\mathrm{d}A_k=\det(S_k^{-1})\|S_k\mathbf{n}_{0,k}\|_2\mathrm{d}A_{0}$. Finally, we perform another transformation of the integral to a patch $\widehat{\Omega}_0$ with unit fill distance, such that $\mathrm{d}A_{0}=h_0^{d-1}\mathrm{d}\hat{A}$. We define $\hat{\mathbf{y}}=\mathbf{y}^\prime/h_0$ and $\hat{\psi}(\hat{\mathbf{y}}) = \tilde{\psi}(\mathbf{y^\prime})$. The squared norm then becomes
{\small
\begin{align}
\begin{split}
    & \|\kappa_1\sigma(\mathbf{v}^{(k)})\cdot\mathbf{n} + \kappa_0\mathbf{v}^{(k)}\|^2_{L_2(\partial\Omega\cap \Omega_k)} \\ 
    = & \sum_{i,j=1}^n\mathbf{w}^{(k)}\left(\mathbf{x}_i\right)\int_{\partial\Omega_{0,k}\cap  \Omega_0}\left(\tilde{\kappa}_1^{(k)}\tilde{\mathcal{B}}^S_1\tilde{\psi}_i^{(k)}+\tilde{\kappa}_0^{(k)}\tilde{\mathcal{B}}_0\tilde{\psi}_i^{(k)}\right)^T\left(\tilde{\kappa}_1^{(k)}\tilde{\mathcal{B}}^S_1\tilde{\psi}_j^{(k)}+\tilde{\kappa}_0^{(k)}\tilde{\mathcal{B}}_0\tilde{\psi}_j^{(k)}\right) \det(S_k^{-1})\|S_k\mathbf{n}_{0,k}\|_2\mathrm{d}A_{0} \mathbf{w}^{(k)}\left(\mathbf{x}_j\right), \\
    = & h_0^{d-1}\sum_{i,j=1}^n\mathbf{w}^{(k)}\left(\mathbf{x}_i\right)\int_{\partial\widehat{\Omega}_{0,k}\cap  \widehat{\Omega}_0}\left(\hat{\kappa}_1^{(k)}h_0^{-1}\tilde{\mathcal{B}}^S_1\hat{\psi}_i^{(k)}+\hat{\kappa}_0^{(k)}\tilde{\mathcal{B}}_0\hat{\psi}_i^{(k)}\right)^T\left(\hat{\kappa}_1^{(k)}h_0^{-1}\tilde{\mathcal{B}}^S_1\hat{\psi}_j^{(k)}+\hat{\kappa}_0^{(k)}\tilde{\mathcal{B}}_0\hat{\psi}_j^{(k)}\right) \det(S_k^{-1})\|S_k\mathbf{n}_{0,k}\|_2\mathrm{d}\hat{A}_{0} \mathbf{w}^{(k)}\left(\mathbf{x}_j\right),
\end{split}
\label{eq:refBndNorm}
\end{align}}%
where $\partial\Omega_{0,k}$ is the image of the boundary in patch $\Omega_k$, and where the operators remain the same in the second transformation, while the scaling of derivatives in terms of $h_0$ is explicitly extracted. A similar transformation was also used in~\eqref{eq:derivativeRef2PhysicalL2NormBnd_1}.
%
We extract the leading power of $h_0$ and write the norm relation on matrix form
using the rotated nodal value vector defined by $\mathbf{w}_k={\mathbf{w}}^{(k)}(X_0) = ({w}_1^{(k)}(X_0)^T,\dots,{w}_d^{(k)}(X_0)^T)^T$. The norm then becomes
\begin{align}
\begin{split}
    \|\kappa_1\sigma(\mathbf{v}^{(k)})\cdot\mathbf{n} + \kappa_0\mathbf{v}^{(k)}\|^2_{L_2(\partial\Omega\cap \Omega_k)} & 
     = h_0^{-2}h_0^{d-1}\,\,\mathbf{w}_k^T\mathbf{B}_k\mathbf{w}_k
    \label{eq:denominator_3}
\end{split}
\end{align}
where $\mathbf{B}_{k}$ is a matrix with $d\times d$ blocks of size $n\times n$ and with elements defined by the final row of~\eqref{eq:refBndNorm}. 
Similarly, we can represent the other patchwise norms in \eqref{eq:numerator_2} as
\begin{align}
    \begin{split}
        \|\nabla_{\Gamma}(\kappa_1\sigma(\mathbf{v}^{(k)})\cdot\mathbf{n} + \kappa_0\mathbf{v}^{(k)})\|^2_{L_2(\partial\Omega\cap\Omega_k)} & = h_0^{-4}h_0^{d-1}\mathbf{w}^T {\mathbf{B}}^1_{k}{\mathbf{w}}_k, \\
        \|\nabla_{\Gamma}(\kappa_1\mathbf{v}^{(k)})\|^2_{L_2(\partial\Omega\cap\Omega_k)} & = h_0^{-2}h_0^{d-1}{\mathbf{w}}_k^T\mathbf{K}^1_{k}{\mathbf{w}}_k, \\
        \|\kappa_1\mathbf{v}^{(k)}\|^2_{L_2(\partial\Omega\cap\Omega_k)} & = h_0^{d-1}{\mathbf{w}}_k^T\mathbf{M}_k{\mathbf{w}}_k,
        \label{eq:numerator_3}
    \end{split}
\end{align}
where ${\mathbf{B}}^1_{k}$ and ${\mathbf{K}}^1_{k}$ are full block matrices with $d\times d$ blocks of size $n\times n$, while the mass matrix ${\mathbf{M}}_{k}$ is a block-diagonal matrix as in \eqref{eq:localMassMatrix}.
Substituting \eqref{eq:denominator_3} and \eqref{eq:numerator_3} into \eqref{eq:numerator_2} and consolidating the different constants, while noting that $\|\nabla\mathbf{n}\|^2$ evaluated in the scaled patch gives another $h_0^{-2}$, leads to the global bound
%
%
\begin{align}
\begin{split}
&\|\kappa_1\sigma(\mathbf{v})\cdot\mathbf{n}+\kappa_0\mathbf{v}\|^2_{L_2(\partial\Omega)}\|\nabla_{\Gamma}\left(\kappa_1\sigma\left(\mathbf{v}\right)\cdot\mathbf{n} + \kappa_0\mathbf{v}\right)\|^2_{L_2(\partial\Omega)} 
\\ & \leq C_{n}C'_{n} h_0^{-2}h_0^{-4}h_0^{2(d-1)}\Biggl(\sum_{k=1}^P{\mathbf{w}}_k^T\biggl(\mathbf{B}_k + {\mathbf{M}}_k\biggr){\mathbf{w}}_k\Biggr) 
\Biggl(\sum_{k=1}^P {\mathbf{w}}_k^T\biggl( \mathbf{B}_{k} + {\mathbf{M}}_k +  {\mathbf{B}}^1_k + {\mathbf{K}}^1_{k}\biggr){\mathbf{w}}_k \Biggr)
\label{eq:inverseBnd_2}
\end{split}
\end{align}
%
The scaled quadratic forms will be bounded on each patch by a specifically chosen combination of norms, where the terms are balanced with respect to the fill distance. We first define the region $K_k$ as the non-overlapped part of $\Omega_k$. Then we scale the Dirichlet part of the boundary term with the inverse fill distance $h^{-1}$, which is related to $h_0^{-1}$ by a constant, see~\eqref{eq:referenceFilltoPhysicalFill}. This is especially relevant for patches where the traction term disappears. Since the basis functions can be linearly dependent on a manifold $\mathcal{M}\subset \mathbb{R}^d$, with zero measure in $\mathbb{R}^d$, the modified boundary term could (with low probability) be identically zero on $(\partial\Omega\cap K_k)\neq\varnothing$, while being non-zero on $\partial\Omega\cap\Omega_k$. Therefore, we also add a domain term $\|\nabla\cdot\sigma(\mathbf{v}^{(k)})\|^2_{L_2(\Omega\cap K_k)}$, which is strictly positive for $\mathbf{w}_k\neq\mathbf{0}$. 
Regarding the positivity we refer to the paragraph following \eqref{eq:scaled_K} where we discuss linear independence and positivity of norms over domains of positive measure. Note that the basis functions $\mathcal{L}\tilde{\psi}^{(k)}_j$ are linearly independent since the PDE operator defined in \eqref{eq:Elasticity_Operators} is elliptic (see Appendix \ref{app:CoercivityBounds}). Using the same transformation steps as above for the bounding term that will be used, we get
\begin{equation}
    \|\kappa_1\sigma(\mathbf{v}^{(k)})\cdot\mathbf{n} + h^{-1}\kappa_0\mathbf{v}^{(k)}\|^2_{L_2(\partial\Omega\cap K_k)} + h_0\|\nabla\cdot\sigma(\mathbf{v}^{(k)})\|^2_{L_2(\Omega\cap K_k)}
    = h_0^{-2}h_0^{d-1}\mathbf{w}_k^T\mathbf{B}_k^K\mathbf{w}_k+h_0h_0^{-4}h_0^d\mathbf{w}_k^T\mathbf{L}_k^K\mathbf{w}_k.
\end{equation}
The local bounds are defined through the following generalized eigenproblems:
\begin{align}
    \begin{split}
        \lambda_k^\mathcal{B} & = \max_{\|\mathbf{w}_k\|=1}\frac{\mathbf{w}_k^T\left(\mathbf{B}_k+\mathbf{M}_k\right)\mathbf{w}_k}{\mathbf{w}_k^T\left(\mathbf{B}_k^K+\mathbf{L}_k^K\right)\mathbf{w}_k},\\
        \lambda_k^{\mathcal{B}^1} & = \max_{\|\mathbf{w}_k\|=1}\frac{\mathbf{w}_k^T\left(\mathbf{B}_k+\mathbf{M}_k+ \mathbf{B}_k^1+\mathbf{M}_k^1\right)\mathbf{w}_k}{\mathbf{w}_k^T\left(\mathbf{B}_k^K+\mathbf{L}_k^K\right)\mathbf{w}_k},
    \end{split}
    \label{eq:geneigs}
\end{align}%
where the largest eigenvalues depend on the shape and size of $\partial\Omega\cap\Omega_k$, the local basis functions, and the scale transformation $S_k$. The leading order term is independent of $h_0$, but in the numerator there are additional terms with positive powers of $h_0$ due to the unbalanced scaling of the two parts of the boundary term, see~\eqref{eq:refBndNorm}.
%
Using the eigenvalue bounds together with \eqref{eq:denominator_3} and \eqref{eq:numerator_3} yields
{
\begin{align}
    \begin{split}
& \frac{\|\kappa_1\sigma(\mathbf{v}^{(k)})\cdot\mathbf{n} + \kappa_0\mathbf{v}^{(k)}\|^2_{L_2(\partial\Omega\cap\Omega_k)} + h_0^{-2}\|\kappa_1\mathbf{v}^{(k)}\|^2_{L_2(\partial\Omega\cap\Omega_k)}}{\|\kappa_1\sigma(\mathbf{v}^{(k)})\cdot\mathbf{n} + h^{-1}\kappa_0\mathbf{v}^{(k)}\|^2_{L_2(\partial\Omega\cap K_k)} + h_0\|\nabla\cdot\sigma(\mathbf{v}^{(k)})\|^2_{L_2(\Omega\cap K_k)}} 
            \leq 
           \lambda_k^{\mathcal{B}}, \\[1ex]
         & \frac{\|\nabla_{\Gamma}(\kappa_1\sigma(\mathbf{v}^{(k)})\cdot\mathbf{n} + \kappa_0\mathbf{v}^{(k)})\|^2_{L_2(\partial\Omega\cap\Omega_k)} + h_0^{-2}\|\kappa_1\sigma(\mathbf{v}^{(k)})\cdot\mathbf{n} + \kappa_0\mathbf{v}^{(k)}\|^2_{L_2(\partial\Omega\cap\Omega_k)}}{\|\kappa_1\sigma(\mathbf{v}^{(k)})\cdot\mathbf{n} + h^{-1}\kappa_0\mathbf{v}^{(k)}\|^2_{L_2(\partial\Omega\cap K_k)} + h_0\|\nabla\cdot\sigma(\mathbf{v}^{(k)})\|^2_{L_2(\Omega\cap K_k)}} \\
        & + \frac{h_0^{-2}\|\nabla_{\Gamma}(\kappa_1\mathbf{v}^{(k)})\|^2_{L_2(\partial\Omega\cap\Omega_k)} + h_0^{-4}\|\kappa_1\mathbf{v}^{(k)}\|^2_{L_2(\partial\Omega\cap\Omega_k)}}{\|\kappa_1\sigma(\mathbf{v}^{(k)})\cdot\mathbf{n} + h^{-1}\kappa_0\mathbf{v}^{(k)}\|^2_{L_2(\partial\Omega\cap K_k)} + h_0\|\nabla\cdot\sigma(\mathbf{v}^{(k)})\|^2_{L_2(\Omega\cap K_k)}} 
         \leq  
         h_0^{-2}\lambda_k^{\mathcal{B}^1}.
       \label{eq:generalisedEigenvalueBnd}
    \end{split}
\end{align}}

Finally, using the local bounds in the global bound \eqref{eq:inverseBnd_2} gives
\begin{align}
\begin{split}
&\|\kappa_1\sigma(\mathbf{v})\cdot\mathbf{n}+\kappa_0\mathbf{v}\|^2_{L_2(\partial\Omega)}\|\nabla_{\Gamma}\left(\kappa_1\sigma\left(\mathbf{v}\right)\cdot\mathbf{n} + \kappa_0\mathbf{v}\right)\|^2_{L_2(\partial\Omega)} \\ 
     &\leq C_{n} C'_n h_0^{-2} \left(\sum_{k=1}^P\lambda_k^{\mathcal{B}}\left(\|\kappa_1\sigma(\mathbf{v}^{(k)})\cdot\mathbf{n} + h^{-1}\kappa_0\mathbf{v}^{(k)}\|^2_{L_2(\partial\Omega\cap K_k)}  + h_0\|\nabla\cdot\sigma(\mathbf{v}^{(k)})\|^2_{L_2(\Omega\cap K_k)}\right)\right) \\ & \mathbin{\hphantom{\leq C_{n} C'_n h_0^{-2}}}\left(\sum_{k=1}^P\lambda_k^{\mathcal{B}^1}\left(\|\kappa_1\sigma(\mathbf{v}^{(k)})\cdot\mathbf{n} + h^{-1}\kappa_0\mathbf{v}^{(k)}\|^2_{L_2(\partial\Omega\cap K_k)} + h_0\|\nabla\cdot\sigma(\mathbf{v}^{(k)})\|^2_{L_2(\Omega\cap K_k)}\right)\right) \\
& \leq C_nC'_n  \lambda_{m}h_0^{-2}\left(\|\kappa_1\sigma(\mathbf{v})\cdot\mathbf{n} + h^{-1}\kappa_0\mathbf{v}\|^2_{L_2(\partial\Omega)} + h_0\|\nabla\cdot\sigma(\mathbf{v})\|^2_{L_2(\Omega)}\right)^2 \\
& \leq 2C_nC'_n\lambda_m h_0^{-2}\|\kappa_1\sigma(\mathbf{v})\cdot\mathbf{n} + h^{-1}\kappa_0\mathbf{v}\|^4_{L_2(\partial\Omega)} + 2C_nC'_n\lambda_m\|\nabla\cdot\sigma(\mathbf{v})\|^4_{L_2(\Omega)},
\label{eq:inverseBnd_3}
\end{split}
\end{align}
where $\lambda_m = \max_{1\leq k\leq P}\lambda_k^{\mathcal{B}^1}\lambda_k^{\mathcal{B}}$.
Similarly, for the second term in~\eqref{eq:inverseEstimateGradient}, we get the bound
\begin{eqnarray}
    \|\kappa_1\sigma(\mathbf{v})\cdot\mathbf{n}+\kappa_0\mathbf{v}\|^4_{L_2(\partial\Omega)}&\leq&
    h_0^2\left(C_n\lambda^{\mathcal{B}}h_0^{-1}\|\kappa_1\sigma(\mathbf{v})\cdot\mathbf{n} + h^{-1}\kappa_0\mathbf{v}\|^2_{L_2(\partial\Omega)} + \|\nabla\cdot\sigma(\mathbf{v})\|^2_{L_2(\Omega)}\right)^2\nonumber\\
    &\leq& 2C_n^2(\lambda^\mathcal{B})^2h_0^2\left(h_0^{-2}\|\kappa_1\sigma(\mathbf{v})\cdot\mathbf{n} + h^{-1}\kappa_0\mathbf{v}\|^2 + \|\nabla\cdot\sigma(\mathbf{v})\|^4_{L_2(\Omega)}\right),
    \label{eq:invterm2}
\end{eqnarray}
where $\lambda^\mathcal{B}=\max_{1\leq k\leq P}\lambda_k^{\mathcal{B}}$.
To arrive at the final inverse estimate we use \eqref{eq:inverseBnd_3} and~\eqref{eq:invterm2} in \eqref{eq:inverseEstimateGradient}, which gives
\begin{eqnarray}
        \|\kappa_1\sigma\left(\mathbf{v}\right)\cdot\mathbf{n} + \kappa_0\mathbf{v}\|^4_{H^{\frac{1}{2}}(\partial\Omega)} &\leq &  2C^2_iC_n\left(C'_n\lambda_m h_0^{-2}+C_n(\lambda^\mathcal{B})^2\right)\|\kappa_1\sigma(\mathbf{v})\cdot\mathbf{n} + h^{-1}\kappa_0\mathbf{v}\|^4_{L_2(\partial\Omega)} \nonumber\\ & + & 2C^2_iC_n(C_n'\lambda_m+C_n(\lambda^\mathcal{B})^2h_0^2)\|\nabla\cdot\sigma(\mathbf{v})\|^4_{L_2(\Omega)}\\ 
        &\leq & C_{\mathrm{inv}}^2h^{-2}\|\kappa_1\sigma(\mathbf{v})\cdot\mathbf{n} + h^{-1}\kappa_0\mathbf{v}\|^4_{L_2(\partial\Omega)} + C_{\mathrm{inv}}^2\|\nabla\cdot\sigma(\mathbf{v})\|^4_{L_2(\Omega)}, \nonumber
        \label{eq:inverseBnd_4}
\end{eqnarray}
where we used the same argumentation as in \eqref{eq:integrationError} to bound the reference patch fill distance $h_0$ by the global fill distance $h$ given in \eqref{eq:fillX}. Taking the square root of the above gives \eqref{eq:boundaryInverseEstimate}.

The generalised eigenvalues~\eqref{eq:geneigs} are bounded, since the numerator is bounded and the denominator is positive and bounded. They depend continuously on the problem parameters, and we specifically note that the node layout in the reference patch, apart from scaling, is fixed under patch refinement, which excludes adverse effects due to non-unisolvent or near non-unisolvent node layouts.  Computing the generalised eigenvalues \eqref{eq:geneigs} accurately is computationally expensive. Therefore, we only provide a few examples here. Similarly to Figure~\ref{fig:Eig}, we compute the largest eigenvalues for similar but simplified versions of the bounds. The examples include the scalar valued function $v = \sum_{j=1}^n\tilde{\psi}_jv(\mathbf{x}_j)$ and some of its derivatives integrated over the boundary in a reference patch, $\Omega_0$. Moreover, we assume the boundary itself is parallel to the faces of the cylindrical patch and that the physical patch coordinate system, height and radius are exactly equal to the reference patch to simplify computations. The eigenfunctions corresponding to the maximum eigenvalues for various problem formulations are shown in Figure~\ref{fig:Bnd}. The results may be unbounded without the added domain term in \eqref{eq:generalisedEigenvalueBnd}, but for the examples here it is not needed and therefore omitted from the computation. The subplots are representative of the types of terms arising in the cases where $\kappa_0\approx 0$ or $\kappa_1\approx 0$. 
\begin{figure}[!htb]
    \centering
    \includegraphics[width=0.24\linewidth]{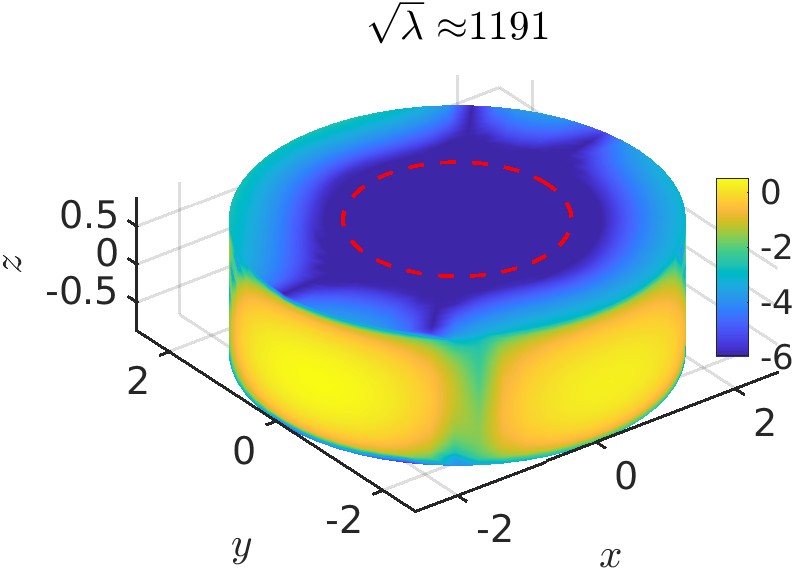}\hspace{0.1cm}
\includegraphics[width=0.24\linewidth]{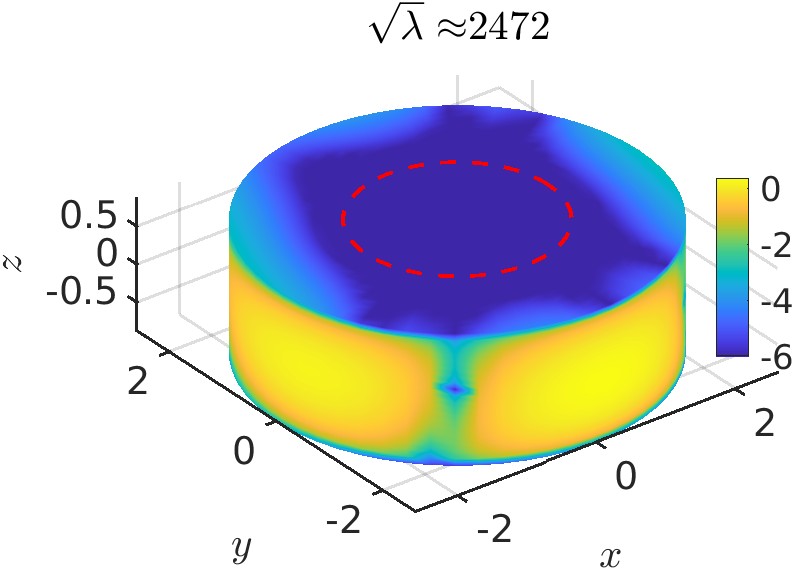}\hspace{0.1cm}
\includegraphics[width=0.24\linewidth]{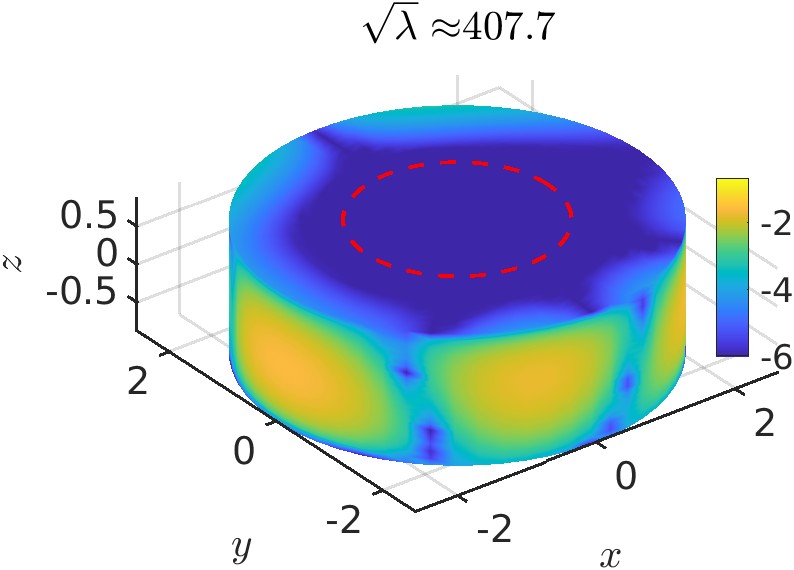}\hspace{0.1cm}
\includegraphics[width=0.24\linewidth]{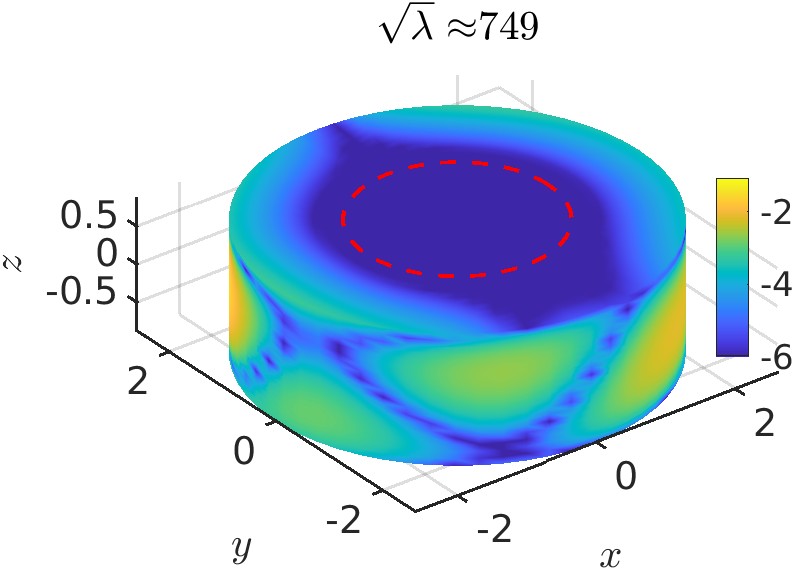}
    \caption{The boundary is denoted by two horizontal planes that cut the patch and which we denote by $\partial\Omega^p$, and their interior part by $\partial\Omega^p_I$ (red dashed curve). Only the slice of the patch that is in between the two boundary planes is shown. We consider eigenvalue problems of the form $\lambda = Q(\mathcal{L}_1v,\mathcal{L}_2v)=\max_v\|\mathcal{L}_1v\|^2_{L_2(\partial\Omega^p)}/\|\mathcal{L}_2v\|^2_{L_2(\partial\Omega_I^p)}$. The subplots show $\log_{10}|\mathcal{L}_2v|$ for the following eigenproblems: $\lambda = Q(v,v)$, 
    $\lambda = Q(v_x,v)$,
    $\lambda = Q(v_x,v_x)$, 
    and $\lambda = Q(v_{xx},v_x)$. The number of local points $n=56$, the planes that cut the patch are located at $\pm H/3$, the aspect ratio $H/R = 1$ and the shape parameter $\varepsilon=0.5$.}
    \label{fig:Bnd}
\end{figure}

\bibliographystyle{siamplain}
\bibliography{refs}

@misc{invive_project,
  author = {Elisabeth Larsson},
  title = {The {INVIVE} project},
  howpublished = {Research project web page: \url{https://www.it.uu.se/research/scientific_computing/project/rbf/biomech.html}},
  year = 	 {2026},
}

@article{llano2012mechanisms,
  title={Mechanisms underlying {ICU} muscle wasting and effects of passive mechanical loading},
  author={Llano-Diez, Monica and Renaud, Guillaume and Andersson, Magnus and Marrero, Humberto Gonzales and Cacciani, Nicola and Engquist, Henrik and Corpe{\~n}o, Rebeca and Artemenko, Konstantin and Bergquist, Jonas and Larsson, Lars},
  journal={Crit. care},
  volume={16},
  pages={1--16},
  year={2012},
  publisher={Springer}
}

@article{stableRBFFD2021,
author = {Tominec, Igor and Larsson, Elisabeth and Heryudono, Alfa},
title = {A Least Squares Radial Basis Function Finite Difference Method with Improved Stability Properties},
journal = {SIAM J. Sci. Comp.},
volume = {43},
number = {2},
pages = {A1441--A1471},
year = {2021},
doi = {10.1137/20M1320079}
}

@article{diaphragmRBFFD2022,
title = {An unfitted radial basis function generated finite difference method applied to thoracic diaphragm simulations},
journal = {J. Comput. Phys.},
volume = {469},
eid = {111496},
year = {2022},
issn = {0021-9991},
doi = {10.1016/j.jcp.2022.111496},
author = {Igor Tominec and Pierre-Frédéric Villard and Elisabeth Larsson and Víctor Bayona and Nicola Cacciani}
}

@misc{larsson2024rbfpartitionunitymethod,
    title={An {RBF} partition of unity method for geometry reconstruction and {PDE} solution in thin structures},
    author={Larsson, Elisabeth and Villard, Pierre-Fr{\'e}d{\'e}ric and Tominec, Igor and Sundin, Ulrika and Michael, Andreas and Cacciani, Nicola},
    year={2024},
    eprint={2403.01486},
    archivePrefix={arXiv},
    primaryClass={math.NA}
}

@article{larsson2022numerical,
  title={A numerical investigation of some {RBF-FD} error estimates},
  author={Larsson, Elisabeth and Mavri{\v{c}}, Bo{\v{s}}tjan and Michael, Andreas and Pooladi, Fatemeh},
  fjournal={Dolomites Research Notes on Approximation},
  journal = {Dolomites Res. Notes Approx.},
  volume={15},
  number={5},
  pages={78--95},
  year={2022},
  doi = {10.14658/PUPJ-DRNA-2022-5-8},
  publisher={Padova University Press}
}

@article{LarsonShcherbakovHeryudono2017,
author = {Larsson, Elisabeth and Shcherbakov, Victor and Heryudono, Alfa},
title = {A Least Squares Radial Basis Function Partition of Unity Method for Solving {PDE}s},
journal = {SIAM J. Sci. Comp.},
volume = {39},
number = {6},
pages = {A2538--A2563},
year = {2017},
doi = {10.1137/17M1118087}
}

@book{TheoryOfElasticity1951,
language = {eng},
publisher = {McGraw-Hill},
title = {Theory of elasticity },
year = {1951},
author = {Timoshenko, Stephen and Goodier, James Norman},
address = {New York},
booktitle = {Theory of elasticity},
edition = {2. rev.},
}

@book {brenner2008mathematical,
    AUTHOR = {Brenner, Susanne C. and Scott, L. Ridgway},
     TITLE = {The mathematical theory of finite element methods},
    SERIES = {Texts in Applied Mathematics},
    VOLUME = {15},
   EDITION = {Third},
 PUBLISHER = {Springer, New York},
      YEAR = {2008},
     PAGES = {xviii+397},
      ISBN = {978-0-387-75933-3},
   MRCLASS = {65-01 (65-02)},
  MRNUMBER = {2373954},
       DOI = {10.1007/978-0-387-75934-0},
}

@article{ADN_2,
author = {Agmon, S. and Douglis, A. and Nirenberg, L.},
title = {Estimates near the boundary for solutions of elliptic partial differential equations satisfying general boundary conditions {II}},
journal = {Communications on Pure and Applied Mathematics},
volume = {17},
number = {1},
pages = {35--92},
doi = {10.1002/cpa.3160170104},
year = {1964}
}

@article{FractionalSobolev2012,
title = {Hitchhiker's guide to the fractional Sobolev spaces},
fjournal = {Bulletin des Sciences Mathématiques},
journal = {Bull. Sci. Math.},
volume = {136},
number = {5},
pages = {521--573},
year = {2012},
issn = {0007-4497},
doi = {10.1016/j.bulsci.2011.12.004},
author = {Eleonora {Di Nezza} and Giampiero Palatucci and Enrico Valdinoci}
}

@article{RBFQR,
author = {Fornberg, Bengt and Larsson, Elisabeth and Flyer, Natasha},
title = {Stable Computations with {G}aussian Radial Basis Functions},
journal = {SIAM J. Sci. Comp.},
volume = {33},
number = {2},
pages = {869--892},
year = {2011},
doi = {10.1137/09076756X}
}

@article{Wendland1995PiecewisePP,
  title={Piecewise polynomial, positive definite and compactly supported radial functions of minimal degree},
  author={Holger Wendland},
  journal={Adv. Comput. Math.},
  year={1995},
  volume={4},
  pages={389--396},
  DOI = {10.1007/BF02123482},
}

@book {necas2011direct,
    AUTHOR = {Ne\v{c}as, Jind\v{r}ich},
     TITLE = {Direct methods in the theory of elliptic equations},
    SERIES = {Springer Monographs in Mathematics},
 PUBLISHER = {Springer, Heidelberg},
      YEAR = {2012},
     PAGES = {xvi+372},
      ISBN = {978-3-642-10454-1; 978-3-642-10455-8},
   MRCLASS = {35-02 (35J08 35J20 35J25 35J50 35J57)},
  MRNUMBER = {3014461},
MRREVIEWER = {Norbert\ Ortner},
       DOI = {10.1007/978-3-642-10455-8},
}

@book {bochev2009least,
    AUTHOR = {Bochev, Pavel B. and Gunzburger, Max D.},
     TITLE = {Least-squares finite element methods},
    SERIES = {Applied Mathematical Sciences},
    VOLUME = {166},
 PUBLISHER = {Springer, New York},
      YEAR = {2009},
     PAGES = {xxii+660},
      ISBN = {978-0-387-30888-3},
   MRCLASS = {65-02 (35A35 65M60 65N30 74S05 76M10)},
  MRNUMBER = {2490235},
MRREVIEWER = {Tsu-Fen\ Chen},
       DOI = {10.1007/b13382},
}

@article {JaRoĭtberg_1969,
    AUTHOR = {Ro\u{i}tberg, Ja.\ A. and \v{S}eftel$'$, Z. G.},
     TITLE = {A homeomorphism theorem for elliptic systems, and its
              applications},
   JOURNAL = {Mat. Sb. (N.S.)},
  FJOURNAL = {Matematicheski\u i\ Sbornik. Novaya Seriya},
    VOLUME = {78(120)},
      YEAR = {1969},
     PAGES = {446--472},
       doi = {10.1070/SM1969v007n03ABEH001099},
      ISSN = {0368-8666},
   MRCLASS = {35.44},
  MRNUMBER = {247270},
MRREVIEWER = {L.\ Arkeryd},
}

@book{knops1971uniqueness,
  title={Uniqueness Theorems in Linear Elasticity},
  author={Knops, R.J. and Payne, L.E.},
  isbn={9783540052531},
  lccn={70138813},
  series={Springer tracts in natural philosophy},
  volume = 19,
  year={1971},
  publisher={Springer-Verlag},
  address = {New York-Berlin}
}

@book {Holzapfel,
    AUTHOR = {Holzapfel, Gerhard A.},
     TITLE = {Nonlinear solid mechanics: A continuum approach for engineering},
 PUBLISHER = {John Wiley \& Sons, Ltd., Chichester},
      YEAR = {2000},
     PAGES = {xiv+455},
      ISBN = {0-471-82304-X},
   MRCLASS = {74-01 (74B20)},
  MRNUMBER = {1827472},
MRREVIEWER = {Stan\ Chiri\c t\u a},
}

@article{rieger2010sampling,
  title={Sampling inequalities for infinitely smooth functions, with applications to interpolation and machine learning},
  author={Rieger, Christian and Zwicknagl, Barbara},
  journal={Adv. Comput. Math.},
  volume={32},
  pages={103--129},
  year={2010},
  publisher={Springer}
}

@article{Larsson_2024,
doi = {10.1088/1742-6596/2766/1/012158},
year = {2024},
publisher = {IOP Publishing},
volume = {2766},
number = {1},
eid = {012158},
author = {Larsson, Elisabeth and Mavri\v{c}, Boštjan and Michael, Andreas and Pooladi, Fatemeh and Tominec, Igor},
title = {Meshfree {RBF–FD} methods for numerical simulation of {PDE} problems},
journal = {J. Phys.: Conf. Ser.}
}

@book {Wendland_2004,
    AUTHOR = {Wendland, Holger},
     TITLE = {Scattered data approximation},
    SERIES = {Cambridge Monographs on Applied and Computational Mathematics},
    VOLUME = {17},
 PUBLISHER = {Cambridge University Press, Cambridge},
      YEAR = {2005},
     PAGES = {x+336},
      ISBN = {978-0521-84335-5; 0-521-84335-9},
   MRCLASS = {41-02 (41A10 41A63 65D10)},
  MRNUMBER = {2131724},
}

@inproceedings{Shepard1968ATI,
 author = {Shepard, Donald},
 title = {A Two-dimensional Interpolation Function for Irregularly-spaced Data},
 booktitle = {Proceedings of the 1968 23rd ACM National Conference},
 series = {ACM '68},
 year = {1968},
 pages = {517--524},
 numpages = {8},
 doi = {10.1145/800186.810616},
 acmid = {810616},
 publisher = {ACM},
 address = {New York, NY, USA},
}

@article{Schoenberg1938,
 ISSN = {0003486X, 19398980},
 author = {I. J. Schoenberg},
 journal = {Annals of Mathematics},
 number = {4},
 pages = {811--841},
 publisher = {[Annals of Mathematics, Trustees of Princeton University on Behalf of the Annals of Mathematics, Mathematics Department, Princeton University]},
 title = {Metric Spaces and Completely Monotone Functions},
 volume = {39},
 year = {1938},
 DOI = {10.2307/1968466}
}

@article{wendland2005approximate,
  title={Approximate interpolation with applications to selecting smoothing parameters},
  author={Wendland, Holger and Rieger, Christian},
  journal={Numer. Math.},
  volume={101},
  pages={729--748},
  year={2005},
  publisher={Springer}
}

@article {Hardy2006,
    AUTHOR = {Hardy, Michael},
     TITLE = {Combinatorics of partial derivatives},
   JOURNAL = {Electron. J. Combin.},
  FJOURNAL = {Electronic Journal of Combinatorics},
    VOLUME = {13},
      YEAR = {2006},
    NUMBER = {1},
 PAGETOTAL = 13,
       EID = {Research Paper R1},
      ISSN = {1077-8926},
   MRCLASS = {26B05 (05A18)},
  MRNUMBER = {2200529},
MRREVIEWER = {Arnold\ Knopfmacher},
       DOI = {10.37236/1027},
}

@book{Comtet1974,
author = {Comtet, Louis},
address = {Dordrecht, Holland},
booktitle = {Advanced combinatorics: the art of finite and infinite expansions},
edition = {Rev. and enl.},
isbn = {9027703809},
language = {eng},
publisher = {D.~Reidel Publishing Company},
title = {Advanced Combinatorics : The Art of Finite and Infinite Expansions },
year = {1974},
}

@article{wilson1987,
  title={Geometry and respiratory displacement of human ribs},
  author={Wilson, TA and Rehder, K and Krayer, S and Hoffman, EA and Whitney, CG and Rodarte, JR},
  journal={J. Appl. Physiol.},
  volume={62},
  number={5},
  pages={1872--1877},
  year={1987}
}

@incollection {villard2011,
    AUTHOR = {Villard, Pierre-Fr{\'e}d{\'e}ric and Boshier, Piers and Bello, Fernando and Gould, Derek},
     TITLE = {Virtual reality simulation of liver biopsy with a respiratory component},
 BOOKTITLE = {Liver Biopsy},
    EDITOR = {Hirokazu Takahashi},
 PUBLISHER = {InTech, London, UK},
      YEAR = {2011},
      doi = {10.5772/22033}
}

@article{Mirtk,
  author={Rueckert, D. and Sonoda, L.I. and Hayes, C. and Hill, D.L.G. and Leach, M.O. and Hawkes, D.J.},
  journal={IEEE Trans. Med. Imag.}, 
  title={Nonrigid registration using free-form deformations: application to breast {MR} images}, 
  year={1999},
  volume={18},
  number={8},
  pages={712--721},
  doi={10.1109/42.796284}}

@article{Tominec2025,
title = {Stability estimates for radial basis function methods applied to linear scalar conservation laws},
journal = {Appl. Math. Comput.},
volume = {485},
eid = {129020},
year = {2025},
issn = {0096-3003},
doi = {10.1016/j.amc.2024.129020},
author = {Igor Tominec and Murtazo Nazarov and Elisabeth Larsson}
}

@article{tensors2009,
author = {Kolda, Tamara G. and Bader, Brett W.},
title = {Tensor Decompositions and Applications},
journal = {SIAM Rev.},
volume = {51},
number = {3},
pages = {455--500},
year = {2009},
doi = {10.1137/07070111X}
}

@book{lai2009continuumMech,
  title={Introduction to continuum mechanics},
  author={Lai, W Michael and Rubin, David and Krempl, Erhard},
  year={1993},
  edition={Third},
  publisher={Butterworth-Heinemann},
  address = {Boston},
  doi = {10.1016/B978-0-7506-8560-3.X0001-1}
}

@book{mclean2000strongly,
  title={Strongly elliptic systems and boundary integral equations},
  author={McLean, William Charles Hector},
  year={2000},
  publisher={Cambridge University Press, Cambridge}
}

@article{SurfPDEs_Dziuk_Elliott_2013, 
    title={Finite element methods for surface {PDE}s}, 
    volume={22}, 
    DOI={10.1017/S0962492913000056}, 
    journal={Acta Numer.}, 
    author={Dziuk, Gerhard and Elliott, Charles M.}, 
    year={2013}, 
    pages={289--396}
}

@book {adams2003sobolev,
    AUTHOR = {Adams, Robert A. and Fournier, John J. F.},
     TITLE = {Sobolev Spaces},
    SERIES = {Pure and Applied Mathematics},
    VOLUME = {140},
   EDITION = {Second},
 PUBLISHER = {Elsevier/Academic Press, Amsterdam},
      YEAR = {2003},
     PAGES = {xiv+305},
      ISBN = {0-12-044143-8},
   MRCLASS = {46E35 (46-01 46-02 46B70 46Exx)},
  MRNUMBER = {2424078},
}

@article{schaback1996approximation,
  title={Approximation by radial basis functions with finitely many centers},
  author={Schaback, Robert},
  journal={Constr. Approx.},
  volume={12},
  number={3},
  pages={331--340},
  year={1996},
  publisher={Springer}
}

@article{SOMMARIVA2026116983,
title = {Unisolvence of unsymmetric random {K}ansa collocation by {G}aussians and other analytic {RBF} vanishing at infinity},
journal = {J. Comput. Appl. Math.},
volume = {475},
eid = {116983},
year = {2026},
issn = {0377-0427},
doi = {10.1016/j.cam.2025.116983},
author = {Alvise Sommariva and Marco Vianello}
}

@article{kuchment2016overview,
  title={An overview of periodic elliptic operators},
  author={Kuchment, Peter},
  fjournal={Bulletin of the American Mathematical Society},
  journal={Bull. Amer. Math. Soc.},
  volume={53},
  number={3},
  pages={343--414},
  year={2016},
  DOI = {10.1090/bull/1528}
}

@incollection{THOMPSON2003717,
title = {Convex Sets},
editor = {Robert A. Meyers},
booktitle = {Encyclopedia of Physical Science and Technology (Third Edition)},
publisher = {Academic Press},
edition = {Third Edition},
address = {New York},
pages = {717--737},
year = {2003},
isbn = {978-0-12-227410-7},
doi = {10.1016/B0-12-227410-5/00146-0},
author = {A. C. Thompson}
}

@article{edelsbrunner1994triangulating,
    AUTHOR = {Edelsbrunner, Herbert and Shah, Nimish R.},
     TITLE = {Triangulating topological spaces},
      NOTE = {Tenth Annual ACM Symposium on Computational Geometry (Stony Brook, NY, 1994)},
   JOURNAL = {Internat. J. Comput. Geom. Appl.},
  FJOURNAL = {International Journal of Computational Geometry \&
              Applications},
    VOLUME = {7},
      YEAR = {1997},
    NUMBER = {4},
     PAGES = {365--378},
      ISSN = {0218-1959,1793-6357},
   MRCLASS = {57Q05 (57Q15 68U05)},
  MRNUMBER = {1460843},
MRREVIEWER = {Nigel\ Martin},
       DOI = {10.1142/S0218195997000223},
}

@article{Yan2009,
  TITLE = {Isotropic Remeshing with Fast and Exact Computation of Restricted {V}oronoi Diagram},
  AUTHOR = {Yan, Dong-Ming and L{\'e}vy, Bruno and Liu, Yang and Sun, Feng and Wang, Wenping},
  JOURNAL = {Computer Graphics Forum},
  PUBLISHER = {Wiley},
  VOLUME = {28},
  NUMBER = {5},
  PAGES = {1445--1454},
  YEAR = {2009},
  DOI = {10.1111/j.1467-8659.2009.01521.x},
  HAL_ID = {inria-00547790},
  HAL_VERSION = {v1},
}

@article{WrightFuselier2012,
author = {Fuselier, Edward and Wright, Grady B.},
title = {Scattered Data Interpolation on Embedded Submanifolds with Restricted Positive Definite Kernels: Sobolev Error Estimates},
fjournal = {SIAM Journal on Numerical Analysis},
journal = {SIAM J. Numer. Anal.},
volume = {50},
number = {3},
pages = {1753--1776},
year = {2012},
doi = {10.1137/110821846}
}

@article{halton1960efficiency,
  title={On the efficiency of certain quasi-random sequences of points in evaluating multi-dimensional integrals},
  author={Halton, John H},
  journal={Numer. Math.},
  volume={2},
  number={1},
  pages={84--90},
  year={1960},
  publisher={Springer},
  doi = {10.1007/BF01386213}
}

@article{Larsson_Heryudono_Michael_Piret_2026, title={A {M}atlab code for stable differentiation of radial basis function approximations}, volume={3}, 
 DOI={10.13135/3103-1935/12293}, abstractNote={&amp;lt;p&amp;gt;Radial basis function approximation allows for fitting of scattered data and for solving partial differential equations in non-trivial geometries. This requires solving potentially ill-conditioned linear systems of equations. The RBF-QR class of methods provides a change of basis that results in well-conditioned matrices in the same approximation space. In this paper, we derive improved algorithms for differentiation in the RBF-QR basis that solves some accuracy issues of the original implementation. We provide all second derivatives in three spatial dimensions, while previously only the Laplacian was implemented. We also provide a high level interface to compute differentiation matrices using either RBF-QR or the direct evaluation method.&amp;lt;/p&amp;gt;}, 
 number={1}, 
 fjournal={Journal of Approximation Software}, 
 journal={J. Approx. Softw.}, 
 author={Larsson, Elisabeth and Heryudono, Alfa and Michael, Andreas and Piret, Cécile}, 
 year={2026}}

@article{Gaur2016,
title = {Characterisation of human diaphragm at high strain rate loading},
fjournal = {Journal of the Mechanical Behavior of Biomedical Materials},
journal = {J. Mech. Behav. Biomed. Mater.},
volume = {60},
pages = {603--616},
year = {2016},
issn = {1751-6161},
doi = {10.1016/j.jmbbm.2016.02.031},
author = {Piyush Gaur and Anoop Chawla and Khyati Verma and Sudipto Mukherjee and Sanjeev Lalvani and Rajesh Malhotra and Christian Mayer}
}

@article{Gaur2019,
author = {Gaur, Piyush and Verma, Khyati and Chawla, A. and Mukherjee, Sudipto and Jain, Mohit and Mayer, Christian and Chitteti, Ravi and Ghosh, Pronoy and Malhotra, Rajesh and Lalwani, Sanjeev},
year = {2019},
pages = {284--298},
title = {A bilinear structural constitutive model for strain rate-dependent behaviour of human diaphragm tissue},
volume = {25},
fjournal = {International Journal of Crashworthiness},
journal = {Int. J. Crashworthiness},
doi = {10.1080/13588265.2019.1583423}
}

@article{DavisSuiteSparse,
author = {Davis, Timothy A.},
title = {Algorithm 915, {SuiteSparseQR}: Multifrontal multithreaded rank-revealing sparse {QR} factorization},
year = {2011},
issue_date = {November 2011},
publisher = {Association for Computing Machinery},
address = {New York, NY, USA},
volume = {38},
number = {1},
issn = {0098-3500},
doi = {10.1145/2049662.2049670},
journal = {ACM Trans. Math. Softw.},
articleno = {8},
pagetotal = {22}
}

@book{ryan2010anatomy,
  title={Anatomy for Diagnostic Imaging},
  author={Ryan, Stephanie and McNicholas, Michelle and Eustace, Stephen J},
  edition={3},
  year={2011},
  publisher={Saunders/Elsevier},
  address = {Philadelphia, PA},
  isbn = {978-0-7020-2971-4}
}

@article{wilson2001respiratory,
  title={Respiratory effects of the external and internal intercostal muscles in humans},
  author={Wilson, Theodore A and Legrand, Alexandre and Gevenois, Pierre-Alain and De Troyer, Andr{\'e}},
  journal={J. Physiol.},
  volume={530},
  number={2},
  pages={319--330},
  year={2001},
  publisher={Wiley Online Library}
}

@article{BastirThoraxMotion2017,
author = {Bastir, Markus and García-Martínez, Daniel and Torres-Tamayo, Nicole and Sanchis-Gimeno, Juan Alberto and O'Higgins, Paul and Utrilla, Cristina and Torres Sánchez, Isabel and García Río, Francisco},
title = {In Vivo {3D} Analysis of Thoracic Kinematics: Changes in Size and Shape During Breathing and Their Implications for Respiratory Function in Recent Humans and Fossil Hominins},
fjournal = {The Anatomical Record},
journal = {Anat. Rec.},
volume = {300},
number = {2},
pages = {255--264},
doi = {10.1002/ar.23503},
year = {2017}
}

@book{guyton2006textbook,
  title={Textbook of medical physiology},
  author={Guyton, Arthur C and Hall, John E},
  edition = {11th},
  year={2006},
  publisher={W.B. Saunders Co Ltd/Elsevier},
  address = {Philadelphia, PA}
}

@article{Pato11,
	Author = {Pato, M. P. and Santos, N. J. and Areias, P. and Pires, E. B. and {de Carvalho}, M. and Pinto, S. and Lopes, D. S.},
	Journal = {Comput. Methods Biomech. Biomed. Engin.},
	Number = 6,
	Pages = {505--513},
	Title = {Finite element studies of the mechanical behaviour of the diaphragm in normal and pathological cases},
	Volume = 14,
	Year = 2011}

@article{CheBor22,
title = {A theoretical framework for mechanics of diaphragm},
journal = {J. Biomech.},
volume = {138},
eid = {111090},
year = {2022},
issn = {0021-9290},
doi = {10.1016/j.jbiomech.2022.111090},
author = {Chen, Yichao  and Aladin M. Boriek},
}

@Article{Zhang16,
  author = 	 {Zhang, Guangzhi and Chen, Xian and Ohgi, Junji and  Miura, Toshiro and  Nakamoto, Akira and  Matsumura, Chikanori and  Sugiura, Seiryo and Hisada, Toshiaki},
  title = 	 {Biomechanical simulation of thorax deformation using finite element approach},
  journal = 	 {Biomed. Eng. Online},
  year = 	 2016,
  eid =          18,
  volume = 	 15,
  pagetotal =    18,
  doi =          {10.1186/s12938-016-0132-y}}

@PhdThesis{Coelho18,
  author = 	 {Coelho, Brett},
  title = 	 {A Composite Material-based Computational Model for Diaphragm Muscle Biomechanical Simulation},
  school = 	 {The University of Western Ontario},
  year = 	 2018,
  type = 	 {{ME}ng thesis},
  address = 	 {London, ON, Canada},
  url = {https://hdl.handle.net/20.500.14721/31312}}

@ARTICLE{Ladjal25,
  author={Ladjal, Hamid and Beuve, Michael and Shariat, Behzad},
  journal={IEEE Trans. Med. Robot. Bionics}, 
  title={Patient-Specific Biomechanical Diaphragm-Ribs Respiratory Motion Model for Radiation Therapy}, 
  year={2025},
  volume={7},
  number={2},
  pages={802--813},
  doi={10.1109/TMRB.2025.3560383}}

@article {BabMel97,
    AUTHOR = {Babu{\v{s}}ka, I. and Melenk, J. M.},
     TITLE = {The partition of unity method},
   JOURNAL = {Internat. J. Numer. Methods Engrg.},
  FJOURNAL = {International Journal for Numerical Methods in Engineering},
    VOLUME = {40},
      YEAR = {1997},
    NUMBER = {4},
     PAGES = {727--758},
      ISSN = {0029-5981},
     CODEN = {IJNMBH},
   MRCLASS = {73V05},
  MRNUMBER = {1429534 (97j:73071)},
MRREVIEWER = {Carsten Carstensen},
DOI = {10.1002/(SICI)1097-0207(19970228)40:4<727::AID-NME86>3.3.CO;2-E}}

@incollection {Wend02,
    AUTHOR = {Wendland, Holger},
     TITLE = {Fast evaluation of radial basis functions: methods based on
              partition of unity},
 BOOKTITLE = {Approximation theory, {X} ({S}t. {L}ouis, {MO}, 2001)},
    SERIES = {Innov. Appl. Math.},
     PAGES = {473--483},
 PUBLISHER = {Vanderbilt Univ. Press},
   ADDRESS = {Nashville, TN},
      YEAR = {2002},
   MRCLASS = {41A30 (65D10)},
  MRNUMBER = {1924902 (2003f:41033)},
MRREVIEWER = {Vitaly E. Ma{\u\i}orov},
}

@article{LaMo02,
title = {Radial basis functions for the multivariate interpolation of large scattered data sets},
journal = {J. Comput. Appl. Math.},
volume = {140},
number = {1},
pages = {521--536},
year = {2002},
issn = {0377-0427},
doi = {10.1016/S0377-0427(01)00485-X},
author = {Damiana Lazzaro and Laura B. Montefusco},
}

@article{ShcheLa16,
title = {Radial basis function partition of unity methods for pricing vanilla basket options},
journal = {Comput. Math. Appl.},
volume = {71},
number = {1},
pages = {185--200},
year = {2016},
issn = {0898-1221},
doi = {10.1016/j.camwa.2015.11.007},
author = {Victor Shcherbakov and Elisabeth Larsson},
}

@article{AhlShche17,
title = {A meshfree approach to non-Newtonian free surface ice flow: Application to the {H}aut {G}lacier d'{A}rolla},
journal = {J. Comput. Phys.},
volume = {330},
pages = {633--649},
year = {2017},
issn = {0021-9991},
doi = {10.1016/j.jcp.2016.10.045},
author = {Josefin Ahlkrona and Victor Shcherbakov},
}

@article{Mirzaei21,
author = {Mirzaei, Davoud},
title = {The Direct Radial Basis Function Partition of Unity ({D-RBF-PU}) Method for Solving {PDE}s},
journal = {SIAM J. Sci. Comp.},
volume = {43},
number = {1},
pages = {A54--A83},
year = {2021},
doi = {10.1137/19M128911X},    
}

@incollection {LaxMilgramTheorem,
    AUTHOR = {Lax, P. D. and Milgram, A. N.},
     TITLE = {Parabolic equations},
 BOOKTITLE = {Contributions to the theory of partial differential equations},
    SERIES = {Ann. of Math. Stud.},
    VOLUME = {33},
     PAGES = {167--190},
 PUBLISHER = {Princeton Univ. Press, Princeton, NJ},
      YEAR = {1954},
   MRCLASS = {35.0X},
  MRNUMBER = {67317},
MRREVIEWER = {L.\ G\aa rding},
}

@article{Goda2026_HaltonQuasiUniformity,
title = {Disproving the quasi-uniformity of the {H}alton sequences and of some {H}alton-type sequences},
journal = {J. Complexity},
volume = {95},
eid = {102047},
year = {2026},
issn = {0885-064X},
doi = {10.1016/j.jco.2026.102047},
author = {Takashi Goda and Roswitha Hofer and Kosuke Suzuki}
}

@MastersThesis{cao2025using,
  author = 	 {Ruihang Cao},
  title = 	 {Using the {RBF}-{PU} method to solve a mixed formulation of a linear elasticity problem},
  school = 	 {Dept.~of Information Technology, Uppsala Univ.},
  year = 	 2025,
  type = 	 {{M.S}c.\ thesis, {IT}~{mTBV}~25~007},
  address = 	 {Uppsala, Sweden}}

@MastersThesis{yangmei2025investigation,
  author = 	 {Yangmei Lin},
  title = 	 {An Investigation of the Viability of Quadrature-Based Localized Galerkin-{RBF} Methods for {PDE} Problems},
  school = 	 {Dept.~of Information Technology, Uppsala Univ.},
  year = 	 2025,
  type = 	 {{M.S}c.\ thesis, {IT}~{mTBV}~25~014},
  address = 	 {Uppsala, Sweden}}

@article{BrunoDavydov2025,
title = {Meshless moment-free quadrature formulas arising from numerical differentiation},
fjournal = {Computer Methods in Applied Mechanics and Engineering},
journal = {Comput. Methods Appl. Mech. Engrg.},
volume = {445},
eid = {118199},
year = {2025},
issn = {0045-7825},
doi = {10.1016/j.cma.2025.118199},
author = {Oleg Davydov and Bruno {Degli Esposti}}
}

@article{GiordanoFactorialBounds2001,
title = {Inequalities and monotonicity properties for the gamma function},
journal = {J. Comput. Appl. Math.},
volume = {133},
number = {1},
pages = {387--396},
year = {2001},
issn = {0377-0427},
doi = {10.1016/S0377-0427(00)00659-2},
author = {C. Giordano and A. Laforgia}
}

\end{document}